\documentclass{book}

\makeindex

\usepackage[mems, lang=french]{ems-book} 

\newcount\Cor\Cor=0
\def\newcor{\global\advance\Cor by 1
\par\medskip\noindent \Err(\romannumeral\Cor)~ -- ~}
\newcommand{\Err}{Err~}

\usepackage{mathrsfs,euscript}
\usepackage[all]{xy}
\usepackage{bbm}
\newtheorem{theorem}{\textsc{Th\'eor\`eme}}[section]
\newtheorem{proposition}[theorem]{\textsc{Proposition}}
\newtheorem{lemma}[theorem]{\textsc{Lemme}}
\newtheorem{corollary}[theorem]{\textsc{Corollaire}}
\newtheorem{remark}[theorem]{\textsc{Remarque}}
\newtheorem{convention}[theorem]{\textsc{Convention}}
\newtheorem{definition}[theorem]{\textsc{D\'efinition}}
\newtheorem{hypoths}[theorem]{\textsc{Hypoth\`eses}}

\def\add#1{{\color{magenta}#1}}

\catcode`\Ž=13
\catcode`\=13
\catcode`\ˆ=13
\catcode`\=13
\catcode`\=13
\catcode`\‰=13
\catcode`\™=13
\catcode`\"=13
\catcode`\•=13
\catcode`\ž=13
\catcode`\=13
\catcode`\'=13
\catcode`\Ë=13
\newcommand{Ž}{\'e}
\newcommand{}{\`e}
\newcommand{ˆ}{\`a}
\newcommand{}{\`u}
\newcommand{}{\^e}
\newcommand{‰}{\^a}
\newcommand{™}{\^o}
\newcommand{"}{\^\i}
\newcommand{•}{\"\i}
\newcommand{ž}{\^u}
\newcommand{}{\c c}
\newcommand{'}{\"e}
\newcommand{Ë}{\`A}
\catcode`\:=13
\newcommand{:}{~\string:}
\catcode`\;=13
\newcommand{;}{~\string;}
\catcode`\!=13
\newcommand{!}{~\string!}

\newcommand{\A}{A}
\newcommand{\G}{G}

\newcommand{\Y}{Y}

\newcommand{\AM}{\mathbb{A}}
\newcommand{\CM}{\mathbb C}
\newcommand{\FM}{\mathbb F}
\newcommand{\GM}{\mathbb G}
\newcommand{\NM}{\mathbb N}
\newcommand{\RM}{\mathbb R}
\newcommand{\QM}{\mathbb Q}
\newcommand{\UM}{\mathbb U}
\newcommand{\ZM}{\mathbb Z}

\newcommand{\adef}{\AM_F}

\newcommand{\bsbbc}{\mathbbm{c}}

\newcommand{\bfHom}{\mathrm{Hom}}
\newcommand{\vol}{\mathrm{vol}}

\newcommand{\brT}[1]{[T]_{#1}}
\newcommand{\brTo}[1]{[T_0]_{#1}}
\newcommand{\brTX}[1]{[T-X]_{#1}}
\newcommand{\ES}[1]{\EuScript{#1}}

\newcommand{\wt}[1]{\widetilde{#1}}
\newcommand{\bs}[1]{\boldsymbol{#1}}

\renewcommand{\AA}{\mathfrak A}
\newcommand{\BB}{\mathfrak B}
\newcommand{\EE}{\mathfrak E}
\newcommand{\FF}{\mathfrak F}
\newcommand{\HH}{\mathfrak H}
\newcommand{\LL}{\mathfrak L}
\newcommand{\NN}{\mathfrak N}
\newcommand{\OO}{\mathfrak O}
\newcommand{\UU}{\mathfrak U}
\newcommand{\VV}{\mathfrak V}
\newcommand{\TT}{\mathfrak T}
\newcommand{\XX}{\mathfrak X}
\newcommand{\YY}{\mathfrak Y}
\newcommand{\ZZ}{\mathfrak Z}

\newcommand{\Siegel}{\bs{\mathfrak{S}}}

\newcommand{\ag}{\mathfrak{a}}
\newcommand{\agp}{\ag_{P}}
\newcommand{\cc}{\mathfrak c}
\newcommand{\uu}{\mathfrak u}
\newcommand{\oo}{\mathfrak o}
\newcommand{\pp}{\mathfrak p}

\newcommand{\bfA}{{\mathbf A}}
\newcommand{\bfC}{{\mathbf C}}
\newcommand{\bfD}{{\mathbf D}}
\newcommand{\bfH}{{\mathbf H}}
\newcommand{\bfM}{{\mathbf M}}
\newcommand{\bfW}{{\mathbf W}}

\newcommand{\Rho}{\bs\rho}
\newcommand{\tRho}{\wt{\Rho}}
\newcommand{\bsmu}{\bs{\mu}}
\newcommand{\bsPi}{\bs{\Pi}}
\newcommand{\bsigma}{\bs{\sigma}}
\newcommand{\tbsrho}{\wt{\bs{\rho}}}

\newcommand{\EV}{\bs E}
\newcommand{\bsA}{\bs{A}}
\newcommand{\bsJ}{\bs{J}}
\newcommand{\bsK}{\bs{K}}
\newcommand{\bsS}{\bs{S}}
\newcommand{\bsX}{\bs{X}}
\newcommand{\ovbsX}{\overline{\bsX}}
\newcommand{\bsY}{\bs{Y}}
\renewcommand{\bsc}{\bs{c}}
\newcommand{\bsd}{\bs{d}}
\newcommand{\bscMF}{\bsc_{M,F}}
\newcommand{\bscM}{\bsc_{M}}
\newcommand{\bsO}{\bs{\Omega}}
\newcommand{\bso}{\bs{\omega}}
\newcommand{\brabsO}{[\bsO]}
\newcommand{\brabso}{[\bso]}

\newcommand{\bESP}{\bs{\ESP}}
\newcommand{\bPI}{\bs{\Pi}}

\newcommand{\fun}{f}
\newcommand{\tdf}{f}
\newcommand{\tG}{{\widetilde G}}

\newcommand{\tL}{{\widetilde L}}
\newcommand{\tM}{{\widetilde M}}
\newcommand{\tP}{{\widetilde P}}
\newcommand{\tQ}{{\widetilde Q}}
\newcommand{\tR}{{\widetilde R}}
\newcommand{\tS}{{\widetilde S}}
\newcommand{\tZ}{{\widetilde Z}}
\newcommand{\tu}{{\widetilde u}}

\newcommand{\ESA}{\ES A}
\newcommand{\ESC}{\ES C}
\newcommand{\ESB}{\ES B}
\newcommand{\ESD}{\ES D}
\newcommand{\ESE}{{\ES E}}
\newcommand{\ESF}{\ES F}

\newcommand{\ESH}{\ES H}

\newcommand{\ESL}{\ES L}
\newcommand{\ESM}{\ES M}

\newcommand{\ESO}{\ES O}
\newcommand{\ESP}{\ES P}
\newcommand{\ESR}{\ES R}
\newcommand{\ESS}{\ES S}

\newcommand{\Dcal}{\mathcal{D}}

\newcommand{\hag}{\wh{\ag}}
\newcommand{\hagp}{\wh{\ag}_P}

\newcommand{\bsl}{\backslash}
\newcommand{\dd}{\mskip 2mu \mathrm{d}}

\newcommand{\cM}{\mid \wh\bsbbc_M\mid }
\newcommand{\Stab}{\mathrm{Stab}_M(\sigma)}
\newcommand{\stab}{\vert \Stab\vert}
\newcommand{\unstab} {\frac{\cM}{\stab}}
\newcommand{\frGM}{\frac{1}{w^G(M)}}
\newcommand{\frPpM}{\frac{1}{w^{\nP}(M)}}

\newcommand{\gammaMF}{\gamma_{M,F}}

\newcommand{\Jres}{\mathfrak J}

\newcommand{\bJres}{\bs{\Jres}}

\newcommand{\ttM}{{_tM}}
\newcommand{\ttS}{{_tS}}
\newcommand{\mf}{m}

\newcommand{\ZTX}{H_Z^{T-X}} 
 
\newcommand{\ZTSs}{H^T_{Z,S''}}

\newcommand{\nP}{P} 
\newcommand{\Pp}{P'}
\newcommand{\Qo}{{Q_0}}

\newcommand{\Kappa}[1]{\kappa^{#1}}
\newcommand{\kopa}{\chi}
\newcommand{\cMsig}{\wh c_M(\sigma)}
\newcommand{\cMSsig}{\wh{c}_{M_S}(\sigma)}
\newcommand{\bsESEsigma}{\bs{\ESE}(\sigma)}

\newcommand{\disc}{\mathrm{disc}}
\newcommand{\cusp}{\mathrm{cusp}}
\newcommand{\spec}{\mathrm{spec}}
\newcommand{\geom}{\mathrm{g\acute{e}om}}
\newcommand{\trace}{\mathrm{trace}}
\newcommand{\spur}{\mathfrak{Sp}}
\newcommand{\spurs}{\spur_\sigma}
\newcommand{\Autom}{\bs{\mathcal{A}}}
\newcommand{\Automd}{\Autom}
\newcommand{\Base}{\bs\Psi}

\newcommand{\fetu}{\ESL(\sigma,\tu)}
\newcommand{\Fetu}{\bs{\mathfrak{l}}(\sigma,\tu)}

\newcommand{\st}{{\mathrm{st}}}
\newcommand{\prim}{{\mathrm{prim}}}

\newcommand{\bmMtunu}{\bsmu_{M,\tu}(\nu)}
\newcommand{\bmMtunuo}{\bsmu_{M,\tu}(\nu,0)}
\newcommand{\azx}{a_{\tZ,\mu},}
\newcommand{\dmu}{\dot\mu}

\newcommand{\oD}{\Delta}

\newcommand{\V}{V}
\newcommand{\wh}{\widehat}
\newcommand{\mun}{^{-1}}

\newcommand{\com}[1]{\quad\hbox{#1}\quad}

\newcommand{\comm}[1]{\qquad\hbox{#1}\qquad}

\newcommand{\vtt}{\vartheta}

\newcommand{\bydef}{\buildrel \mathrm{d\acute{e}f}\over{=}}

\newcommand{\Lbra}{[\mskip -2mu[}
\newcommand{\Rbra}{]\mskip -2mu]}

\renewcommand{\ni}{\noindent}
\newcommand{\pni}{\par\noindent}
\newcommand{\ptf}{\mskip 2mu .}
\newcommand{\vg}{\mskip 2mu ,}
\newcommand{\vgq}{\mskip 2mu ,\quad}

\title{La formule des traces tordue pour un corps global de caract\'eristique $p>0$}

\author{Jean-Pierre Labesse \& Bertrand Lemaire\\
Institut de Math\'ematique de Marseille, CNRS, UMR 7373\\
Aix-Marseille Universit\'e, France}

\begin{document}

\frontmatter

\maketitle


\begin{abstract}
Dans ce travail nous adaptons
au cas d'un corps global $F$ de caract\'eristique positive, c'est-\`a-dire un corps de fonctions 
sur un corps fini $\FM_q$,
les r\'esultats prouv\'es pour un corps de nombres dans le livre de Labesse-Waldspurger:
 \textit{La formule des traces tordue  d'aprs le Friday Morning Seminar}.
En d'autres termes, nous \'etablissons la formule des traces 
pour un $G$-espace tordu $\wt G$ o $G$ 
est un groupe r\'eductif connexe d\'efini sur $F$.
C'est une premire Žtape vers la forme stabilisŽe de la formule des traces tordue,
nŽcessaire pour la plupart des applications, qui elle est l'objet de travaux en cours.
\end{abstract}



\classification[20G35, 22E55]{11F72}

\begin{ack}
Nous remercions J.-L. Waldspurger pour d'utiles critiques et remarques
sur une version pr\'eliminaire de ce texte.
\end{ack}





\tableofcontents


\chapter*{Introduction}

 \section*{Un peu d'histoire} 
 
 \subsection*{DŽcomposition spectrale et formule des traces}
 
La formule des traces, pour l'action ˆ droite de l'algbre des fonctions sphŽriques sur l'espace localement 
symŽtrique $\Gamma\bsl \mathfrak h$ quotient du demi-plan de PoincarŽ 
$\mathfrak h$  par le groupe modulaire $\Gamma=\mathrm{SL}(2,\ZM)$,
est devenue depuis l'exposŽ de Selberg \cite{Se} ˆ Bombay en 1956
un des outils fondamentaux de la th\'eorie des formes modulaires. 
  
  L'Žtablissement de la formule des traces suppose connu la description de la dŽcomposition spectrale 
de l'espace de Hilbert $L^2(\Gamma\backslash \mathfrak h)$ au moyen du prolongement mŽromorphe
des SŽries d'Eisenstein. Ds 1964, le cas gŽnŽral de la d\'ecomposition spectrale de 
$L^2(\Gamma\backslash G(\RM))$, o 
 $\Gamma$ est un sous-groupe arithmŽtique de $G(\QM)$ avec $G$ rŽductif connexe dŽfini sur $\QM$,
 Žtait obtenu par Langlands (ce travail ne fut publi\'e que 12 ans plus tard \cite{Lan}).
 
 L'\'etude des op\'erateurs de Hecke am\`ene \`a travailler avec des limites projectives de revtements 
d\'efinis par des sous-groupes de congruence, ce qui revient 
essentiellement \`a consid\'erer le quotient $G(\QM)\backslash G(\AM_\QM)$ o $\AM_\QM$ est l'anneau des adles de $\QM$. 
Le cas des groupes arithmŽtiques qui ne sont pas de congruence
relve plus de la gŽomŽtrie diffŽrentielle que de l'arithmŽtique. Tout ce qui suit est dans le cadre adlique.
 
 La thŽorie gŽnŽrale des SŽries d'Eisenstein ouvrait la voie 
 ˆ la g\'en\'eralisation de la formule des traces  ˆ tous les groupes rŽductifs
mais aussi ˆ la dŽcouverte, par Langlands, des fonctions $L$ attachŽes aux reprŽsentations automorphes
et ˆ ses conjectures sur la fonctorialitŽ.

 \subsection*{La formule des traces d'Arthur-Selberg}
 
Soit $G$ un groupe r\'eductif connexe d\'efini sur un corps de nombres $F$. 
On note $\adef$ l'anneau des adles de $F$.
La formule des traces  -- dite d'Arthur-Selberg --
pour le quotient $G(F)\backslash G(\adef)$ 
est une identit\'e entre deux expressions pour une {\og trace renormalisŽe \fg} 
de l'opŽrateur de convolution dŽfini par une fonction $f\in C_c^\infty(\G(\adef))$
dans $L^2(\G(F)\backslash \G(\adef))$.
La {\og trace \fg} doit tre {\og renormalisŽe \fg} car l'opŽrateur de convolution n'est ˆ trace que dans le spectre discret et
une correction, via des troncatures, est nŽcessaire pour tenir compte du spectre continu.
On aura d'une part un d\'eveloppement g\'eom\'etrique indexŽ par les classes de conjugaison dans $G(F)$
et d'autre part un d\'eveloppement spectral:
\begin{equation*} J_{\mathrm {g\acute{e}om}}(f)=J_{\mathrm {spec}}(f)\ptf
\end{equation*}
L'Žtablissement de la formule des traces 
peut tre divis\'e en quatre \'etapes principales: Arthur obtient successivement
\begin{enumerate}[(1)]
\item la forme grossi\`ere (\og coarse form\fg) \cite{A1,A2}; 
\item la forme fine (\og fine form\fg) \cite{A3,A4,A5,A6};
\item la forme invariante (\og invariant form\fg) \cite{A7,A8}; 
\item la forme stabilis\'ee (\og stable form\fg) \cite{A10,A11,A12}. 
\end{enumerate}
Chacune des \'etapes fournit des d\'eveloppements g\'eom\'etriques et spectraux
de plus en plus pr\'ecis. C'est bien sžr le d\'eveloppement ultime, la formule des traces stabilis\'ee, 
qui permet le plus d'applications. 

\subsection*{La variante tordue}

La variante tordue de la formule des traces est apparue  
dans le travail de Saito et Shintani, immŽdiatement gŽnŽralisŽ par
Langlands, pour Žtudier le changement de base cyclique pour $\mathrm{GL}(2)$ \cite{Lan2}. 
On sait le r™le jouŽ par ce travail dans la preuve du ThŽorme de Fermat.
Le cadre g\'en\'eral de la \textit{formule des traces tordue} est le suivant: on consid\`ere un $G$-espace tordu 
$\wt{G}$ d\'efini sur $F$ 
et de plus on tord les reprŽsentations par un caract\`ere unitaire $\omega$ de $G(\adef)$ 
trivial sur $G(F)$. 
Dans la majoritŽ des applications   $\wt{G}=G\rtimes\theta$ o $\theta$ 
est un automorphisme d'ordre fini de $G$ et $\omega$ est trivial; par exemple pour le changement
de base cyclique $\wt{G}(F)=GL(n,E)\rtimes\theta$ o $\theta$ est un gŽnŽrateur du groupe
de Galois d'une extension finie cyclique $E/F$.
Toutefois le cas de la torsion des reprŽsentations
par  un caractre $\omega$ non trivial mais sans torsion sur le groupe ($\theta$ trivial)
intervient aussi naturellement. 

{La formule des traces tordue pour les corps de nombres}
a fait l'objet du \textit{Friday Morning Seminar} de Princeton en 1983-1984. 
On y donnait le dŽveloppement spectral fin mais seulement le dŽveloppement
gŽomŽtrique grossier. Les notes de ce s\'eminaire ont \'et\'e reprises et compl\'et\'ees dans \cite{LW}. 
La forme fine du dŽveloppement gŽomŽtrique a ŽtŽ obtenue plus tard par Arthur
en combinant les rŽsultats du Friday Morning Seminar avec
des rŽsultats d'analyse harmonique locale. 

\subsection*{Lemme fondamental et stabilisation}

La formule des traces, sous sa forme invariante, est une ŽgalitŽ entre deux
sommes de distributions invariantes par conjugaison sous $G(F_S)$ o $S$ est une ensemble fini
de places fixŽ, mais arbitrairement grand. 
Or, pour Žtablir certains cas de fonctorialitŽ on est amenŽ ˆ comparer des distributions 
qui ne sont invariantes que
sous une forme plus grossire de la conjugaison appelŽe \textit{conjugaison stable}.
La stabilisation consiste en la rŽcriture, au moyen de \textit{transferts endoscopiques}, de
chaque terme de la formule des traces invariante comme une somme de
distributions stablement invariantes sur des groupes auxiliaires appelŽs \textit{groupes endoscopiques}.
On en dŽduit des transferts de reprŽsentations automorphes entre groupes diffŽrents.
Les premiers exemples sont le transfert de Jacquet-Langlands \cite[Chapter 16]{JL},
la formule des traces stable pour $\mathrm{SL}(2)$ \cite{L, LL}
et le changement de base cyclique pour $\mathrm{GL}(2)$ \cite{Lan2}.
La stabilisation suppose en particulier
connu le \textit{Lemme fondamental} aux places en dehors de $S$. 
Aprs les travaux de nombreux auteurs --
parmi lesquels il faut citer (par ordre alphabŽtique)
Arthur, Kottwitz, Langlands,  Laumon, Ng™ B\'au Ch‰u, Shelstad et Waldspurger --
la stabilisation de la formule des traces tordue pour les corps de nombres a \'et\'e achev\'ee  
par M\oe glin et Waldspurger \cite{MW2}. Pour une bibliographie complte sur ce sujet
nous renvoyons aux deux tomes de cet ouvrage.

C'est au moyen de la  \textit{formule des traces tordue stabilis\'ee}
qu'Arthur a pu d\'ecrire dans \cite{A13} les repr\'esentations automorphes discr\`etes des groupes classiques 
par comparaison avec celles du groupe lin\'eaire \textit{lorsque $F$ est un corps de nombres}. 
De fait, les groupes classiques quasi-d\'eploy\'es se r\'ealisent comme des \textit{groupes endoscopiques} 
de l'espace tordu $\wt{G}=\mathrm{GL}(n)\rtimes\theta$ o
$\theta$ est l'automorphisme  $g\mapsto  {^t\!g^{-1}}$.

\section*{Le cas des corps globaux de caractŽristique $p>0$} 

\subsection*{Etat des lieux}

On s'attend  \`a ce qu'un analogue de tout ce qui prŽcde 
existe aussi lorsque $F$ est un corps global de caractŽristique $p>0$, c'est-ˆ-dire
un corps de fonctions sur un corps fini, et on espre en dŽduire 
des rŽsultats similaires ˆ ceux obtenus pour les reprŽsentations automorphes
sur les corps de nombres. 

Dans toute la suite de ce mŽmoire nous dirons {\og corps de fonctions \fg} 
pour {\og corps de fonctions sur un corps fini \fg}.
Examinons l'Žtat de la littŽrature sur ce sujet.

La d\'ecomposition spectrale pour un groupe r\'eductif connexe g\'en\'eral
a \'et\'e \'etendue au cas des corps de fonctions par Morris \cite{Mo1,Mo2}, 
en prolongement de travaux de Harder \cite{H1,H2,H3}. Le livre de 
M\oe glin-Waldspurger \cite{MW1} en redonne une preuve pour tout corps global. 

La formule des traces (sous sa forme fine)
pour le groupe $\mathrm{GL}(2)$ -- et ses formes int\'erieures -- est
ŽnoncŽe dans Jacquet-Langlands \cite[Chapter 16]{JL} pour les corps globaux 
\textit{en toute caractŽristique};
toutefois la dŽmonstration en est seulement esquissŽe.
Pour le groupe $\mathrm{GL}(n)$ 
(ou une forme int\'erieure), on dispose de plusieurs versions en caractŽristique $p>0$: 
Drinfeld pour $n=2$ \cite{D1,D2}, Laumon \cite{Lau}, Lafforgue \cite{Laf1,Laf2}. 
En adaptant la m\'ethode de Lafforgue, 
 Ng™ Dac Tu‰n a obtenu dans \cite{N} le d\'eveloppement spectral fin de la formule des traces 
non tordue  pour les groupes r\'eductifs connexes d\'eploy\'es. 

On ne dispose donc, dans la littŽrature, que de quelques cas particuliers du 
d\'eveloppement spectral fin en caractŽristique $p>0$. 
Et pour le dŽveloppement gŽomŽtrique, en dehors de $\mathrm{GL}(n)$, 
il n'y a que tr\`es peu de r\'esultats.

 \subsection*{Objet du prŽsent M\'emoire et perspectives}
 
Notre projet est d'Žtablir une formule des traces tordue stabilis\'ee 
pour les corps globaux sans restriction sur la caractŽristique.
Ce mŽmoire est une premire Žtape o 
nous suivons pas ˆ pas la preuve donnŽe dans \cite{LW}, 
en l'adaptant au cas d'un corps global de caract\'eristique $p>0$. 
Nous obtenons le d\'eveloppement spectral fin mais, comme dans \cite{LW},
nous n'obtenons que le d\'eveloppement grossier du c™t\'e g\'eom\'etrique. 

Pour le dŽveloppement spectral fin, la caractŽristique $p$ ne semble jouer pratiquement 
aucun r™le et les formules que nous obtenons pour les corps de fonctions sont 
essentiellement identiques ˆ celles pour les corps de nombres. 

Il n'en est pas de mme du c™tŽ gŽomŽtrique. 
Si $p$ est grand par rapport au rang de $G$, rien de vraiment nouveau n'appara"t,
mais le traitement des petites caractŽristiques nŽcessite des idŽes nouvelles.

Pour le dŽveloppement gŽomŽtrique grossier, donnŽ ici, 
la principale diffŽrence avec le cas des corps de nombres est que
l'on doit remplacer la notion d'ŽlŽment semi-simple elliptique rŽgulier
pas celle plus gŽnŽrale d'ŽlŽment primitif englobant ainsi, par exemple, des ŽlŽments 
qui deviennent unipotents sur une extension insŽparable
mais ne sont contenus dans le radical unipotent d'aucun 
sous-groupe parabolique propre dŽfini sur $F$.  De tels ŽlŽments, qui
n'apparaissent  que pour des caractŽristiques assez petites,
ne gŽnrent toutefois pas de difficultŽs sŽrieuses ˆ ce stade.

Par contre des phŽnomnes nouveaux compliquent considŽrablement 
la preuve et la structure mme du dŽveloppement gŽomŽtrique fin si $p$ est petit. 
Par exemple, pour $G=\mathrm{SL}(2)$, $p=2$ et $v$ une 
place de $F$, le nombre de $G(F_v)$-orbites unipotentes est infini. 
La forme fine de la contribution unipotente aura donc une expression de nature fort diff\'erente de celle 
donnŽe par Arthur pour les corps de nombres. 
L'ins\'eparabilit\'e perturbe aussi violemment la descente centrale \`a la Harish-Chandra.
On espre qu'il sera possible d'en Žtablir une variante
permettant de ramener le d\'eveloppement g\'eom\'etrique fin \`a celui 
de la contribution unipotente. 

L'extension au cas des corps de fonctions des rŽsultats d'Arthur \cite{A5,  A6},
nŽcessaires pour le dŽveloppement gŽomŽtrique fin, est 
l'objet d'un travail en cours du second auteur.

Il restera encore ˆ stabiliser la formule des traces mais nous sommes loin
d'avoir une vue claire du travail ˆ faire. Dans la littŽrature, seuls le transfert 
de Jacquet-Langlands \cite{JL} et la stabilisation pour $\mathrm{SL}(2)$  \cite{L} ont ŽtŽ 
traitŽs en toute caractŽristique.
 
\section*{Structure du texte}

\subsection*{DiffŽrences techniques}

Nos r\'ef\'erences principales seront \cite{LW} et \cite{W} 
et on supposera que le lecteur a ces deux textes sous la main. 
L'organisation de cet article suit celle de \cite{LW} et
on se contentera souvent de citer sans d\'emonstration les r\'esultats de cet ouvrage
 pour peu qu'ils passent sans autre forme de proc\`es \`a la caract\'eristique positive.
 
Toutefois, lorsqu'on essaie de remplacer dans \cite{LW} le corps de nombres par un corps de fonctions, 
on rencontre imm\'ediatement  une diff\'erence technique, 
analogue \`a celle qui existe entre les cas archim\'ediens et non-archim\'ediens
dans les travaux d'Arthur \cite{A9} et Waldspurger \cite{W} sur la formule des traces locale. 
En effet, pour un corps global $F$ quelconque, 
on dispose pour chaque sous-groupe parabolique $P$ de $G$ dŽfini sur $F$ d'un espace vectoriel r\'eel $\ag_P$
et d'un morphisme 
\begin{equation*}\bfH_P: P(\adef)\rightarrow\ag_P\ptf\end{equation*}
On note $\ESA_P$ son image
et  $\ESB_P$ l'image de $A_P(\adef)$ o $A_P$ est le tore d\'eploy\'e maximal dans le centre de $P/U_P$. 
On a les inclusions
\begin{equation*}\ESB_P\subset\ESA_P\subset\ag_P\ptf\end{equation*}
Pour les corps de nombres, ces trois groupes sont \'egaux et
on dispose d'un rel\`evement canonique $\AA_P$ de 
$\ESA_P$ dans $A_P(\adef)$. Il est usuel de
se limiter \`a traiter l'espace des formes automorphes pour $P$ qui sont invariantes
par $\AA_P$; cela ne restreint pas la g\'en\'eralit\'e puisque l'on peut par torsion
par un caract\`ere automorphe de $P(\adef)$ passer \`a un caract\`ere quelconque sur $\AA_P$.
Mais, pour les corps de fonctions, les morphismes $\bfH_P$ ne sont plus surjectifs: ${\ESA_P}$
est un r\'eseau de $\ag_P$ et en g\'en\'eral $\ESA_P\ne\ESB_P$ de sorte
qu'un rel\`evement central de $\ESA_P$ n'existe pas.
Ceci complique donc certains arguments et a des cons\'equences sur la pr\'esentation des r\'esultats.
On prendra garde en particulier ˆ ce que les espaces $\bsX_G$ et $\bsY_{\!Q}$ qui interviennent 
ici jouent un r™le analogue mais ne sont pas identiques
aux espaces ${\mathbf X}_G$ et ${\mathbf Y}_{\mskip -2mu Q}$ de \cite{LW}.

L'autre diffŽrence technique -- venant de ce que des rŽseaux ont remplacŽ les espaces vectoriels 
pour la combinatoire -- est que les polyn™mes en la variable de troncature $T$, qui jouent un r™le essentiel
dans l'Žtablissement de la formule des traces pour les corps de nombres, sont ici remplacŽs
par des fonctions de type {\og Polyn™mes-exponentielles \fg}. Des passages ˆ la limite sont nŽcessaires
pour les Žliminer et, par exemple, obtenir le dŽveloppement spectral fin. Ceci complique encore les preuves.
Par contre la compacitŽ du dual de Pontryagin des rŽseaux simplifie le contr™le des convergences d'intŽgrales.
 
 Comme dans \cite{LW} le texte est divis\'e en quatre parties. Un appendice corrige
 des bŽvues de \cite{LW}.

 \subsection*{Partie I. G\'eom\'etrie et combinatoire}
Dans le chapitre 1, apr\`es le rappel des notations usuelles,
on adapte \`a notre cadre le calcul des transform\'ees de Laplace (ou anti-Laplace) 
des fonctions caract\'eristiques de polytopes \cite[1.9]{LW} 
ainsi que les r\'esultats de \cite[1.10]{LW} sur les $(G,M)$-familles. 
Les int\'egrales sur les polytopes qui apparaissent dans \cite{LW} 
doivent ici tre remplac\'ees par des sommes sur l'intersection de ces polytopes
avec des r\'eseaux. La combinatoire des polytopes est certes identique
mais la manipulation des intersections est plus d\'elicate; on doit en particulier tenir compte 
des groupes finis $\bsbbc_P=\ESB_P \backslash \ESA_P$. 
Leur traitement requiert l'utilisation
de divers outils techniques emprunt\'es \`a Arthur \cite{A9} et Waldspurger \cite{W}. 
Pour les corps de nombres, les $(G,M)$-familles fournissent des polyn™mes en la variable
de troncature $T$ qui permettent de contr™ler le comportement asymptotique des
divers termes de la formule des traces. Ici les $(G,M)$-familles fournissent, pour les $T$ {\og rationnels \fg},
des expressions du type PolExp c'est-\`a-dire des combinaisons lin\'eaires de polyn™mes et
d'exponentielles. Par un passage \`a la limite on d\'efinit un polyn™me qui
est l'analogue des polyn™mes asymptotiques pour les corps de nombres.
Dans le chapitre 2, on g\'en\'eralise au cas tordu les relations et 
propri\'et\'es du chapitre 1. L'analogue de l'\'etude dans le cas tordu de la notion d'\'el\'ement 
semi-simple \cite[2.6]{LW} est renvoy\'ee ici au chapitre suivant.
Le chapitre 3 contient des rappels sur la th\'eorie de la r\'eduction. 
Sur un corps global de caract\'eristique $p>0$, 
cette th\'eorie est essentiellement due \`a Harder \cite{H2}. En reprenant les id\'ees de Harder 
sur la descente galoisienne, Springer \cite{Sp} a donn\'e un traitement uniforme 
(valable pour tout corps global) des principaux r\'esultats de cette th\'eorie. 
L\`a o c'est n\'ecessaire, on remplace donc par \cite{Sp} la r\'ef\'erence au livre de Borel \cite{B1} dans \cite[ch.~3]{LW}. 
Par ailleurs, on d\'efinit un ersatz de la d\'ecomposition de Jordan (inutilisable ici car
en g\'en\'eral elle n'est pas rationnelle): on remplace la notion d'\'el\'ement 
quasi semi-simple r\'egulier elliptique par celle d'\'el\'e\-ment \textit{primitif} (qui d'ailleurs appara"t en \cite[3.7]{LW}). 
Toute paire $(\tM,\delta)$ form\'ee d'un facteur de Levi $\tM$ de $\tG$ d\'efini sur $F$ 
et d'un \'el\'e\-ment primitif $\delta$ de $\tM(F)$, d\'efinit un sous-ensemble $\ESO_\oo$ de 
$\tG(F)$ stable par $G(F)$-conjugaison. Ces ensembles $\ESO_\oo$
 joueront le r™le des classes de ss-conjugaison de \cite{LW}.
 
 \subsection*{Partie II. Th\'eorie spectrale, troncatures et noyaux}
Le chapitre 4 concerne l'op\'erateur de troncature $\bs\Lambda^T$. 
Ses propri\'et\'es sont essentiellement les mmes que dans le cas des corps de nombres, 
except\'ees les propri\'et\'es de d\'ecroissance qui sont ici beaucoup plus fortes. On pose
\begin{equation*}\bsX_G=G(F)\backslash G(\adef)\quad \hbox{et}
 \quad \overline{\bsX}_G= 
A_G(\adef)G(F)\backslash G(\adef)\ptf\end{equation*}
Pour un corps de fonctions, l'op\'erateur de troncature $\bs\Lambda^T$ appliqu\'e \`a une fonction 
lisse et $K$-finie $\varphi$ sur $\bsX_G$ fournit une fonction $\bs\Lambda^T\varphi$ sur $\bsX_G$ 
\`a support d'image compacte dans $\overline{\bsX}_G$. 
De plus la d\'ecomposition \begin{equation*}\bs\Lambda^T = \bfC^T+(\bs\Lambda^T - \bfC^T)\end{equation*}
o $\bfC^T$ est la multiplication par la fonction caract\'eristique d'un domaine de Siegel tronqu\'e se simplifie ici car,
si $T$ est assez r\'egulier, on a $(\bs\Lambda^T-\bfC^T)\varphi = 0$.
Le chapitre 5 introduit les op\'erateurs d'entrelacement et les s\'eries d'Eisenstein dont les propri\'et\'es, 
en particulier leur prolongement m\'eromorphe, ont \'et\'e \'etablies par Morris \cite{Mo1,Mo2} et reprises
dans \cite[II, IV]{MW1}.
La preuve de la formule pour le produit scalaire de deux s\'eries d'Eisenstein tronqu\'ees 
reprend celle de \cite[5.4.3]{LW}. Dans le cas o les fonctions induisantes sont cuspidales, on obtient 
une formule exacte essentiellement due \`a Langlands. 
Dans le cas o les fonctions ne sont pas cuspidales, on dispose d'une formule asymptotique 
qui se d\'eduit du cas cuspidal. Cette formule asymptotique fournit une majoration uniforme
lorsque les param\`etres $\lambda$ et $\mu$ sont imaginaires purs.
En effet, ici, les espaces de param\`etres sont compacts, ce qui n'\'etait pas le cas pour les 
corps de nombres.
Dans le chapitre 6 est introduit le noyau int\'egral. Le th\'eor\`eme de factorisation de Dixmier-Malliavin, 
et la propri\'et\'e de $\Autom$-admissibilit\'e, sont trivialement vrais ici. 
On \'enonce la principale propri\'et\'e du noyau tronqu\'e (l'op\'erateur de troncature agissant sur la premi\`ere variable): sa restriction \`a 
$\Siegel^*\times\Siegel^*$ est born\'ee et \`a support compact, o $\Siegel^*$ 
 est un domaine de Siegel pour le quotient $\overline{\bsX}_G$.
Dans le chapitre 7 on rappelle la d\'ecomposition spectrale de $L^2(\bsX_G)$ 
due \`a Langlands pour les corps de nombres \cite{Lan} et Morris pour les corps de fonctions \cite{Mo1,Mo2} puis
r\'edig\'ee pour tout corps global par M\oe glin et Waldspurger \cite{MW1}. 
On remarquera que si $A_G$, le tore d\'eploy\'e maximal dans le centre de $G$, n'est pas trivial,
il n'y a pas de spectre discret dans $L^2(\bsX_G)$ et, pour parer ce fait,
il est usuel de donner la d\'ecomposition spectrale en ayant fix\'e un caract\`ere unitaire d'un sous-groupe 
co-compact de $A_G(F)\backslash A_G(\adef)$. Pour les corps de nombres on 
se limite aux fonctions sur $\bsX_G$ qui sont invariantes par le rel\`evement $\AA_G$ de $\ESA_G$
dans $A_G(\adef)$.
Comme d\'ej\`a observ\'e ci-dessus, pour les corps de fonctions, en g\'en\'eral un tel rel\`evement n'existe pas.
En l'absence d'un choix naturel
nous ne ferons aucune hypoth\`ese sur le comportement des fonctions
sur $A_G(\adef)$ et la d\'ecomposition spectrale
comprendra, comme premi\`ere \'etape, la d\'ecomposition suivant les caract\`eres 
unitaires de $A_G(F)\backslash A_G(\adef)$.

 \subsection*{Partie III. La formule des traces grossi\`ere}
Dans le chapitre 8 on donne, \`a titre d'exemple, la formule des traces dans le cas compact. 
Comme nous ne faisons aucune hypoth\`ese sur l'action du centre de $G$
l'int\'egrale du noyau sur la diagonale porte ici sur 
\begin{equation*}\bsY_{\!G}=A_\tG(\adef)G(F)\backslash G(\adef)\end{equation*}au lieu de 
$\AA_GG(F)\backslash G(\adef)$ dans \cite{LW}.
Puis on \'enonce l'identit\'e fondamentale entre les deux troncatures, pour le noyau, 
que l'on doit consid\'erer pour traiter le cas g\'en\'eral. 
Le chapitre 9 d\'ecrit le d\'eveloppement g\'eom\'etrique dit \og grossier\fg. 
Le r\'esultat principal est la convergence de l'expression 
\begin{equation*}\sum_{\oo}\int_{\bsY_G}\vert k_\oo^T(x)\vert \dd x\end{equation*}o $\oo$ 
parcourt les classes d'\'equivalence de paires primitives dans $\tG(F)$. 
Il implique en particulier que les int\'egrales orbitales des \'el\'ements primitifs sont absolument convergentes. 
Comme en \cite[9.3]{LW}, on ne donnera des formules explicites que pour les orbites primitives et pour les orbites 
quasi semi-simples. Le d\'eveloppement g\'eom\'etrique fin qui explicite les contributions unipotentes
n'est pas abord\'e ici. 
Le chapitre 10 \'etablit la convergence d'une premi\`ere forme du d\'eveloppement spectral. Il s'agit de montrer la convergence
de chaque terme d'une somme index\'ee par des sous-groupes paraboliques.
Les r\'esultats principaux sont les propositions \ref{[LW,10.1.6]} et \ref{cot\'espec} analogues
de \cite[10.1.6, 10.2.3]{LW}. 
Dans le chapitre 11, pour chaque terme des d\'eveloppements g\'eom\'etriques et spectraux,
 un \'el\'ement de l'ensemble PolExp est d\'efini ainsi que le polyn™me limite. On obtient
 la premi\`ere forme, dite \og grossi\`ere\fg, de la formule des traces. 

 \subsection*{Partie IV. Forme fine des termes spectraux}
Il reste \`a exploiter la d\'ecomposition spectrale pour obtenir le d\'eveloppement spectral fin. 
Les preuves des \'enonc\'es des chapitres 12 et 13 suivent pas \`a pas celles des chapitres 12 et 13 de \cite{LW}, 
\`a deux diff\'erences (simplificatrices) pr\`es: d'une part 
il est inutile d'introduire une fonction $B$ car ici
les param\`etres spectraux \'evoluent dans un espace compact; d'autre part (comme dit plus haut) la d\'ecomposition 
de l'op\'erateur de troncature en $\bs\Lambda^T= \bfC^T+(\bs\Lambda^T - \bfC^T)$ devient ici, 
pour $T$ assez r\'egulier, $\bs\Lambda^T=\bfC^T$.
Le chapitre 14 contient la combinatoire finale donnant lieu aux formules explicites. 
La preuve de la proposition \ref{propcombfin} est plus technique que celle de
son analogue \cite[14.1.8]{LW} 
et fournit une formule moins simple 
\`a cause d'une inversion de Fourier relative \`a un accouplement qui n'est pas parfait.
C'est un probl\`eme d\'ej\`a pr\'esent dans le cas local. On s'inspire 
du traitement de cela dans \cite{W} pour obtenir les formules finales.

\subsection*{Appendice}
On trouvera en appendice un Erratum corrigeant les lapsus et erreurs que nous avons 
pu rep\'erer dans \cite{LW}.

\mainmatter

\part{G\'eom\'etrie et combinatoire}

\chapter{Racines, convexes et $(G,M)$-familles}

\renewcommand{\adef}{\AM}

\section{Le corps $F$}\label{le corps de base}
Dans tout cet article $F$ est un corps global et, sauf mention expresse
du contraire lorsque nous faisons le parall\`ele avec le cas des corps de nombres, il est de caract\'eristique $p>0$.

Soit ${\FM}_q$ {\og le\fg} corps fini \`a $q$ \'el\'ements, pour une puissance $q$ du nombre premier $p$. 
Soit $\ES V$\index{V@$\ES V$, $\vert \ES V \vert$} une courbe projective lisse et 
g\'eom\'etriquement connexe sur ${\FM}_q$, de corps de fonctions $F$. L'ensemble $\vert\ES V\vert$ 
des points ferm\'es de $\ES V$ est en bijection avec l'ensemble des places de $F$. Pour $v\in\vert\ES V\vert$, 
on note $F_v$ le corps compl\'et\'e de $F$ en $v$, $\oo_v$ l'anneau des entiers de $F_v$, $\pp_v$ 
l'id\'eal maximal de $\oo_v$, et $\kappa_v$ le corps r\'esiduel $\oo_v/\pp_v$. 
Ce dernier est une extension finie de ${\FM}_q$, de degr\'e $\deg(v)$ appel\'e \og degr\'e de $v$\fg, 
et de cardinal $q_v= q^{\deg(v)}$. Pour $v\in \vert \ES V \vert$, on note encore $v$ la valuation sur $F_v$ 
normalis\'ee par $v(F_v^\times)=\ZM$, c'est-\`a-dire par $v(\varpi_v)=1$ pour une uniformisante 
$\varpi_v$ de $F_v$, et on note $\vert\;\vert_v$ la valeur absolue sur $F_v$ d\'efinie par
\begin{equation*}\vert x\vert_v = q_v^{-v(x)}\vgq x\in F_v\ptf\end{equation*}
Soit $\adef$\index{A@$\AM$, $\AM^1$} l'anneau des ad\`eles\footnote{Le corps $F$ Žtant fixŽ dans toute la suite 
nous allgeons la notation en notant simplement $\AM$
l'anneau $\AM_F$.} et $\adef^{\times}$ le groupe des id\`eles de $F$. 
On dispose de l'application degr\'e\index{daeg@$\deg$}
\begin{equation*}\deg: \adef^{\times}\rightarrow \ZM\end{equation*}
d\'efinie par
\begin{equation*}\deg(a)=\sum_{v\in\vert\ES V\vert} -v(a_v)\deg(v)\qquad\hbox{pour}\qquad
a= (a_v)_{v\in\vert\ES V\vert}\in \adef^{\times}\end{equation*}
et on pose $\vert a\vert =\prod_v\vert a_v\vert_v=q^{\deg (a)}$. 
Le groupe $F^\times$ est un sous-groupe discret de $\adef^{\times}$ contenu dans 
\begin{equation*}\adef^1=\{a\in \adef^{\times}:\deg(a)=0\}\vg\end{equation*}
et le quotient $F^\times\bsl \adef^1$ est un compact.

 \section{Espaces vectoriels et r\'eseaux}
\label{espace vectoriel et r\'eseau}

Soit $P$ un groupe alg\'ebrique lin\'eaire connexe d\'efini sur $F$.
On note $X_F(P)$ le groupe des \textit{caract\`eres alg\'ebriques} de $P$ d\'efinis sur $F$.
Si $R$ est un anneau on pose
\begin{equation*}\ag_{P,R}\bydef\bfHom(X_F(P),R)\ptf\end{equation*}
\index{aaagpr@$\ag_{P,R}$}
Le $\ZM$-module libre de type fini $\ag_{P,\ZM}$
est un r\'eseau de l'espace vectoriel r\'eel  
\begin{equation*}\agp\bydef \ag_{P,\RM}\ptf\end{equation*}
\index{aaagp@$\ag_P$}
Pour $x\in P(\adef)$ on note $\bfH_P(x)$ \index{Haap@$\bfH_P$}
l'\'el\'ement de $\ag_{P}$ tel que
\begin{equation*}\langle \chi,\bfH_P(x)\rangle= \deg \chi(x) 
\quad
\hbox{pour tout}\quad \chi\in X_F(P)\ptf\end{equation*}
L'application $x\mapsto\bfH_P(x)$ est un morphisme
\begin{equation*}\bfH_P: P(\adef)\rightarrow\ag_{P}\end{equation*}
dont l'image -- not\'ee $\ag_{P,F}$ dans dans \cite{A9} et \cite{W} --
sera ici not\'ee $\ESA_P$:
\index{AabESAP@$\ESA_P$}
\begin{equation*}\ESA_P\bydef \bfH_P(P(\adef))\ptf\end{equation*}
 Pour un corps de fonctions, 
 $\ESA_P$ est un sous-groupe d'indice fini de $\ag_{P,\ZM}$
et donc un r\'eseau de $\ag_P$ (alors que $\ESA_P=\ag_P$ pour un corps de nombres).
Une inclusion de $F$-groupes alg\'ebriques lin\'eaires connexes $P\subset Q$ induit des homomorphismes
\begin{equation*}\ag_{P,\bullet}\to\ag_{Q,\bullet}\ptf\end{equation*}
On pose
\begin{equation*}\ag^Q_P\bydef\ker[\agp\to\ag_Q]\index{aaagpq@$\ag_P^Q$}\quad \hbox{et} \quad \ESA_P^Q\bydef\ker[\ESA_P\to\ESA_Q]=\ESA_P\cap \ag_P^Q\index{AaxESAPQ@$\ESA_P^Q$}\ptf\end{equation*}

Pour $H\in\ESA_P$ on notera $P(\adef;H)$
l'image r\'eciproque de $H$. Le noyau de $\bfH_P$, usuellement not\'e $P(\adef)^1$, 
n'est autre que $P(\adef;0)$. Le groupe $P(\adef)^1=P(\adef;0)$
 op\`ere par translations sur $P(\adef;H)$ ainsi que le groupe
$P(F)$ des points $F$-rationnels de $P$ qui est un sous-groupe de 
 $P(\adef)^1$, et donc $P(\adef)^1$
op\`ere \`a droite sur le quotient $P(F)\bsl P(\adef;H)$.

Supposons que le radical unipotent $U_P$ de $P$ soit d\'efini sur $F$ et qu'il existe une section 
$\iota: \overline{P} =P/U_P\rightarrow P$ elle aussi d\'efinie sur 
$F$ (c'est toujours le cas si $P$ est un $F$-sous-groupe parabolique d'un 
groupe alg\'ebrique r\'eductif connexe d\'efini sur $F$). On note $Z_P$ le centre {\og sch\'ematique \fg} 
de $\overline{P}$, et $A_P\index{AaaaP@$A_P$}
\subset Z_P$ le tore $F$-d\'eploy\'e maximal de $Z_P$. On identifie $A_P$ \`a $\iota(A_P)$. 
L'homomorphisme naturel 
\begin{equation*}\ESA_{A_P}\to\ESA_{P}\end{equation*}
est bien d\'efini (il ne d\'epend pas de la section $\iota$). Il est injectif mais n'est pas surjectif en g\'en\'eral; 
son image sera not\'ee 
\begin{equation*}\ESB_P\index{Baesp@$\ESB_P$}
\bydef\bfH_P(A_P(\adef))\ptf\end{equation*}
C'est un sous-groupe d'indice fini de $\ESA_P$; on pose
\begin{equation*}\bsbbc_P \index{cbsp@$\bsbbc_P$}
\bydef \ESB_P\bsl\ESA_P\simeq \A_P(\adef)P(\adef)^1\bsl P(\adef)\ptf\end{equation*}
Pour $P\subset Q$ deux $F$-sous-groupes paraboliques d'un groupe alg\'ebrique r\'eductif 
connexe d\'efini sur $F$, l'inclusion $A_P \subset A_Q$ 
(d\'efinie via le choix de sections $\overline{P}\rightarrow P$ et $\overline{Q}\rightarrow Q$ d\'efinies sur $F$) 
induit une section $\ag_Q\rightarrow \ag_P$ de l'homomorphisme surjectif $\ag_P \rightarrow \ag_Q$ 
et donc une d\'ecomposition
\begin{equation*}\ag_P= \ag_Q \oplus \ag_P^Q\ptf\end{equation*}
Pour $X=X_P\in \ag_P$, cette d\'ecomposition s'\'ecrit $X= X_Q+X^Q$. 


 \section{Dualit\'e et mesures}\label{dualit\'e}

On appelle \textit{caract\`ere} d'un groupe topologique 
un homomorphisme continu dans $\CM^\times$, et \textit{caract\`ere unitaire} un caract\`ere \`a valeurs 
dans le groupe ${\UM}$ des nombres complexes de module $1$. Le dual de Pontryagin d'un groupe
ab\'elien localement compact est le groupe topologique de ses caract\`eres unitaires.

Si $\ag$ est un espace vectoriel r\'eel de dimension finie on notera $\ag^*$ l'espace vectoriel r\'eel dual,
$\langle \Lambda,X\rangle$ le produit scalaire de $X\in\ag$ et $\Lambda\in\ag^*$, 
et $\wh\ag$ le dual de Pontryagin de $\ag$ que l'on peut identifier, au moyen de l'exponentielle, avec le sous-espace 
$i\ag^*$ des vecteurs imaginaires purs dans $\ag^*\otimes\CM$. Si $\ESR$ est un r\'eseau
de $\ag$ on notera $\ESR^\vee$ l'ensemble des $\Lambda\in i\ag^*$ 
tels que $\langle \Lambda,X\rangle\in{2i\pi}Z$ pour tout $X\in \ESR$.
C'est l'orthogonal de $\ESR$ du point de vue de la dualit\'e de Pontryagin: 
le groupe compact $\wh\ag/\ESR^\vee$ s'identifie au dual de Pontryagin de $\ESR$:
\begin{equation*}\wh\ESR\simeq\wh\ag/\ESR^\vee\ptf\end{equation*}
Pour une fonction \`a d\'ecroissance rapide $f$ sur $\ESR$, on note 
$\wh{f}$ la fonction lisse sur $\wh{\ESR}$ d\'efinie par
\begin{equation*}\wh{f}(\Lambda) = \sum_{X\in \ESR}{f}(X)e^{\langle\Lambda , X\rangle}\vgq \Lambda \in \wh{\ESR}\ptf\end{equation*}
Si $\wh\ESR$ est muni de la mesure duale de la mesure de comptage sur
$\ESR$, c'est-\`a-dire telle que $\vol(\wh\ESR)=1$, on a la formule d'inversion
\begin{equation*}f(X)= \int_{\wh{\ESR}}
\wh{f}(\Lambda) e^{-\langle \Lambda,X\rangle}\dd \Lambda\vgq X\in \ESR\ptf\end{equation*}
Pour all\'eger l\'eg\`erement les notations, pour $P$ comme en \ref{espace vectoriel et r\'eseau}, 
on se permettra d'\'ecrire $\wh{\ag}_P$\index{abgphat@$\wh{\ag}_P$}, $\wh\ESA_P$\index{AbESAPhat@$\wh\ESA_P$}, etc. en place de $\wh{\ag_P}$, $\wh{\ESA_P}$, etc.
On pose
\begin{equation*}\bsmu_P\index{maup@$\bsmu_P$}
\bydef \wh\ESA_P\ptf\end{equation*}

Notons $\Xi(P)$, resp. $\Xi(P)^1$\index{XiP@$\Xi(P)$, $\Xi(P)^1$}, l'ensemble des caract\`eres unitaires de $A_P(\adef)$, resp. $A_P(\adef)^1$, 
qui sont triviaux sur $A_P(F)$. Ce sont deux groupes ab\'eliens localement compact, et $\Xi(P)^1$ \index{XiPp@$\Xi(P)^+$}
est un quotient discret de $\Xi(P)$. On notera $\Xi(P)^+$ l'ensemble des caractres non nŽcessairement unitaires.
Le groupe $\Xi(P)$ s'ins\`ere dans la suite exacte courte 
\begin{equation*}0 \rightarrow \wh{\ESB}_P \rightarrow \Xi(P)\rightarrow \Xi(P)^1 \rightarrow 0\ptf\end{equation*}
\begin{convention}\label{convention-mesures} 
Pour les mesures de Haar sur les groupes ab\'eliens localement compacts, on adoptera les normalisations suivantes. 
On impose la compatibilit\'e aux suites exactes courtes et \`a la dualit\'e de Pontryagin. 
Les r\'eseaux de $\ag_P$ sont munis de la mesure de comptage. Cela implique que le groupe fini 
$\bsbbc_P$ est lui aussi muni de la mesure de comptage, et que l'on a\footnote{On observera que 
cette convention n'est pas celle utilis\'ee dans \cite{A9}.}
\begin{equation*}\vol (\bsmu_P) =\vol (\wh\ESB_P)= \vol ( \wh\bsbbc_P )=1\ptf\end{equation*}
On impose aussi que la mesure de Haar sur $A_P(F)\backslash A_P(\adef)$ v\'erifie
\begin{equation*}\vol (A_P(F)\bsl A_P(\adef)^1)=1\ptf\end{equation*}
Ceci implique en particulier que le groupe discret
$\Xi(P)^1$, dual du groupe compact $A_P(F)\bsl A_P(\adef)^1$,
est muni de la mesure de comptage.
\end{convention}

Tout caract\`ere $\chi$ de $P(\adef)$ s'\'ecrit de mani\`ere unique
\begin{equation*}\chi=\chi_\mathrm{u} \vert \chi \vert\end{equation*}
avec $\vert \chi \vert (g) = \vert \chi(g) \vert$ et $\chi_\mathrm{u}$ unitaire.
Le caract\`ere $\vert \chi\vert$ est trivial sur $P(\adef)^1$. 
On note $\chi^1=\chi^1_\mathrm{u}$ la restriction de $\chi$ \`a $P(\adef)^1$.
Tout \'el\'ement $\nu\in \ag_{P,\CM}^*$ d\'efinit un caract\`ere de $P(\adef)$ trivial sur $P(\adef)^1$:
\begin{equation*}p\mapsto e^{\langle \nu,\bfH_P(p)\rangle}\vgq p\in P(\adef)\ptf\end{equation*}
Ce caract\`ere ne d\'epend que de l'image de $\nu$ dans $\ag_{P,\CM}^*/\ESA_P^\vee$ 
et sa restriction \`a $A_P(\adef)$ ne d\'epend que 
de l'image de $\nu$ dans $\ag_{P,\CM}^*/\ESB_P^\vee$. 
La suite exacte courte\index{cbsphat@$\wh\bsbbc_P$}
\begin{equation*}0 \rightarrow \wh\bsbbc_P \rightarrow \ag_{P,\CM}^*/\ESA_P^\vee \rightarrow 
\ag_{P,\CM}^*/\ESB_P^\vee \rightarrow 0\end{equation*}
correspond \`a la restriction des caract\`eres de $P(\adef)$ \`a $A_P(\adef)$.
Si $\pi$ est une repr\'esentation de $P(\adef)$, pour $\nu \in \ag_{P,\CM}^*/\ESA_P^\vee$
on note $\pi_\nu$ la repr\'esentation de $P(\adef)$ d\'efinie par
\begin{equation*}\pi_\nu(p)= \pi(p)e^{\langle \nu,\bfH_P(p)\rangle}\vgq p\in P(\adef)\ptf\end{equation*}
On la notera aussi parfois $\pi \star \nu$\index{painu@$\pi\star \nu = \pi_\nu$}. 
De mme, si $\xi$ est un caract\`ere de $A_P(\adef)$, pour $\nu \in \ag_{P,\CM}^*/\ESB_P^\vee$ 
on note $\xi_\nu=\xi\star\nu$ 
le caract\`ere $a \mapsto \xi(a)e^{\langle \nu, \bfH_P(a)\rangle}$ de $A_P(\adef)$.


 \section{Sous-groupes paraboliques}
Soit $G$ un groupe alg\'ebrique r\'eductif connexe d\'efini sur $F$. 
Tous les sous-groupes paraboliques de $G$, ainsi que leurs composantes de Levi,
consid\'er\'es dans la suite sont suppos\'es d\'efinis sur $F$. 
On fixe un sous-groupe parabolique minimal $P_0$ de $G$ et une composante de Levi $M_0$ de $P_0$. 
On note $A_0$ le tore $F$-d\'eploy\'e maximal du centre $Z_{M_0}$ de $M_0$. Ainsi $M_0$ est 
le centralisateur de $A_0$ dans $G$. Un sous-groupe parabolique de $G$ est dit 
\og standard\fg, resp. \og semi-standard\fg, s'il contient $P_0$, resp. $M_0$. Un facteur de Levi de $G$, 
c'est-\`a-dire une composante de Levi d'un sous-groupe parabolique de $G$, est dit {\og semi-standard\fg} 
s'il contient $M_0$, et il est dit {\og standard\fg} si c'est la composante de Levi semi-standard d'un sous-groupe 
parabolique standard. Dans la suite tous les sous-groupe paraboliques et
tous les sous-groupe de Levi seront semi-standards; nous omettrons parfois de le pr\'eciser.

On note $\ESP=\ESP^G$,\index{Parond@$\ESP$, $\ESP_{\mathrm {st}}$}
resp. $\ESL=\ESL^G$\index{Lrond@$\ESL$}, 
l'ensemble des sous-groupes paraboliques, resp. facteurs de Levi, de $G$,
(semi-standards) et $\ESP_\st=\ESP^G_\st$ 
le sous-ensemble de $\ESP$ form\'e des \'el\'ements standards. 
Pour $P\in \ESP$, 
on note $M_P$ ou simplement $M$
la composante de Levi (semi-standard) de $P$, et $U_P$ ou simplement $U$ le radical unipotent de $P$; 
on a les identifications
\begin{equation*}A_P = A_{M}\vgq \ag_P=\ag_{M}\quad\hbox{et} \quad\ESA_P=\ESA_{M}\ptf\end{equation*}
On pose
\begin{equation*}\quad a_P\index{abP@$a_P$}
=\dim(\ag_P)\ptf\end{equation*}
Pour $P,\mskip 2mu  Q\in\ESP$ tels que $P\subset Q$, on a not\'e $\ag_P^Q$ le noyau de l'homomorphisme 
naturel $\ag_P\rightarrow \ag_Q$. On pose\index{abPQ@$a_P^Q$}
\begin{equation*}a_P^Q= \dim(\ag_P^Q)= a_P-a_Q\ptf\end{equation*}

Pour $M\in\ESL$ on note\index{PbrondM@$\ESP(M)$} $\ESP(M)$, resp. $\ESF(M)$,\index{FrondM@$\ESF(M)$}
le sous-ensemble des $Q\in\ESP$ avec comme composante de Levi
 $M_Q=M$, resp. $M_Q\supset M$. Pour $Q\in\ESF(M)$, on pose
\begin{equation*}
\ESP^Q(M)=\{P\in\ESP(M): P\subset Q\},\index{PbQrondM@$\ESP^Q(M)$}
\quad\ESF^Q(M)=\{P\in\ESF(M): P\subset Q\}\ptf\index{FQrondM@$\ESF^Q(M)$}\end{equation*}
On se permettra de remplacer l'indice $P$ par un indice $M$ 
dans les objets qui ne d\'ependent pas du choix de $P\in\ESP(M)$. Par exemple, pour $Q\in\ESF(M)$ et 
$P\in\ESP^Q(M)$, on \'ecrira parfois $\ag_M^Q$ au lieu de $\ag_P^Q$, 
et $X_M^Q$ au lieu de $X_P^Q$ pour $X\in\ag_{P_0}$. 
On remplacera souvent l'indice {\og$P_0$\fg} par un indice \og 0\fg: ainsi on 
\'ecrira $A_0$ pour $A_{P_0}$, 
$\ag_0$ pour $\ag_{P_0}$, $\ag_0 =\ag_Q\oplus\ag_0^Q$, etc.

On fixe une forme quadratique d\'efinie positive $(\cdot,\cdot)$ sur $\ag_0$, 
invariante par le groupe de Weyl $\bfW= N^G(A_0)/M_0$, o $N^G(A_0)$ est 
le normalisateur de $A_0$ dans $G$. Pour tout 
$X\in\ag_0$, on note $\lVert  X\rVert =(X,X)^{1/2}$ la norme de $X$. 
Pour $M\in\ESL$, la forme 
$(\cdot ,\cdot)$ induit par restriction une forme quadratique d\'efinie positive sur $\ag_M$, 
invariante par le groupe de Weyl $\bfW^M\index{WaM@$\bfW^M$}
=N^M(A_0)/M_0$. 
Pour $Q\in\ESF(M)$, $\ag_M^Q$ n'est autre que l'orthogonal de $\ag_Q$ 
dans $\ag_M$ pour cette forme.


Pour $P\in\ESP_\st$, on dispose des racines de $A_0$ dans $M_P$
et des coracines, qui sont de ŽlŽments de $\ag_0^P$. Pour toute racine $\alpha$ 
et toute coracine $\check\beta$ 
on a \begin{equation*}\langle\alpha,\check\beta\rangle=N_{\alpha,\beta}\in\ZM\ptf\end{equation*}
Les coracines appartiennent donc \`a $\ag_{0,\QM}^P$.
Soit  $\Delta^P_0$\index{DeltaOP@$\Delta^P_0$}
l'ensemble des racines simples de $A_0$ dans $M_P$ 
pour l'ordre d\'efini par $P_0\cap M_P$; 
on note $\check{\Delta}_0^P$\index{DeltacheckOP@$\check\Delta^P_0$}
la base de $\ag_0^P$ form\'ee par les coracines simples
et $\hat{\Delta}^P_0$\index{DeltahatOP@$\hat\Delta^P_0$} 
la base des poids pour $M_P$ c'est-\`a-dire la base de $(\ag_0^P)^*$ 
duale de $\check{\Delta}_0^P$.  
Lorsque $P=G$ on Žcrira souvent $\Delta_0$ pour $\Delta_0^G$. 
On observe que $\Delta^P_0\subset\Delta_0$ et que
$\hat{\Delta}^P_0$ est
 l'ensemble des restrictions non nulles des \'el\'ements de $\hat{\Delta}_0$ 
au sous-espace $\ag_0^P$ de $\ag_0^G$.

Plus g\'en\'eralement, pour $P,\mskip 2mu  Q\in\ESP_\st$ tels que $P\subset Q$, on note $\Delta^Q_P$ 
l'ensemble des restrictions non nulles des \'el\'ements de $\Delta^Q_0$ au sous-espace 
$\ag_P^Q$ de $\ag_0^Q$. On note $\check{\Delta}^Q_P$\index{DeltacheckPQ@$\check{\Delta}^Q_P$} l'ensemble 
des projections non nulles des \'el\'ements de $\check{\Delta}^Q_0$ sur l'espace $\ag^Q_P$ 
par rapport \`a la d\'ecomposition $\ag_0^Q =\ag_0^P\oplus\ag_P^Q$. On note 
 $\hat{\Delta}^Q_P$ la base de $(\ag_P^Q)^*$ duale de $\check{\Delta}^Q_P$:
c'est le sous-ensemble des \'el\'ements de $\hat{\Delta}^Q_0$ nuls sur $\ag_0^P$. 
On consid\`ere les \'el\'ements 
de $\Delta^Q_P$\index{DeltaPQ@$\Delta_P^Q$} et $\hat{\Delta}^Q_P$\index{DeltahatPQ@$\hat{\Delta}_P^Q$} comme des formes lin\'eaires sur 
$\ag_0$ gr‰ce \`a la d\'ecomposition 
\begin{equation*}\ag_0 =\ag_0^P\oplus\ag_P^Q\oplus\ag_Q\ptf\end{equation*}
 Rappelons que deux r\'eseaux $\ESR_1$ et $\ESR_2$ d'un ${\RM}$-espace vectoriel de 
 dimension finie sont dits 
\textit{commensurables} si leur intersection $\ESR_1\cap \ESR_2$ est d'indice fini dans chacun d'eux. 
En particulier on a le
\begin{lemma}\label{commensurable}
Les r\'eseaux ${\ZM}(\check{\Delta}^Q_P)$ sont commensurables aux $\ESA_P^Q$. 
\end{lemma}

Ces d\'efinitions s'\'etendent \`a toute paire de sous-groupes paraboliques $(P,Q)$ de $G$ 
tels que $P\subset Q$ : on choisit un \'el\'ement $g\in G(F)$ tel que $g^{-1}Pg\supset P_0$ et, 
par transport de structures via le 
$F$-automorphisme $\mathrm{Int}_g$ de $G$, on d\'efinit les analogues des objets 
ci-dessus; cela ne d\'epend pas du choix de $g$.

On note $\ag_0^+$\index{ag0+@$\ag_0^+$} l'ensemble des $X\in\ag_0$ tels que $\langle\alpha,X\rangle >0$ 
pour tout $\alpha\in\Delta_0$. Un \'el\'ement de $\ag_0$ est dit \textit{r\'egulier} 
s'il appartient \`a $\ag_0^+$. Pour $X\in\ag_0$, on pose\index{dbX@$\bs{d}_0(X)$}
\begin{equation*}\bsd_0(X)=\inf_{\alpha\in\Delta_0}\langle\alpha,X\rangle\ptf\end{equation*}
Ainsi $X$ est r\'egulier si et seulement si $\bsd_0(X)>0$.

Soit $M$ un facteur de Levi de $G$.
Pour $P\in \ESP(M)$, les homomorphismes de $\RM$-espaces vectoriels $\ag_{M}\to\ag_P$
et de $\ZM$-modules $ \ESA_M\to\ESA_P $ sont des isomorphismes. 
Pour $P,\mskip 2mu Q \in \ESP$ tels que $P\subset Q$, on a not\'e $\ESA_P ^Q$ le r\'eseau 
$\ker[\ESA_P \rightarrow \ESA_Q] =\ESA_P \cap \ag_P^Q$ de $\ag_P^Q$.

\begin{lemma}\label{surj}
On a une suite exacte courte de r\'eseaux:
\begin{equation*}0\rightarrow\ESA_P ^Q\rightarrow\ESA_P \rightarrow\ESA_Q\rightarrow 0\end{equation*}
\end{lemma}

\begin{proof} 
Il convient d'\'etablir la surjectivit\'e de la fl\`eche $\ESA_P\rightarrow\ESA_Q$.
Pour cela on invoque la d\'ecomposition d'Iwasawa (rappel\'ee en \ref{d\'ecomposition d'Iwasawa}): 
tout $q\in Q(\adef)$ peut s'\'ecrire $q=pk$ avec $k\in\bs{K}$ o $\bs{K}$ est un bon sous-groupe compact 
maximal dans $G(\adef)$
et $p\in P(\adef)$, donc $\bfH_Q(q)=\bfH_Q(p)=\bfH_P(p)$.
\end{proof}

On pose dualement\index{maupq@$\bsmu_P^Q$}
\begin{equation*}\bsmu_P^Q\bydef\bsmu_P/\bsmu_Q=\wh\ESA_P^{\mskip 2mu Q}\ptf\end{equation*}
Notons qu'en g\'en\'eral il n'y a pas de section canonique relevant $\ESA_Q$ dans
$\ESA_P$. En revanche, on a toujours l'inclusion
\begin{equation*}\ESB_Q= \bfH_Q(A_Q(F))= \bfH_P(A_Q(F))\subset\ESB_P\ptf\end{equation*}
Puisque $A_Q$ est un sous-tore de $A_P$, on a l'\'egalit\'e:
\begin{equation*}\ESB_Q= \ESB_P \cap \ag_Q\ptf\end{equation*}
En r\'esum\'e, les inclusions
\begin{equation*}A_Q(\adef) \subset A_P(\adef) \subset M_P(\adef) \subset M_Q(\adef)\end{equation*}
donnent le diagramme commutatif suivant:
\begin{equation*}\xymatrix{\ESA_P \ar[r]&\ESA_Q\\
\ESB_P \ar[u] & \ESB_Q \ar[u] \ar[l]}\end{equation*}
o la fl\`eche horizontale du haut est surjective alors que les trois autres sont injectives. 
On pose
\begin{equation*}\ESB_P^Q\index{Baespq@$\ESB_P^Q$}
\bydef \ESB_Q\backslash \ESB_P \quad \hbox{et} \quad\ESC_P^Q\index{Cespq@$\ESC_P^Q$}
 \bydef \ESB_Q \backslash \ESA_P\ptf\end{equation*}
On observe que $\ESB_P^Q$ est un r\'eseau de $\ag_P^Q$ et que $\ESC_P^Q$ est un $\ZM$-module 
de type fini qui s'ins\`ere dans la suite exacte courte
\begin{equation*}0 \rightarrow \ESB_P^Q \rightarrow \ESC_P^Q \rightarrow \bsbbc_P\rightarrow 0\ptf\end{equation*}


 \section{Familles orthogonales et $(G,M)$-familles}\label{ortho}\label{GMF}
Fixons un facteur de Levi $M\in\ESL$ et consid\'erons une famille 
\begin{equation*}\XX= (X_{P})_{P\in\ESF(M)}\end{equation*}
d'\'el\'ements $X_P\in\ag_P$. 
Une telle famille est dite {$M$-\textit orthogonale} (ou simplement \textit{orthogonale})
si pour tous $P$ et $Q$ dans $\ESF(M)$ tels que $P\subset Q$, 
la projection $(X_P)_Q$ de $X_P$ dans $\ag_Q$ est \'egale \`a $X_Q$. 
Pour $M'\in\ESL$ tel que $M\subset M'$, 
une famille $M$-orthogonale d\'etermine par restriction une famille $M'$-orthogonale. 
Il suffit, pour d\'efinir une famille $M$-orthogonale, de se donner des $X_P\in\ag_M$ pour chaque
${P\in\ESP(M)}$ v\'erifiant la propri\'et\'e suivante: si $P$ et $P'\in \ESP(M)$
sont adjacents et si $\alpha$ est l'unique racine de $A_M$ positive pour 
$P$ et n\'egative pour $P'$ alors $X_P-X_{P'}$ est un multiple de $\check{\alpha}$. 
La famille est dite \textit{r\'eguli\`ere} si les $X_P-X_{P'}$ sont des multiples positifs de $\check{\alpha}$. 
On dit que la famille est \textit{enti\`ere}, resp. \textit{rationnelle}, 
si $X_P\in\ESA_M$, resp. $X_P\in\ag_{M,\QM}$, pour tout $P\in\ESP(M)$. 
En ce cas $X_Q\in\ESA_Q $, resp. $X_Q\in \ag_{Q,\QM}$, pour tout $Q\in\ESF(M)$.

On notera $\HH_{G,M}$ ou simplement $\HH_M$\index{HbgothM@$\HH_M$} 
si aucune confusion n'en r\'esulte, le $\RM$-espace vectoriel de dimension finie form\'e 
des familles $M$-orthogonales, et
\begin{equation*}\ESH_M\index{HbesM@$\ESH_M$}
= \ESH_{G,M}\subset\HH_{M}\end{equation*}
le r\'eseau form\'e 
des familles qui sont enti\`eres. On note
\begin{equation*}\wh\HH_M= i \HH_M^*\end{equation*}
le dual de Pontryagin de $\HH_M$. Le groupe compact 
\begin{equation*}\wh\ESH_M=\wh\HH_M/\ESH_M^\vee\end{equation*}
s'identifie au dual de Pontryagin de $\ESH_M$. 
Pour chaque $P\in\ESF(M)$, 
on dispose d'une application 
\begin{equation*}\pi_P:\HH_{M}\rightarrow\agp\vg\; \XX\mapsto X_P\end{equation*}
qui est surjective et on note \begin{equation*}\iota_P:\hagp\rightarrow \wh\HH_M\end{equation*}
l'application injective transpos\'ee de $\pi_P$. 

Une mani\`ere tr\`es simple de construire une famille $M_0$-orthogonale est de fixer un \'el\'ement $T\in\ag_0$. 
Pour $P\in\ESP(M_0)$, si $w$ est l'unique \'el\'ement de $\bfW$ tel que 
${^w(P_0)}= P$, on pose $\brT{P}= wT$.\index{Tabbrap@ $\brT{P}$}
On a donc $\brT{P_0}=T$ et l'ensemble 
\begin{equation*}\TT=(\brT{P})_{P\in\ESP(M_0)}\end{equation*}
d\'efinit une famille $M_0$-orthogonale. Pour $Q\in\ESF(M_0)$, on a
$\brT{Q}= (\brT{P})_Q$ pour un (i.e. pour tout) $P\in\ESP^Q(M_0)$. Cette famille est r\'eguli\`ere, 
resp. rationnelle, si et seulement si $T\in \ag_0^+$, resp. 
$T\in\ag_{0,\QM}$.
 Soit $\XX=(X_P)$ une famille $M$-orthogonale, et soit $T$ un \'el\'ement de $\ag_{0}$.
On d\'efinit une autre famille $M$-orthogonale en posant:
\begin{equation*}\XX(T)=\XX+\TT\index{XmmT@$\XX(T)$}\qquad \hbox{c'est-\`a-dire}\qquad\XX(T)_{P}= X_P+\brT{P} \ptf\end{equation*}

Une $(G,M)$-famille est la donn\'ee d'une famille de fonctions 
\begin{equation*}\bsc = (\Lambda\mapsto\bsc(\Lambda,P)\mskip 2mu \mid \mskip 2mu {P\in\ESF(M)})\end{equation*}
\`a valeurs dans $\CM$, ou plus g\'en\'eralement \`a valeurs
dans un espace vectoriel de dimension finies $\EV$, 
v\'erifiant les conditions:
\begin{itemize}
\item pour tout $P\in\ESF(M)$, la fonction $\Lambda\mapsto\bsc(\Lambda,P)$ est lisse sur $\wh\ag_P$ ;
\item pour tous $P,\mskip 2mu  Q\in\ESF(M)$ tels que $P\subset Q$ (i.e. $P\in\ESF^Q(M)$), on a
\begin{equation*}\bsc(\cdot ,P)\vert_{\wh\ag_Q }=\bsc(\cdot , Q)\ptf\end{equation*}
\end{itemize}
Pour chaque $P\in\ESF(M)$, on prolonge $\bsc(\cdot ,P)$ en une fonction 
sur $\wh\ag_0$ constante sur les fibres de la projection
$\wh\ag_0\rightarrow \wh\ag_P$.
Comme pour les familles $M$-orthogonales, il suffit pour d\'efinir une $(G,M)$-famille 
de se donner 
des fonctions lisses 
\begin{equation*}\bsc(\cdot ,P):\wh\ag_M \rightarrow\EV\qquad \hbox{pour}\qquad P\in\ESP(M)\end{equation*}
qui v\'erifient la propri\'et\'e suivante:
pour $P,\mskip 2mu P'\in \ESP(M)$ deux \'el\'ements adjacents correspondant \`a des chambres s\'epar\'ees 
par le mur $\ag_R$, o $R$ est l'\'el\'ement de $\ESF(M)$ engendr\'e par $P$ et $P'$, on a 
l'\'egalit\'e
\begin{equation*}\bsc(\cdot ,P)\vert_{\wh\ag_R }=\bsc(\cdot ,P')\vert_{\wh\ag_R }\ptf\end{equation*}
Les fonctions $c(\cdot,Q)$ pour $Q\in\ESF(M)\smallsetminus\ESP(M)$ s'en d\'eduisent par restriction. 

Une $(G,M)$-famille $\bsc=({\bs c}(\cdot,P))$ est dite \textit{p\'eriodique} 
si pour tout $P\in\ESF(M)$,
 la fonction $\Lambda\mapsto\bsc(\Lambda,P)$ est invariante par translation par 
 $\ESA_P^\vee$, i.e. se factorise 
par $\bsmu_P\bydef\wh{\ESA}_P$. 
Pour qu'une $(G,M)$-famille $\bsc=({\bs c}(\cdot,P))$ soit p\'eriodique, 
il suffit que pour tout $P\in\ESP(M)$, la 
fonction $\Lambda\mapsto\bsc(\Lambda,P)$ soit $\ESA_M^\vee$-p\'eriodique. On notera 
$\Dcal(G,M)$\index{Drond@$\Dcal(G,M)$} le $\CM$-espace vectoriel form\'e des $(G,M)$-familles p\'eriodiques.

Soit $m$ une mesure de Radon \`a d\'ecroissance rapide sur $\HH_M$
et \`a valeurs dans $\EV$. On lui associe une
$(G,M)$-famille en posant, pour $\Lambda\in \hagp$:
\begin{equation*}\bsc(\Lambda,P)=\int_{\HH_M}e^{\langle\Lambda,X_P\rangle}\dd m(\XX)\end{equation*}
o $X_P=\pi_P(\XX)$. 
La $(G,M)$-famille $\bsc$ est p\'eriodique si et seulement si 
la mesure $m$ est produit d'une fonction \`a d\'ecroissance rapide sur le r\'eseau $\ESH_M$
par la mesure de comptage. Par abus de notation nous noterons encore $\mf$ cette fonction.
Sa transform\'ee de Fourier
\begin{equation*}\wh\mf(\LL)=
\sum_{\UU\in\ESH_M} e^{\langle\LL,\UU\rangle} m(\UU)\quad \hbox{pour}\quad \LL\in \wh{\mathfrak{H}}_M\end{equation*}
se factorise par $\wh\ESH_M$ 
c'est-\`a-dire est invariante par $\ESH_M^\vee$. 
On notera $\bsc_\mf$ la $(G,M)$-famille p\'eriodique 
d\'efinie par \begin{equation*}\bsc_m(\Lambda,P)=\sum_{\UU\in\ESH_M} e^{\langle\Lambda,U_P\rangle}\mf(\UU)=
( \wh\mf\circ\iota_P)(\Lambda)\ptf\end{equation*}

 On munit l'espace vectoriel de dimension finie $\HH_M$ d'une structure euclidienne et
on note $\lVert \XX\rVert $ la norme du vecteur $\XX\in\HH_M$. 
L'espace $\ESS(\ESH_M)$ des fonctions $m$ \`a d\'ecroissance rapide sur $\ESH_M$ est muni
d'une structure d'espace de Fr\'echet au moyen des normes $n_d$ pour $d\in\NM$:
\begin{equation*}n_d(m)=\sup_{\UU\in\ESH_M}(1+ \lVert \UU\rVert )^d \mid  m(\UU) \mid  \ptf\end{equation*}
Tout op\'erateur diff\'erentiel \`a coefficients constants $D$ sur $\hag_M$ permet de d\'efinir
une semi-norme $N_D$ sur l'espace $\Dcal(G,M)$ en posant
\begin{equation*}N_D(\bsc)=\sum_{P\in \ESP(M)}\sup_{\Lambda\in \bsmu_M} \vert D\bsc(\Lambda,P)\vert\vg\end{equation*}
munissant ainsi $\Dcal(G,M)$ d'une structure d'espace de Fr\'echet. 
 L'application lin\'eaire
 \begin{equation*}\bsS :\ESS(\ESH_M)\to\Dcal(G,M)\end{equation*}
 d\'efinie par $m\mapsto \bsc_\mf$ est continue. 
\begin{lemma}\label{invfour}
Toutes les $(G,M)$-familles 
p\'eriodiques sont obtenues de cette mani\`ere. En d'autres termes, l'application $\bsS$ est surjective.
\end{lemma}

\begin{proof} C'est une variante de \cite[1.10.1]{LW}. La preuve en est identique 
\`a ceci pr\`es que la fonction $\chi$ sur $\RM$ servant \`a construire la globalisation
sur $\wh{\HH}_M$ de la $(G,M)$-famille $\bsc$ donn\'ee, qui ici est p\'eriodique,
doit ici tre prise p\'eriodique au lieu de lui imposer d'tre \`a support compact. 
Plus pr\'ecis\'ement, on fixe une base ${\mathcal B}_G$ de $\hag_G= i \ag_G^*$, et 
pour chaque $P\in \ESF(M)$, on pose
\begin{equation*}{\mathcal B}_P= \iota_P({\mathcal B}_G \cup i \wh{\Delta}_P)\subset \wh{\HH}_M \ptf\end{equation*}
On note ${\mathcal B}$ la base de $\wh{\HH}_M$ form\'ee par l'union de ces ensembles ${\mathcal B}_P$.
C'est l'union de ${\mathcal B}_G$ et des $e_Q= \iota_Q(i\varpi_Q)$ o $Q$ parcourt l'ensemble 
$\ESF_\mathrm{max}(M)$ des sous-groupes paraboliques maximaux propres de $G$ contenant $M$
et $\varpi_Q$ est 
l'unique \'el\'ement de $\wh{\Delta}_Q$. Pour chaque $P\in \ESF(M)$, on a une partition de ${\mathcal B}$ en
\begin{equation*}{\mathcal B}= {\mathcal B}_P\cup {\mathcal B}^P\quad \hbox{avec} \quad {\mathcal B}^P=
 \{e_Q \mskip 2mu \vert\mskip 2mu  Q \in \ESF_\mathrm{max}(M),\mskip 2mu  Q \not\supset P\}\end{equation*}
qui induit une d\'ecomposition de $\wh{\HH}_M$ en somme directe. Pour $\lambda\in \wh{\HH}_M$, 
on note $\lambda= \lambda_P+ \lambda^P$ la 
d\'ecomposition associ\'ee et l'on pose
\begin{equation*}\chi_P(\lambda^P)= \prod_{Q \not\supset P} \chi(x_Q(\lambda))\quad \hbox{si} \quad \lambda^P = 
\sum_{Q\not\supset P} x_Q(\lambda) e_Q\ptf\end{equation*}
On d\'efinit la fonction
\begin{equation*}f(\lambda)= \sum_{P\in \ESF(M)} (-1)^{a_Q-a_M} \bs{c}(\lambda_P,P)\chi_P(\lambda^P)\end{equation*}
o l'on a identifi\'e $\lambda_P\in \iota_P(\hag_P)$ \`a un \'el\'ement de $\hag_P$.
Notons $\EuScript{Z}$ le r\'eseau de $\RM$ engendr\'e par les $x_Q(\lambda)$ pour $\lambda \in \ESH_M^\vee$ et $Q\in \ESF_\mathrm{max}(M)$. 
Il suffit de prendre pour $\chi$ une fonction lisse et $\EuScript{Z}$-invariante sur $\RM$ telle que $\chi(0)=1$. N'importe quel caract\`ere 
de $\EuScript{Z}\backslash \RM$ convient, par exemple $\chi= 1$.
\end{proof}

Ce r\'esultat permet de d\'efinir d'autres normes sur l'espace $\Dcal(G,M)$:

\begin{definition}\label{normes} Pour $d\in\NM$ on pose
\begin{equation*}
N_d(\bsc)=\inf\{n_d(m)\mskip 2mu \mid \mskip 2mu  \bsc_\mf=\bsc\}\ptf
\end{equation*}
\end{definition}


 \section{Fonctions caract\'eristiques de c™nes et de convexes}
\label{fonctions caract\'eristiques}

Pour $P,\mskip 2mu  Q\in\ESP$ tels que $P\subset Q$, on note\footnote{Comme dans \cite{LW}, 
toutes les fonctions index\'ees par $P$ sont des fonctions sur $\ag_0$ qui se factorisent \`a travers 
$\ag_0^P\backslash \ag_0 \;(\simeq\ag_P)$. Dans \cite{W}, 
elles sont consid\'er\'ees comme des fonctions sur $\ag_P$.} $\tau^Q_P$,\index{taupq@$\tau^Q_P$} 
resp. $\wh{\tau}^{\hspace{0.7pt}Q}_P$,\index{taupqh@$\wh{\tau}^Q_P$} 
la fonction caract\'eristique du c™ne ouvert dans $\ag_0$ d\'efini par $\Delta^Q_P$, 
resp. $\hat{\Delta}^Q_P$:
\begin{equation*}
\tau^Q_P(X)=1\Leftrightarrow\langle\alpha,X\rangle >0,\mskip 2mu \forall\alpha\in\Delta^Q_P\vg
\end{equation*}
\begin{equation*}
\wh{\tau}^{\hspace{0.7pt}Q}_P(X)=1\Leftrightarrow\langle\varpi,X\rangle >0,\mskip 2mu \forall\varpi\in\hat{\Delta}^Q_P\ptf
\end{equation*}
Pour $Q=G$, on \'ecrira souvent $\tau_P=\tau_P^G$, $\wh{\tau}_P=\wh{\tau}_P^{\hspace{0.7pt}G}$.
La propri\'et\'e essentielle pour la combinatoire est que les matrices, index\'ees par les couples 
de sous-groupes paraboliques standards, $\tau =(\tau_{P,Q})$ et $\wh{\tau}=(\wh{\tau}_{P,Q})$
d\'efinies par
\begin{equation*}\tau_{P,Q}=\left\{\begin{array}{ll}(-1)^{a_P}\tau^Q_P &\hbox{si $P\subset Q$}\\
0 &\hbox{sinon}\end{array}\right.\qquad
\hbox{et}\qquad
\wh{\tau}_{P,Q}=\left\{\begin{array}{ll}(-1)^{a_P}\wh{\tau}^{\hspace{0.7pt}Q}_P &\hbox{si $P\subset Q$}\\
0 &\hbox{sinon}\end{array}\right.\end{equation*}
sont inverses l'une de l'autre: $\tau\wh{\tau} =\wh{\tau}\tau =1$ (cf. \cite[1.7.2]{LW}).

Soit $M\in\ESL$. Fixons un \'el\'ement $Q\in\ESP(M)$. 
Cet \'el\'ement d\'efinit un ordre sur l'ensemble des racines de $A_M$ dans $G$. 
On \'ecrit $\alpha >_Q 0$, resp. $\alpha <_Q 0$, pour signifier qu'une racine 
$\alpha$ est positive, resp. n\'egative, pour cet ordre. Pour $P\in \ES{P}(M)$, 
notons $\phi_{P,Q}^G$\index{phigpq@$\phi_{P,Q}^G$} la 
fonction caract\'eristique des $X\in\ag_P^G$ suivante:
\begin{equation*}\phi^{G}_{P,Q}(X)=1\Leftrightarrow
\left\{\begin{array}{ll}\langle\varpi_\alpha,X\rangle\leq 0 
&\hbox{pour $\alpha\in\Delta_P$ tel que $\alpha>_Q 0$}\\
\langle\varpi_\alpha,X\rangle >0&\hbox{pour $\alpha\in\Delta_P$ tel que $\alpha <_Q 0$}
\end{array}
\right.\end{equation*}
o $\{\varpi_\alpha :\alpha\in\Delta_P\}=\hat{\Delta}_P$ est la base de $(\ag_P^G)^*$ 
duale de $\{\check{\alpha}:\alpha\in\Delta_P\}=\check{\Delta}_P$. On note $a(P,Q)$
le nombre des $\alpha \in\Delta_P$ tels que $\alpha <_Q 0$. 
Par d\'efinition, $\phi_{P,Q}^G$ se factorise par $\ag_P^G$. 
Observons que $\phi_{P,Q}^G$ est not\'e $\phi_{M,s}$ dans \cite{LW} lorsque $P=s(Q)$
pour un \'el\'ement $s$ dans le groupe de Weyl. Plus g\'en\'eralement pour $R\in\ESF(M)$ et 
$P,\mskip 2mu Q\in\ESP^R(M)$, on d\'efinit la fonction $\phi_{Q,P}^R$ comme ci-dessus en rempla\c{c}ant 
$G$ par $L= M_R$, $P$ par $P\cap L$ et $Q$ par $Q\cap L$.
Nous aurons aussi besoin de $\phi^Q_P$\index{pphipq@$\phi^Q_P$}
 la fonction caract\'eristique du c™ne ferm\'e dans $\ag_0$ 
d\'efini par $-\hat{\Delta}^Q_P$:
\begin{equation*}\phi^Q_P(X)=
1\Leftrightarrow\langle\varpi,X\rangle\leq 0,\mskip 2mu \forall\varpi\in\hat{\Delta}^Q_P\ptf\end{equation*}
On observera que $\phi_P^Q=\phi_{P,P}^Q$.

Nous allons citer des r\'esultats emprunt\'es \`a la section 1.8 de \cite{LW}. Leurs
preuves reposent sur le lemme \cite[1.8.1, page 22]{LW} dont la d\'emonstration est incorrecte.  
Cette erreur a \'et\'e observ\'ee par P.-H. Chaudouard. Une preuve alternative
en est donn\'ee dans l'Annexe \ref{Erratum} (cf.   \Err(i) Lemme~\ref{cpqrx}).

Pour $P$ et $R$ dans $\ESP$ tels que $P\subset R$, on d\'efinit la fonction 
$\Gamma^R_P$\index{Gammarp@$\Gamma^R_P$} sur $\ag_0\times\ag_0$ par
\begin{equation*}\Gamma^R_P(H,X)=
\sum_{\{Q\mid P\subset Q\subset R\}}(-1)^{a_Q-a_R}\tau_P^Q(H)\wh{\tau}_Q^{R}(H-X)\ptf\end{equation*}
D'apr\`es \cite[1.8.3]{LW} la projection dans $\ag_P^R$ 
du support de la fonction $H\mapsto\Gamma_P^R(H,X)$ est 
une union d'ensembles $C(P,Q,R,X)$ et est donc compacte. 
Pr\'ecis\'ement, il existe une constante $c>0$ telle que si 
$\Gamma_P^R(H,X)\neq 0$ on a
\begin{equation*}\lVert  H_P^R\rVert \leq c\lVert  X_P^R\rVert \end{equation*}
o $ X\mapsto X_P^R$ 
est la projection sur l'orthogonal de $\ag_0^P\oplus\ag_R$. 
De plus, si $P$ est standard et $X$ est r\'egulier, on a
\begin{equation*}\Gamma_P^R(H,X)=\tau_P^R(H)\phi_P^R(H-X)\ptf\end{equation*}

Soit $M\in\ESL$. Pour une famille $M$-orthogonale $\XX= (X_P)$ et 
pour $R\in\ESF(M)$, on note $\Gamma_M^R(\cdot ,\XX)$
\index{GammaRMX@$\Gamma^R_M(\cdot,\XX)$} la fonction 
sur $\ag_0$ d\'efinie par
\begin{equation*}\Gamma_M^R (H,\XX)=\sum_{P\in\ESF^R(M)}(-1)^{a_P-a_R}\wh{\tau}^R_P(H-X_P)\ptf\end{equation*}
D'apr\`es \cite[1.6.5]{LW}, si la famille $\XX$ est r\'eguli\`ere, alors $\Gamma_M^R(\cdot ,\XX)$ 
est la fonction caract\'eristique de l'ensemble 
des $H\in\ag_0$ dont la projection $H^R_M$ dans $\ag^R_M$ appartient \`a l'enveloppe 
convexe des points $X_P^R$ pour $P\in\ESP^R(M)$. En g\'en\'eral, la projection du support de la fonction 
$\Gamma^R_M(\cdot,\XX)$ dans $\ag_M^R$ est compacte \cite[1.8.5]{LW}. 
Pr\'ecis\'ement, il existe une constante $c>0$ (ind\'ependante de la famille $\XX$) tel que pour 
tout $H\in\ag_0$ tel que $\Gamma^R_M(H,\XX)\neq 0$, on ait
\begin{equation*}\lVert  H^R_M\rVert \leq c\sup_{P\in\ESP^R(M)}\lVert X_P^R\rVert \ptf\end{equation*}
On se limite maintenant au cas $R=G$. 
\begin{lemma}\label{gammaphi} Soit $Q\in \ESP(M)$. On a
\begin{equation*}\Gamma^G_M(H,\XX)= \sum_{P\in \ESP(M)}(-1)^{a(P,Q)}\phi^G_{P,Q}(H- X_P)\ptf\end{equation*}
\end{lemma}
\begin{proof} Ceci r\'esulte de \cite[1.8.7~(2)]{LW}.
\end{proof} 

Pour les (nombreuses) autres \'egalit\'es reliant les fonctions $\tau_P^Q$, $\wh{\tau}_P^{\hspace{0.7pt}Q}$, 
$\phi_{P,Q}^G$, $\Gamma_P^Q$ et $\Gamma_M^Q$, 
on renvoie \`a \cite[1.7, 1.8]{LW} et \cite[1.3]{W}. 

\label{epsi}\label{les fonctions gamma}
Pour $P,\mskip 2mu Q\in\ESP$ tels que $P\subset Q$ les fonctions m\'eromorphes 
 $\epsilon^Q_P$ sur $\ag_{0,\CM}^*$ sont d\'efinies dans \cite[1.9]{LW}. 
On rappelle leur d\'efinition:

\begin{definition}\label{epspq}
On munit l'espace vectoriel $\ag_P^Q$ d'une mesure de Haar. 
Pour $\Lambda\in \ag_{0,\CM}^*$ en dehors des murs, on pose
\begin{equation*}\epsilon_P^Q\index{eapsilonpq@$\epsilon_P^Q$}
(\Lambda) = \vol (\ag_P^Q/\ZM(\check{\Delta}_P^Q))
\prod_{\alpha\in\Delta_P^Q}\langle \Lambda,\check{\alpha}\rangle^{-1}\end{equation*}
o $\ZM(\check{\Delta}_P^Q)$ d\'esigne le r\'eseau de $\ag_P^Q$ engendr\'e par $\check{\Delta}_P^Q$, 
et $\ag_P^Q/\ZM(\check{\Delta}_P^Q)$ est muni de la mesure quotient de la mesure sur $\ag_P^Q$ 
par la mesure de comptage sur $\ZM(\check{\Delta}_P^Q)$. 
\end{definition}
On observera que, si $\langle \Re(\Lambda),\alpha \rangle>0$ pour tout $\alpha \in \Delta_P^Q$, on a
\begin{equation*}\epsilon_P^Q(\Lambda) =\int_{\agp^Q}
\phi_P^Q(H)e^{\langle\Lambda,H\rangle}\dd H\ptf\end{equation*}
Pour $X_P\in \ag_P$ on pose aussi\index{eapsilonpqX@$\epsilon_P^{Q,X_P}$}
\begin{equation*}\epsilon_P^{Q,X_P}(\Lambda) =\int_{\agp^Q}
\phi_P^Q(H-X_P)e^{\langle\Lambda,H\rangle}\dd H= 
e^{\langle \Lambda,X_P^Q\rangle}\epsilon_P^Q(\Lambda)\ptf\end{equation*}
En sommant sur des r\'eseaux au lieu d'espaces vectoriels 
on d\'efinit, comme en \cite[1.5]{W}, des variantes $\varepsilon_P^{Q,X_P}(Z;\Lambda)$
des fonctions $\epsilon_P^{Q,X_P}$. Pour $Z\in\ESA_Q $, 
on pose\index{AmbESAPQZ@$\ESA_P^Q(Z)$}
\begin{equation*}\ESA_P^Q(Z)\bydef\{H\in\ESA_P : H_Q =Z\}\ptf\end{equation*}
Si $Z'\in\ESA_P$ est tel que
$Z'_Q=Z$, on a alors
\begin{equation*}\ESA_P^Q(Z) = Z' +\ESA_P^Q\ptf\end{equation*}
Pour $\Lambda\in\ag_{0,\CM}^*$, $Z\in\ESA_Q $ et $X_P\in\ag_P$, on pose
\begin{equation*}\varepsilon_P^{Q,X_P}\index{eapsilonvarpqXZ@$\varepsilon_P^{Q,X_P}(Z;\cdot)$}
(Z;\Lambda)=
\sum_{H\in\ESA_P^Q(Z)}\phi^Q_{P}(H-X_P)e^{\langle \Lambda,H\rangle}\ptf\end{equation*}

\begin{lemma}\label{serievareps}
La s\'erie d\'efinissant $\varepsilon_P^{Q,X_P}(Z;\Lambda)$
est absolument convergente si $\langle\Re \Lambda,\check{\alpha}\rangle >0$ pour tout 
$\alpha\in\Delta^Q_P$.
Dans ce domaine, la s\'erie ne d\'epend que de la projection de $\Lambda$ dans 
$\ag^{*}_{P,\CM}/ \ESA^{\vee}_P$. Pour $X_P\in \ag_{P,\QM}$, 
la fonction \begin{equation*}\Lambda\mapsto\varepsilon_P^{Q,X_P}(Z;\Lambda)\end{equation*}se prolonge m\'eromorphiquement 
\`a tout $\Lambda\in\ag^*_{0,\CM}$.
\end{lemma}

\begin{proof}
Montrons le prolongement m\'eromorphe, pour $X_P\mskip -2mu\in\ag_{P,\QM}$. 
La preuve ci-apr\`es est classique (cf. \cite{A9} et \cite{W}). 
Pour tout entier $k\geq 1$ on pose
\begin{equation*}\ESD_k=k\mun \ESD \qquad \hbox{o}\qquad\ESD=\ZM(\check{\Delta}^Q_P)
\subset \ag^Q_P\ptf\end{equation*}
Choisissons $k$ tel que $\ESD_k$ contienne $\ESA_P^Q$ et $(X_P-Z')^Q$. 
Le r\'eseau dual $\ESD_k^\vee$, form\'e des $\Lambda\in\ag^{Q,*}_P\otimes\CM$ 
tels que $\langle \Lambda,Y\rangle\in {2i\pi}Z$ pour tout 
$Y\in\ESD_k$, v\'erifie l'inclusion
\begin{equation*}\ESD_k^\vee\subset\ESA_P^{Q,\vee}\ptf\end{equation*}
Consid\'erons les fonctions m\'eromorphes
\footnote{Dans \cite{A9}, cette fonction est not\'ee $(\theta^Q_{P,k^{-1}})^{-1}$.}
\begin{equation*}\varepsilon^Q_{P,k}(\Lambda)= \prod_{\alpha\in\Delta^Q_P}
 (1- e^{-k\mun\langle\Lambda,{ \check{\alpha}}\rangle})^{-1}
  \ptf\end{equation*}
 L'ensemble des $H\in\ESD_k$ tels que $\phi^Q_P(H)=1$ est celui form\'e des 
\begin{equation*}\sum_{\alpha\in\Delta^Q_P}{n_\alpha }{ k}\mun \check{\alpha}\end{equation*}
pour des entiers $n_\alpha\leq 0$. 
Pour $\Lambda\in\ag^*_{0,\CM}$ 
tel que $\langle\Re \Lambda,\check{\alpha}\rangle >0$ pour tout $\alpha\in\Delta_P^Q$,
on a donc: 
\begin{equation*}\varepsilon^Q_{P,k}(\Lambda)=\sum_{H\in\ESD_k}\phi^Q_P(H)e^{\langle \Lambda,H\rangle}\ptf\end{equation*}
Par inversion de Fourier sur le groupe ab\'elien fini $\ESA_P^Q\backslash \ESD_k$, on obtient
\begin{align*}
\varepsilon_P^{Q,X_P}(Z;\Lambda) &=&\frac{e^{\langle \Lambda,Z'\rangle}}{ [\ESD_k:\ESA_P^Q]}
\sum_{\nu\in \ESA_P^{Q,\vee}/\ESD_k^\vee}
\sum_{H\in\ESD_k}\phi^Q_P(H+Z'-X_P)e^{\langle \Lambda+\nu,H\rangle}\\
& =&\frac{e^{\langle \Lambda,Z'\rangle}}{ [\ESD_k:\ESA_P^Q]}
\sum_{\nu\in \ESA_P^{Q,\vee}/\ESD_k^\vee}
\sum_{H\in\ESD_k}\phi^Q_P(H)e^{\langle\Lambda+\nu, H+(X_P-Z')^Q\rangle}
\end{align*}
et donc
\begin{equation*}\varepsilon_P^{Q,X_P}(Z;\Lambda)=
\frac{e^{\langle \Lambda,Z'\rangle}}{ [\ESD_k:\ESA_P^Q]}
\sum_{\nu\in \ESA_P^{Q,\vee}/\ESD_k^\vee}
e^{\langle \Lambda+\nu,(X_P-Z')^Q\rangle}\varepsilon^Q_{P,k}(\Lambda+\nu)\ptf\end{equation*}
Le lemme en r\'esulte.
\end{proof}

\begin{lemma}\label{aussi} 
Le lemme \ref{serievareps} reste vrai pour tout $X_P\in \ag_P $. 
\end{lemma}
\begin{proof} Notons $\ES{Z}$ le r\'eseau de $\RM$ engendr\'e par les $\langle \varpi,H \rangle$ pour $\varpi \in \hat{\Delta}_P^Q$ et $H\in \ESA_P$. 
Pour $\varpi\in \hat{\Delta}_P^Q$, $X_P\in\agp$ et $H\in \ESA_P$, posons
$x_\varpi=\langle \varpi,X_P\rangle\in \RM$ et $h_\varpi= \langle \varpi,H\rangle\in \ES{Z}$.
Notons $\hat{\Delta}_1$ l'ensemble des $\varpi \in \hat{\Delta}_P^Q$ avec $x_\alpha\in\ES{Z}$.
Il existe un $Y_P\in\ag_{P,\QM}$ tel que les coordonn\'ees $y_\varpi=\langle \varpi,Y_P\rangle$ v\'erifient:
\begin{itemize}
\item $y_\varpi=x_\varpi$ pour tout $\varpi\in\hat{\Delta}_1$;
\item $(y_\varpi-h_\varpi)(x_\varpi-h_\varpi)>0$ pour tout $\varpi \in \hat{\Delta}_P^Q\smallsetminus \hat{\Delta}_1$ 
et tout $H\in \ESA_P$.
\end{itemize}
Pour un tel $Y_P$ on a
\begin{equation*}\phi^Q_{P}(H-Y_P)=\phi^Q_{P}(H-X_P)\end{equation*}
et donc 
\begin{equation*}\varepsilon_P^{Q,Y_P}(Z;\Lambda)=\varepsilon_P^{Q,X_P}(Z;\Lambda)\ptf\end{equation*}
\end{proof}
Pour $P,\mskip 2mu  Q \in \ESP(M)$, $Z\in \ESA_Q$ et $X_P\in \ag_P$, on pose
\begin{equation*}\varepsilon^{G,X_P}_{P,Q}(Z;\Lambda)=
\sum_{H\in \ESA_P^G(Z)}\phi^G_{P,Q}(H-X_P)e^{\langle H,\Lambda \rangle}\end{equation*}
la s\'erie \'etant absolument convergente si $\langle \Re \Lambda,\check{\alpha}\rangle >0$ 
pour tout $\alpha\in \Delta_Q$. Elle ne d\'epend que de la projection de 
$\Lambda$ dans $\ag_{P,\CM}^*/\ESA_P^\vee$. 

\begin{lemma} \label{app}\label{extres} Pour $X_P\in \ag_{P}$, la fonction 
$\Lambda\mapsto \varepsilon^{G,X_P}_{P,Q}(Z;\Lambda)$ 
ne d\'epend que de l'image de $\Lambda$ dans $\ag_{P,{\CM}}^*/\ESA_P^\vee$. Elle 
se prolonge m\'eromorphiquement \`a tout $\Lambda\in \ag^*_{0,\CM}$, et on a l'\'egalit\'e
\begin{equation*}\varepsilon^{G,X_P}_{P,Q}(Z;\Lambda)=(-1)^{a(P,Q)}\varepsilon^{G,X_P}_P(Z;\Lambda)\ptf\end{equation*}
\end{lemma}
\begin{proof} Compte tenu de \ref{aussi}, c'est l'assertion (2) de \cite[1.5]{W}.
\end{proof}

Soit $\XX=(X_P)_{P\in\ESF(M)}$ une famille $M$-orthogonale. On pose\footnote{L'indice $F$ et la lettre $Z$ indiquent que l'on somme sur le translatŽ 
$\ESA_M^Q(Z)=Z'+ \ESA_M^Q$ du rŽseau $\ESA_M^Q$ de $\ag_M^Q$; ˆ ne pas confondre avec l'intŽgrale 
$\int_{\ag_M^G}\Gamma_M^Q(H,\XX)e^{\langle \Lambda, H \rangle}\dd H$ (cf. \cite[1.9.3]{LW}) qui appara"tra dans la preuve de 
\ref{passlim}. Idem pour l'expression $\bscMF^{Q,\XX}(Z;\Lambda)$, voir plus loin. }:
\begin{equation*}\gammaMF^{Q,\XX}\index{gbmmamfqx@$\gammaMF^{Q,\XX}$}
(Z;\Lambda)=\sum_{H\in\ESA_M^Q(Z)}
\Gamma^Q_M(H,\XX)e^{\langle \Lambda,H\rangle}\ptf\end{equation*}

\begin{lemma}\label{gamma-fini}
 La s\'erie
\begin{equation*}\sum_{H\in\ESA_M^Q(Z)}\Gamma^Q_M(H,\XX)e^{\langle \Lambda,H\rangle}\end{equation*}
est une somme finie. 
La fonction $\Lambda\mapsto\gammaMF^{Q,\XX}(Z;\Lambda)$ 
est une fonction enti\`ere de $\Lambda\in\ag^*_{0,\CM}$ qui ne d\'epend que de l'image de $\Lambda$ 
dans $\ag_{M,{\CM}}^*/\ESA_M^\vee$.
Pour $\Lambda$ en dehors des murs,
on a l'identit\'e suivante:
\begin{equation*}\gammaMF^{Q,\XX}(Z;\Lambda)=\sum_{P\in\ESP^Q(M)}\varepsilon^{Q,X_P}_P(Z;\Lambda)\ptf\end{equation*}
\end{lemma} 
\begin{proof} 
La compacit\'e de la projection sur $\ag_M^Q$
du support de la fonction \begin{equation*}H\mapsto \Gamma^Q_M(H,\XX)\end{equation*}(\cite[1.8.5]{LW}) implique que la s\'erie
d\'efinissant $\gammaMF^{Q,\XX}$ est une somme finie. Elle d\'efinit donc une fonction enti\`ere.
 Pour la seconde assertion on invoque
l'expression \ref{gammaphi}\footnote{On observera que l'identit\'e \ref{gammaphi}, qui r\'esulte de \cite[1.8.7]{LW},
remplace (avantageusement) la d\'ecomposition \cite[1.9.3~(3)]{LW}
dont ni la formulation ni la preuve ne s'\'etendent au cas des corps de fonctions:
en effet, les murs peuvent contenir
des points du r\'eseau qui donnent alors une contribution non nulle.}
de $\Gamma_M$ au moyen des $\phi_{P,Q}$ et \ref{app}.
\end{proof}

Pour $P\in \ESP^Q(M)$ et $\lambda\in \ag_{P,\CM}^*$, notons $d(\lambda)$ le cardinal de
l'ensemble des $\alpha\in\Delta^Q_P$ tels que $\langle\lambda,\check\alpha\rangle\in 2ik\pi\ZM$.

\begin{lemma} \label{croisspol} 
On suppose que la famille $M$-orthogonale $\XX$ est rationnelle. 
Choisissons un entier $k$ tel que, pour tout $P\in\ESP^Q(M)$, le r\'eseau
$\ESD_k$ dans $\ag_M^Q$ 
contienne $\ESA_M^Q$ et $(X_P-Z')^Q$. La valeur en $\Lambda$ de la fonction 
$\gamma_{M,F}^{Q,\XX}(Z; \Lambda)$ peut s'\'ecrire
\begin{equation*}\gamma_{M,F}^{Q,\XX}(Z; \Lambda)=\sum_{P\in \ESP^Q(M)}
\sum_{\nu} p_{P,\Lambda+\nu}(X_P^Q)e^{\langle\Lambda+\nu,X_P^Q\rangle}\end{equation*}
o les $\nu$ varient dans $\ESA_M^{Q,\vee}/\ESD_k^\vee$ et les $p_{P,\lambda}$ 
sont des polyn™mes en $X_P^Q$ de degr\'e $d(\lambda)$.
\end{lemma}
\begin{proof} Pour $\Lambda$ en dehors des murs, on a
\begin{equation*}\gamma_{M,F}^{Q,\XX}(Z;\Lambda)=
\sum_{P\in\ESP^Q(M)}\varepsilon_P^{Q,X_P}(Z;\Lambda)\ptf\end{equation*}
 On a vu dans la preuve de \ref{serievareps} que
\begin{equation*}\varepsilon_P^{Q,X_P}(Z;\Lambda)=
\frac{e^{\langle \Lambda,Z'\rangle}}{ [\ESD_k:\ESA_M^Q]}
\sum_{\nu\in \ESA_M^{Q,\vee}/\ESD_k^\vee}
e^{\langle \Lambda+\nu,(X_P-Z')^Q\rangle}\varepsilon^Q_{P,k}(\Lambda+\nu)\ptf\end{equation*}
Fixons $\lambda = \Lambda + \nu$.
Pour $t\in\RM$ et $\xi$ en position g\'en\'erale,
on dispose du d\'eveloppement de Laurent au voisinage de $t=0$ des fonctions $\varepsilon^Q_{P,k}(t\xi + \lambda)$. 
On rappelle que
\begin{equation*}\varepsilon^Q_{P,k}(\lambda)= \prod_{\alpha\in\Delta^Q_P}
 (1- e^{-k\mun\langle\lambda,{\check{\alpha}}\rangle})^{-1}
  \end{equation*}
et donc
\begin{equation*}e^{\langle t\xi + \lambda ,(X_P-Z')^Q\rangle}\varepsilon^Q_{P,k}(t\xi+\lambda)=
t^{-d(\lambda)}e^{\langle t\xi,X_P^Q\rangle}f_P(t,\xi,\lambda)e^{\langle \lambda,X_P^Q\rangle}\end{equation*}
o $d(\lambda)$ est le nombre de racines $\alpha\in\Delta^Q_P$ telles que 
$e^{k\mun\langle\lambda,{\check{\alpha}}\rangle}=1$
c'est-\`a-dire $\langle\lambda,\check\alpha\rangle\in 2ik\pi\ZM$ et o, pour $\lambda$ et $\xi$ fix\'e, $f_P(t,\xi,\lambda)$ 
est une fonction de $t$ lisse au voisinage de $t=0$, ind\'ependante de $X_P$,
v\'erifiant $f_P(0,\xi,\lambda)\ne0$. 
La d\'eriv\'ee par rapport \`a $t$ d'ordre $d(\lambda)$ de 
\begin{equation*}e^{\langle t\xi,X_P^Q\rangle}f_P(t,\xi,\lambda)\end{equation*}est un polyn™me en $X_P^Q$
de degr\'e $d(\lambda)$. Le terme de degr\'e z\'ero dans le d\'eveloppement de Laurent au voisinage de $t=0$ de la fonction
\begin{equation*}\frac{e^{\langle t\xi +\lambda,Z'\rangle}}{ [\ESD_k:\ESA_M^Q]}
e^{\langle t\xi +\lambda,(X_P-Z')^Q\rangle}\varepsilon^Q_{P,k}(t\xi +\lambda)\end{equation*}
 est donc de la forme
$ p_{P,\lambda}(X_P^Q)e^{\langle\lambda,X_P^Q\rangle}$.
Comme d'apr\`es \ref{gamma-fini} la fonction
\begin{equation*}t\mapsto\gamma_{M,F}^{Q,\XX}(Z; \Lambda+t\xi)\end{equation*}
est lisse, les parties polaires se compensent
\footnote{On remarquera que contrairement au cas des corps de nombres les polyn™mes obtenus
ne sont en g\'en\'eral pas homog\`enes.}.
\end{proof}

Soit $\bsc=({\bs c}(\cdot,P))$ une $(G,M)$-famille. Pour $Q\in\ESF(M)$, 
$Z\in\ESA_Q$ et une famille $M$-orthogonale $\XX= (X_P)$, on note 
$\bscMF^{Q,\XX}(Z;\cdot)$ la fonction d\'efinie pour $\Lambda\in \hag_0$ en dehors des murs
 par\footnote{Notons que cette fonction ne d\'epend que de la $(Q,M)$-famille 
$(\bsc(\cdot, P))_{P\in\ESF^Q(M)}$ et de la famille $(Q,M)$-orthogonale $(X_P)_{P\in \ESF^Q(M)}$ 
d\'eduites de $\bsc$ et $\XX$ par restriction.}\index{cMFQXZ@$\bscMF^{Q,\XX}(Z;\cdot)$}
\begin{equation*}\bscMF^{Q,\XX}(Z;\Lambda)=
\sum_{P\in\ESP^Q(M)}\varepsilon_P^{Q,X_P}(Z;\Lambda)\bsc(\Lambda,P)\ptf\end{equation*}

\begin{proposition}\label{lissb}
Soient $Q\subset R$ deux sous-groupes paraboliques dans $\ESF(M)$.
Consid\'erons une $(R,M)$-famille {\rmfamily p\'eriodique} $\bsc$ associ\'ee \`a une fonction $\mf$
\`a d\'ecroissance rapide sur $\ESH_{R,M}$, et une famille $M$-othogonale 
$\XX$. Alors,
\begin{equation*}\bscMF^{Q,\XX}(Z;\Lambda)=
\sum_{\UU\in\ESH_{R,M}}\mf(\UU)\gammaMF^{Q,\XX+\UU}(Z+U_Q;\Lambda)\end{equation*}
et la fonction \begin{equation*}\Lambda\mapsto\bscMF^{Q,\XX}(Z; \Lambda)\end{equation*}
est une fonction lisse sur $\bsmu_M$. 
\end{proposition}

\begin{proof}
On exprime la $(R,M)$-famille $\bsc$ au moyen de la fonction $m$: pour $P\in \ESP(M)$ et $\Lambda \in \bsmu_M$ on a
\begin{equation*}\bsc(\Lambda,P)=\sum_{\UU\in\ESH_{R,M}}
 e^{\langle\Lambda,U_P\rangle}\mf(\UU)\end{equation*}
et donc
\begin{equation*}\bscMF^{Q,\XX}(Z;\Lambda)=\sum_{P\in\ESP^Q(M)}
\sum_{\UU\in\ESH_{R,M}} \mf(\UU)e^{\langle\Lambda,U_P\rangle}\varepsilon^{Q,X_P}_P(Z;\Lambda)\ptf\end{equation*}
Fixons un $P'\in \ESP(M)$. Pour $\Lambda$ dans le c™ne positif associ\'e \`a $P'$ on a
\begin{equation*}\varepsilon_P^{Q,X_P}(Z;\Lambda)=(-1)^{a(P,P')}
\sum_{H\in\ESA_M^Q(Z)}\phi^Q_{P,P'}(H-X_P)e^{\langle \Lambda,H\rangle}\ptf\end{equation*}
Donc $\bscMF^{Q,\XX}(Z;\Lambda)$ est \'egal \`a
\begin{equation*}\sum_{\UU\in\ESH_{R,M}} \mf(\UU)\sum_{P\in\ESP^Q(M)}(-1)^{a(P,P')}
\sum_{H\in\ESA_M^Q(Z)}\phi^Q_{P,P'}(H-X_P)e^{\langle \Lambda,H+U_P\rangle}\end{equation*}
qui est encore \'egal \`a
\begin{equation*}\sum_{\UU\in\ESH_{R,M}} \mf(\UU)
\sum_{P\in\ESP^Q(M)}(-1)^{a(P,P')}
\sum_{H\in\ESA_M^Q(Z+U_Q)}\phi^Q_{P,P'}(H-(X_P+U_P))e^{\langle \Lambda,H\rangle}\end{equation*}
Donc, vu \ref{gammaphi},
\begin{equation*}\bscMF^{Q,\XX}(Z;\Lambda)=\sum_{\UU\in\ESH_{R,M}} \mf(\UU)
\sum_{H\in\ESA_M^Q(Z+U_Q)}\Gamma_M^Q(H,\XX+\UU)e^{\langle \Lambda,H\rangle}\end{equation*}
et on obtient la formule de l'\'enonc\'e gr‰ce \`a \ref{gamma-fini}. 
On observe maintenant que pour $\Lambda\in\bsmu_M$, la fonction sur $\ESH_{R,M}$ 
\begin{equation*}\UU\mapsto \gamma_{M,F}^{\XX+\UU}(Z+U_Q;\Lambda)= 
\sum_{H\in\ESA_M^Q(Z+U_Q)}\Gamma_M^Q(H,\XX+\UU)e^{\langle \Lambda ,H \rangle}\end{equation*}
est major\'ee par
\begin{equation*}\UU\mapsto\sum_{H\in\ESA_M^Q(Z+U_Q)} \vert \Gamma_M^Q(H, \XX + \UU)\vert \ptf\end{equation*}
La fonction \begin{equation*}H \mapsto \vert \Gamma_M^Q(H, \XX + \UU) \vert 
\qquad\hbox{pour}\qquad H\in\ESA_M\end{equation*}
ne prend qu'un nombre fini de valeurs enti\`eres, born\'ees ind\'ependamment de $\UU$. Elle est
\`a support compact inclus dans une boule de rayon major\'e par un polyn™me en $\UU$. 
Sa somme sur $H$ est donc 
born\'ee par un polyn™me en $\UU$. Par ailleurs $m$ est \`a d\'ecroissance rapide 
ce qui prouve la convergence absolue, uniforme en $\Lambda$, de la s\'erie en $\UU$. 
L'expression $\bscMF^{Q,\XX}(Z;\Lambda)$ est donc une fonction continue en $\Lambda$.
Plus g\'en\'eralement, les d\'eriv\'ees en $\Lambda$ correspondent \`a des s\'eries analogues 
o l'op\'erateur diff\'erentiel sur $\bsc$ se traduit en transform\'ee de Fourier par la multiplication 
par un polyn™me en $\UU$.
On a encore la convergence uniforme des s\'eries vu la d\'ecroissance rapide de $m$, d'o la lissit\'e de la 
fonction $\Lambda \mapsto \bscMF^{Q,\XX}(Z;\Lambda)$.
\end{proof}

\begin{lemma} \label{croisspolb} On reprend les hypoth\`eses de \ref{croisspol}. 
La valeur en $\Lambda$ de la fonction 
$\bsc_{M,F}^{Q,\XX}(Z; \Lambda)$
 peut s'\'ecrire
\begin{equation*}\bsc_{M,F}^{Q,\XX}(Z; \Lambda)=\sum_{P\in \ESP^Q(M)}\sum_{\nu} q_{P,\Lambda+\nu}(X_P^Q)
e^{\langle\Lambda+\nu,X_P^Q\rangle}\end{equation*}
o les $\nu$ varient dans $\ESA_M^{Q,\vee}/\ESD_k^\vee$,
les $q_{P,\Lambda + \nu}$ sont des polyn™mes en $X_P^Q$ de degr\'e inf\'erieur ou \'egal \`a
$d(\Lambda+\nu)$.
\end{lemma}

\begin{proof} La preuve est analogue \`a celle de \ref{croisspol}: les fonctions
$f_P(t,\xi,\lambda)$ doivent tre remplac\'ees par les
\begin{equation*}g_P(t,\xi,\Lambda+\nu)=f_P(t,\xi,\Lambda+\nu)c(\Lambda+t\xi,P)\ptf\end{equation*}
Il convient ensuite d'observer que l\`a aussi les singularit\'es des $\varepsilon_P^{Q,X_P}(Z;{\Lambda})$
se compensent puisque la fonction \begin{equation*}t\mapsto\bscMF^{Q,\XX}(Z; \Lambda+t\xi)\end{equation*}
est lisse d'apr\`es \ref{lissb}. Mais les degr\'es des polyn™mes peuvent s'abaisser l\`a o les
$\bsc(\Lambda,P)$ ont des z\'eros.
\end{proof}

Pour une $(G,M)$-famille $\bsc$ (p\'eriodique ou non) et une famille $M$-orthogonale quelconque $\XX$, 
on pose pour 
$\Lambda \in \hag_0$ en dehors des murs\footnote{Rappelons que la fonction $\epsilon_P^{Q,X_P}({\Lambda})$ 
d\'epend du choix d'une mesure de Haar 
sur $\ag_P^Q=\ag_M^Q$. On prend bien sžr la mme mesure pour tous les $P\in \ESP^Q(M)$.}
\begin{equation*}\bsc_M^{Q,\XX}(\Lambda)= \sum_{P\in \ESP^Q(M)}\epsilon_P^{Q,X_P}(\Lambda) \bsc(\Lambda,P)\ptf\end{equation*}
Cela d\'efinit une fonction lisse sur $\hag_0$ (les singularit\'es des fonctions $\epsilon$ sur les murs 
sont compens\'ees par des annulations dues aux propri\'et\'es des $(G,M)$-familles).
Si $\XX$ est la famille triviale, on \'ecrit simplement
\begin{equation*}\bsc_M^G(\Lambda) = \bsc_M^{G,\XX=0}(\Lambda)\ptf\end{equation*}
Pour $\UU \in \mathfrak{H}_M$, on note $\bsc(\UU)$ la $(G,M)$-famille d\'efinie par
\begin{equation*}\bsc(\UU;\Lambda,P)= e^{\langle \Lambda,U_P\rangle} \bsc(\Lambda,P)\ptf\end{equation*}
Pour $\Lambda\in \wh{\ag}_0$ en dehors des murs, on pose
\begin{equation*}\bsc_M^Q(\UU;\Lambda)= \sum_{P\in \ESP^Q(M)}\epsilon_P^Q(\Lambda) \bsc(\UU;\Lambda,P)\ptf\end{equation*}
On a\footnote{Notons que dans \cite{W}, c'est la $(G,M)$-famille $\bsc(\UU^G)$ qui est not\'ee \og $\bsc(\UU)$\fg. }
\begin{equation*}\bsc_M^{Q}(\UU;\Lambda)= e^{\langle \Lambda , U_Q \rangle} \bsc_M^{Q}(\UU^Q;\Lambda)\end{equation*}
o $\UU^Q$ est la famille $M$-orthogonale $(U_P^Q)$ dans $M_Q$. 
Plus g\'en\'eralement, pour $\XX$ et $\UU$ deux familles $M$-orthogonales quelconques, 
on pose \begin{equation*}\bsc_M^{Q,\XX}(\UU;\Lambda)\bydef 
\sum_{P\in\ESP^Q(M)}\epsilon_P^{Q,X_P}({\Lambda}) \bsc(\UU;{\Lambda},P)\ptf\end{equation*}
 Observons que
 \begin{equation*}\bsc_M^{Q,\XX}(\UU;\Lambda) = \bsc_M^{Q}(\XX^Q + \UU;\Lambda)\ptf\end{equation*}
 Si $\bsc$ est une $(G,M)$-famille p\'eriodique et $\UU$ une famille
$M$-orthogonale rationnelle, la $(G,M)$-famille $\bsc(\UU)$ 
est p\'eriodique si et seulement si la famille $\UU$ est enti\`ere. Auquel cas, 
si $\bsc= \bsc_m$ pour une fonction 
$m$ \`a d\'ecroissance rapide sur $\ESH_M$, alors $\bsc(\UU)= \bsc_{{m'}}$ o 
\begin{equation*}m'(\VV)= m (\VV-\UU)\ptf\end{equation*}
Pour $\UU$ enti\`ere\footnote{On observera que la famille $\UU^Q$ est \`a priori seulement 
rationnelle.} et $\XX$ quelconque
on pose
\begin{equation*}\bscMF^{Q,\XX}(Z,\UU;{\Lambda})= \sum_{P\in\ESP^Q(M)}
\varepsilon_P^{Q,X_P}(Z;{\Lambda}) \bsc(\UU;\Lambda,P)\ptf\end{equation*}
On a
\begin{equation*}\bscMF^{Q,\XX}(Z,\UU;{\Lambda})= \bscMF^{Q,\XX+\UU}(Z+ U_Q;{\Lambda})\ptf\end{equation*}
On pose aussi
\begin{equation*}\gamma_{M,F}^{Q,\XX}(Z,\UU;\Lambda)\buildrel\mathrm{d\acute{e}f}\over{=} 
\gamma_{M,F}^{Q,\XX+\UU}(Z+ U_Q;\Lambda)=
\sum_{P\in \ESP^Q(M)}\varepsilon_P^{Q,X_P}(Z;\Lambda)e^{\langle \Lambda,U_P\rangle}\ptf\end{equation*}
Si $\XX = \TT$ pour un $T\in \ag_{0}$, on remplacera l'exposant $\TT$ par un simple 
$T$ dans les expressions ci-dessus. Avec cette convention on a le

\begin{corollary}\label{formule d'inversion pour les (G,M)-familles}
Soit $\bsc= \bsc_\mf$ une $(G,M)$-famille p\'eriodique donn\'ee par 
une fonction $\mf$ \`a d\'ecroissance rapide sur $\ESH_M$. Pour $T\in \ag_{0}$ on a
\begin{equation*}\bscMF^{Q,T}(Z;\Lambda)= \sum_{\UU\in \ESH_M}\mf(\UU) \gammaMF^{Q,T}(Z,\UU;\Lambda)\ptf\end{equation*}
\end{corollary}

\begin{proof}
C'est un corollaire de \ref{lissb}.
\end{proof}

Soit $\HH_{Q,M}$ le ${\RM}$--espace vectoriel form\'e des familles $M$--orthogonales dans $M_Q$ -- 
c'est-\`a-dire qu'on remplace les conditions sur 
$P\in \ESP(M)$ par des conditions sur $P\in \ESP^Q(M)$. Soit $\ESH_{Q,M}\subset \HH_{Q,M}$ 
le r\'eseau form\'e des familles qui sont enti\`eres. L'application naturelle (qui n'est 
\`a priori pas surjective)
\begin{equation*}\HH_M=\HH_{G,M}\rightarrow \HH_{Q,M}\end{equation*}
envoie $\ESH_M=\ESH_{G,M}$ dans $\ESH_{Q,M}$. Elle donne par dualit\'e une application
\begin{equation*}\wh\HH_{Q,M}\rightarrow \wh\HH_{M}\end{equation*}
qui se factorise en une application
\begin{equation*}\wh\ESH_{Q,M} \rightarrow \wh\ESH_{M}\ptf\end{equation*}
Toute fonction lisse $h$ sur $\wh\HH_M = i\HH_{M}^*$ 
d\'efinit donc par composition une fonction lisse $h_Q$ sur $\wh\HH_{Q,M}= i\HH_{Q,M}^*$, et si 
$h$ est p\'eriodique alors $h_Q$ l'est aussi. En ce cas, $h= \wh\mf$ pour une (unique) fonction \`a d\'ecroissance rapide $\mf$ 
sur le r\'eseau $\ESH_M$, et on note $m_Q$ la fonction \`a d\'ecroissance rapide sur le r\'eseau $\ESH_{Q,M}$ d\'efinie par 
$\wh{m_Q}= h_Q$. Cette fonction vaut $0$ en dehors de l'image de l'application 
$\ESH_{G,M} \rightarrow \ESH_{Q,M}$, et pour $\UU$ dans cette image on a
\begin{equation*}m_Q(\UU) = \sum_{\VV \in \ESH_{Q,M}^G(\UU)}m(\VV)\end{equation*}
o $\ESH_{Q,M}^G(\UU)\subset \ESH_{G,M}$ est la fibre au-dessus de $\UU$.

\begin{corollary}\label{deuxi\`eme formule d'inversion}
Soit $\bsc= \bsc_\mf$ une $(G,M)$-famille p\'eriodique donn\'ee par
 une fonction $\mf$ \`a d\'ecroissance rapide sur $\ESH_{G,M}$. Pour $T\in \ag_{0}$
on a
\begin{equation*}\bscMF^{Q,T}(Z;\Lambda)= \sum_{\UU\in \ESH_{Q,M}}{\mf_Q}(\UU) 
\gammaMF^{Q,T}(Z,\UU;\Lambda)\end{equation*}
o
\begin{equation*}\gammaMF^{Q,T}(Z,\UU;\Lambda)=\gammaMF^{Q,\UU(T)}(Z+ U_Q;\Lambda)\ptf\end{equation*}
\end{corollary}
\begin{proof} C'est encore un corollaire de \ref{lissb}.
\end{proof}
 Par inversion de Fourier ceci se reformule comme suit:
\begin{lemma}\label{TMGMF}
Soit $\XX$ une famille $(Q,M)$-orthogonale, et soit $\bsc$ 
une $(Q,M)$-famille p\'eriodique donn\'ee par une fonction \`a d\'ecroissance rapide 
$m$ sur $\ESH_{Q,M}$. Pour $Z\in \ESA_Q$ et $V\in \ESA_M$ on pose
\begin{equation*}\wh{\bsc}_{M,F}^{\mskip 2mu Q,\XX}(Z;V)= 
\int_{\bsmu_M} \bsc_{M,F}^{Q,\XX}(Z;\Lambda)e^{-\langle \Lambda ,V\rangle} \dd \Lambda\ptf\end{equation*}
On a alors
\begin{equation*}\wh{\bsc}_{M,F}^{\mskip 2mu Q,\XX}(Z;V)=
\sum_{\substack{\UU \in \ESH_{Q,M}\\ Z+ U_Q=V_Q}}m(\UU) \Gamma_M^Q(V,\XX + \UU)\ptf\end{equation*}
\end{lemma}
\begin{proof}
On sait d'apr\`es \ref{lissb} que
\begin{equation*}{\bsc}_{M,F}^{\mskip 2mu Q,\XX}(Z;V)= 
\sum_{\UU \in \ESH_{Q,M}} m(\UU){\gamma}_{M,F}^{\mskip 2mu Q,\XX+\UU}(Z + U_Q;V)\end{equation*}
 donc
\begin{equation*}\wh{\bsc}_{M,F}^{\mskip 2mu Q,\XX}(Z;V)= 
\sum_{\UU \in \ESH_{Q,M}} m(\UU) \wh{\gamma}_{M,F}^{\mskip 2mu Q,\XX+\UU}(Z + U_Q;V)\end{equation*}
et on observe que
\begin{equation*}\wh{\gamma}^{\mskip 2mu Q,\XX+\UU}_{M,F}(Z+U_Q;V)=
\left\{\begin{array}{ll}\Gamma_M^{Q}( V,\XX+\UU)
 & \hbox{si $ Z +U_Q=V_Q $}\\
0 & \hbox{sinon}
\end{array}\right.\ptf\end{equation*}
\end{proof}
%


 \section{L'ensemble PolExp} \label{l'ensemble PolExp}

Nous avons besoin de deux lemmes \'el\'ementaires. Faute de r\'ef\'erence nous en donnons une preuve.

\begin{lemma}\label{prepar}
On consid\`ere, pour $k=1,\ldots , m$, des nombres complexes $a_k$ et des nombres complexes $b_k$ deux \`a deux distincts 
de module $1$. On suppose que
\begin{equation*}\lim_{n\to+\infty}\sum_{k=1}^m a_k b_k^n=0\ptf\leqno(1)\end{equation*}
Alors les $a_k$ sont tous nuls. 
\end{lemma}

\begin{proof}
On le d\'emontre par r\'ecurrence sur $m$. Si $m=1$ c'est \'evident. 
Supposons-le vrai pour $m-1$. En posant $d_k=b_k/b_m$ et $c_k=a_k(1-d_k)$
la condition (1) implique
\begin{equation*}\lim_{n\to+\infty}\bigg(\sum_{k=1}^{m-1} a_k d_k^n-\sum_{k=1}^{m-1} a_k d_k^{n+1}\bigg)
=\lim_{n\to+\infty}\sum_{k=1}^{m-1} c_k d_k^n=0\ptf\end{equation*}
Les $b_k$ sont tous diff\'erents donc les $d_k$ sont tous diff\'erents.
L'hypoth\`ese de r\'ecurrence impose $c_k=0$ pour $1\le k\le m-1$.
Comme $1-d_k\ne0$ les $a_k$ sont nuls pour $1\le k\le m-1$; ceci implique $a_m=0$.
\end{proof}

Soit $\ag$ un espace vectoriel r\'eel de dimension finie, et soit $\ESR$ un r\'eseau de $\ag$.
On consid\`ere pour $T\in\ESR$ une combinaison lin\'eaire de produits de polyn™mes et d'exponentielles:
\begin{equation*}\phi(T)=\sum_{\nu\in E} p_\nu(T)e^{\langle \nu, T\rangle }\end{equation*}
o $E$ est un sous-ensemble fini de $\wh{\ESR}= \hag/ \ESR^\vee$ et les $p_\nu$ sont des polyn™mes sur $\ag$.

\begin{lemma} \label{unicit\'e}
Soit $C$ un c™ne ouvert non vide de $\ag$, et soit $T_\star\in \ESR$.
On suppose que $\phi(T)$ tend vers $0$ lorsque $\lVert T\rVert \mskip -2mu$ tend vers l'infini
pour $T$ dans $T_\star+(C\cap \ESR)$. Alors $p_\nu=0$ pour tout $\nu$.
\end{lemma}

\begin{proof} La preuve se fait par r\'ecurrence sur la dimension de $\ag$. 
Le cas de la dimension z\'ero est fourni par \ref{prepar}. 
Le r\'eseau $\ESR$ peut tre d\'ecompos\'e en une somme directe 
$\ESR=\ZM v\oplus\ESR_1$ avec 
$v\in C\cap\ESR$ primitif 
(c'est-\`a-dire que $v=nv_1$ avec $v_1\in\ESR$ et $n\in\ZM$ implique $n=\pm1$).
Consid\'erons $T=nv+T_1\in\ESR$ avec $n\in\NM$ et $T_1\in\ESR_1$. On peut \'ecrire
$\phi(nv+T_1)$ sous la forme
\begin{equation*}\phi(nv+T_1)=\sum_{\mu\in E(v)} q_\mu(n, T_1)b_\mu^n\end{equation*}
o les $b_\mu=e^{\langle\mu,v\rangle}$ sont des nombres complexes de module $1$ deux \`a deux distincts et
$E(v)$ est le quotient de $E$ d\'efini par la restriction \`a $\ZM v$. Les fonctions $q_\mu$ sont de la forme:
\begin{equation*}q_\mu(n, T_1)=\sum_{s=0}^{d_\mu} \sum_{\tau\in E_\mu} r_{s,\tau}(T_1)e^{\langle \tau,T_1\rangle}n^s\end{equation*}
o les $s$ sont entiers, les $r_{s,\tau}$ sont des polyn™mes sur $\ag_1$, l'espace vectoriel engendr\'e 
par $\ESR_1$, et $E_\mu$ est un sous-ensemble fini de $\wh{\ESR}_1 = \hag_1/\ESR_1^\vee$.
Soit $d$ est le degr\'e maximal des polyn™mes $q_\mu$ en $n$, et soit $a_\mu(T_1)$
 le coefficient de $n^{d}$ dans $q_\mu$: 
 \begin{equation*}a_\mu(T_1)\bydef \sum_{\tau\in E_\mu} r_{d,\tau}(T_1)e^{\langle \tau,T_1\rangle}\end{equation*}
 o, par convention, $r_{d,\tau}(T_1)=0$ si $d>d_\mu$. On a
\begin{equation*}\lim_{n\to+\infty}\bigg(n^{-d}\phi(nv+T_1)-\sum_\mu a_\mu(T_1) b_\mu^n\bigg)=0\ptf\leqno(2)\end{equation*}
On observe que pour $T_1$ fix\'e et $n$ assez grand on a
\begin{equation*}nv+T_1\in T_\star +( C\cap\ESR)\ptf\end{equation*}
Par hypoth\`ese \begin{equation*}\lim_{n\to+\infty}\phi(nv+T_1)=0\ptf\leqno(3)\end{equation*}
On d\'eduit de (2) et (3) que \begin{equation*}\lim_{n\to+\infty}\sum_\mu a_\mu(T_1) b_\mu^n=0\ptf\end{equation*}
D'apr\`es le lemme \ref{prepar} les $a_\mu(T_1)$ sont tous nul.
Par r\'ecurrence descendante sur le degr\'e on obtient que pour tout entier $s$ et tout $T_1\in\ESR_1$:
\begin{equation*}\sum_{\tau\in E_\mu}r_{s,\tau}(T_1)e^{\langle \tau,T_1\rangle }=0\ptf\end{equation*}
Par hypoth\`ese de r\'ecurrence sur la dimension de $\ag$ cela implique que les $r_{s,\tau}$ sont nuls.
On en d\'eduit que les $p_\nu$ le sont.
\end{proof}

Nous introduisons maintenant comme dans \cite[1.7]{W} l'ensemble PolExp:

\begin{definition}\label{polexp}
On note {\rmfamily PolExp} l'espace vectoriel des fonctions 
$\phi:\ag_{0,\QM}\rightarrow \CM$ telles que,
pour tout r\'eseau $\ESR$ de $\ag_{0,\QM}$, il existe 
 une famille index\'ee par les $\nu\in \wh\ESR$ de polyn™mes sur $\ag_0$ 
 \begin{equation*}T\mapsto p_{\ESR,\nu}(\phi,T)\end{equation*}
avec les propri\'et\'es suivantes:
\begin{itemize}
\item les $\nu\in \wh\ESR$ tels que $p_{\ESR,\nu}\neq 0$ sont en nombre fini;
\item pour $T\in\ESR$, on a l'\'egalit\'e 
$\phi(T)=\sum_{\nu\in \wh\ESR} p_{\ESR,\nu}(\phi,T)e^{\langle\nu,T\rangle}$.
\end{itemize}
\end{definition}

D'apr\`es \ref{unicit\'e}, les $p_{\ESR,\nu}$ sont uniquement d\'etermin\'es par une 
approximation de $\phi_\ESR=\phi\vert_{\ESR}$, sur l'intersection de $\ESR$ et 
 d'un c™ne ouvert non vide de $\ag_0$. Observons que si $\ESR$ et $\ESR'$ 
 sont deux r\'eseaux de $\ag_{0,\QM}$ tels que
 $\ESR\subset\ESR'$, pour $\nu\in \wh{\ESR}$ on a 
\begin{equation*}p_{\ESR,\nu}(\phi,T)=\sum_{\nu'\in \ESR^\vee/\ESR'^\vee}p_{\ESR'\mskip -2mu,\nu +\nu'}(\phi,T)\ptf\end{equation*}
 La famille $(p_{\ESR,\nu})_{\nu\in \wh\ESR}$ 
se d\'eduit donc de la famille $(p_{\ESR'\mskip -2mu,\nu'})_{\nu'\in \wh{\ESR'}}$.

\begin{proposition}\label{passlim}
Soient $\bsc$ une $(Q,M)$-famille p\'eriodique \`a valeurs scalaires, 
$\XX$ une famille $M$-orthogonale rationnelle, $Z\in \ESA_Q$ et $\Lambda\in \wh{\ag}_M$. 
On note $\mu$ l'image de ${\Lambda^Q}$ dans $\bsmu_M^Q= \bsmu_M/\bsmu_Q$. 
Soit $\ESR$ un r\'eseau de $\ag_{0,\QM}$, et pour $k\in \NM^*$ 
soit $\ESR_k= k^{-1}\ESR$. Alors:
\begin{enumerate}[(i)]
\item La fonction
$T\mapsto\phi( T)= \bscMF^{Q,{\XX(T)}}(Z;{\Lambda})$ appartient \`a $\mathrm{PolExp} $. 
\item Si $\mu\ne0$, il existe un entier 
$k_0\geq 1$, ne d\'ependant que de $\ESR$, tel que $p_{\ESR_k,0}(\phi,T)=0$ 
pour tout entier $k\geq k_0$.
\item Si $\mu=0$, i.e. ${\Lambda}\in \Lambda_Q + \ESA_M^\vee$, 
alors
\begin{equation*}\lim_{k\to +\infty} p_{\ESR_k,0}(\phi,T)={\vol}(\ESA_M^Q\backslash \ag_M^Q)\mun
e^{\langle \Lambda_Q,Z\rangle}\bscM^Q(\XX(T)^Q;\Lambda_Q)\end{equation*}
et en particulier, cette limite est ind\'ependante de $\ESR$.
Plus pr\'ecis\'ement, il existe un r\'eel $b>0$ ne d\'ependant que de $\ESR$, $\XX$ et $T$ 
tel que pour tout entier $k\geq 1$ on ait la majoration
\begin{equation*}\vert p_{\ESR_k,0}(\phi,T)-{\vol}(\ESA_M^Q\backslash \ag_M^Q)\mun
e^{\langle \Lambda_Q,Z\rangle}\bscM^Q(\XX(T)^Q;\Lambda_Q) \vert\leq b\mskip 2mu N_d(\bsc)\mskip 2mu k\mun\end{equation*}
o $N_d$ est la norme pour les $(G,M)$-familles p\'eriodiques d\'efinie en \ref{normes} et $d$
la dimension de $\ag_M^Q$.
\end{enumerate}
\end{proposition}
\begin{proof} 
L'assertion (i) est une cons\'equence imm\'ediate de \ref{croisspolb}.
Relevons $Z$ en un \'el\'ement $Z'$ de $\ESA_M$ et choisissons un entier $k'$ tel que, 
pour tout $P\in\ESP^Q(M)$, le r\'eseau $\ESD_{k'}$ de $\ag_M^Q$ 
contienne $\ESA_M^Q$ et $(X_P-Z')^Q$. Pour $P\in \ES{P}^Q(M)$, notons 
$[\ES{R}]_P^Q$ le r\'eseau de $\ag_M^Q$ image de $\ES{R}$ 
par l'application $T \mapsto [T]_P^Q =([T]_P)^Q$. 
On suppose $k$ assez grand de sorte que pour tout $P\in\ESP^Q(M)$ on ait
\begin{equation*}[\ESR_{k}]_P^{Q,\vee}\subset\ESD_{k'}^{\vee}\subset \ESA_M^{Q,\vee}\ptf\end{equation*}
 Avec les notations de \ref{croisspolb} on sait que:
\begin{equation*}p_{\ESR_k,0}(\phi;T)=\sum_{P\in \ES{P}^Q(M)}\sum_{\nu\in E_P} q_{P,\Lambda+\nu}(\brT{P}^Q)\end{equation*}
o
\begin{equation*}E_P=\{\nu\in \ESA_M^{Q,\vee}/[\ESR_{k}]_P^{Q,\vee}\mskip 2mu \mid \mskip 2mu \Lambda^Q+ \nu=0\}\ptf\end{equation*}
On voit que $E_P$ poss\`ede un unique \'el\'ement si $\Lambda^Q\in \ESA_M^{Q,\vee}$ c'est-\`a-dire si
$\mu=0$ et est vide sinon; (ii) en r\'esulte.
Lorsque ${\mu}=0$ on va esquisser une d\'emonstration de (iii), 
diff\'erente de celle de \cite[1.7, lemme~(ii)]{W}.
Pour all\'eger les notations on commence par traiter le cas $\Lambda_Q=0$. 
D'apr\`es la proposition \ref{lissb} on a
\begin{equation*}{\phi(T)}=\sum_{\UU\in\ESH_{M}}\mf(\UU)\gammaMF^{Q,\XX(T)}(Z, \UU;\Lambda)\end{equation*}
avec $\mf$ \`a d\'ecroissance rapide sur le r\'eseau $\ESH_M$ et 
\begin{equation*}\gammaMF^{Q,\XX(T)}(Z, \UU;\Lambda)= \gammaMF^{Q,(\XX+\UU)(T)}(Z+U_Q;\Lambda)\ptf\end{equation*}
Fixons $\UU \in \ESH_M$ et posons $\VV=\XX+\UU$. On rappelle que
\begin{equation*}\gammaMF^{Q,\VV(T)}(Z+ U_Q;\Lambda)= 
\sum_{H\in \ESA_M^Q(Z+U_Q)}\Gamma_M^Q(H,\VV(T))e^{\langle \Lambda ,H \rangle}\ptf\end{equation*}
Relevons $U_Q$ en un \'el\'ement $U'_Q$ de $\ESA_M$, et posons
$Z''=Z'+ U'_Q$. 
On obtient
\begin{equation*}\gammaMF^{Q,\VV(T)}(Z+U_Q;\Lambda)=e^{\langle \Lambda , Z'' \rangle}
\sum_{H \in \ESA_M^Q}\Gamma_M^Q(Z''+H,\VV(T))e^{\langle \Lambda, H \rangle}\ptf\end{equation*}
Notons $\ESR_M^Q$ l'image de $\ESR$ par l'application $H \mapsto H_M^Q$.
 On suppose $\ESR$ assez fin de sorte que $\ESR_M^Q$ contienne $\ESA_M^Q$ ainsi que
l'image $(Z'')^Q$ de $Z''$ dans $\ag_M^Q$. 
Par inversion de Fourier sur le groupe fini $\ESA_M^Q \backslash \ESR_M^Q$ on a
\begin{equation*}\gammaMF^{Q,\VV(T)}(Z+ U_Q;\Lambda)=\frac{e^{\langle \Lambda , Z''\rangle}}{[\ESR_M^Q: \ESA_M^Q]}
\sum_{\nu \in \ESA_M^{Q,\vee}/\ESR_M^{Q,\vee}} \sum_{H \in \ESR_M^Q}
\Gamma_M^Q(Z''+H,\VV(T))e^{\langle \Lambda+ \nu, H \rangle}\ptf\end{equation*}
Le polyn™me en $T$ attach\'e \`a $\Lambda=\nu=0$, que nous noterons $p_{\ESR,0}(\VV,T)$, vaut:
\begin{equation*}p_{\ESR,0}(\VV,T)=\frac{1}{[\ESR_M^Q: \ESA_M^Q]} \sum_{H \in \ESR_M^Q}\Gamma_M^Q(Z''+H,\VV(T))\ptf\end{equation*}
La restriction \`a $\ag_M^Q$ de la fonction \begin{equation*}H\mapsto\Gamma^Q_M(Z''+H,\VV(T))\end{equation*}
est, d'apr\`es \cite[1.8.4~(2), 1.8.3]{LW}, combinaison lin\'eaire \`a coefficients dans $\{-1,+1\}$ 
d'une famille finie de fonctions caract\'eristiques de polytopes du type $C(P,Q,R,X)$ qui sont 
convexes, born\'es (mais en g\'en\'eral non ferm\'es); en particulier elle est \`a support compact
de rayon born\'e par un polyn™me en $\VV(T)$.
Lorsque l'on remplace $\ESR$ par $\ESR_k=k\mun\ESR$ et que l'on fait tendre $k$ vers l'infini
 $p_{\ESR_k,0}(\VV,T)$ a pour limite une int\'egrale au sens de Riemann:
\begin{equation*}\lim_{k\to\infty}p_{\ESR_k,0}(\VV,T)=\vol(\ESA_M^Q\bsl\ag_M^Q)\mun
\int_{H\in\ag_M^Q}\Gamma^Q_M(H,\VV(T))\dd H\end{equation*}
soit encore
\begin{equation*}\lim_{k\to\infty}p_{\ESR_k,0}(\VV,T)=
\vol(\ESA_M^Q\bsl\ag_M^Q)\mun\gamma_M^Q(\VV(T);0)\ptf\end{equation*}
Le terme d'erreur est major\'e par le volume des hypercubes (les mailles du r\'eseau $k\mun\ESR_M^Q$)
rencontrant la fronti\`ere des polytopes; ces hypercubes sont inclus dans un voisinage tubulaire 
de la fronti\`ere des polytopes, de rayon $a/k$ o $a$ est une constante ne d\'ependant
que de la taille des mailles de $\ESR$. 
Le voisinage tubulaire a un volume born\'e par le produit de $a/k$ et de $r(\ESR,\VV(T))$ qui est
la somme des mesures des fronti\`eres des polytopes. Donc
\begin{equation*}\mid p_{\ESR_k,0}(\VV,T)-\vol(\ESA_M^Q\bsl\ag_M^Q)\mun\gamma_M^Q(\VV(T);0)\mid 
\le \frac{a}{k} r(\ESR,\VV(T))\ptf\end{equation*}
On passe de $\gamma_{M,F}^{Q,\XX}$ \`a $\bsc_{M,F}^{Q,\XX}$ en sommant sur $\UU$ 
cette in\'egalit\'e contre $m(\UU)$ d'o:
\begin{equation*}\vert p_{\ESR_k,0}(\phi,T)-{\vol}(\ESA_M^Q\backslash \ag_M^Q)\mun
\bscM^Q(\XX(T)^Q;0) \vert \le \frac{a}{k}
\sum_{\UU\in\ESH_M}r(\ESR,\UU+\XX(T))\mid m(\UU)\mid \ptf\end{equation*}
Comme $r(\ESR,\VV(T))$ est major\'e par un polyn™me en $\VV(T)$ 
on voit (avec les notations de \ref{normes}) qu'il existe un entier $d$ et une fonction $b(\ESR,\XX(T))$ telle que
\begin{equation*}\sum_{\UU\in\ESH_M}r(\ESR,\UU+\XX(T))\mid m(\UU)\mid \leq b(\ESR,\XX(T)) n_d(m)\ptf\end{equation*}
En prenant l'infimum sur les $m$, on obtient l'assertion souhait\'ee
lorsque $\Lambda_Q=0$. Maintenant on observe que
\begin{equation*}\bscMF^{Q,\XX(T)}(Z; \Lambda_Q)= 
e^{\langle \Lambda_Q,Z\rangle} \bsd_{M,F}^{\mskip 2mu Q,\XX(T)}(Z; 0)\end{equation*}
o $\bsd$ est la $(Q,M)$-famille p\'eriodique d\'eduite de $\bsc$ par translation par $\Lambda_Q$. 
Le cas g\'en\'eral en r\'esulte.
\end{proof}

L'expression \begin{equation*}{\vol}(\ESA_M^Q\backslash \ag_M^Q)\mun
e^{\langle \Lambda_Q,Z\rangle}\bscM^{Q,T_1}(\XX^Q;\Lambda)\end{equation*}
est ind\'ependante du choix de la mesure de Haar sur $\ag_M^Q$.
Dans les applications que nous avons en vue
la normalisation naturelle semble tre la suivante: pour chaque $M\in \ESL$, on munit $\ag_M$ 
de la mesure de Haar telle que
\begin{equation*}\mathrm{vol}(\ESB_M\backslash \ag_M)=1\end{equation*}
et pour $Q\in \ESF(M)$, on munit $\ag_M^Q$ de la mesure de Haar compatible aux mesures sur $\ag_M$ et 
$\ag_Q= \ag_{M_Q}$ et \`a la d\'ecomposition $\ag_M= \ag_Q \oplus \ag_M^Q$. Alors on a
\begin{equation*}\mathrm{vol}(\ESB_M^Q\backslash \ag_M^Q) =
1 \quad \hbox{et} \quad \mathrm{vol}(\ESA_M^Q\backslash \ag_M^Q)=
 \vert \bsbbc_M\vert^{-1} \vert \bsbbc_Q \vert\ptf\end{equation*}


 \chapter{Espaces tordus}\label{espaces tordus}

Tous les r\'esultats de \cite[ch.~2]{LW} sont vrais ici, \`a l'exception de 2.6 et 2.10. 
L'adaptation au cas tordu de la version {\og corps de fonctions\fg} des r\'esultats du chapitre
pr\'ec\'edent
sur les transform\'ees de Laplace des fonctions caract\'eristiques de c™nes et les $(G,M)$-familles
est imm\'ediat. Nous serons tr\`es succincts. 

 \section{Hypoth\`eses}\label{hypoth\`eses}
Soit $(\tG,G)$ un $G$-espace tordu. On rappelle que $\tG$\index{Gtilde@$\tG$}
est une vari\'et\'e alg\'ebrique affine, 
munie d'une action alg\'ebrique de $G$ \`a gauche qui en fait un $G$-espace principal homog\`ene, et d'une application
\begin{equation*}\tG\rightarrow \mathrm{Aut}(G),\mskip 2mu \delta\mapsto \mathrm{Int}_\delta
\qquad\hbox{telle que}\qquad\mathrm{Int}_{g\delta} = \mathrm{Int}_g\circ \mathrm{Int}_\delta\end{equation*}
pour tout $g\in G$ et tout $\delta\in\tG$. 
On en d\'eduit une action \`a droite de $G$ sur $\tG$, donn\'ee par 
\begin{equation*}\delta g = \mathrm{Int}_\delta(g)\delta\ptf\end{equation*}
On suppose que $\tG$ est d\'efini sur $F$, c'est-\`a-dire que les actions \`a gauche et \`a droite de $G$ sur $\tG$ 
sont d\'efinies sur $F$, et que $\tG(F)$ est non vide.
L'ensemble $\tG(\adef)$ des points ad\'eliques de $\tG$ est un espace tordu sous $G(\adef)$, et on a
\begin{equation*}\tG(\adef)= G(\adef)\tG(F)=\tG(F)G(\adef)\ptf\end{equation*}
On notera souvent $\theta$ l'automorphisme de $G$ d\'efini par $\mathrm{Int}_\delta$ pour un $\delta\in \tG(F)$.
On observe que l'automorphisme induit par $\theta$ sur $\ag_G$ ne d\'epend que de $\tG$.
On pose\index{apgtilde@$\ag_\tG$}
\begin{equation*}\ag_\tG=\ag_G^\theta\quad\hbox{et}\quad a_\tG=\dim\ag_\tG\ptf\end{equation*}
On suppose, comme en \cite[2.5]{LW}\footnote{Dans \cite{W}, l'hypoth\`ese est un peu plus forte 
que celle de \cite[2.5]{LW}: le $F$-auto\-morphisme $\theta$ de $Z_G$ est suppos\'e d'ordre fini, 
ce qui assure l'existence d'un $F$-groupe alg\'ebrique affine $G^+$ de composante neutre $G$, tel que $\tG$ 
soit une composante connexe de $G^+$.},
que l'application naturelle \begin{equation*}\ag_G^\theta \rightarrow \ag_G/ (1-\theta)\ag_G\end{equation*}
est un isomorphisme. Dans ce cas on a une d\'ecomposition en somme directe
\begin{equation*}\ag_G = \ag_\tG \oplus \ag_G^\tG\quad\hbox{en posant}\quad\ag_G^\tG = (1-\theta)\ag_G\index{apgtildeg@$\ag_G^\tG$}\ptf\end{equation*}
On observe que \begin{equation*}\det(\theta-1 \vert \ag_G^\tG)\neq 0\ptf\end{equation*}
Notons $X$ le $\ZM$-module libre des caract\`eres du tore $A_G$. Soit $X_\theta$ le 
groupe des co-invariants sous $\theta$ dans $X$ et $\wt{X}$
le $\ZM$-module libre quotient de $X_\theta$ par son sous-groupe de torsion. On notera 
$A_\tG$ le tore d\'eploy\'e dont le groupe des caract\`eres est $\wt{X}$. C'est aussi le tore d\'eploy\'e 
dont le groupe des co-caract\`eres est le sous-groupe $Y^\theta$ des invariants sous $\theta$ du groupe 
$Y$ des co-caract\`eres de $A_G$.
Le morphisme $X\to \wt{X}$ induit un homomorphisme $\A_\tG\to\A_G$ qui identifie $A_\tG$ \`a la composante neutre 
du sous-groupe $A_G^\theta$ de $A_G$ form\'e des points fixes sous $\theta$. En particulier $A_\tG(\adef)$ est un sous-groupe 
d'indice fini de $A_G(\adef)^\theta= A_G^\theta(\adef)$. 
Soit \begin{equation*}\bfH_\tG: G(\adef)\rightarrow\ag_\tG \index{Hgtilde@$\bfH_\tG$}\end{equation*}
l'application compos\'ee de $\bfH_G:G(\adef)\rightarrow\ag_G$ et de la projection sur $\ag_\tG$. On note 
$\ESA_\tG$ l'image de $\bfH_\tG$, c'est-\`a-dire l'image de ${\ESA_G}$ 
par la projection orthogonale par 
rapport \`a $\ag_G^\tG$. C'est un r\'eseau de $\ag_\tG$. Comme dans le cas non tordu, on a un morphisme naturel 
injectif $\ESA_{A_\tG}\rightarrow \ESA_\tG$. On note $\ESB_\tG\;(= \bfH_\tG(A_\tG(\adef)))$ son image, 
qui est un sous-groupe d'indice fini de $\ESA_\tG$, et on pose
\begin{equation*}\bsbbc_\tG \bydef \ESB_\tG \backslash \ESA_\tG\ptf\end{equation*}
Notons que d'apr\`es ce qui pr\'ec\`ede, $\ESB_\tG$ co\"{\i}ncide avec le sous-groupe $Y^\theta$ de $Y=\ESB_G$ 
form\'e des points fixes sous $\theta$: on a
\begin{equation*}\ESB_\tG = \ESB_G^\theta = \ESB_G \cap \ag_\tG\ptf\end{equation*}
On pose
\begin{equation*}\ESB_G^\tG = \ESB_\tG \backslash \ESB_G\quad \hbox{et}\quad \ESC_G^\tG = \ESB_\tG \backslash \ESA_G\ptf\end{equation*}
On observe que $\ESB_G^\tG$ est un r\'eseau de $\ag_G^\tG$, et que $\ESC_G^\tG$ est un $\ZM$-module 
de type fini qui s'ins\`ere dans la suite exacte courte
\begin{equation*}0 \rightarrow \ESB_G^\tG \rightarrow \ESC_G^\tG \rightarrow \bsbbc_G \rightarrow 0\ptf\end{equation*}

On suppose, ce qui est loisible, que la paire parabolique d\'efinie sur $F$ minimale $(P_0,A_0)$ de $G$ 
a \'et\'e choisie de telle sorte qu'elle 
soit stable par $\mathrm{Int}_{\delta_0}$ pour un \'el\'ement $\delta_0\in\tG(F)$, d\'etermin\'e 
de mani\`ere unique modulo $M_0(F)$. 
On fixe un tel $\delta_0$, et on pose $\theta_0 = \mathrm{Int}_{\delta_0}$, 
$\tP_0 = \delta_0P_0$ et $\tM_0 = \delta_0M_0$. Alors le $F$-automorphisme $\theta_0$ de $G$ 
induit par fonctorialit\'e un automorphisme de $\ag_0$, que l'on note encore $\theta_0$. 
Puisque le $F$-automorphisme $\theta_0$ pr\'eserve $A_0$ et $P_0$, il induit une permutation
de l'ensemble fini $\Delta_0$ et donc un automorphisme 
d'ordre fini de $\ag_0^G$. 

On renvoie \`a \cite[2.7, 2.8]{LW} pour la dŽfinition des sous-ensembles 
(ou sous-espaces) paraboliques et sous-ensembles de Levi
et l'adaptation des autres notions. En particulier, un sous-groupe parabolique standard $P$
dont le normalisateur dans $\tG$ est non vide
est $\theta_0$ stable et l'ensemble $\tP=P\delta_0$ est un sous-espace parabolique
standard.

L'extension au cas tordu de la notion de famille orthogonale, de $(G,M)$-famille et
de la combinatoire des fonctions $\tau$, $\wh\tau$, $\phi$ et $\Gamma$,
est imm\'ediate (cf. \cite[2.9]{LW}). On dispose de plus ici de la notion de famille 
$\tM$-orthogonale enti\`ere et de $(\tG,\tM)$-famille p\'eriodique. 
Toute famille $M$-orthogonale $\XX= (X_P)$ 
d\'efinit par projection une famille $\tM$-orthogonale $(X_\tP)$, et si $\XX$ est 
enti\`ere alors $(X_\tP)$ l'est aussi. En particulier, tout \'el\'ement $T\in\ag_0$ 
d\'efinit une famille $\tM$-orthogonale $(\brT{\tP})$. 
Toutes les relations de \cite[1.7, 1.8]{LW} et \cite[1.3]{W} 
sont valables pour ces nouvelles fonctions. Par exemple,
si $\XX= (X_\tP)$ est une famille $\tM$-orthogonale, pour $\Lambda\in\ag_{0,\CM}^*$, on pose
\begin{equation*}\gamma_{\tM,F}^{\tQ,\XX}(Z;\Lambda)=
\sum_{H\in\ESA^{\tQ}_\tP(Z)}\Gamma^{\tQ}_{\tM}(H,\XX)e^{\langle \Lambda,H\rangle}\ptf\end{equation*}
Comme dans le cas non tordu, $\Lambda\mapsto\gamma_{\tM,F}^{\tQ,\XX}(Z;\Lambda)$ 
est une fonction enti\`ere de $\Lambda\in\ag_{0,\CM}^*$, et on a la d\'ecomposition
pour $\Lambda$ en dehors des murs
\begin{equation*}\gamma_{\tM,F}^{\tQ,\XX}(Z;\Lambda) =
\sum_{\tP\in\ESP^{\tQ}(\tM)}\varepsilon_\tP^{\tQ,\XX}(Z;\Lambda)\ptf\end{equation*}
Pour une $(\tG,\tM)$-famille $\bsc= (\bsc(\cdot,\tP))$, 
comme en \ref{fonctions caract\'eristiques} 
et modulo le choix d'une 
mesure de Haar sur l'espace $\ag_{\tM}^{\tQ}$ on d\'efinit pour $\Lambda\in \wh\ag_0$
 en dehors des murs
\begin{equation*}\bsc_{\tM}^{\tQ}(\Lambda)=
\sum_{\tP\in\ESP^{\tQ}(\tM)}\epsilon_\tP^{\tQ}(\Lambda)\bsc(\Lambda,\tP)\ptf\end{equation*}
De mme, si $Z\in\ESA_\tQ$ et $\XX=(X_\tP)$ est une famille $\tM$-orthogonale, on pose 
\begin{equation*}\bsc_{\tM,F}^{\tQ,\XX}(Z;\Lambda)=
\sum_{\tP\in\ESP^{\tQ}(\tM)}\varepsilon_\tP^{\tQ,\XX}(Z;\Lambda)\bsc(\Lambda,\tP)\ptf\end{equation*}
Ces fonctions v\'erifient les mmes propri\'et\'es que dans le cas non tordu. 
En particulier, toute $(\tG,\tM)$-famille p\'eriodique $\bsc$ s'\'ecrit $\bs{c}= \bs{c}_{m}$ pour une fonction \`a d\'ecroissance rapide 
$m$ sur le r\'eseau $\ESH_\tM$ des familles $\tM$-orthogonales qui sont enti\`eres. On a une formule d'inversion de Fourier analogue de 
celle de \ref{lissb} dans le cas tordu, et la fonction $\Lambda\mapsto\bsc_{\tM,F}^{\tQ,\XX}(Z; \Lambda)$ sur $\wh\ag_0$ 
est lisse et invariante par $\ESA_\tM^\vee$.
On a aussi une variante de cette formule d'inversion de Fourier, lorsque $\bs{c}$ se prolonge en une $(G,M)$-famille p\'eriodique:

\begin{lemma}\label{inversion (cas tordu)}
Soient $\tM\in\wt{\ESL}$, $\tQ\in\ESF(\tM)$ et $Z\in\ESA_\tQ$. 
Soit $\XX$ une famille $\tM$-orthogonale, et soit $\bsc= (\bsc(\cdot,\tP))$ une $(\tG,\tM)$-famille 
p\'eriodique. Supposons que $\bsc$ se prolonge en une $(G,M)$-famille p\'eriodique $ (\bsc(\cdot,P))$, et 
soit $\mf$ une fonction 
\`a d\'ecroissance rapide sur $\ESH_M$ telle que $\bsc=\bsc_\mf$. Alors 
\begin{equation*}\bsc_{\tM,F}^{\tQ,\XX}(Z;\Lambda)=
\sum_{\UU\in\ESH_M}\mf(\UU)\gamma_{\tM,F}^{\tQ,\XX}(Z,\UU;\Lambda)\end{equation*}
avec
\begin{equation*}\gamma_{\tM,F}^{\tQ,{\XX}}(Z,\UU;\Lambda) =\gamma_{\tM,F}^{\tQ,{\UU'+\XX}}(Z+U_{\tQ};\Lambda)\end{equation*}
o $\UU'$ est la famille $\tM$-orthogonale enti\`ere d\'eduite de $\UU$ par projection.
\end{lemma}

\begin{proof}
Notons $m'$ la fonction \`a d\'ecroissance rapide sur $\ESH_\tM$ d\'efinie par
\begin{equation*}m'(\UU')= \sum_{\UU \in \ESH_M(\UU')}m(\UU)\end{equation*}
o $\ESH_M(\UU')\subset \ESH_M$ est la fibre au-dessus de $\UU'$. Il suffit de voir que la 
$(\tG,\tM)$-famille $\bs{c}$ est associ\'ee \`a $m'$: on a $\bs{c}= \bs{c}_{m'}$.
\end{proof}

La preuve de \ref{passlim}
s'\'etend au cas tordu et fournit le
\begin{lemma}\label{gammaMpol}
Soient $\tM\in\wt{\ESL}$, $\tQ\in\ESF(\tM)$ et $Z\in\ESA_\tQ$. 
Soit $\XX$ une famille $\tM$-orthogonale rationnelle, et soit $\bsc= (\bsc(\cdot,\tP))$ une $(\tG,\tM)$-famille 
p\'eriodique. Pour $\Lambda\in \wh\ag_0$, la fonction $T\mapsto\bsc_{\tM,F}^{\tQ,\XX(T)}(Z;\Lambda)$ 
appartient \`a {\rmfamily PolExp}. On a aussi l'analogue tordu des points (ii) et (iii) de \ref{passlim}.
\end{lemma}

Nous aurons besoin d'une variante de ce qui pr\'ec\`ede. Le ${\ZM}$-module de type fini
\begin{equation*}\ESC_{\tM}^{\tQ}\bydef \ESB_\tQ \backslash \ESA_{\tM}\end{equation*}
s'ins\`ere dans la suite exacte courte
\begin{equation*}0 \rightarrow\ESB_\tQ \backslash \ESB_\tM \rightarrow \ESC_\tM^\tQ \rightarrow \bsbbc_\tM\rightarrow 0\ptf\end{equation*}
On note $\ESB_{\tM}^{\tQ}(Z)\subset\ESC_\tM^\tQ$ la fibre au-dessus de $Z\in \bsbbc_{\tM}$. 
C'est un espace principal homog\`ene 
sous $\ESB_\tM^\tQ = \ESB_\tQ \backslash \ESB_\tM$. 
Pour $\Lambda \in (\ag_{0,{\CM}}^G)^* \oplus \ESB_\tQ^\vee$, $\tP\in \ESP^{\tQ}(\tM)$, $T\in \ag_0$ et $X\in \ag_0$
on pose 
\begin{equation*}\eta_{\tP\mskip -2mu,F}^{\tQ,T}(Z;X,\Lambda)= 
\sum_{H\in \ESB_\tM^\tQ(Z)}\Gamma_\tP^\tQ(H-X,T)e^{\langle \Lambda,H \rangle} \ptf\end{equation*}
L'expression
\begin{equation*}\eta_{\tP\mskip -2mu,F}^{\tQ,T}(Z;X)= \eta_{\tP\mskip -2mu,F}^{\tQ,T}(Z;X,0)\leqno{(1)}\end{equation*}
ne d\'epend que de l'image de $T$ dans $\ESB_\tP^{\tQ}\backslash \ag_\tP^{\tQ}$. 
La proposition suivante est une variante de \ref{croisspol} et \ref{passlim}. 

\begin{proposition}\label{etapol}
Pour $X\in \ag_{0,\QM}$, la fonction
\begin{equation*}T\mapsto \phi(T)=\eta_{\tP\mskip -2mu,F}^{\tQ,T}(Z;X)\end{equation*}
est un \'el\'ement de \hbox{\rmfamily PolExp}:
pour tout r\'eseau $\ESR$ de $\ag_{0,\QM}$, sa restriction \`a $\ESR$ s'\'ecrit
\begin{equation*}\phi_\ESR(T)=\sum_{\nu\in E} p_{\ESR,\nu}(T)e^{\langle\nu,T\rangle}\end{equation*}
o $E$ est un sous-ensemble fini de $\wh\ESR$ et
les $p_{\ESR,\nu}$ sont des polyn™mes de degr\'e major\'e par $a_\tP -a_\tQ$. 
Les polyn™mes $p_{\ESR_k,0}$ ont pour limite, lorsque $k\to\infty$, un polyn™me 
qui est ind\'ependant du r\'eseau $\ESR$.
\end{proposition}

\begin{proof}
Puisque \begin{equation*}\eta_{\tP\mskip -2mu,F}^{\tQ,T}(Z;X)= e^{\langle \Lambda,Z' \rangle} \eta_{\tP\mskip -2mu,F}^{\tQ,T}(0;X-Z')\end{equation*}
 on peut supposer $Z=0$ et il suffit de traiter le cas $\tQ = \tG$. Posons
\begin{equation*}\eta_{\tP\mskip -2mu,F}^{\tG,T}(X,\Lambda)=\eta_{\tP\mskip -2mu,F}^{\tG,T}(0;X,\Lambda)\ptf\end{equation*}
On rappelle que 
\begin{equation*}\Gamma_\tP^\tG(H,T)=\sum_{\{\tR \vert 
\tP\subset \tR\}}(-1)^{a_\tR-a_\tG}\tau_\tP^\tR(H)\hat{\tau}_\tR^\tG(H-T)\ptf\leqno(2)\end{equation*}
et que la projection dans $\ag_\tP^\tG$ du support de la fonction $H\mapsto \Gamma_\tP^\tG(H,T)$ est compacte. 
Pour $\Lambda \in \ag_{0,\CM}^\tG$ sa transform\'ee anti-Laplace
\begin{equation*}\eta_{\tP\mskip -2mu,F}^{\tG,T}(X,\Lambda)=\sum_{H\in\ESB_\tP^\tG}\Gamma_\tP^\tG(H-X,T)e^{\langle\Lambda,H\rangle}\end{equation*}
est donc une fonction holomorphe de $\Lambda$. Comme dans \ref{serievareps}
on consid\`ere un r\'eseau $\ESD_k$ de $\ag_\tP^\tG$
assez fin pour que $\ESB_\tP^\tG$ et les images de $X$ et $T$ dans $\ag_\tP^\tG$
soient contenus dans ce r\'eseau. On a
\begin{equation*}\eta_{\tP\mskip -2mu,F}^{\tG,T}(X,\Lambda)=
{c^{-1}} \sum_{\nu\in\NN}\sum_{H\in\ESD_k}\Gamma_\tP^\tG(H-X,T)e^{\langle\Lambda+\nu,H\rangle}\end{equation*}
o $\nu$ parcourt le dual $\NN= \ESB_\tP^{\tG,\vee}\mskip -2mu/ \ESD_k^\vee$ 
de $\ESB_\tP^\tG \backslash \ESD_k$ {et $c$ est l'indice de $\ESB_\tP^\tG$ dans $\ESD_k$}.
La somme en $H$ peut se calculer au moyen de l'expression (2) lorsque $\Re(-\Lambda)$ est r\'egulier: 
\begin{equation*}\eta_{\tP\mskip -2mu,F}^{\tG,T}(X,\Lambda)=\sum_{\{\tR \vert 
\tP\subset \tR\}}\eta_{\tP\mskip -2mu,\tR}^T(X,\Lambda)\end{equation*}
avec
\begin{equation*}\eta_{\tP\mskip -2mu,\tR}^{T}(X,\Lambda)=c^{-1}(-1)^{a_\tR-a_\tG}\sum_{\nu\in\NN}\sum_{H\in\ESD_k}
\tau_\tP^\tR(H-X)\hat{\tau}_\tR(H-{T-X})e^{\langle\Lambda+\nu,H\rangle}\end{equation*}
qui est une fonction m\'eromorphe en $\Lambda$ 
ayant un p™le d'ordre $a_\tP^\tG = a_\tP - a_\tG$ en $\Lambda=0$.
On conclut comme dans \ref{croisspol} en consid\'erant les d\'eveloppements de Laurent des 
$\eta_{\tP\mskip -2mu,\tR}^T(X,\Lambda)$. Pour la derni\`ere assertion on proc\`ede comme dans
la preuve de \ref{passlim}. 
\end{proof}


 \section{Les fonctions $\sigma$ et $\wt{\sigma}$}\label{les fonctions sigma}
D'apr\`es \cite[2.11.1]{LW}, pour un $Q\in\ESP_\st$, il existe un plus petit $\tQ^+\in\wt\ESP_\st$ 
et un plus grand $\tQ^-\in\wt\ESP_\st$ tels que
\begin{equation*}Q^-\subset Q\subset Q^+\ptf\end{equation*}
De plus \cite[2.11.2]{LW}, pour $Q,\mskip 2mu  R\in\ESP_\st$ tels que $Q^+\subset R^-$, 
on a $(\ag_Q^R)^{\theta_0}=\ag_{\tQ^+}^{\tR^-}$.

Pour $Q,\mskip 2mu  R\in\ESP$ tels que $Q\subset R$, on note $\sigma_Q^R$\index{siagmaqr@$\sigma_Q^R$}
la fonction caract\'eristique de l'ensemble des $H\in\ag_0$ tels que
\begin{equation*}\left\{\begin{array}{ll}\langle \alpha,H \rangle >0 &\hbox{pour $\alpha\in\Delta_Q^R$}\\
\langle \alpha,H\rangle\leq 0&\hbox{pour $\alpha\in\Delta_Q\smallsetminus\Delta_Q^R$}\\
\langle \varpi,H\rangle >0 &\hbox{pour tout $\varpi\in\hat{\Delta}_R$}
\end{array}
\right.\end{equation*}
Si de plus $Q\in \ESP_\st$ et $Q^+\subset R^-$, il existe un $\tP\in\wt\ESP$ tel que $Q\subset P\subset R$ alors, on d\'efinit 
la variante tordue $\wt{\sigma}_Q^R$\index{siatgmaqr@$\wt\sigma_Q^R$} de la fonction $\sigma_Q^R$
en rempla\c{c}ant la troisi\`eme condition par 
\begin{equation*}\langle \wt{\varpi},H\rangle >0\quad
\hbox{pour tout} \quad\wt{\varpi}\in\hat{\Delta}_\tP\ptf\end{equation*}
D'apr\`es \cite[2.11.3]{LW}, la fonction $\wt{\sigma}_Q^R$ est
ind\'ependante du choix du $\wt{P}\in\wt\ESP$ avec $Q\subset P\subset R$
utilis\'e pour la d\'efinir, ce qui justifie la notation.

 \section{La fonction \textit{q}}\label{qqq}
Pour $Q\in\ESP_\st$, consid\'erons l'application lin\'eaire\footnote{Notons que notre d\'efinition de $q_Q$ 
diff\`ere de celle de \cite[2.13]{LW}, puisqu'on projette sur $\ag_Q^\tG$ et non pas 
sur $\ag_Q^G$. Cela ne change pas grand chose \`a l'affaire puisque par hypoth\`ese, 
l'application $1-\theta$ est un automorphisme de $\ag_G^\tG$.}\index{qQ@$q=q_Q$}
\begin{equation*}q=q_Q :\ag_0\rightarrow\ag_Q^\tG\end{equation*}
d\'efinie par
\begin{equation*}q(X)= ((1-\theta_0)X^\tG)_Q = ((1-\theta_0)X)_Q^\tG\ptf\end{equation*}
Elle se factorise \`a travers la projection orthogonale $\ag_0\rightarrow\ag_\Qo^\tG$ 
avec
\begin{equation*}Q_0 = Q\cap\theta_0^{-1}(Q)\in\ESP_\st\ptf\end{equation*}
Tous les r\'esultats de \cite[2.12, 2.13]{LW} sont vrais ici, mutatis mutandis.

 \chapter{Th\'eorie de la r\'eduction}\label{th\'eorie de la r\'eduction}

 \section{D\'ecomposition d'Iwasawa}
\label{d\'ecomposition d'Iwasawa}
Pour $v\in\vert\ES V\vert$, on fixe une paire parabolique d\'efinie sur $F_v$ minimale 
$(P_{v,0},A_{v,0})$ de $G_v=G\times_FF_v$, et on note 
$M_{v,0}$ le centralisateur de $A_{v,0}$ dans $G_v$. 
On suppose que
\begin{equation*}P_{v,0}\subset P_{0,v}= P_0\times_F F_v,\quad 
A_{v,0}\supset A_{0,v}= A_0\times_F F_v\ptf\end{equation*}
Un sous-groupe compact de $G(F_v)$ est dit {\og $M_{v,0}$-admissible\fg} s'il est 
sp\'ecial -- donc maximal -- et correspond 
\`a un sommet de l'immeuble de $G(F_v)$ qui appartient \`a l'appartement associ\'e \`a $A_{v,0}$. 
Rappelons qu'un sous-groupe compact maximal $M_{v,0}$-admissible $\bsK_v$ de $G(F_v)$ 
v\'erifie les propri\'et\'es suivantes (cf. \cite[3.1.1]{LW}):
\begin{itemize}
\item $G(F_v)=P_{v,0}(F_v)\bsK_v$ (d\'ecomposition d'Iwasawa);
\item tout \'el\'ement de $N_G(M_{v,0})(F_v)/M_{v,0}(F_v)$ a un repr\'esentant dans $\bsK_v$;
\item pour tout sous-groupe parabolique $P$ de $G$ contenant $M_{v,0}$ et d\'efini sur $F_v$, 
notant $M$ la composante de Levi de $P$ contenant 
$M_{v,0}$ (elle est d\'efinie sur $F_v$), on a la d\'ecomposition 
$\bsK_v\cap P(F_v)= (\bsK_v\cap M(F_v))(\bsK_v\cap U_P(F_v))$ et $\bsK_v\cap M(F_v)$ 
est un sous-groupe compact $M_{v,0}$-admissible de $M(F_v)$.
\end{itemize}
Nous dirons que $\bs{K}$ est un sous-groupe compact maximal {\og $M_0$-admissible\fg} 
de $G(\adef)$ s'il est de la forme 
\begin{equation*}\bsK=\prod_{v\in\vert\ES V\vert}\bsK_{v}\end{equation*}
o les $\bs{K}_v$ v\'erifient les propri\'et\'es suivantes:
\begin{itemize}
\item pour tout $v\in\vert\ES V\vert$, $\bsK_{v}$ est un sous-groupe compact $M_{v,0}$-admissible de $G(F_v)$;
\item pour tout $F$-plongement $G\hookrightarrow \mathrm{GL}_n$, on a $\bsK_v= \mathrm{GL}_n(\oo_v)\cap G(F_v)$ 
pour presque tout $v\in\vert\ES V\vert$;
\end{itemize}

Fixons un sous-groupe compact maximal $M_0$-admissible $\bsK=\prod_{v\in\vert\ES V\vert}\bsK_v$ 
de $G(\adef)$. Alors on a la d\'ecomposition d'Iwasawa
\begin{equation*}G(\adef)= P_0(\adef)\bsK\end{equation*}
et tout \'el\'ement de $N_{G(\adef)}(M_0)/M_0(\adef)$ a un repr\'esentant dans $\bsK$. 
Plus g\'en\'eralement, pour tout $P\in\ESP$ on a $G(\adef)= P(\adef)\bsK$.

Pour $P\in\ESP$, gr‰ce \`a la d\'ecomposition d'Iwasawa, 
on \'etend les morphismes $\bfH_P$ en des fonctions sur $G(\adef)$ tout entier
que, par abus de notation, on note encore $\bfH_P$:
pour $g\in G(\adef)$, on \'ecrit $g=pk$ avec $p\in P(\adef)$ et $k\in\bsK$, et on pose
\begin{equation*}\bfH_P(g)= \bfH_P(p)\ptf\end{equation*}
Pour $\tP\in\wt\ESP$, on note $\wt\bfH_P:\tG(\adef)\rightarrow\ag_P$ la fonction d\'efinie par
\begin{equation*}\wt\bfH_P(\delta_0g)= \bfH_P(g)\ptf\end{equation*}
Suivant la convention habituelle, on pose $\bfH_0= \bfH_{P_0}$ et $\wt\bfH_0 =\wt\bfH_{P_0}$~.

La construction de hauteurs dans \cite[3.2]{LW}, qui reprend essentiellement celle de \cite[I.2.2]{MW1}, 
 est valable pour un corps global de caract\'eristique quelconque. 
 Pour la notion de \textit{hauteur} sur un $F$-espace vectoriel de dimension finie, on renvoie \`a \textit{loc.~cit}. 
On suppose donn\'e un $F$-plongement \begin{equation*}\rho: G\rightarrow \mathrm{GL}(V)\end{equation*}pour un 
$F$-espace vectoriel de dimension finie $V$. On choisit une \textit{hauteur} $\lVert \mskip -1mu\cdot\mskip -1mu\rVert $ 
sur le $F$-espace vectoriel $\mathrm{End}(V)\times \mathrm{End}(V)$, et pour $x\in G(\adef)$, on pose
\begin{equation*}\vert x\vert =\lVert  (\rho(x),{^\mathrm{t}\rho(x^{-1})})\rVert \ptf\end{equation*}


 \section{Point central}\label{l'\'el\'ement T_0}
 
D'apr\`es \cite[3.3.3]{LW}, il existe un point $T_0\in\ag_0^G$ tel que 
pour tout \'el\'ement $s\in \bfW$, et pour tout repr\'esentant $w_s$ de $s$ dans $G(F)$, on ait
\begin{equation*}\bfH_0(w_s)= T_0 - sT_0\vgq \bfH_0(w_s^{-1})=T_0 - s^{-1}T_0\ptf\end{equation*}
Cet point est donn\'e par 
\begin{equation*}T_0\index{tbo@$T_0$}
=\sum_{\alpha\in\Delta_0} t_\alpha(1)\check{\varpi}_\alpha\end{equation*}
o $\check{\varpi}_\alpha\in\ag_0^G$ est l'\'el\'ement 
correspondant \`a $\alpha\in\Delta_0$ dans la base duale et $t_\alpha(1)\in \RM$ est d\'efini par 
\begin{equation*}\bfH_0(w_\alpha)=t_\alpha(1)\check{\alpha}\end{equation*}
o $w_\alpha$ est un repr\'esentant dans $G(F)$ de la sym\'etrie $s_\alpha$. 
Puisque les $\check{\varpi}_\alpha$ sont dans
$\ag_{0,\QM}$ il existe un entier $k\geq 1$ tel que $T_0\in k\mun\ESA_0$. 
Pour $x\in G(\adef)$, $T\in\ag_0$ et $s\in \bfW$, 
on pose\footnote{Dans \cite[3.3 et 5.3]{LW}, les \'el\'ements $\Y_{x,T,s}$, 
$\Y_{T,s}$, $\Y_s$ sont not\'es respectivement ${Y}_s(x,T)$, ${Y}_s(T)$, ${Y}_s$.}
\begin{equation*}\Y_{x,T,s}= s^{-1}(T-\bfH_0(w_sx))\ptf\end{equation*}
Si $x=1$, on \'ecrit simplement
\begin{equation*}\Y_{T,s}\index{YaTs@$\Y_{T,s}$}
\bydef\Y_{1,T,s}= s^{-1}T+(T_0-s^{-1}T_0),\end{equation*}
et si $T=0$, on pose $\Y_s=\Y_{0,s}$. 
Pour $P\in\ESP(M_0)$ et $s\in \bfW$ tel que $s(P)=P_0$, on 
pose $\Y_{T,P}=\Y_{T,s}$. 
Ceci d\'efinit comme en \cite[3.3]{LW} une famille orthogonale 
\begin{equation*}\YY(T)\index{YT@$\YY(T)$}
\bydef(\Y_{T,P})_{P\in\ESP}\quad\hbox{avec}\quad\Y_{T,P}=\brT{P}+(T_0 - \brTo{P})\ptf\end{equation*}
On note ${\YY}= (\Y_P)$ la famille $M_0$-orthogonale d\'efinie par 
 \begin{equation*}\Y_P\bydef\Y_{0,P}=T_0 - \brTo{P} =T_0-s^{-1}T_0=\Y_s\ptf\end{equation*}
Puisque $\Y_{s(P_0)} = \bfH_0(w_s^{-1})\in\ESA_0$, la famille ${\YY}$ est enti\`ere et
on a $\YY(T)=\YY+\TT$.

Plus g\'en\'eralement, pour $x\in G(\adef)$ et $T\in\ag_0$, on d\'efinit comme en \cite[3.3.2~(iii)]{LW} 
une famille $M_0$-orthogonale $(\Y_{x,T,P})$: pour $P\in\ESP(M_0)$ et $s\in \bfW$ tel que $s(P)=P_0$, 
on pose $\Y_{x,T,P}=\Y_{x,T,s}$. On a donc $\Y_{T,P}=\Y_{1,T,P}$. La famille $(\Y_{x,T,P})$ est rationnelle 
si $T\in\ag_{0,\QM}$. De plus (\textit{loc.~cit}.), il existe une constante $c$ telle que si $\bsd_0(T)>c$, cette famille 
est r\'eguli\`ere.

 \section{\'El\'ements primitifs}\label{primitifs}
En caract\'eristique positive, la d\'ecomposition de Jordan 
n'est en g\'en\'eral pas d\'efinie sur le corps de base; il convient donc ici 
de remplacer la notion d'\'el\'ement quasi semi-simple r\'egulier elliptique par celle d'\'el\'ement primitif \cite[3.7, page~76]{LW}: 
un \'el\'ement de $\tG(F)$ est dit \textit{primitif} (dans $\tG$) s'il n'appartient \`a aucun sous-espace parabolique propre 
de $\tG$ d\'efini sur $F$, autrement dit, si son orbite sous $G(F)$ ne rencontre aucun $\tP(F)$ pour 
$\tP\ne\tG$. On note $\tG(F)_\prim$
l'ensemble des \'el\'ements primitis de $\tG(F)$. Pour $\tM\in\wt{\ESL}$, 
on dispose plus g\'en\'eralement de la notion d'\'el\'ement primitif de $\tM(F)$ 
et de l'ensemble $\tM(F)_\prim$. 

On appelle \textit{paire primitive} (dans $\tG$)
une paire $(\tM,\delta)$ o $\tM\in \wt{\ESL}$ et $\delta$ est un \'el\'ement primitif de $\tM(F)$. 
Deux paires primitives $(\tM,\delta)$ et $(\tM',\delta')$ sont dites \'equivalentes si elles sont conjugu\'ees i.e. 
s'il existe un \'el\'ement $x\in G(F)$ tel que ${\tM'}= \mathrm{Int}_x(\tM)$ et $\delta'= \mathrm{Int}_x(\delta)$. 
On note $ [\tM,\delta]$ la classe d'\'equivalence de $(\tM,\delta)$ et $\OO$ l'ensemble de ces classes.

Pour un \'el\'ement $\gamma\in\tG(F)$, on note $\ESO(\gamma)$ sa classe de $G(F)$-conjugaison.
Consid\'erons un espace parabolique $\tP=\tM U\in\wt\ESP$ tel que 
$\ESO(\gamma)\cap\tP(F)\neq\emptyset$ avec $\tP$ minimal pour cette propri\'et\'e.
On choisit $g\in G(F)$ tel que $g^{-1}\gamma g\in\tP(F)$. On peut \'ecrire
$g^{-1}\gamma g=\delta u$ avec $\delta\in\tM(F)$ et $u\in U(F)$.
La condition de minimalit\'e assure que $\delta$ est primitif dans $\tM(F)$.

\begin{lemma} \label{corrorb}
La correspondance $\gamma\mapsto (\tM,\delta)$ induit une
application surjective \begin{equation*}\zeta:\tG(F)\rightarrow\OO\ptf\end{equation*}
\end{lemma}
\begin{proof}
Il convient de montrer que deux paires primitives associ\'ees \`a un mme $\gamma$ sont \'equivalentes.
Soient donc $\tP=\tM  U$ et $\tP'=\tM' U'$ deux \'el\'ements de $\wt{\ESP}$ tels que
\begin{equation*}\gamma=\delta u=\delta' u'\end{equation*}pour $\delta\in\tM(F)$, $u\in U(F)$, $\delta'\in\tM'(F)$ 
 et $u'\in U'(F)$; de plus $\wt{P}$ et $\wt{P}'$ sont minimaux pour ces conditions. L'ensemble
\begin{equation*}\wt{H}= \wt{P}\cap \wt{P}'= (P\cap P')\gamma = \gamma (P\cap P')\end{equation*}est un espace tordu 
de groupe sous-jacent $H= P \cap P'$. Soit $L$ une composante de Levi de 
$H$ d\'efinie sur $F$ (cf. \cite[4.7]{BT}). Puisque $\mathrm{Int}_{\gamma}(L)$ en est une autre, 
d'apr\`es \textit{loc.~cit}. il existe un unique \'el\'ement $u_H\in U_{H}(F)$ 
tel que $\mathrm{Int}_{\gamma}(L)= \mathrm{Int}_{u_H}(L)$, o $U_H$ est le radical unipotent de $H$. Ainsi 
\begin{equation*}\wt{L}= \eta L = L \eta \com{avec} \eta = u_H^{-1}\gamma\in \wt{H}(F)\end{equation*}est 
une composante de Levi de $\wt{H}$ d\'efinie sur $F$: on a l'ŽgalitŽ $\wt{H}= \wt{L} \ltimes U_H$. 
Comme $\wt{H}$ normalise $U$ (resp. $U'$), d'aprs \cite[4.4\,(b)]{BT} $Q=\wt{H}U$ (resp. $\wt{Q}=\wt{H}U'$) est un sous-espace parabolique 
de $\wt{G}$ dŽfini sur $F$ et contenu dans $\wt{P}$ (resp. $\wt{P}'$). Par construction $\gamma \in \wt{Q}(F)$ (resp. $\gamma \in \wt{Q}'(F)$). Par minimalitŽ on a donc $\wt{Q}=\wt{P}$ 
et $\wt{Q}'= \wt{P}'$. D'aprs \textit{loc.~cit}., cela entra"ne que $\wt{L}$ est une composante de Levi de $\wt{P}$ (resp. $\wt{P}'$) et $U_H\subset U$ (resp. $U_H\subset U'$). On en dŽduit qu'il  
existe un unique $v\in U(F)$ (resp. $v'\in U'(F)$) tel que 
$\wt{L} = \mathrm{ Int}_v(\wt{M})$ (resp. $\wt{L} \subset \mathrm{ Int}_{v'}(\wt{M})$). On a donc \begin{equation*}\wt{M}'= \mathrm{Int}_y(\wt{M})\quad\hbox{avec}\quad y= {v'}^{-1}v\ptf\end{equation*}
Il reste \`a montrer que $\delta$ et $\delta'$ sont conjugu\'es par ce $y$.
\'Ecrivons $\gamma=\delta''u''$ avec $\delta''\in \wt{L}(F)$ et $u''\in U_H(F)$. On a donc 
$$\gamma = v^{-1}\delta''v u_1 \quad \hbox{avec} \quad u_1= v^{-1}\mathrm{ Int}_{\delta''^{-1}}(v) u''\ptf$$ Or 
$v^{-1}\delta'' v$ appartient ˆ $\wt{M}$ et $u_1$ appartient ˆ $U(F)$. Puisque $\gamma = \delta u$, cela entra"ne 
$v^{-1} \delta'' v = \delta$ et $u_1=u$. On obtient de la mme manire l'ŽgalitŽ $v'^{-1} \delta'' v' = \delta'$. 
On a donc bien $\delta' = \mathrm{ Int}_y(\delta)$. 
\end{proof}

\begin{lemma} \label{OO} On a les assertions suivantes:
\begin{enumerate}[(i)]
\item L'application $\zeta$ fournit une partition de l'ensemble des classes de $G(F)$-conjugaison
dans $\tG(F)$ et pour $\oo=[\tM,\delta]\in\OO$ l'ensemble
\begin{equation*}\ESO_\oo=\bigcup_{\tQ\in\ESP(\tM)}\{g^{-1}\delta u g\mskip 2mu \vert \mskip 2mu  g\in G(F),\mskip 2mu  u\in U_Q(F)\}\end{equation*}
 est la fibre de $\zeta$ au desus de $\oo$. 
\item Pour $\tQ\in \wt\ESP$, on a la d\'ecomposition 
\begin{equation*}\ESO_\oo\cap\tQ(F)= (\ESO_\oo\cap\tM_Q(F))U_Q(F)\ptf\end{equation*}
\end{enumerate}
\end{lemma}

\begin{proof}
Le point (i) est clair. Prouvons (ii).
Soit $(\tM,\delta)$ une paire primitive dans la classe $\oo$.
On distingue deux cas: ou bien il n'existe aucun \'el\'ement $x\in G(F)$ 
tel que $\tM\subset \mathrm{Int}_x(\tM_Q)$, auquel cas les ensembles \`a gauche 
et \`a droite de l'\'egalit\'e (2) sont vides. Ou bien il existe un tel $x$ et,
quitte \`a remplacer $\tQ$ par $\mathrm{Int}_x(\tQ)$, on peut supposer que $\tM\subset\tM_Q$. 
Un $\gamma\in\ESO_\oo\cap\tQ(F)$ peut s'\'ecrire $\gamma =\gamma_1 u$ 
avec $\gamma_1\in\tM_Q(F)$ et $u\in U_Q(F)$. Soient $\tP_1$ un $F$-sous-espace parabolique 
 de $\tM_Q$ minimal pour la condition $\gamma_1\in\tP_1(F)$ et
 $\tM_1$ une composante de Levi de $\tP_1$ (d\'efinie sur $F$). 
On a $\gamma_1 =\delta_1u_1$ avec $\delta_1\in\tM_1(F)$ et $u_1\in U_{P_1}(F)$. 
La paire $(\tM_1,\delta_1)$ est conjugu\'ee \`a $(\tM,\delta)$
et donc $\gamma_1$ appartient \`a $\ESO_\oo\cap\tM_Q(F)$. D'o l'inclusion $\subset$. 
L'inclusion en sens inverse s'obtient de mani\`ere similaire.
\end{proof}


 \section{Ensembles de Siegel, partitions et lemme de finitude}
\label{Siegel}\label{la fonction F} 

Rappelons qu'on a not\'e $G({\adef})^1$ le noyau de $\bfH_G$, et $P_0({\adef})^1= 
M_0({\adef})^1 U_0({\adef})$ celui de $\bfH_0=\bfH_{P_0}$. 
Pour $t\in {\RM}$, on note $A_0^G(t)$ l'ensemble des $a\in A_0({\adef})\cap G(\adef)^1$ 
tels que 
\begin{equation*}\langle\alpha\mskip 2mu ,\mskip 2mu  \bfH_0(a)\rangle> t\quad \hbox{pour toute racine}\quad \alpha\in\Delta_0\ptf\end{equation*}
On peut choisir un ensemble fini $\FF$ \index{FxF@$\FF$}
dans $M_0(\adef)\cap G({\adef})^1$ de sorte que
\begin{equation*}(\A_0(\adef)\cap G(\adef)^1) M_0(\adef)^1\FF=M_0(\adef)\cap G({\adef})^1\ptf\end{equation*}
D'apr\`es \cite{Sp}, on sait que:
\pni(1) le quotient $G(F)\bsl G({\adef})^1$ est compact si et seulement si $G_\mathrm{der}$ est anisotrope;
\pni(2) il existe un $t\in {\RM}$ tel que 
\begin{equation*}G({\adef})^1= G(F)P_0({\adef})^1A_0^G(t)\FF\bsK\ptf\end{equation*}
Puisque $M_{0\mathrm{,\mskip 2mu  der}}$ est anisotrope, l'assertion (1) montre qu'il existe un sous-ensemble 
compact $\Omega_0$ de $M_{0\mathrm{,\mskip 2mu  der}}({\adef})$ tel que 
$M_{0\mathrm{,\mskip 2mu  der}}({\adef})= M_0(F)\Omega_0$.
D'apr\`es \cite[1.7]{Sp}, l'ensemble $U_0(F)\bsl U_0({\adef})$ est compact, 
il existe donc un sous-ensemble compact $\Omega_1$ de $U_0({\adef})$ tel que $U_0(F)\Omega_1=U_0(\adef)$.
Il r\'esulte de l'assertion (2) qu'il existe un $t\in {\RM}$ tel que, en posant 
$\Omega =\Omega_1\Omega_0\subset P_{0}({\adef})$ on ait
\begin{equation*}G({\adef})^1= G(F)\Omega A_0^G(t)\FF\bsK\ptf \leqno{(3)}\end{equation*}
Maintenant consid\'erons une section du morphisme compos\'e
\begin{equation*}A_0(\adef)\rightarrow A_G(\adef)\backslash A_0(\adef) \rightarrow \ESB_0^G = 
\ESB_{G}\backslash \ESB_{0}\;(=\ESA_{A_G}\backslash \ESA_{A_0})\end{equation*}
et notons $\BB_0^G$ \index{Bbog@$\BB_0^G$}
son image. Une telle section s'obtient en choisissant (arbitrairement) des repr\'esentants dans l'image 
r\'eciproque d'une $\ZM$-base de $\ESB_0^G$ et en prenant le sous-groupe engendr\'e
\footnote{Ë priori $\BB_0^G$ n'est pas invariant sous l'action de $\bfW$.}. 
On pose 
\begin{equation*}\BB_0^G(t) = \BB_0^G\cap A_0^G(t)\ptf\end{equation*}
Pour $t\in {\RM}$ et $\Omega$ un sous-ensemble compact de $P_0({\adef})^1$, on pose 
\begin{equation*}\Siegel_{t,\FF,\Omega}^1  \index{SiegeltFOm@$\Siegel_{t,\FF,\Omega}^1$}
=\Omega \BB_0^G(t)\FF\bsK\ptf\end{equation*}
D'apr\`es (3), on voit que pour $t$ assez petit et $\Omega$ assez gros, on a 
\begin{equation*}G({\adef})^1= G(F)\Siegel_{t,\FF,\Omega}^1\ptf\leqno{(4)}\end{equation*}
On notera simplement $\Siegel^1$\index{Siegelun@$\Siegel^1$}
un tel domaine $\Siegel_{t,\FF,\Omega}^1$ 
pour le quotient $G(F)\bsl G(\adef)^1$.

La propri\'et\'e de finitude usuelle pour un corps de nombres, \`a savoir que si $\Siegel^1$ est un domaine de
Siegel pour le quotient $G(F)\backslash G(\adef)^1$, l'ensemble des $\gamma\in\G(F)$ tels que 
$\gamma\Siegel^1\cap\Siegel^1\ne\emptyset$ est fini,
n'est plus vraie ici. En effet, pour tout corps global, tout \'el\'ement unipotent $u\in U_0(\adef)$ et
tout voisinage ouvert relativement compact de l'identit\'e $\mathcal V$ dans $U_0(\adef)$, on a
$a\mun ua\in\mathcal V$ pour $a\in\A_0(\adef)$ avec $\bfH_0(a)$
assez loin dans la chambre de Weyl positive (i.e. $\bsd_0(\bfH_0(a))$ assez grand). 
Maintenant, pour un corps de fonctions le sous-groupe 
compact maximal $\bsK$ est ouvert et donc
$\gamma\Siegel^1\cap\Siegel^1\ne\emptyset$ pour tout $\gamma\in U_0(F)$; il en r\'esulte que
la propri\'et\'e de finitude est en d\'efaut.

Pour $P\in\ESP_\st$, notons $P(F)_\mathrm{st-prim}$ l'ensemble des $\gamma\in P(F)$
 tels que $\gamma\notin Q(F)$ pour tout $Q\in\ESP_\st$ tel que $Q\subsetneq P$. 
 Tout \'el\'ement primitif de $G(F)$ est contenu dans $G(F)_\mathrm{st-prim}$, mais la r\'eciproque 
 est fausse en g\'en\'eral. 
 
Pour $\alpha\in\Delta_0$, notons 
$P_\alpha$ l'unique \'el\'ement de $\ESP_\st$ tel que $\Delta_0\smallsetminus\{\alpha\}$ 
soit une base du syst\`eme de racines de $M_{P_\alpha}$. 
Un \'el\'ement $\gamma\in G(F)$ est dans $G(F)_\mathrm{st-prim}$ si et seulement s'il n'appartient 
\`a aucun $P_\alpha(F)$ pour $\alpha\in\Delta_0$. Soient $\gamma\in G(F)$ et $g,\mskip 2mu  g'\in \Siegel^1$ 
tels que $g =\gamma g'$. \'Ecrivons $g= yax$ et $g'= y'a'x'$ avec 
$y,\mskip 2mu y'\in\Omega$, $a,\mskip 2mu a'\in \BB_0^G(t)$ et 
$x,\mskip 2mu x'\in\FF\bsK$. D'apr\`es \cite[2.6]{Sp}, pour chaque $\alpha\in\Delta_0$, il existe une constante $c_\alpha > t$ 
telle que si $\log\vert\alpha (a)\vert\geq c_\alpha$ ou 
$\log\vert\alpha(a')\vert\geq c_\alpha$, alors $\gamma\in P_\alpha(F)$. 
On en d\'eduit que
l'ensemble des $\gamma\in G(F)_\mathrm{st-prim}$ tels que 
$\gamma\Siegel^1\cap\Siegel^1\neq\emptyset$, est fini. 
Le lemme ci-dessous est une simple g\'en\'eralisation de ce r\'esultat. 
Nous l'\'enonons pour un travail ult\'erieur (il ne sera pas utilis\'e ici).

\begin{lemma}\label{finiprim}Soit $P=MU\in\ESP_\st$. Pour $\gamma\in P(F)$, on note $\gamma_M$ 
la projection de $\gamma$ sur $M(F)= U(F)\backslash P(F)$. 
Alors l'ensemble des projections $\gamma_M$ des \'el\'ements \hbox{$\gamma\in P(F)_\mathrm{st-prim}$}
tels que $\gamma\Siegel^1\cap\Siegel^1\neq\emptyset$, est fini.
\end{lemma}

\begin{proof} Soient $\gamma\in P(F)$ et $g,\mskip 2mu g'\in\Siegel^1$ tel que $g=\gamma g'$. 
Comme plus haut, on \'ecrit $g=yax$, $g'=y'a'x'$. Alors $l=xx'^{-1}\in P(\adef)$. 
Pour $p\in P(\adef)$, \'ecrivons $p=p_Up_M$ avec $p_U\in U(\adef)$ et $p_M\in M(\adef)$. 
L'\'equation $g =\gamma g'$ se r\'ecrit
\begin{equation*}y_U\mathrm{Int}_{y_M a}(l_U) y_M a\mskip 2mu  l_{\mskip -2muM}
 =\gamma_U \mathrm{Int}_{\gamma_M}(y'_U)\gamma_M y'_M a' \end{equation*}
soit encore
\begin{equation*}y_U\mathrm{Int}_{y_M a}(l_U)=\gamma_U \mathrm{Int}_{\gamma_M}(y'_U)\quad \hbox{et}\quad 
y_M a\mskip 2mu  l_{\mskip -2muM} =\gamma_M y'_M a' \ptf\end{equation*}
Puisque $l_{\mskip -1muM}$ appartient au compact $(P(\adef) \cap \FF\bsK \FF^{-1})_M$ de $M(\adef)$
et $\gamma_M$ appartient \`a $M(F)_\mathrm{st-prim}$, l'\'equation 
$y_M a\mskip 2mu  l_{\mskip -1muM} =\gamma_M y'_M a' $ assure que les projections $\gamma_M$ sont dans un ensemble fini. 
\end{proof}

Fixons comme ci-dessus une section du morphisme $A_G(\adef) \rightarrow \ESB_G$ et notons $\BB_G$ 
\index{Bbxg@$\BB_G$} son image. 
Puisque le groupe $\BB_G G(\adef)^1\backslash G(\adef)= \BB_G(M_0(\adef) \cap G(\adef)^1)\backslash M_0(\adef)$ est fini, 
on peut fixer aussi un sous-ensemble fini $\EE_G \index{EleG@$\EE_G$}
\subset M_0(\adef)$ tel que
\begin{equation*}G(\adef)= \BB_G G(\adef)^1 \EE_G= \BB_G \EE_G G(\adef)^1\ptf \end{equation*}
Posons
\begin{equation*}\Siegel^*\index{Siegelw@$\Siegel^*$, $\Siegel$}  = \EE_G \Siegel^1\quad \hbox{et}\quad \Siegel= \BB_G \Siegel^* 
=\BB_G \EE_G \Siegel^1\ptf\end{equation*}
Ainsi $\Siegel^*$ est un domaine de Siegel pour le quotient $\BB_GG(F)\backslash G(\adef)$, 
et $\Siegel$ est un domaine de 
Siegel pour le quotient $G(F)\backslash G(\adef)$. 
Les r\'esultats vrais pour $\Siegel^1$ s'\'etendent sans difficult\'e \`a $\Siegel^*$.

Pour $L\in \ESL$, on peut d\'efinir de la mme mani\`ere des domaines 
de Siegel $\Siegel^{L,1}$, $\Siegel^{L,*}=\EE_L\Siegel^{L,1}$ et $\Siegel^{L}=\BB_L \Siegel^{L,*}$. 
On peut bien sžr imposer, mme 
si ce n'est pas vraiment n\'ecessaire, que ces domaines soient compatibles \`a la conjugaison: si $L,\mskip 2mu L'\in \ESL$ 
sont tels que $A_{L'}=\mathrm{Int}_g(A_L)$ 
pour un $g\in G(F)$, on demande que $\Siegel^{L'\mskip -1mu,1}= \mathrm{Int}_g(\Siegel^{L,1})$, 
$\EE_{L'}= \mathrm{Int}_g(\EE_L)$ et 
$\BB_{L'}= \mathrm{Int}_g(\BB_L)$. 

Fixons un \'el\'ement $T_1\in\ag_0$. Fixons aussi une section du morphisme $A_0(\adef)\rightarrow \ESB_0$, 
et notons $\mathfrak{B}_0$ son image (on peut prendre 
$\mathfrak{B}_0 = \mathfrak{B}_G\times \mathfrak{B}_0^G$). 
Pour $Q\in\ESP_\st$ et $T\in\ag_0$, on note
\begin{equation*}\Siegel_{P_0}^Q(T_1,T)\end{equation*}
l'ensemble des 
$x=uac\in G(\adef)$ avec $u\in U_Q(\adef)$, $a\in\BB _0$ 
et $c\in C_Q$ o $C_Q$ 
est un sous-ensemble compact de $G(\adef)$, tels que
\begin{equation*}\left\{\begin{array}{ll}\langle \alpha, \bfH_0(x)-T_1\rangle >0 &\hbox{pour tout $\alpha\in\Delta_0^Q$}\\
\langle \varpi,\bfH_0(x)-T \rangle\leq 0 &\hbox{pour tout $\varpi\in\hat{\Delta}_0^Q$}
\end{array}
\right..\end{equation*}
D'apr\`es \cite[1.8.3]{LW}, si $T-T_1$ est r\'egulier (ce que l'on suppose)
la condition ci-dessus est \'equivalente \`a $\Gamma_{P_0}^Q(\bfH_0(x)-T_1,T-T_1)=1$. 
On note
\begin{equation*}F_{P_0}^Q(\cdot,T)\index{FxpoqT@$F_{P_0}^Q(\cdot,T)$}
\end{equation*}
la fonction caract\'eristique de l'ensemble $Q(F)\Siegel_{P_0}^Q(T_1,T)$. 
Elle d\'epend du compact $C_Q$ et aussi de l'\'el\'ement $T_1\in\ag_0$. 
En pratique, on prendra $C_Q$ assez gros, et $-T_1$ et $T$ assez r\'eguliers. 
En particulier, on supposera toujours que $F^Q_{P_0}(\cdot,T)$ 
est invariante \`a gauche par $\BB_QQ(F)$ o $\BB_Q$ est l'image d'une section du morphisme 
$A_Q(\adef) \rightarrow \ESB_Q$.  Observons que puisque 
$A_Q(F)\mathfrak{B}_Q\backslash A_Q(\adef) = A_Q(F)\backslash A_Q(\adef)^1$ est compact, 
quitte \`a grossir le compact $C_Q$, on peut mme la supposer invariante \`a gauche par $A_Q(\adef)$. 
On la supposera aussi invariante \`a droite par $\bs{K}$.

Tous les r\'esultats de \cite[3.6]{LW} sont vrais ici, en particulier \cite[3.6.3]{LW} 
qui est l'analogue pour $M_P(F)U_P(\adef)\bsl G(\adef)$ 
de la partition \cite[1.7.5]{LW} de $\ag_0$.

Les lemmes 3.7.1, 3.7.2 et 3.7.3 de \cite{LW} sont vrais ici. 
Quant au lemme 3.7.4 de \cite{LW}, il suffit d'en modifier l'\'enonc\'e de la mani\`ere suivante:

\begin{lemma}\label{finitude}
Soit $\Omega$ un compact de $\tG(\adef)$, et soit $\tP\in\wt\ESP$. 
L'ensemble des \'el\'ements $\delta\in\tM_P(F)$ tels que $\delta$ soit $M_P(F)$-conjugu\'e \`a un 
\'el\'ement $\delta_1\in\tM_{P_1}(F)_\prim$ 
pour un $\tP_1\in\wt\ESP$ tel que $\tP_1\subset\tP$ (i.e. $P_1\subset P$), 
et $x^{-1}\delta u x\in\Omega$ pour des \'el\'ements $x\in G(\adef)$ et $u\in U_P(\adef)$, 
appartient \`a un ensemble fini de classes de $M_P(F)$-conjugaison.
\end{lemma}

\part{Th\'eorie spectrale, troncatures et noyaux}

 \chapter{L'op\'erateur de troncature}\label{troncature}

 \section{Terme constant}\label{terme constant}
Une fonction 
$\varphi: G(\adef)\rightarrow \CM$ est dite \textit{\`a croissance lente} s'il existe des r\'eels 
$c,\mskip 2mu  r >0$ tels que pour tout 
$g\in G(\adef)$, on ait
\begin{equation*}\vert\varphi(g)\vert\leq c\vert g\vert^r\ptf\end{equation*}
On \'ecrit aussi \og $\vert\varphi (g)\vert\ll\vert g\vert^r$ pour tout $g\in G(\adef)$\fg.

Soit $P\in\ESP$, et soit $\varphi$ une fonction sur $U_P(F)\bsl G(\adef)$ mesurable et localement $L^1$. 
On d\'efinit le terme constant $\varphi^{\phantom{*}}_P=\Pi_P\varphi$ de $\varphi$ le long de $P$ par
\begin{equation*}\varphi^{\phantom{*}}_P(x)=\int_{U_P(F)\bsl U_P(\adef)}\varphi(ux)\dd u\vgq x\in G(\adef)\vg\end{equation*}
o $\dd u$ est la mesure de Tamagawa sur $U_P(\adef)$ -- i.e. celle qui donne le volume $1$ au quotient 
$U_P(F)\bsl U_P(\adef)$. Alors $\varphi^{\phantom{*}}_P$ est une fonction sur $U_P(\adef)\bsl G(\adef)$ 
mesurable et localement $L^1$. De plus, si $\varphi$ est \`a croissance lente, resp. lisse, alors 
$\varphi^{\phantom{*}}_P$ est \`a croissance lente, resp. lisse.

Pour $P\in\ESP_\st$, on note $\ESR_0^{P,+}$ l'ensemble des racines de $T_0$ dans $M_P$ 
qui sont positives par rapport \`a $\Delta_0^{P}$. 
Rappelons que l'on a fix\'e en \ref{Siegel} un domaine de Siegel $\Siegel=\mathfrak{B}_G\Siegel^*$ 
pour le quotient $G(F)\bsl G(\adef)$.
Fixons aussi un sous-groupe ouvert compact $\bsK'$ de $G(\adef)^1$.

Le lemme suivant \cite[I.2.7]{MW1} est le r\'esultat technique clef pour l'\'etude 
du terme constant dans le cas des corps de fonctions. 

\begin{lemma}\label{lemme1TC}
Soit $P\in\ESP_\st$. Il existe une constante $c_P>0$ telle que :
si $g\in\Siegel$ v\'erifie $\langle \bfH_0(g),\alpha\rangle >c_P$ 
pour tout $\alpha\in\ESR_0^{G,+}\smallsetminus\ESR_0^{P,+}$
alors, pour toute fonction $\varphi$ sur $G(F)\bsl G(\adef)$ invariante \`a droite par 
$\bsK'$, on a $\varphi^{\phantom{*}}_P(g) =\varphi(g)$. 
\end{lemma}

Ce lemme se g\'en\'eralise aux fonctions $\varphi$ sur $U_{P'}(\adef)M_{P'}(F)\bsl G(\adef)$ 
pour $P'\in\ESP_\st$ tel que $P\subset P'$: en remplaant $\ESR_0^{G,+}$ 
par $\ESR_0^{P'\mskip -3mu,+}$ dans la condition sur $g$, on obtient de mme $\varphi^{\phantom{*}}_P(g)=\varphi(g)$. 
On a aussi la variante suivante \cite[I.2.8]{MW1}:

\begin{lemma}\label{lemme2TC}
Il existe une constante $c'>0$ telle que que pour tout $T'\in\ag_0$ tel que 
$\bsd_0(T') >c'$, la propri\'et\'e suivante soit v\'erifi\'ee: pour tout $P\in\ESP_\st$, tout $g\in\Siegel$ tel que
\begin{equation*}\left\{\begin{array}{ll}\langle \alpha,\bfH_0(g)-T'\rangle >0 &\hbox{pour tout $\alpha\in\Delta_P$}\\
\langle \varpi ,\bfH_0(g)-T'\rangle\leq 0 &\hbox{pour tout $\varpi\in\hat{\Delta}_0^P$}
\end{array}\right.\end{equation*}
et toute fonction $\varphi$ sur $G(F)\bsl G(\adef)$ invariante \`a droite par $\bsK'$, 
on a $\varphi^{\phantom{*}}_P(g) =\varphi(g)$. 
\end{lemma}

Pour $Q,\mskip 2mu  R\in\ESP$ tels que $Q\subset R$, et $\psi$ une fonction sur $U_Q(F)\bsl G(\adef)$ 
mesurable et localement $L^1$, on pose\footnote{Dans \cite[4.3]{LW}, cette fonction est not\'ee $\Theta\psi$.} 
\begin{equation*}\Pi_{Q,R}\psi =\sum_{\{P\in\ESP \vert Q\subset P\subset R\}} (-1)^{a_P-a_R}\psi^{\phantom{*}}_P\ptf\end{equation*}
C'est encore une fonction sur $U_Q(F)\bsl G(\adef)$ mesurable et localement $L^1$.

\begin{lemma}\label{lemme3TC}
Il existe une constante $c''>0$ telle que pour tous les couples de sous-groupes parboliques
$Q,\mskip 2mu  R\in\ESP_\st$ 
avec $Q\subsetneq R$, tout $g\in\Siegel$ v\'erifiant
\begin{equation*}\langle \alpha, \bfH_0(g)\rangle > c''\qquad\hbox{pour tout} \qquad \alpha\in\Delta_Q^R\end{equation*}
et toute fonction $\varphi$ sur $G(F)\bsl G(\adef)$ invariante \`a droite par $K'$, on ait
\begin{equation*}\Pi_{Q,R}\varphi(g)=0\ptf\end{equation*}
\end{lemma}

\begin{proof}
Il suffit d'adapter celle de \cite[I.2.9]{MW1}. Par d\'efinition de $\Siegel$, il existe une constante $c_1<0$ 
telle que $\langle \bfH_0(g),\delta\rangle > c_1$ 
pour tout $g\in\Siegel$ et 
tout $\delta\in\ESR_0^{G,+}$. Fixons aussi une constante $c_2>0$ 
telle que pour tous $P,\mskip 2mu P'\in\ESP_\st$ tels que $P\subset P'$, 
on ait la version g\'en\'eralis\'ee du lemme \ref{lemme1TC}: pour tout $g\in\Siegel$ tel que $\langle \bfH_0(g),\alpha\rangle >c_2$ 
pour tout $\alpha\in\ESR_0^{P'\mskip -2mu,+}\smallsetminus\ESR_0^{P,+}$, et pour toute fonction $\psi$ 
sur $U_{P'}(\adef)M_{P'}(F)\bsl G(\adef)$ invariante \`a droite par $\bsK'$, on a $\psi^{\phantom{*}}_{P'}(g) =\varphi(g)$. 
Posons $c''= c_2-c_1$.

Soient $Q,\mskip 2mu R\in\ESP_\st$ tels que $Q\subsetneq R$, et soit $g\in\Siegel$. Posons
\begin{equation*}\Delta_Q^R(g)=\{\alpha\in\Delta_Q^R:\langle \alpha, \bfH_0(g)\rangle\leq c''\}\ptf\end{equation*}
L'ensemble des $P\in\ESP$ tels que $Q\subset P\subset R$ est en bijection avec 
l'ensemble des couples $(\Theta,\Theta')$ avec 
$\Theta\subset\Delta_Q^R(g)$ et $\Theta'\subset\Delta_Q^R\smallsetminus\Delta_Q^R(g)$: 
pour un tel couple $(\Theta,\Theta')$, on note $P(\Theta,\Theta')$ l'\'el\'ement 
de $\ESP_\st$ tel que $Q\subset P(\Theta,\Theta')\subset R$ d\'efini par
\begin{equation*}\Delta_Q^{P(\Theta,\Theta')}=\Theta\cup\Theta'\ptf\end{equation*}
Puisque
\begin{equation*}a_{P(\Theta,\Theta')} - a_R= a_Q -a_R - (\vert\Theta\vert +\vert\Theta'\vert)\end{equation*}
on a
\begin{equation*}\Pi_{Q,R}\varphi(g) = (-1)^{a_Q-a_R} \sum_{(\Theta,\Theta')} 
(-1)^{\vert\Theta\vert +\vert\Theta'\vert}\varphi^{\phantom{*}}_{P(\Theta,\Theta')}(g)\end{equation*}
o $(\Theta,\Theta')$ parcourt les couples comme ci-dessus. Fix\'e un tel couple, 
toute racine $\alpha\in\ESR^{P(\Theta,\Theta'),+}\smallsetminus\ESR^{P(\Theta,\emptyset),+}$ 
s'\'ecrit $\alpha =\beta +\delta$ avec $\beta\in\Theta'$ et $\delta\in\ESR^{P(\Theta,\Theta'),+}\cup\{0\}$, et l'on a
\begin{equation*}\langle \alpha, \bfH_0(g)\rangle =\langle \beta,\bfH_0(g)\rangle +\langle \delta,\bfH_0(g)\rangle > c'' +\inf\{c_1,0\}\geq c_2\ptf\end{equation*}
Par cons\'equent $\varphi^{\phantom{*}}_{P(\Theta,\Theta')}(g)=\varphi^{\phantom{*}}_{(\Theta,\emptyset)}(g)$. On a donc
\begin{equation*}\Pi_{Q,R}\varphi(g) = (-1)^{a_Q-a_R}
\Big(\sum_{\Theta'\subset\Delta_Q^R\smallsetminus\Delta_Q^R(g)} (-1)^{\vert\Theta'\vert}\Big)
\sum_{\Theta\subset\Delta_Q^R(g)}(-1)^{\vert\Theta\vert}\varphi_{P(\theta,\emptyset)}(g)\ptf\end{equation*}
Or la somme sur $\Theta'$ est nulle si $\Delta_Q^G(g)\neq\Delta_Q^R$, ce qui prouve le lemme. 
\end{proof}

Pour $Q=P_0$ et $R=G$, puisque l'ensemble des $g\in\Siegel^*$ tels que 
$\langle \bfH_0(g),\alpha\rangle\leq c'' $ est compact, on a en particulier 
 \cite[I.2.9]{MW1}:

\begin{lemma}\label{lemme4TC}
Il existe un sous-ensemble compact $C=C_{K'}$ de $\Siegel^*$ tel que 
pour toute fonction $\varphi$ sur $G(F)\bsl G(\adef)$ invariante 
\`a droite par $\bsK'$, le support de \begin{equation*}\Pi_{P_0,G}\varphi\vert_{\Siegel^*}\end{equation*}soit contenu dans $C$. 
\end{lemma} 

Soit $Q\in\ESP_\st$. 
Pour $T\in\ag_0$, on d\'efinit un op\'erateur de troncature $\bs\Lambda^{T,Q}$,\index{LtroncTQ@$\bs\Lambda^{T,Q}$}
pour une fonction 
$\varphi\in L^1_\mathrm{loc}(Q(F)\bsl G(\adef))$, par
\begin{equation*}\bs\Lambda^{T,Q}\varphi(x)=
\sum_{P\in\ESP_\st,P\subset Q} (-1)^{a_P-a_Q}
\sum_{\xi\in P(F)\bsl Q(F)}\wh{\tau}_P^{\mskip 2mu Q}(\bfH_0(\xi x)-T)\varphi_P(\xi x)\ptf\end{equation*}
D'apr\`es \cite[3.7.1]{LW}, la somme sur $\xi$ est finie. Notons que l'op\'erateur 
$\bs\Lambda^{T,Q}$ ne d\'epend que de la projection 
$T^Q$ de $T$ sur $\ag_0^Q$. On pose
\begin{equation*}\bs\Lambda^T =\bs\Lambda^{T,G}\ptf\end{equation*}
Les r\'esultats de \cite[4.1]{LW} sur les propri\'et\'es de $\bs\Lambda^{T}$ sont vrais ici. 
En particulier, pour $T$ assez r\'egulier (i.e. tel que $\bsd_0(T)\geq c$ pour une constante $c$ d\'ependant de $G$), 
l'op\'erateur $\bs\Lambda^T$ est un idempotent 
 \cite[4.1.3]{LW}: on a $\bs\Lambda^T \circ\bs\Lambda^T\varphi=\bs\Lambda^T\varphi$. 

\begin{definition}\label{defbrT}
Pour $X\in\ag_Q^G$ et $T\in\ag_0^G$, on d\'efinit
\footnote{On a utilis\'e les doubles crochets pour \'eviter les confusions 
avec la famille $M_0$-orthogonale $(\brT{P})$ d\'efinie par un \'el\'ement $T\in\ag_0$.} 
 \begin{equation*}T\Lbra X\Rbra \index{TbrbrX@$T\Lbra X\Rbra$}
 = T\Lbra X\Rbra^Q\in\ag_0^Q\end{equation*}
en posant \begin{equation*}{T-X}=\sum_{\alpha\in\Delta_0}x_\alpha\check{\alpha}\qquad \hbox{et}\qquad 
T\Lbra X\Rbra=\sum_{\alpha\in\Delta_0^Q}x_\alpha\check{\alpha}\ptf\end{equation*}
\end{definition}
Le raffinement \cite[4.2.2]{LW} des propri\'et\'es de 
$\bs\Lambda^{T,Q}$ est encore vrai ici.

 \section{Troncature et support}\label{troncature et support}Cette section adapte au cas des
corps de fonctions les r\'esultats de \cite[4.3]{LW}. On fixe un $T\in\ag_0$, assez 
r\'egulier. La proposition suivante \cite[I.2.16~(2)]{MW1} joue ici le r™le de
 \cite[4.3.2]{LW}. Elle est tr\`es simple \`a prouver et 
 cependant fournit des d\'ecroissances beaucoup plus radicales.

\begin{proposition}\label{supportcompact}
Soit $\bsK'$ un sous-groupe ouvert compact de $G(\adef)$. Il existe un sous-ensemble ferm\'e
$\Omega=\Omega_{T,K'}$ de $G(F)\bsl G(\adef)$ d'image compacte dans 
\begin{equation*}\BB_G G(F)\backslash G(\adef)\end{equation*}
tel que pour toute fonction 
$\varphi$ sur $G(F)\bsl G(\adef)$ invariante \`a droite par $\bsK'$, le support de 
la fonction tronqu\'ee $\bs\Lambda^T\varphi$ soit contenu dans $\Omega$. 
\end{proposition}

\begin{proof}
On reprend celle de \cite[I.2.16~(2)]{MW1}. 
Fixons un \'el\'ement $T'\in\ag_0$ assez r\'egulier: on demande que 
la conclusion du lemme \ref{lemme2TC} soit v\'erifi\'ee pour $\bsK'$. Pour $P\in\ESP_\st$, 
notons $G(\adef)_{P,T'}$ l'ensemble des $x\in G(\adef)$ v\'erifiant
\begin{equation*}\left\{\begin{array}{ll}\langle \alpha,\bfH_0(x)-T'\rangle >0 &\hbox{pour tout $\alpha\in\Delta_P$}\\
\langle \varpi,\bfH_0(x)-T' \rangle\leq 0 &\hbox{pour tout $\varpi\in\hat{\Delta}_0^P$}
\end{array}\right.\end{equation*}
et posons
\begin{equation*}\Siegel^*_{P,T'}= \Siegel^* \cap G(\adef)_{P,T'}\ptf\end{equation*}
D'apr\`es \cite[4.1.1]{LW}, pour $P\in\ESP_\st$ et $x\in G(\adef)$, si $(\bs\Lambda^T\varphi)_P(x)\neq 0$ alors
\begin{equation*}\langle \varpi,\bfH_0(x)-T\rangle\leq 0 \quad\hbox{pour tout}\quad \varpi\in\hat{\Delta}_P\ptf\end{equation*}
Gr‰ce \`a \cite[1.2.8]{LW}, on en d\'eduit que le support de 
$(\bs\Lambda^T\varphi)_P\vert_{\Siegel^*_{P,T'}}$ est contenu dans un compact de $\Siegel^*$ 
ind\'ependant de $\varphi$. En appliquant le lemme \ref{lemme2TC} \`a la fonction 
$\bs\Lambda^T\varphi$, on obtient que le 
support de $\bs\Lambda^T\varphi\vert_{\Siegel^*_{P,T'}}$ est contenu 
dans un compact de $\Siegel^*$ ind\'ependant de $\varphi$. 
Puisque \cite[1.7.5]{LW}
\begin{equation*}\sum_{P\in\ESP_\st}\phi_{P_0}^P\tau_P^G=1\end{equation*}
on a $G(\adef)=\bigcup_{P\in\ESP_\st}G(\adef)_{P,T'}$. D'o la proposition.
\end{proof}

Pour d\'emontrer la proposition \ref{supportcompact}, on a utilis\'e la partition \cite[1.7.5]{LW} de $\ag_0$. 
On observe aussi que d'apr\`es \cite[3.6.4]{LW} on a
\begin{equation*}\bs\Lambda^T\varphi(x) =
\sum_{\substack{Q,R\in\ESP_\st\\ Q\subset R}} A_{Q,R}^T\varphi(x)\leqno{(1)}\end{equation*}
avec
\begin{equation*}A_{Q,R}^T\varphi(x) =\sum_{\xi\in Q(F)\bsl G(F)} F^Q_{P_0}(\xi x ,T)\sigma_Q^R(\bfH_0(\xi x)-T )
\Pi_{Q,R}\varphi(\xi x)\ptf \leqno{(2)}\end{equation*}
Si $Q=R$, alors $\sigma_Q^R=0$ sauf si $Q=R=G$, auquel cas
\begin{equation*}A_{G,G}^T\varphi(x) = F_{P_0}^G(x,T)\varphi(x)\ptf\end{equation*}
On note $\bfC^T$\index{CT@$\bfC^T$} l'op\'erateur $A_{G,G}^T$. Puisque la fonction $F^G_{P_0}(\cdot,T)$ est invariante \`a gauche par 
$\BB_GG(F)$ et que son support est 
d'image compacte dans $\BB_GG(F)\bsl G(\adef)$, il existe un sous-ensemble ferm\'e 
$\Omega^*=\Omega^*_{T,K'}$ de $G(F)\bsl G(\adef)$ d'image compacte dans $\BB_G G(F)\backslash G(\adef)$ 
tel que pour toute fonction $\varphi$ sur $G(F)\bsl G(\adef)$ invariante \`a droite par $\bsK'$, 
le support de $(\bs\Lambda^T-\bfC^T)\varphi$ soit contenu dans $\Omega^*$. On a aussi la 
variante de \cite[4.3.3]{LW}:

\begin{proposition}\label{propK}
Soit $\bsK'$ un sous-groupe ouvert compact de $G(\adef)$. Il existe une constante $c=c_{\bsK'}$ 
(qui ne d\'epend pas de $T$) telle que si 
$\bsd_0(T)\geq c$, alors pour toute fonction $\varphi$ sur $G(F)\bsl G(\adef)$ invariante 
\`a droite par $\bsK'$, on a
\begin{equation*}(\bs\Lambda^T-\bfC^T)\varphi= 0\ptf\end{equation*}
\end{proposition}

\begin{proof}Pour \'etudier $(\bs\Lambda^T-\bfC^T)\varphi(x)$, on traite s\'epar\'ement 
chaque terme $A_{Q,R}^T\varphi(x)$ avec $Q\neq R$ dans (1). 
On peut prendre $x$ dans $\Siegel$.
Pour $\xi\in Q(F)\bsl G(F)$, il s'agit de contr™ler $\Pi_{Q,R}\varphi(\xi x)$ sous la condition
\begin{equation*}F_{P_0}^Q(\xi x,T)\sigma_Q^R(\bfH_0(\xi x)-T)=1\ptf\end{equation*}
D'apr\`es \cite[3.6.1]{LW}, pour $x$ fix\'e, il y a au plus un $\xi$ modulo $Q(F)$ tel que l'expression ci-dessus soit non nulle. 
Puisqu'on est libre de multiplier 
$\xi x$ par un \'el\'ement de $Q(F)$, on peut supposer que $\xi x= uay\in\Siegel_{P_0}^Q(T_1,T)$ 
avec $u\in U_Q(\adef)$, $a\in\BB _0$ et $y\in C_Q$ -- cf. \ref{la fonction F}. 
On peut mme supposer que $u\in\Omega_Q$ pour un compact 
$\Omega_Q\subset U_Q(\adef)$ tel que $U_Q(F)\Omega_Q = U_Q(\adef)$. 
Rappelons que $C_Q$ est un compact fix\'e (assez gros, 
mais qui ne d\'epend pas de $T$) de $G(\adef)$, et que $H=\bfH_0(\xi x)$ v\'erifie
\begin{equation*}\left\{\begin{array}{ll}\langle \alpha , H - T_1\rangle>0 &\forall\alpha\in\Delta_0^Q\\
\langle \varpi,H -T\rangle\leq 0 &\forall\varpi\in\hat{\Delta}_0^Q\ptf
\end{array}\right.\end{equation*}
Puisque $\sigma_Q^R(H-T)= 1$, on a $\langle \alpha,H\rangle >\langle \alpha,T\rangle$ pour tout $\alpha\in\Delta_Q^R$. Comme 
l'\'el\'ement $T-T_1$ est r\'egulier, il existe $c_1\in \RM$ (ind\'ependant de $T$) tel que $\langle \alpha,H\rangle>c_1$ 
pour tout $\alpha\in\ESR^{R,+}$. D'apr\`es le 
lemme \ref{lemme3TC}, il existe $c''>0$ tel que pour $g\in\Siegel$ tel que $\langle \bfH_0(g),\alpha\rangle > c''$ 
pour tout $\alpha\in\Delta_Q^R$, on ait $\Pi_{Q,R}\varphi(g)=0$ pour toute 
fonction $\varphi$ sur $G(F)\bsl G(\adef)^1$ invariante \`a droite par $\bsK'$. Ici l'\'el\'ement $\xi x=uay$ 
n'appartient pas \`a $\Siegel$, mais $H = \bfH_0(a)+\bfH_0(y)$ et $y$ reste dans un compact fix\'e, 
par cons\'equent il existe $c'_1\in \RM$ (ind\'ependant de $T$) 
tel que $\log\vert\alpha (a)\vert > c'_1$ pour tout $\alpha\in\ESR^{R,+}$. Puisque $u$ reste dans 
un compact fix\'e de $U_Q(\adef)$, pour tout $P'\in\ESP_\st^R$, 
la version g\'en\'eralis\'ee du lemme \ref{lemme1TC} s'applique encore \`a $\xi x$ (cf. la preuve de \cite[I.2.7]{MW1}), 
et quitte \`a modifier la constante $c''$, 
la conclusion du lemme \ref{lemme3TC} s'applique encore \`a $\xi x$. D'o la proposition. 
\end{proof}

La diff\'erence par rapport au cas des corps de nombres est ici spectaculaire: sur un corps de nombres $F$, 
pour toute fonction lisse \`a croissance uniform\'ement lente $\varphi$ sur $\BB_GG(F)\bsl G(\adef)$, 
la fonction $\bs\Lambda^T\varphi$ est seulement \`a d\'ecroissance rapide. Par ailleurs la d\'ecomposition 
\begin{equation*}\bs\Lambda^T=
\bfC^T+(\bs\Lambda^T - \bfC^T)\end{equation*}qui joue un r™le crucial dans 
les estim\'ees de \cite[ch.~12, 13]{LW} est bien plus simple \`a contr™ler car ici, 
pour toute fonction $\bsK'$-invariante \`a droite sur $G(F)\bsl G(\adef)$, 
non seulement $\bs\Lambda^T\varphi$ est \`a support 
d'image compacte dans $\BB_G G(F) \backslash G(\adef)$, mais si $T$ est assez r\'egulier, 
la troncature est encore plus brutale: on a $\bs\Lambda^T\varphi =\bfC^T\varphi$. 
Cela simplifiera la preuve des estim\'ees \`a \'etablir plus loin. 


 \chapter{Formes automorphes et produits scalaires}
\label{formes automorphes}

 \section{Formes automorphes}\label{formesautom}
On fixe une mesure de Haar $\dd g$ sur $G(\adef)$. On note $\dd k$ la mesure de Haar sur 
$\bsK$ telle que $\vol (\bsK)=1$. 
Pour $P\in\ESP$, on note $\dd u_P$, ou simplement $\dd u$, la mesure de Tamagawa 
sur $U_P(\adef)$. Par quotient par la mesure de comptage sur $U_P(F)$, on obtient 
une mesure sur $U_P(F)\bsl U_P(\adef)$ qui v\'erifie
\begin{equation*}\vol (U_P(F)\bsl U_P(\adef))=1\ptf\end{equation*}
Posons $M=M_P$. La mesure de Haar $\dd m$ sur $M(\adef)$ est choisie
de sorte que l'on ait la formule d'int\'egration
\begin{equation*}\int_{G(\adef)}f(g)\dd g=
\int_{U_P(\adef)\times M(\adef)\times\bsK}f(u m k)e^{-\langle 2\rho_P, \bfH_P(m) \rangle}
\dd u \dd m \dd k\end{equation*}
o $\rho_P$ d\'esigne la demi-somme des racines positives de $A_P$. La fonction
\begin{equation*}m\mapsto\bs{\delta}_{P}(m)
\index{dyeltap@\add{$\bs{\delta}_{P}$}}
=e^{\langle 2\rho_P, \bfH_P(m) \rangle}\end{equation*}
est le module de $P(\adef)$:
\begin{equation*}\dd (mum\mun)=\bs{\delta}_{P}(m)\dd u\ptf\end{equation*}
On pose\footnote{On prendra garde \`a ce que l'espace not\'e ici $\bsX_G$ ne co•ncide pas
avec l'espace ainsi not\'e dans [LW] car dans cette r\'ef\'erence il y a en plus un quotient par $\BB_G$.
Son usage correspond plut™t \`a celui de notre $\ovbsX_G$ sans toutefois lui tre \'egal.}
\begin{equation*}\bsX_P\index{XP@$\bsX_P$}
=P(F)U_P(\adef)\bsl G(\adef)\quad\hbox{et}\quad
\ovbsX_{\mskip -2mu P}\index{XPbar@$\ovbsX_P$}
=\A_P(\adef)P(F)U_P(\adef)\bsl G(\adef)\ptf\end{equation*}
En particulier
\begin{equation*}\bsX_G=\G(F)\bsl\G(\adef)\quad\hbox{et}\quad
\ovbsX_{\mskip -2mu \G}=\A_\G(\adef)G(F)\bsl G(\adef)\ptf\end{equation*}
Les groupes $G(\adef)$ et $P(F)U_P(\adef)$ sont unimodulaires; on dispose donc
d'une mesure quotient invariante \`a droite sur $\bsX_P$.
Pour $\phi$ localement int\'egrable et \`a support compact sur $\bsX_P$, on a la formule d'int\'egration:
\begin{equation*}\int_{\bsX_P}\phi(x)\dd x=
\int_{{\bsX_{M}}\times\bsK}\bs{\delta}_P(m)\mun\phi(mk)\dd m \dd k\leqno(1)\end{equation*}
o $\dd x$ est la mesure quotient. Par contre,
si $P\neq G$ il n'y a pas de mesure $\G(\adef)$-invariante \`a droite sur $\ovbsX_{\mskip -2mu P}$. Toutefois, il existe 
une fonctionnelle invariante \`a droite \add{$r_{\ovbsX_{P}}$}
sur l'espace des sections du fibr\'e en droites sur $\ovbsX_{\mskip -2mu P}$ d\'efini par $\bs{\delta}_P$. 
Ces sections sont repr\'esentables par les fonctions sur $\bsX_P$ v\'erifiant
\begin{equation*}\phi(px)=\bs{\delta}_P(p)\phi(x)\quad
\hbox{pour} \quad p\in A_P(\adef)P(F)U_P(\adef)\ptf\end{equation*}
La fonctionnelle est d\'efinie par:
\begin{equation*}\add{r}_{\ovbsX_{\mskip -2mu P}}(\phi)=\int_{{\ovbsX_{\mskip -2mu M}}\times\bsK}\bs{\delta}_P(m)\mun\phi(mk)\dd\dot{m} \dd k\end{equation*}
o $\dd\dot{m}$ est la mesure quotient.
 
Rappelons que l'on a not\'e $\Xi(P)$ le groupe des caract\`eres unitaires de $A_P(\adef)$ qui sont triviaux sur 
$A_P(F)$. Soit $\xi\in \Xi(P)$. On dit qu'une fonction 
$\varphi$ sur $\bsX_P$ {\og se transforme \`a gauche suivant $\xi$ \fg} si 
pour tout $x\in G(\adef)$ on a
\begin{equation*}\varphi(ax)= \xi(a)\varphi(x)\quad \hbox{pour} \quad a\in A_P(\adef)\ptf\end{equation*}
On note $L^2(\bsX_M)_\xi$ 
l'espace de Hilbert form\'e des fonctions 
$\varphi$ sur $\bsX_{M}$ qui se transforment \`a gauche suivant $\xi$ et sont 
de carr\'e int\'egrable sur $\ovbsX_{\mskip -2mu M}$. Le groupe $M(\adef)$ agit 
sur $L^2(\bsX_{M})_\xi$ par translations \`a droite. 
Consid\'erons deux fonctions $\varphi$ et $\psi$ localement int\'egrables sur $\bsX_P$ qui se transforment 
\`a gauche suivant le mme caract\`ere $\xi\in \Xi(P)$. On pose (si l'int\'egrale converge)
\begin{equation*}\langle \varphi, \psi \rangle_P= \int_{\ovbsX_{M}\times \bsK} 
\varphi(mk)\overline{\psi(mk)}\mskip 2mu \dd\dot{m}\mskip 2mu \dd k\ptf\leqno{(2)}\end{equation*}
Ce produit scalaire d\'efinit l'espace de Hilbert $L^2(\bsX_P)_\xi$
si\`ege de la repr\'esentation unitaire \og induite parabolique\fg 
\begin{equation*}\textrm{Ind}_{P(\adef)}^{G(\adef)}L^2(\bsX_M)_\xi\end{equation*}
d\'efinie par
\begin{equation*}(\Rho(y)\varphi)(x)=\bs{\delta}^{-1/2}_P(x)\bs{\delta}^{1/2}_P(xy)\varphi(xy)\ptf\end{equation*}


Pour la notion g\'en\'erale de forme automorphe nous renvoyons le lecteur \`a \cite[I.2.17]{MW1}.
Soit $M\in \ESL$. Un caract\`ere de $A_M(\adef)$ est automorphe s'il est 
trivial sur $A_M(F)$. Ainsi $\Xi(M)$ est le groupe des caract\`eres unitaires automorphes de $A_M(\adef)$. 
Pour $P\in\ESP(M)$, une forme automorphe $\varphi$ sur $\bsX_P$ est une fonction
$\bsK$-finie \`a droite telle que la fonction $m\mapsto\varphi(mx)$ sur $\bsX_{M}$ est
automorphe. Elle est dite cuspidale si pour tout $Q\in\ESP$ 
tel que $Q\subsetneq P$, le terme constant $\varphi_Q$ est nul -- ou, ce qui revient au mme, si 
pour tout $x$, la fonction $m\mapsto\varphi(mx)$ sur $\bsX_{M}$ est cuspidale. Notons 
$\Autom_\cusp(\bsX_P)$\index{automcusp@$\Autom_\cusp(\bsX_P)$}
l'espace des formes automorphes cuspidales sur $\bsX_P$.
Pour $\xi\in \Xi(P)$, notons
\begin{equation*}\Autom_\cusp(\bsX_P)_\xi \subset \Autom_\cusp(\bsX_P) \end{equation*}
le sous-espace form\'e des fonctions qui se transforment \`a gauche suivant $\xi$. 

\begin{definition} Soit $P\in \ESP(M)$. 
\begin{itemize}
\item On appelle repr\'esentation automorphe discr\`ete modulo le centre -- 
ou simplement discr\`ete -- de $M(\adef)$, une sous-repr\'esentation 
irr\'eductible de $M(\adef)$ dans l'espace $L^2(\bsX_{M})_\xi$ pour un $\xi\in\Xi(M)$. On note $L^2_\disc(\bsX_{M})_\xi$ 
le sous-espace ferm\'e de $L^2(\bsX_{M})_\xi$ engendr\'e par ces repr\'esentations.
\item On appelle \textit{forme automorphe discr\`ete pour $P$} une fonction $\bsK$-finie sur
$\bsX_P$ telle que pour tout $x$,
la fonction $m\mapsto \varphi(mx)$ sur $M(\adef)$ soit un vecteur d'une repr\'esentation automorphe,
discr\`ete modulo le centre, de $M$.
\end{itemize}
\end{definition}

Une repr\'esentation automorphe irr\'eductible de $M(\adef)$ est discr\`ete modulo le centre
si et seulement si sa restriction \`a $M(\adef)^1$ est une somme finie de 
repr\'esentations irr\'eductibles dans \begin{equation*}L^2(M(F)\bsl M(\adef)^1)\ptf\end{equation*}
Pour $\xi\in \Xi(P)$, on note
 $\Autom_\disc(\bsX_P)_\xi$\index{automdisc@$\Autom_\disc(\bsX_P)$} 
 l'espace engendr\'e par les formes
 automorphes discr\`etes qui se transforment \`a gauche suivant $\xi$ sur $\bsX_P$. 
Le produit scalaire $\langle \cdot ,\cdot \rangle_P$ munit 
$\Autom_\disc(\bsX_P)_\xi$ d'une structure d'espace pr\'e-hilbertien.
On sait (gr‰ce \`a \ref{supportcompact} pour les corps de fonctions) que
\begin{equation*}\Autom_\cusp(\bsX_P)_\xi\subset \Autom_\disc(\bsX_P)_\xi \ptf\end{equation*}
Ainsi $\Autom_\cusp(\bsX_P)_\xi$ est lui aussi muni d'une structure d'espace pr\'e-hilbertien. 

 \section{Op\'erateurs d'entrelacement et s\'eries d'Eisenstein}\label{OE et SE}
Soient $P,\mskip 2mu  Q\in\ESP$ deux sous-groupes parabolique associ\'es, i.e. tels que $M_P$ et $M_Q$ 
soient conjugu\'es dans $G(F)$. 
Consid\'erons une fonction $\Phi$ lisse sur $\bsX_P$ se transformant \`a gauche suivant un caract\`ere unitaire automorphe 
$\xi$ de $A_M(\adef)$. Pour $\lambda\in\ag^*_{P,\CM}$ et $x\in G(\adef)$, posons
\begin{equation*}\Phi(x,\lambda)=e^{\langle \lambda+\rho_P,\bfH_P(x)\rangle}\Phi(x)\ptf\end{equation*}
 La fonction $x\mapsto\Phi(x,\lambda)$ ne d\'epend que de l'image de $\lambda$ dans 
$\ag_{P,\CM}^*/\ESA_P ^\vee$.
Pour \hbox{$s\in \bfW(\ag_P,\ag_Q)$} et $\lambda\in\ag^*_{P,\CM}$ 
\og assez r\'egulier\fg
\footnote{En notant $\ESR_P$ l'ensemble des racines de $A_P$ dans $P$, 
on demande ici que $\langle\check{\alpha},\Re\lambda-\rho_P\rangle >0$ pour toute racine $\alpha\in\ESR_P$ telle que 
$s\alpha\in -\ESR_Q$.}
dans la chambre associ\'ee \`a $P$ dans $\ag_{P,\CM}^*$, 
on a une expression d\'efinie par une int\'egrale convergente:
\begin{equation*}(\bfM_{Q\vert P}(s,\lambda)\Phi)(x,s\lambda)=\int_{U_{s,P,Q}(\adef)}\Phi(w_s^{-1}nx,\lambda)\dd n\end{equation*}
o l'on a pos\'e 
\begin{equation*}U_{s,P,Q}= (U_Q\cap w_sU_P w_s^{-1})\bsl U_Q\ptf\end{equation*}
On obtient ainsi un op\'erateur
\begin{equation*}\bfM_{Q\vert P}(s,\lambda)\index{Mpq@$\bfM_{Q\vert P}(s,\lambda)$}:
\Autom_\disc(\bsX_P)_\xi\rightarrow \Autom_\disc(\bsX_Q)_{s\xi}\ptf\end{equation*}
 Pour $P$ et $Q$ standards et $P$ fix\'e, alors $Q$ est d\'etermin\'e par 
$s$, et on pose \begin{equation*}\bfM(s,\lambda)=\bfM_{Q\vert P}(s,\lambda)\ptf\end{equation*}
Pour $s=1$ on \'ecrira 
\begin{equation*}\bfM_{Q\vert P}(\lambda)= \bfM_{Q\vert P}(1,\lambda)\ptf\end{equation*}
Dans le cas particulier o $Q=s(P)$ on a (cf. \cite[5.2.1]{LW}):
\begin{equation*}\bfM_{s(P)\vert P}(s,\lambda)= e^{\langle\lambda +\rho_P,\Y_s\rangle }\bs{s}\qquad
\hbox{o} \qquad
\Y_s= \bfH_0(w_s^{-1})= T_0-s^{-1}T_0\end{equation*}
et \begin{equation*}\bs{s}:\Autom_\disc(\bsX_P)_\xi \rightarrow \Autom_\disc(\bsX_Q)_{s\xi}\qquad
\hbox{est d\'efini par}\qquad \bs{s}\Phi(x) =\Phi(w_s^{-1}x)\ptf\end{equation*}


\begin{definition}\label{Dnu}
Pour $\add{\mu}\in\bsmu_P$, on pose
\begin{equation*}\varphi_\add{\mu}(x) = e^{\langle \add{\mu} , \bfH_P(x)\rangle}\varphi(x)\end{equation*}
et on note $\bfD_\add{\mu}$\index{Dznu@$\bfD_\add{\mu}$}
 l'op\'erateur $\varphi\mapsto\varphi_\add{\mu}$, i.e. $\bfD_\add{\mu}\varphi=\varphi_\add{\mu}$. 
\end{definition}
\begin{lemma}\label{eqf}
Pour $P,\mskip 2mu Q\in\ESP$ associ\'es, 
$s\in \bfW(\ag_P,\ag_Q)$, $\add{\mu}\in\bsmu_P$ et $\lambda\in\ag^*_{P,\CM}$ 
assez r\'egulier, l'op\'erateur $\bfD_\add{\mu}$ v\'erifie l'\'equation fonctionnelle: 
\begin{equation*}\bfM_{Q\vert P}(s,\lambda) \bfD_\add{\mu} = \bfD_{s\add{\mu}} \bfM_{Q\vert P}(s,\lambda +\add{\mu})\ptf\end{equation*}
\end{lemma}
\begin{proof}Il suffit d'observer que $(\bfD_\add{\mu}\Phi)(x,\lambda)=\Phi(x,\lambda+\add{\mu})$.
\end{proof}

Soient $P,\mskip 2mu  Q\in\ESP$ tels que $P\subset Q$. Pour $\Phi\in \Autom_\disc(\bsX_P)_\xi$ et 
$\lambda\in\ag^*_{P,\CM}$ assez r\'egulier, on d\'efinit une s\'erie d'Eisenstein sur 
$\bsX_Q$ par la formule:
\begin{equation*}E^Q(x,\Phi,\lambda) =\sum_{\gamma\in P(F)\bsl Q(F)}\Phi(\gamma x ,\lambda)\ptf\end{equation*}
Pour $Q=G$, on pose $E(\cdot ,\Phi,\lambda)= E^G(\cdot ,\Phi,\lambda)$. 
Le th\'eor\`eme \cite[5.2.2]{LW} est vrai ici (mutatis mutandis)
\footnote{Les propri\'et\'es de rationalit\'e dans le cas cuspidal sont \'etablies en \cite[IV.4]{MW1}. 
Le cas g\'en\'eral est trait\'e en \cite[Appendice~II]{MW1}.}:
pour $\Phi\in \Autom_\disc(\bsX_P)_\xi$ et $x\in\bsX_Q$, les fonctions 
\begin{equation*}\lambda\mapsto (\bfM_{Q\vert P}(s,\lambda)\Phi)(x)\quad\hbox{et}\quad\lambda\mapsto E(x,\Phi,\lambda)\end{equation*}
admettent un prolongement m\'eromorphe d\'efinissant des fonctions
rationnelles sur le cylindre $\ag_{P,\CM}^*/\ESA_P ^\vee={\bfHom}(\ESA_P ,\CM^\times)$.

 \section{La $(G,M)$-famille spectrale}\label{GMspec}
Soient $M\in\ESL$, $P\in\ESP(M)$ et $\lambda\in\ag_{P,\CM}^*$. 
On d\'efinit une $(G,M)$-famille p\'eriodique \`a valeurs op\'erateurs \cite[5.3.2]{LW}: 
pour $Q\in\ESP(M)$ et $\Lambda\in \wh\ag_M$, on pose
\begin{equation*}\ESM(P,\lambda; \Lambda,Q)= \bfM_{Q\vert P}(\lambda)^{-1}\bfM_{Q\vert P}(\lambda +\Lambda)\ptf\end{equation*}
Soit $T\in\ag_{0}$. 
Rappelons que l'on a d\'efini en \ref{l'\'el\'ement T_0} une famille $M_0$-orthogonale, qui est rationnelle si 
$T\in \ag_{0,\QM}$:
\begin{equation*}\YY(T)=(\Y_{T,P})\end{equation*}o, pour $P\in\ESP(M_0)$, on a pos\'e 
\begin{equation*}\Y_{T,P}= \brT{P} +\Y_{P}\qquad\hbox{et}\qquad \Y_{P}= T_0 - \brTo{P}\ptf\end{equation*}
Suivant la convention habituelle, 
pour $Q\in\ESP$ et $P\in \ESP(M_0)$ tels que $P\subset Q$, on pose 
$\Y_{T,Q}= (\Y_{T,P})_Q$ et $\Y_Q= (\Y_{P})_Q$. Rappelons que pour 
$P\in\ESP(M_0)$, l'\'el\'ement $\Y_{P}$ appartient \`a $\ESA_0$. 
En particulier 
la famille $M_0$-orthogonale $(\Y_{P})$ est enti\`ere.
On peut donc d\'efinir une autre $(G,M)$-famille p\'eriodique \`a valeurs op\'erateurs: 
pour $Q\in\ESP(M)$ et $\Lambda\in \wh\ag_M$ en posant
\begin{equation*}\ESM({\YY};P,\lambda;\Lambda,Q)= e^{\langle\Lambda,\Y_Q\rangle }\ESM(P,\lambda; \Lambda,Q)\ptf\end{equation*}
Le lemme suivant r\'esulte de \ref{lissb}:

\begin{lemma}\label{lissgm}
Fixons un \'el\'ement $Z\in{\ESA_G}$. Les fonctions m\'eromorphes de 
$\lambda$ et $\Lambda$ \`a valeurs op\'erateurs\footnote{La notion de m\'eromorphie 
invoqu\'ee pour un op\'erateur disons $A(\lambda)$ l'est au sens faible, c'est \`a dire la m\'eromorphie
pour les fonctions $\lambda\mapsto A(\lambda)\Phi$ pour $\Phi$ dans un espace de Banach.}
\begin{equation*}\ESM_{M,F}^{G,T}(Z,\YY;P,\lambda;\Lambda)=
\sum_{Q\in\ESP(M)}\varepsilon_Q^{G,\brT{Q}}(Z;\Lambda)\ESM({\YY};P,\lambda;\Lambda,Q)\end{equation*}
sont lisses pour les valeurs imaginaires pures de $\lambda$ et $\Lambda$.
\end{lemma}

Observons que l'expression $\ESM_{M,F}^{G,T}(Z,\YY;P,\lambda;\Lambda)$ est \'egale \`a
\begin{equation*}\ESM_{M,F}^{G,T}(Z;P,\lambda;\Lambda)=
\sum_{Q\in\ESP(M)}\varepsilon_Q^{G,\brT{Q}}(Z;\Lambda)\ESM(P,\lambda;\Lambda,Q)\end{equation*}
si $\YY=0$ et donc par exemple si $G$ est d\'eploy\'e.

Soit $Z\in{\ESA_G}$. Pour $Q,\mskip 2mu  R\in\ESP_\st$, on introduit la fonction m\'eromorphe de $\lambda\in\ag_{Q,\CM}^*$ 
et $\mu\in\ag_{R,\CM}^*$, \`a valeurs op\'erateurs, 
\begin{equation*}\bsO_{R\vert Q}^{T}(Z;\lambda,\mu)\index{OmegaTrq@$\bsO_{R\vert Q}^{T}$}
=\sum_{S,s,t}\varepsilon_S^{G,T_S}(Z;s\lambda-t\mu)\bfM(t,\mu)^{-1}\bfM(s,\lambda)\end{equation*}
o $S$ parcourt les \'el\'ements de $\ESP_\st$ qui sont associ\'es \`a $Q$, $s$ parcourt les \'el\'ements de $\bfW(\ag_Q,\ag_S)$ et 
$t$ parcourt les \'el\'ements de $\bfW(\ag_R,\ag_S)$. Notons que $\bsO_{R\vert Q}^{T}(Z;\lambda,\mu)$ 
ne d\'epend que des images de 
$\lambda$ dans $\ag_{Q,\CM}^*/\ESA_Q^\vee$ et $\mu$ dans $\ag_{R,\CM}^*/\ESA_R^\vee$, et que l'on a 
$\bsO_{R\vert Q}^{T}(Z;\lambda ,\mu)=0$ si $R$ et $Q$ ne sont pas associ\'es. 

Le lemme \cite[5.3.4]{LW} est vrai ici. Il entra"ne la variante suivante de 
\cite[5.3.5]{LW}\footnote{Dans l'\'enonc\'e de \textit{loc.~cit}., $M$ est la composante de Levi standard de $Q$.}: 
en posant $M=M_R$, le changement de variables $s\mapsto u=t^{-1}s$, $S\mapsto S'= t^{-1}S$ et 
$s\lambda - t\mu\mapsto \Lambda_u = u\lambda -\mu$, donne
\begin{align*}
\lefteqn{
\bsO_{R\vert Q}^{T}(Z;\lambda,\mu)}\\
&& =\sum_{u\in \bfW(\ag_Q,\ag_R)}\sum_{S'\in\ESP(M)}e^{\langle\Lambda_u,\Y_{S'}\rangle} 
\varepsilon_{S'}^{G,\brT{S'}}(Z; \Lambda_u)\ESM(R,\mu; \Lambda_u , S')\bfM_{R\vert Q}(u,\lambda)\\
&& =\sum_{u\in \bfW(\ag_Q,\ag_R)}\ESM_{M,F}^{G,T}(Z,\YY;R,\mu;\Lambda_u)\bfM_{R\vert Q}(u,\lambda)\ptf
\end{align*}
Puisque pour $\lambda\in \wh\ag_Q $, l'op\'erateur $\bfM_{R\vert Q}(u,\lambda)$ est une isom\'etrie \cite[5.2.2~(2)]{LW}, 
on en d\'eduit que la fonction \`a valeurs op\'erateurs $(\lambda,\mu)\mapsto\bsO_{R\vert Q}^{T}(Z;\lambda,\mu)$ 
est lisse pour les valeurs imaginaires pures de $\lambda$ et $\mu$. 
L'op\'erateur 
\begin{equation*}\bsO_{R\vert Q}^{T}(Z;\lambda,\mu)\end{equation*}entrelace les repr\'esentation de $\G(\adef)$ dans
$\Autom_\disc(\bsX_Q)_\xi$ et $\Autom_\disc(\bsX_R)_{\xi'}$ o 
$\xi$ et $\xi'$ sont des caract\`eres unitaires automorphes de $A_Q(\adef)$ et $A_R(\adef)$ respectivement tels que 
pour un (i.e. pour tout) $u\in \bfW(\ag_Q,\ag_R)$ on ait $\xi'=u\xi$.

\begin{definition}\label{newomega}
On pose
\begin{equation*}\brabsO_{R\vert Q}^{T}(Z;\lambda ,\mu)\index{OmegaTrqbr@$\brabsO_{R\vert Q}^{T}$}
= \vert \wh\bsbbc_R \vert^{-1}\sum_{\nu \in \wh{\bsbbc}_R} 
\bfD_\nu\mskip 2mu \bsO_{R\vert Q}^{T}(Z;\lambda,\mu+\nu)\ptf\end{equation*}
La fonction \`a valeurs op\'erateurs $(\lambda,\mu)\mapsto\brabsO_{R\vert Q}^{T}(Z;\lambda,\mu)$ 
est lisse pour les valeurs imaginaires pures de $\lambda$ et $\mu$. 
\end{definition}

 \section{S\'eries d'Eisenstein et troncature}
\label{SETR}
Soit $M\in \ESL$. Pour $Z\in \ESA_G$ et $H\in\ESA_M$, on pose 
\begin{equation*}\bsX_G(Z)\index{XvgZ@$\bsX_G(Z)$}
=\G(F)\bsl\G(\adef;Z)\vgq \bsX_M(H)=M(F)\bsl M(\adef;H)\end{equation*}
et
\begin{equation*}\ovbsX_{\mskip -2mu M}=A_M(\adef)M(F)\bsl M(\adef)\ptf\end{equation*}
\add{Pour  $P\in \ESP(M)$, soient $\Phi$ et $\Psi$ deux fonctions sur $\bsX_P$ qui se transforment \`a gauche 
suivant le mme caract\`ere unitaire automorphe 
de $\A_M(\adef)$.  On a d\'efini un produit scalaire
\begin{equation*}\langle \Phi, \Psi \rangle_P= \int_{\ovbsX_{\mskip -2mu M}\times \bsK} 
\Phi(mk)\overline{\Psi(mk)}\dd m\dd k\ptf\end{equation*}
}
\begin{lemma}\label{phipsi}
 Pour $H\in\ESA_M$, on pose
\begin{equation*}\langle \Phi, \Psi\rangle_{P,H}=\int_{\bsX_M(H)\times \bsK} \Phi(mk) \overline{\Psi(mk)}\dd m\dd k\ptf\end{equation*}
On a alors
\begin{equation*}\langle \Phi, \Psi\rangle_{P,H}=\vert\wh\bsbbc_M \vert^{-1}
\sum_{\nu\in\wh\bsbbc_M}e^{-\langle\nu,H\rangle}\langle \bfD_\nu\Phi, \Psi\rangle_P\end{equation*}
\end{lemma}

\begin{proof} On observe que puisque
\begin{equation*}\Phi(amk) \overline{\Psi(amk)}= \Phi(mk) \overline{\Psi(mk)}\end{equation*}
pour tout $a\in\A_M(\adef)$, le produit scalaire $\langle \Phi, \Psi\rangle_{P,H}$ ne d\'epend que de l'image
de $H$ dans $\bsbbc_M$. On conclut par transform\'ee de Fourier sur le groupe fini $\bsbbc_M$.
\end{proof}

Soit $\varphi\in L^1_\mathrm{loc}(P(F)\bsl\G(\adef;Z))$. On pose, si la s\'erie converge,
\begin{equation*}E({x,}\varphi)=\sum_{\gamma\in P(F)\bsl\G(F)}\varphi(\gamma x)\ptf\end{equation*}
Soit $\psi\in L^1_\mathrm{loc}(\ovbsX_{\mskip -2mu M}\times\bsK)$, 
c'est-\`a-dire que $\psi$ est une fonction localement int\'egrable 
sur $\bsX_M\times\bsK$ qui est invariante \`a gauche sous $A_M(\adef)$.
On pose, si l'int\'egrale a un sens, pour $\nu\in \wh\bsbbc_M$
\begin{equation*}\wh\psi(\nu)=
\int_{\ovbsX_M\times\bsK}e^{\langle\nu,\bfH_M(m)\rangle}\psi(m,k) \dd m \dd k \ptf\end{equation*}
Nous aurons besoin du calcul formel suivant:
 \begin{lemma} \label{XMZ} 
Notons $\ESA_M(Z)$ l'image r\'eciproque dans $\ESA_M$ de $Z\in\ESA_G$.
Soit $\varphi$ comme ci-dessus et
supposons que 
\begin{equation*}\varphi_P(x)=\int_{U_P(F)\bsl U_P(\adef)}\varphi(ux)\dd u\end{equation*}
soit de la forme
\begin{equation*}\varphi_P(mk)=\bs{\delta}_P(m)e^{\langle \xi,\bfH_P(m)\rangle}\psi(m,k)\leqno(\star)\end{equation*}
pour $ m\in M(\adef)$, $k\in\bsK$,
$\xi\in \bsmu_M$ et $\psi\in L^1_\mathrm{loc}(\overline{\bsX}_{\mskip -2mu M}\times \bs{K})$. 
On a l'\'egalit\'e suivante:
\begin{equation*}\int_{\bsX_G(Z)}E({x,}\varphi)\dd x=\vert\wh\bsbbc_M\vert^{-1}
\sum_{\nu\in\wh\bsbbc_M} \sum_{H\in\ESA_M(Z)} e^{\langle \xi-\nu,H\rangle}
\wh\psi(\nu)\ptf\end{equation*}
\end{lemma}

\begin{proof}
Tout d'abord il est classique d'observer que
\begin{equation*}\int_{\bsX_G(Z)}E({x,}\varphi)\dd x=
\int_{P(F)\bsl\G(\adef;Z)}\varphi( x)\dd x=\int_{P(F)U_P(\adef)\bsl\G(\adef;Z)}\varphi_P( x)\dd x\ptf\end{equation*}
La formule d'int\'egration \ref{formesautom}~(1) montre alors que
\begin{equation*}\int_{\bsX_G(Z)}E({x,}\varphi)\dd x=\sum_{H\in\ESA_M(Z)}
\int_{\bsX_M(H)\times \bsK}\bs{\delta}_P(m)\mun \varphi_P(mk) \mskip 2mu \dd m\mskip 2mu  \dd k\end{equation*}
soit encore, compte tenu de l'hypoth\`ese $(\star)$:
\begin{equation*}\int_{\bsX_G(Z)}E({x,}\varphi)\dd x=\sum_{H\in\ESA_M(Z)}e^{\langle \xi,H\rangle}
\int_{\bsX_M(H)\times\bsK}\psi(mk)\mskip 2mu \dd m\mskip 2mu  \dd k \end{equation*}
et il suffit pour conclure d'observer que 
\begin{equation*}\int_{\bsX_M(H)\times\bsK}\psi(mk)\dd m\mskip 2mu  \dd k 
=\vert\wh\bsbbc_M\vert^{-1}\sum_{\nu\in\wh\bsbbc_M}
e^{\langle-\nu,H\rangle}\wh\psi(\nu)\ptf\end{equation*}
\end{proof}

Nous pouvons maintenant \'etablir l'analogue dans notre cadre de \cite[5.4.3]{LW}. 
Soient $\Phi$ et $\Psi$ des formes automorphes associ\'ees \`a des sous-groupe paraboliques standard. 
Pr\'ecis\'ement: 

\begin{hypoths}\label{hypo} On suppose que:
\begin{enumerate}[(i)]
\item $\Phi\in \Autom_\disc(\bsX_Q)_\xi$ et 
$\Psi\in \Autom_\disc(\bsX_R)_{\xi'}$ pour des sous-groupes
paraboliques associ\'es $Q,\mskip 2mu R\in \ESP_\st$ o $\xi$, resp. $\xi'$, 
est un caract\`ere unitaire automorphe de $A_Q(\adef)$, resp. $A_R(\adef)$;
\item $\lambda\in \ag_{Q,\CM}^*/\ESA_Q^\vee$ et $\mu\in \ag_{R,\CM}^*/\ESA_R^\vee$;
\item $\xi'= w\xi$ pour un (i.e. pour tout) $w\in \bfW(\ag_Q,\ag_R)$. 
\end{enumerate}
\end{hypoths}

\begin{theorem}\label{PSET}
Soit $Z\in\ESA_G$. 
Sous les hypoth\`eses \ref{hypo} on a les assertions suivantes:
\begin{enumerate}[(i)]
\item On suppose que $\Phi$ et $\Psi$ sont cuspidales.
On a l'\'egalit\'e entre fonctions m\'ero\-mor\-phes de $\lambda$ et $\mu$:
\begin{equation*}\int_{\bsX_G(Z)}\bs\Lambda^TE(x,\Phi,\lambda)
\overline{E(x,\Psi, -\bar{\mu})}\dd x= 
\langle \brabsO_{R\vert Q}^{T}(Z;\lambda,\mu)\Phi,\Psi\rangle_R\ptf\end{equation*}
\item On suppose que $\Phi$ et $\Psi$ sont discr\`etes mais non n\'ecessairement cuspidales.
Il existe une constante $c>0$ telle que pour tout $\lambda\in\bsmu_Q$ et tout $\mu\in \bsmu_R$, on ait:
\begin{equation*}\left\lvert \int_{\bsX_G(Z)}
\mskip -39mu
\bs\Lambda^TE(x,\Phi,\lambda)\overline{E(x,\Psi, -\bar{\mu})}\dd x
-\langle \brabsO_{R\vert Q}^{T}(Z;\lambda,\mu)\Phi,\Psi\rangle_{R}
\right\rvert \ll e^{- c\bsd_0(T)}\ptf\end{equation*}
\end{enumerate}
\end{theorem} 

\begin{proof}Prouvons (i). 
Pour $\lambda\in \ag_{Q,\CM}^*$ dans le domaine de convergence de la s\'erie d'Eisenstein 
$E(x,\Phi,\lambda)$, et puisque $\Phi$ est cuspidale, on a \cite[5.4.1]{LW}
\begin{equation*}\bs\Lambda^TE(x,\Phi,\lambda) = \sum_{S,s,\gamma} (-1)^{a(s)}
\phi_{M,s}(s^{-1}(\bfH_0(\gamma x)-T))(\bfM(s,\lambda)(\gamma x, s\lambda)\end{equation*}
o la somme porte sur les $S\in\ESP_\st$ associ\'es \`a $Q$,
$s\in\bfW(\ag_Q,\ag_S)$, $\gamma\in S(F)\bsl G(F)$, et $M=M_Q$.
On d\'eduit de \ref{XMZ} que pour $\lambda$ dans le domaine de convergence 
de $E(x,\Phi,\lambda)$ et $-\bar{\mu}$ dans celui de $E(x,\Psi,-\bar{\mu})$, 
l'int\'egrale de (i) est \'egale \`a
\begin{equation*}\sum_{S,s} \int_{\bsX_S(Z)}(-1)^{a(s)}\phi_{M,s}(s^{-1}(\bfH_0(x)-T))\bfA(x,s)\dd x\leqno{(1)}\end{equation*}
avec
\begin{equation*}\bsX_S(Z)=S(F)U_S(\adef)\bsl\G(\adef;Z)\end{equation*}
{($\bsX_S(Z)$ est l'image de $\bigg(\coprod_{H\in\ESA_M(Z)}\bsX_M(H) \bigg)\times \bs{K}$ dans $\bsX_S$)} 
et
\begin{equation*}\bfA(x,s)= (\bfM(s,\lambda)\Phi)(x,s \lambda)\Pi_S\overline{E(x,\Psi, -\bar{\mu})}\end{equation*}
o $\Pi_SE$ est le terme constant de $E$ le long de $S$.
Notons que $\phi_{M,s}(s^{-1}(\bfH_0(x)-T))$ ne d\'epend que de l'image $(\bfH_S(x)^G - T_S^G)$ 
de $(\bfH_0(x)-T)$ dans $\ag_S^G$. D'apr\`es \cite[5.2.2~(5)]{LW}, on a
\begin{equation*}\bfA(x,s)= \sum_{t\in \bfW^G(\ag_R,\ag_S)} e^{\langle s\lambda - t\mu+2 \rho_S,\bfH_S(x)\rangle}
(\bfM(s,\lambda)\Phi)(x) \overline{\bfM(t,-\bar{\mu})\Psi)(x)}\ptf\end{equation*}
La fonction $\bfM(s,\lambda)\Phi$ appartient \`a $\Autom_\cusp(\bsX_S)_{s\xi}$ et la fonction 
$\bfM(t,-\bar{\mu})\Psi$ appartient \`a $\Autom_\cusp(\bsX_S)_{t\xi'}$. 
Il r\'esulte de \ref{phipsi} et \ref{XMZ} 
que l'expression (1) est \'egale \`a la somme sur $S$, $s$ et $t$ de 
\begin{equation*}\vert \wh\bsbbc_S \vert^{-1}\mskip -5mu\sum_{\nu \in \wh{\bsbbc}_S}
\sum_{H\in \ESA_S(Z)}\mskip -20mu(-1)^{a(s)}\phi_{M,s}(H-T_S)e^{\langle s\lambda - t\mu-\nu,H \rangle}
 \langle \bfD_\nu\bfM(s,\lambda)\Phi, \bfM(t,-\bar{\mu})\Psi\rangle_S\ptf\end{equation*}
 Fixons un triplet $(S,s,t)$ comme ci-dessus. En tenant compte de 
 \ref{extres} on a pour $\lambda$ assez r\'egulier et $\mu$ fix\'e:
\begin{equation*}\sum_{H\in \ESA_S(Z)} (-1)^{a(s)}\phi_{M,s}(H-T_S)e^{\langle s\lambda - t\mu-\nu,H\rangle}=
\varepsilon_{S}^{G,T_S}(Z;s\lambda - t\mu-\nu)\ptf\end{equation*}
On obtient que l'expression (1) est \'egale \`a la somme sur $S$, $s$ et $t$ de 
\begin{equation*}\vert \wh\bsbbc_S \vert^{-1}\sum_{\nu \in \wh{\bsbbc}_S}
\varepsilon_{S}^{G,T_S}(Z;s\lambda - t\mu -\nu)
\langle \bfD_\nu\bfM(s,\lambda)\Phi, \bfM(t,-\bar{\mu})\Psi\rangle_{S}\ptf\leqno{(2)}\end{equation*}
soit encore
\begin{equation*}\vert \wh\bsbbc_R \vert^{-1}\sum_{\nu \in \wh{\bsbbc}_R}
\varepsilon_{S}^{G,T_S}(Z;s\lambda - t(\mu +\nu))
\langle \bfD_{t\nu}\bfM(s,\lambda)\Phi, \bfM(t,-\bar{\mu})\Psi\rangle_{S}\ptf\leqno{(3)}\end{equation*}
et, gr‰ce \`a l'\'equation fonctionnelle \ref{eqf},
on obtient que (3) est \'egal \`a
\begin{equation*}\vert \wh\bsbbc_R \vert^{-1}\sum_{\nu \in \wh{\bsbbc}_R}\varepsilon_{S}^{G,T_S}(Z;s\lambda - t(\mu+\nu))
\langle\bfD_\nu \bfM(t,-(\mu+\nu))^{-1}\bfM(s,\lambda)\Phi,\Psi \rangle_{R}\ptf\leqno{(4)}\end{equation*}
On voit appara"tre la $(G,M)$-famille spectrale
\`a valeurs op\'erateurs pour $M=M_R$ et l'int\'egrale de (i) est donc \'egale \`a
\begin{equation*}\vert \wh\bsbbc_R \vert^{-1}\sum_{\nu \in \wh{\bsbbc}_R}
\langle \bfD_\nu\bsO_{R\vert Q}^{T}(Z;\lambda,\mu+\nu)\Phi,\Psi\rangle_{R}\ptf\end{equation*}
L'assertion (i) en r\'esulte.
Le cas g\'en\'eral (ii) est dž \`a Arthur \cite{A3} dans le cas des corps de nombres. 
La preuve consiste \`a se ramener au cas cuspidal, c'est-\`a-dire \`a la formule de Langlands \cite[5.4.2.(i)]{LW}. 
Dans le cas des corps de fonctions, on prouve de la mme mani\`ere (ii) \`a partir de (i). 
Notons qu'ici, les groupes $\bsmu_Q$ et $\bsmu_R$ sont compacts, 
d'o la borne uniforme en $\lambda$ et $\mu$. 
\end{proof}

Sous les hypoth\`eses (i) et (ii) de \ref{hypo}, pour que l'int\'egrale
\begin{equation*}\int_{\bsX_G(Z)}\bs\Lambda^TE(x,\Phi,\lambda)\overline{E(x,\Psi, -\bar{\mu})}\dd x\end{equation*}
soit non nulle, il faut que
 $w\xi$ et $\xi'$ co\"{i}ncident sur $A_R(F)\bsl A_R({\adef})^1$ pour un (et donc pour tout) $w\in \bfW(\ag_Q,\ag_R)$. 
Cette condition \'equivaut \`a l'existence d'un $\tau\in \bsmu_R$ tel que
\begin{equation*}(w\xi) \star \tau = \xi'\ptf\end{equation*}
Son image dans $\wh{\ESB}_R$ est uniquement d\'etermin\'ee.
\begin{proposition}\label{d\'ecalage} 
Notons $\bs{\EuScript{E}}(\xi,\xi')$ l'ensemble des $\tau\in \bsmu_R$ v\'erifiant l'\'equation
\begin{equation*}(w\xi) \star \tau = \xi'\ptf\end{equation*}
pour un $w\in \bfW(\ag_Q,\ag_R)$. 
S'il est non vide, c'est un espace principal homog\`ene sous $\wh{\bsbbc}_R$,
ind\'ependant du choix de $w$.
Sous les hypoth\`eses (i) et (ii) de \ref{hypo}, le th\'eor\`eme \ref{PSET} reste vrai
sans l'hypoth\`ese (iii) \`a condition de remplacer
$\brabsO_{R\vert Q}^{T}(Z;\lambda,\mu)$ par l'op\'erateur
\begin{equation*}\brabsO_{R\vert Q}^{T}(Z,\xi,\xi ';\lambda,\mu)\bydef
\vert \wh\bsbbc_R \vert^{-1}\sum_{\nu \in \bs{\EuScript{E}}(\xi,\xi')} 
\bfD_{\nu}\mskip 2mu \bsO_{R\vert Q}^{T}(Z;\lambda,\mu+\nu)\ptf\end{equation*}
Par convention $\brabsO_{R\vert Q}^{T}(Z,\xi,\xi ';\lambda,\mu)=0$ si $\bs{\EuScript{E}}(\xi,\xi')$ est vide. 
Si $(\lambda - \mu - \nu)\in\wh\bsbbc_G$ pour $\nu \in \bs{\EuScript{E}}(\xi,\xi')$,
chacun des membres de l'\'egalit\'e ne d\'epend que de l'image de $Z$ dans $\bsbbc_G$.
\end{proposition}

\begin{proof}
Il suffit d'observer que $E(x, \bfD_{\nu}\Psi,\mu)=E(x,\Psi, \mu +\nu)$.
\end{proof}

 \chapter{Le noyau int\'egral}
\label{le noyau int\'egral}


 \section{Op\'erateurs et noyaux}\label{les op\'erateurs}
 Notons $C^\infty_\mathrm{c}(\tG({\adef}))$ l'espace des fonctions lisses et \`a support compact sur $\tG({\adef})$. 
On notera $\dd y$ la mesure $G({\adef})$-invariante \`a droite et \`a gauche sur $\tG({\adef})$ 
d\'eduite de la mesure de Haar $\dd x$ sur $G({\adef})$ en posant: 
\begin{equation*}\int_{\tG({\adef})}f(y)\dd y=\int_{G({\adef})}f(\delta x)\dd x\quad \hbox{avec} \quad \delta\in\tG(F)\ptf\end{equation*}
La mesure ainsi d\'efinie est ind\'ependante du choix de $\delta$. L'espace tordu localement compact 
$\tG({\adef})$ est unimodulaire au sens de \cite[2.1]{LW}. 
Il agit sur $\bsX_G$ de la mani\`ere suivante: pour $x\in\bsX_G$ et $y\in\tG({\adef})$, on choisit 
un repr\'esentant $\dot{x}$ de $x$ dans $G({\adef})$ et un \'el\'ement $\delta$ dans $\tG(F)$. 
Alors $\dot{x}'=\delta^{-1}\dot{x}y$ est un \'el\'ement de $G({\adef})$, dont l'image $x'$ dans 
$\bsX_G$ ne d\'epend pas des choix de $\dot{x}$ et de $\delta$. On pose $x*y=x'$. 

On fixe dans toute la suite un caract\`ere unitaire $\omega$ de $\G(\adef)$ 
 trivial sur le groupe $A_\tG({\adef})G(F)$. La repr\'esentation r\'eguli\`ere droite $\Rho$ de $G({\adef})$ dans 
$L^2(\bsX_G)$ se prolonge naturellement en une repr\'esentation unitaire 
$\tRho$ de $(\tG({\adef}),\omega)$, au sens de \cite[2.3]{LW}: pour $\varphi\in L^2(\bsX_G)$, et 
$x$ et $y$ comme ci-dessus, on pose
\begin{equation*}\tRho(y,\omega)\varphi(x)= (\omega\varphi)(x*y)=
\omega(\delta^{-1}\dot{x}y)\varphi(\delta^{-1}\dot{x}y)\ptf\end{equation*}
Par int\'egration contre une fonction $f\in C^\infty_\mathrm{c}(\tG({\adef}))$, on d\'efinit l'op\'erateur
\begin{equation*}\tRho(f,\omega)=\int_{\tG({\adef})}f(y)\tRho(y,\omega)\dd y\ptf\end{equation*}
Il est repr\'esent\'e par le noyau int\'egral sur $\bsX_G\times\bsX_G$: 
 \begin{equation*}K_\tG(f,\omega;x,y)=\sum_{\delta\in\tG(F)}\omega(y)f(x^{-1}\delta y)\end{equation*}
c'est-\`a-dire que
\begin{equation*}(\tRho(f,\omega)\varphi)(x)=\int_{\bsX_G}K_\tG(f,\omega;x,y)\varphi(y)\dd y\ptf\end{equation*}
Le noyau $K_\tG(f,\omega;x,y)$ sera not\'e $K(f,\omega;x,y)$ si aucune confusion craindre. 
D'apr\`es \cite[6.2.1]{LW} on a le 
\begin{lemma}\label{noyineg}
 Il existe des constantes $c(f)$ et $N$ telles que, pour tout $x$ et tout $y$ dans $G({\adef})$, on ait
\begin{equation*}\vert K(f,\omega;x,y)\vert\leq c(f)\vert x\vert^N\vert y\vert^N\ptf\end{equation*}
\end{lemma}

 \section{Factorisation du noyau}\label{factorisation du noyau}
Pour $f\in C^\infty_\mathrm{c}(\tG({\adef}))$ et $h\in C^\infty_\mathrm{c}(G({\adef}))$, 
on note $f\star h\in C^\infty_\mathrm{c}(\tG({\adef}))$ la fonction d\'efinie par 
\begin{equation*}(f\star h) (x)=\int_{G({\adef})} f(xy\mun)h(y)\dd y\ptf\end{equation*}
Le noyau int\'egral de l'op\'erateur $\tRho(f*h,\omega)$ sur $\bsX_G$ est donn\'e par
\begin{equation*}K(f\star h,\omega;x,y)=\int_{\bsX_G}K(f,\omega;x,z)K_G(\omega h;z,y)\dd z\ptf\end{equation*}
Toute fonction $f\in C^\infty_\mathrm{c}(\tG({\adef}))$ est $\bsK'$-bi-inva\-riante, 
 c'est-\`a-dire invariante \`a droite et \`a gauche, par $\bsK'$ un sous-groupe ouvert compact de $G({\adef})$, 
que l'on peut choisir distingu\'e dans $\bsK$. 
Si on suppose que le caract\`ere $\omega$ est trivial sur $\bsK'$ la fonction
\begin{equation*}\bsX_G\times\bsX_G\rightarrow {\CM},\mskip 2mu  (x,y)\mapsto K(f,\omega;x,y)\end{equation*}
est $(\bsK' \times\bsK')$-invariante (pour l'action \`a droite)
et le noyau $K(f,\omega;x,y)$ 
est {$\Autom$-\textit admissible} au sens de \cite[6.3]{LW}.
Le th\'eor\`eme de factorisation de Dixmier-Mallia\-vin \cite[6.3.1]{LW} est trivialement vrai ici: 
notons $e_{\bsK'}$ la fonction caract\'eristique de $\bsK'$ divis\'ee par $\vol (\bsK')$. C'est un idempotent de 
$C^\infty_\mathrm{c}(G({\adef}))$ et l'on a
\begin{equation*}f= f\star e_{\bsK'}= e_{\bsK'}\star f\;= e_{\bsK'}\star f\star e_{\bsK'}\ptf\end{equation*}
Puisque $\omega\vert_{\bs{K}'}=1$ le noyau $K(f,\omega;x,y)$ s'\'ecrit
\begin{equation*}K(f,\omega;x,y) =\int_{{\bsX}_{G}}K(f,\omega;x,z)K_G(e_{\bsK'};z,y)\dd z\ptf\end{equation*}


 \section{Propri\'et\'es du noyau tronqu\'e}
On a d\'efini en \ref{Siegel} un domaine de Siegel $\Siegel^*=\EE_G\Siegel^1$ pour le quotient
$$\BB_GG(F)\bsl\G(\adef)$$ et on pose $G(\adef)^*=\EE_G G(\adef)^1$.
On note $\bs\Lambda^T_1$ l'op\'erateur de troncature agissant sur la premi\`ere variable d'un noyau 
$K(f,\omega;x,y)$. On a la variante \cite[IV.2.5~(b)]{MW1} des lemmes \cite[6.4.1, 6.4.2]{LW}:

\begin{lemma}\label{prop1nt}
(i) -- Il existe un sous-ensemble compact $\Omega_1$ de $\Siegel^*$ tel que pour tout $y\in G(\adef)^*$, la fonction 
\begin{equation*}\Siegel^*\rightarrow \CM\vgq x\mapsto \bs\Lambda_1^TK(f,\omega;x,y)\end{equation*}
soit \`a support dans $\Omega_1$. De plus la fonction
\begin{equation*}\Siegel^*\times \Siegel^*\rightarrow \CM\vgq (x,y)\mapsto\bs\Lambda_1^TK(f,\omega;x,y)\end{equation*}
est \`a support compact, donc born\'ee. 
\pni (ii) -- Soit $\bsK'$ un sous-groupe ouvert compact de $G({\adef})$.
Il existe un sous-ensemble compact $\Omega_2$ de $\Siegel^* \times \Siegel^*$ 
tel que pour toute fonction $\bsK'$-bi-invariante $f$ dans $C^\infty_\mathrm{c}(\tG({\adef}))$, le support de la restriction 
\`a $\Siegel^* \times \Siegel^*$ 
du noyau tronqu\'e $(x,y)\mapsto\bs\Lambda^T_1K(f,\omega;x,y)$ soit 
contenu dans $\Omega_2$.
\end{lemma}

\begin{proof}
La fonction $(x,y)\mapsto K(f,\omega;x,y)$ sur $\bsX_G\times \bsX_G$ 
est $(\bsK'\times\bsK')$-invariante pour un sous-groupe ouvert compact $\bsK'$ de $G({\adef})$. On peut donc 
appliquer \ref{supportcompact}: il existe un sous-ensemble compact $\Omega_1$ de $\Siegel^*$ tel que 
pour tout $y\in G({\adef})^*$, le support de la fonction 
$x\mapsto\bs\Lambda^T_1K(f,\omega;x,y)$ soit 
contenu dans $\Omega_1$. On proc\`ede ensuite comme dans la preuve de \cite[IV.2.5~(b)]{MW1}.
La derni\`ere assertion r\'esulte de \ref{supportcompact} et de la preuve de \cite[IV.2.5~(b)]{MW1}. 
\end{proof}

 \chapter{D\'ecomposition spectrale}\label{d\'ecomposition spectrale}

Les sorites de \cite[7.1]{LW} sont valables ici.
La d\'ecomposition spectrale de $L^2(\bsX_G)$ a \'et\'e obtenue par Langlands 
pour les corps de nombres \cite{Lan} et Morris pour les corps de fonctions \cite{Mo1,Mo2} puis
r\'edig\'ee pour tout corps global par M\oe glin et Waldspurger \cite{MW1}. 

 \section{Un r\'esultat de finitude}\label{resultatfinitude}
Soit $P\in \ESP$. On observe qu'une fonction $\bsK$-finie sur $\bsX_P$ est forc\'ement $\bsK'$-invariante 
\`a droite pour un sous-groupe 
ouvert $\bsK'$ de $\bsK$. Pour $\xi\in \Xi(P)$, on note $\Autom_{\disc,\bsK'}(\bsX_P)_\xi$ le sous-espace 
de $\Autom_\mathrm{disc}(\bsX_P)_\xi$ form\'e des fonctions qui sont $\bsK'$-invariantes. 
On d\'efinit de la mme mani\`ere 
les espaces $\Autom_{\cusp,\bsK'}(\bsX_P)_\xi$. 
Sur un corps de fonctions on a le r\'esultat de finitude suivant:

\begin{theorem}\label{thmfinitude}
La repr\'esentation de $G(\adef)$ dans $\Autom_\disc(\bsX_P)_\xi$ est {\rmfamily admissible}:
pour tout sous-groupe ouvert compact $\bsK'$ de $G(\adef)$, l'espace $\Autom_{\disc,\bsK'}(\bsX_P)_\xi$ 
est de dimension finie. 
 \end{theorem}

\begin{proof}
D'apr\`es \ref{lemme2TC}, il existe un sous-ensemble compact $\Omega$ de $\bsX_P$ tel que toute forme 
automorphe cuspidale $\bsK'$-invariante sur $\bsX_P$ soit \`a support contenu dans $\Omega$. On en d\'eduit 
que l'espace 
$\Autom_{\cusp,{\bsK'}}(\bsX_P)_\xi$ est de dimension finie. 
En d'autres termes, la repr\'esentation 
de $G(\adef)$ dans $\Autom_\cusp(\bsX_P)_\xi$ est admissible. 
 D'apr\`es la d\'ecomposition spectrale de Langlands, les formes automorphes discr\`etes
 \begin{equation*}\Phi \in \Autom_{\disc,\bsK'}(\bsX_P)_\xi\end{equation*}
 s'obtiennent comme r\'esidus de s\'eries d'Eisenstein $E^P(x,\Phi',\lambda)$ construites \`a partir de formes automorphes cuspidales 
 $\Phi' \in \Autom_{\cusp,\bsK'}(\bsX_Q)_{\xi'}$ pour 
 $Q\in \ESP$ tel que $Q\subset P$, 
 $\xi'\in \Xi(Q)$ et $\lambda \in \bsmu_Q^+ \buildrel \mathrm{d\acute{e}f}\over{=} \ag_{Q,\mathbb{C}}^*/ \ES{A}_Q^\vee$; 
 avec la condition que l'ŽlŽment 
$\xi' \star \lambda $ de $\Xi(Q)^+$ dŽfini par $(\xi'\star \lambda)(a)=\xi'(a)e^{\langle \lambda, {\bf H}_Q(a)\rangle}$ 
pour tout $a\in A_Q(\mathbb{A})$, prolonge $\xi$. 
Observons que $\xi'\star \lambda$ ne dŽpend que de la projection de $\lambda$ sur 
$\ag_{Q,\mathbb{C}}^*/\ES{B}_Q^\vee= \ker [\Xi(Q)^+ \rightarrow \Xi(Q)^1]$.  
Pour $\mu \in \bsmu_Q\;(= \widehat{\ESA}_Q)$, la fonction ${\bf D}_\mu \Phi'$ 
appartient ˆ $\Autom_{\cusp,\bsK'}(\bsX_Q)_{\xi'\star \mu}$ et 
$$E^P(x,{\bf D}_\mu \Phi', \lambda -\mu)= 
E^P(x,\Phi',\lambda)\qquad \hbox{pour tout}\qquad \lambda \in \bsmu_Q^+\ptf$$ 
Soit $\Xi(Q)\times^{\wh{\ES{B}}_Q}\bsmu_Q^+$ le quotient de 
$\Xi(Q)\times \bsmu_Q^+$ par la relation d'Žquivalence dŽfinie comme suit: 
deux couples $(\xi'_1,\lambda_1)$ et $(\xi'_2,\lambda_2)$ sont Žquivalents si et seulement 
s'il existe un $\mu\in \bsmu_Q$ tel que $(\xi'_2,\lambda_2)= (\xi'_1\star \mu, \lambda_1-\mu)$; 
auquel cas $\xi'_2\star \lambda_2= \xi'_1\star \lambda_1$. 
Puisque $\wh\bsbbc_Q=\ker [\bsmu_Q^+ \rightarrow \ag_{Q,\mathbb{C}}^*/\ESB_Q^\vee]$ est fini, 
l'application (surjective) 
$$\Xi(Q)\times^{\widehat{\ESB}_Q} \bsmu_Q^+\rightarrow \Xi(Q)^+,\, (\xi'\!,\lambda)\mapsto \xi'\star \lambda$$ 
est ˆ fibres fines. Comme l'espace $\Autom_{\cusp,\bsK'}(\bsX_Q)_{\xi'}$ est de dimension finie, 
il suffit de voir que les ŽlŽments de $\Xi(Q)^+$ de la forme $\xi' \star \lambda$ 
avec $(\xi'\!,\lambda)\in \Xi(Q)\times  \bsmu_Q^+$ tels qu'il existe une forme 
$\Phi'\in \Autom_{\cusp,\bsK'}(\bsX_Q)_{\xi'}$ pouvant donner naissance par rŽsidu de la sŽrie 
d'Eisenstein $E^P(x,\Phi',\cdot)$ en $\lambda$ ˆ une forme dans $\Autom_{\disc,\bsK'}(\bsX_P)_\xi$, 
appartiennent ˆ un ensemble fini. La projection 
$(\xi'\star \lambda)^1= \xi'\star \lambda\vert_{A_Q(\mathbb{A})^1}$ de 
$\xi'\star \lambda$ sur $\Xi(Q)^1$ co\"{\i}ncide avec celle de $\xi'$, qui est un caractre du groupe fini 
$A_Q(F)(\bs{K}'\cap A_Q(\mathbb{A})^1)\backslash A_Q(\mathbb{A})^1$. 
Par consŽquent les classes $\xi'+ \widehat{\ESB}_Q$ varient dans un sous-ensemble fini 
de $\Xi(Q)/\widehat{\ESB}_Q$. On peut donc fixer $\xi'$ et se contenter de faire varier $\lambda$. 
\'Ecrivons $\lambda= \lambda_{\mathrm{u}} + \lambda^+$ avec $\lambda_{\mathrm{u}}\in \bsmu_Q$ et 
$\lambda^+ \in \ag_Q^*$. La projection $(\lambda^+)_P$ de $\lambda$ sur $\ag_P^*$ est nulle (car 
$\xi'\star \lambda\vert_{A_P(\mathbb{A})} = \xi$) et la projection $(\lambda^+)_Q^P$ de 
$\lambda$ sur $\ag_Q^{P,*}$ varie dans un compact de 
$\ag_Q^{P,*}$. La compacitŽ de $\bsmu_Q$ assure que $\lambda$ varie dans un compact du 
cylindre $\bsmu_Q^+$. L'intersection d'un ensemble compact et d'un ensemble discret 
(les p™les d'une sŽrie d'Eisenstein) est finie, ce qui achve la dŽmonstration du thŽorme. 
\end{proof}

Ce th\'eor\`eme rend inutile le d\'ecoupage suivant les donn\'ees cuspidales utilis\'e par Arthur (et repris dans \cite{LW})
dans le d\'eveloppement spectral de la formule des traces, puisqu'il r\`egle imm\'ediatement 
les \'eventuelles questions de convergence.

 \section{Donn\'ees discr\`etes et d\'ecomposition spectrale}\label{donn\'ees discr\`etes}
Pour $M\in\ESL$, notons 
$\bfW^G(M)$ le quotient de l'ensemble des \'el\'ements $w\in \bfW^G$ tels que $w(M)=M$ par $\bfW^{M}$. 
C'est un groupe et on note $w^G(M)$\index{wgm@$w^G(M)$}
son ordre. 
Rappelons que pour $\sigma$ une repr\'esentation de $M(\adef)$ et $\lambda\in\bsmu_M$, on a not\'e 
$\sigma_\lambda = \sigma\star\lambda$ la repr\'esentation d\'efinie par les op\'erateurs
\begin{equation*}\sigma\star\lambda:x\mapsto e^{\langle\lambda,\bfH_M(x)\rangle}\sigma(x)\ptf\end{equation*}
\begin{definition}\label{defcms}
On appelle \textit{donn\'ee discr\`ete} pour $G$ un couple $(M,\sigma)$ o $\sigma$ est une repr\'esentation 
automorphe irr\'eductible de $M(\adef)$ discr\`ete modulo le centre c'est-\`a-dire
apparaissant comme composant de $L^2_\disc(\bsX_M)_\xi$ --
l'espace de Hilbert engendr\'e par les sous-repr\'esentations irr\'eductibles de $L^2(\bsX_{M})_\xi$ --
pour un caract\`ere unitaire automorphe $\xi$ de $A_M(\adef)$.
Deux donn\'ees discr\`etes $(M,\sigma)$ et $(M',\sigma')$ de $G$ sont dites \textit{\'equivalentes} s'il existe un couple 
$(w,\lambda)\in \bfW^G\times\bsmu_M$ tel que
\begin{equation*}wMw^{-1} =M',\quad w(\sigma\star \lambda )\simeq\sigma'\ptf\end{equation*}
Nous noterons $\Stab$ le sous-groupe de $\wh\bsbbc_M$
form\'e des $\lambda$ tels que $\sigma\star\lambda\simeq\sigma$
et $\cMsig$ son indice:
\begin{equation*}\cMsig\index{cmsig@$\cMsig$}=\unstab\ptf\end{equation*}
\end{definition}

Soit $(M,\sigma)$ une donn\'ee discr\`ete pour $G$ et soit $P\in\ESP(M)$. Soit
$\xi$ la restriction \`a $A_M(\adef)$ du caract\`ere central de $\sigma$. On notera
\begin{equation*}\Automd(\bsX_P,\sigma)\index{automsigma@$\Automd(\bsX_P,\sigma)$}
\subset \Autom_\disc(\bsX_P)_\xi\end{equation*}
le sous-espace des formes automorphes $\varphi$ sur $\bsX_P$ telles que pour tout $x\in G(\adef)$, 
la fonction $m\mapsto\varphi(mx)$ 
sur $\bsX_{M}$ soit un vecteur de la composante isotypique 
de $\sigma$ dans $L^2_\disc(\bsX_{M})_{\xi}$. C'est l'espace des fonctions
$\bsK$-finies \`a droite dans l'espace de la repr\'esentation induite parabolique de $P(\adef)$ \`a $G(\adef)$
de la composante isotypique de $\sigma$ dans $L^2_\disc(\bsX_M)_\xi$.

Consid\'erons $x\in\bsX_P$, $y\in\tG(\adef)$, $\theta=\mathrm{Int}_\delta$ avec $\delta\in\tG(F)$ 
et $\mu\in\ag_{M,{\CM}}^*$. Rappelons que l'on a pos\'e
\begin{equation*}\varphi(x,\mu)= e^{\langle \mu +\rho_P,\bfH_P(x)\rangle}\varphi(x)\ptf\leqno(1)\end{equation*}

\begin{definition}\label{rhotu}
Pour une repr\'esentation automorphe irr\'eductible $\sigma$ de $M_P({\adef})$ discr\`ete modulo le centre, on d\'efinit 
 pour $Q=\theta(P)$ un op\'erateur unitaire
\footnote{On prendra garde \`a ce que, contrairement au cas des corps de nombres, on ne dispose
pas d'un repr\'esentant canonique dans l'orbite de $\sigma$ sous les d\'ecalages par les $\mu\in\bsmu_M$.}
\begin{equation*}\Rho_{P,\sigma,\mu}\index{rhopsigma@$\Rho_{P,\sigma,\mu}$}
(\delta,y,\omega): \Automd(\bsX_P,\sigma)\rightarrow 
\Automd(\bsX_Q,\theta(\omega\otimes \sigma))\end{equation*}
en posant
\begin{equation*}(\Rho_{P,\sigma,\mu}(\delta,y,\omega)\varphi)(x,\theta(\mu))= (\omega\varphi)(\delta\mun xy,\mu)\ptf\leqno(2)\end{equation*}
\end{definition}
Cet op\'erateur r\'ealise un avatar tordu par $\delta$ et $\omega$ de la repr\'esentation induite parabolique 
\begin{equation*}\mathrm{Ind}_{P(\adef)}^{G(\adef)}(\sigma\star\mu)\ptf\end{equation*}
Diff\'erentes r\'ealisations peuvent appara"tre et doivent tre compar\'ees:

\begin{lemma}\label{romu}
Pour $\mu$ et $\lambda \in \bsmu_M$, les avatars tordus
 \begin{equation*}\Rho_1=\Rho_{P,\sigma,\lambda+\mu}(\delta,y,\omega)\qquad\hbox{et}\qquad 
 \Rho_2=\Rho_{P,\sigma\star\lambda,\mu}(\delta,y,\omega)\end{equation*}
sont \'equivalents et l'entrelacement est donn\'e par les op\'erateurs 
$\bfD_\lambda$ et $\bfD_{\theta(\lambda)}$ (d\'efinis en \ref{Dnu}).
En d'autres termes, le diagramme suivant 
\begin{equation*}\xymatrix{\Automd(\bsX_P,\sigma) \ar[r]^(.43){\mathbf{\rho}_1}\ar[d]_{\bfD_\lambda} & 
\Automd(\bsX_Q,\theta(\omega\otimes \sigma))\ar[d]^{\bfD_{\theta(\lambda)}} \\
 \Automd(\bsX_P,\sigma\star \lambda) \ar[r]^(.43){\mathbf{\rho}_2} & 
 \Automd(\bsX_Q,\theta(\omega\otimes \sigma\star\lambda))}\end{equation*}
est commutatif, c'est-\`a-dire que l'on a
\begin{equation*}\bfD_{\theta(\lambda)}\circ\Rho_{P,\sigma,\lambda+\mu}(\delta,y,\omega)=
\Rho_{P,\sigma\star\lambda,\mu}(\delta,y,\omega)\circ \bfD_\lambda\ptf\leqno(3)\end{equation*}
\end{lemma}
\begin{proof} C'est une cons\'equence imm\'ediate des \'equations (1) et (2).
\end{proof}

Par int\'egration contre une fonction 
$f\in C^\infty_\mathrm{c}(\tG({\adef}))$, on d\'efinit l'op\'erateur \begin{equation*}\Rho_{P,\sigma,\mu}(\delta,f,\omega)\end{equation*}
et on pose
\begin{equation*}\tRho_{P,\sigma,\mu}(y,\omega)=\Rho_{P,\sigma,\mu}(\delta_0,y,\omega)\vgq 
\tRho_{P,\sigma,\mu}(f,\omega)=\Rho_{P,\sigma,\mu}(\delta_0, f,\omega)\ptf\end{equation*}
Soit $\varphi: \bsX_G \rightarrow \CM$ une fonction continue et \`a support compact. 
Pour $P\in\ESP$, $\Psi\in \Autom_\disc(\bsX_P)$ et $\mu\in\bsmu_P$, on pose
\begin{equation*}\widehat{\varphi}(\Psi,\mu)=\int_{\bsX_G}\varphi(x)\overline{E(x,\Psi,\mu)}\dd x\ptf\end{equation*}
Pour deux fonctions $\phi,\varphi: \bsX_G \rightarrow \CM$ 
continues et \`a support compact, on pose
\begin{equation*}\langle\phi,\varphi\rangle_{\bsX_G} =\int_{\bsX_G}\phi(x)\overline{\varphi(x)}\dd x\ptf\end{equation*}
 Pour $M\in \ESL$, notons 
\begin{itemize}
\item $\Pi_\disc(M)$ l'ensemble 
des classes d'isomorphisme de repr\'esentations auto\-mor\-phes irr\'eductibles de $M(\adef)$ discr\`etes modulo le centre, 
\item $\bPI_\disc(M)$ le quotient de $\Pi_\disc(M)$ par la relation d'\'equivalence donn\'ee par la torsion par les 
caract\`eres unitaires de $\ESA_M$. 
\item On notera $\Base_P(\sigma)$ \index{Psi@$\Base_P(\sigma)$}
une base orthonormale de l'espace vectoriel pr\'e-hilbertien $\Automd(\bsX_P,\sigma)$.

\end{itemize}
D'apr\`es \cite[VI]{MW1} 
avec les conventions de \ref{convention-mesures} pour la normalisation des mesures ($\mathrm{vol}(\bsmu_M)=1$) 
et les notations de \ref{defcms}, on a le

\begin{theorem}\label{th\'eor\`eme principal de [LW]}
Le produit scalaire $\langle\phi,\varphi\rangle_{\bsX_G}$ admet la d\'ecomposition spectrale suivante:
\begin{equation*}\langle\phi,\varphi\rangle_{\bsX_G} = 
 \sum_{M \in \ESL/\bfW}\frGM
 \sum_{\bs{\sigma} \in \bPI_\disc(M)} \cMsig
 \int_{\bsmu_M}\sum_{\Psi\in\Base_P(\sigma)}
\widehat{\phi}(\Psi,\mu)\overline{\widehat{\varphi}(\Psi,\mu)}\dd\mu\end{equation*}
o l'on a identifi\'e $\ESL/\bfW$ a un ensemble de repr\'esentants dans $\ESL$ et o pour chaque classe 
$\bs{\sigma}\in \bPI_\disc(M)$ on a choisi un repr\'esentant $\sigma\in \Pi_\disc(M)$ 
dans la classe $\bs{\sigma}$\footnote{On observera que pour chaque facteur de 
Levi $M\in \ESL/{\mathbf W}$, on a choisi un sous-groupe parabolique 
$P$ de composante de Levi $M$. Le choix de ces sous-groupes paraboliques $P$ est indiff\'erent. 
Il en est de mme pour les formules des \'enonc\'es 
\ref{decspec} et \ref{KQd} ci-dessous.}.
\end{theorem}


 \section{D\'ecomposition spectrale d'un noyau}\label{d\'ecomposition d'un noyau}
La proposition \cite[7.2.2]{LW} est vraie ici, mutatis mutandis\footnote{On observera que dans \cite{LW} la d\'efinition 
des espaces $\bsX_P$ diff\`ere de la n™tre par un quotient par $\BB_P$; il en r\'esulte que,
pour que la formule \cite[7.2.2~(1)]{LW} soit correcte, il faut la modifier comme indiqu\'e en 
 \Err(viii) dans l'Annexe \ref{Erratum}.}. 
Plus pr\'ecis\'ement, soient $P\in\ESP_\st$ et $\theta$ un $F$-automorphisme de $G$. 
Soit $H(x,y)$ un noyau int\'egral sur $\bsX_{\theta(P)}\times\bsX_P$ de la forme $H=K_1K_2^*$:
\begin{equation*}H(x,y)=\int_{\bsX_P}K_1(x,z)K_2^*(z,y)\dd z\end{equation*}
o $K_1$ (resp. $K_2$) est un noyau $\Autom$-admissible sur 
$\bsX_{\theta(P)}\times\bsX_P$ (resp. $\bsX_P\times\bsX_P$). 
On suppose que pour $S\in\ESP^{P}_\st$, $\sigma\in\Pi_\disc(M_S)$ et $\mu\in\bsmu_{S}$,
on a des op\'erateurs de rang fini et, plus pr\'ecis\'ement, s'annulant en dehors d'un ensemble fini de 
vecteurs de $\Base_S(\sigma)$:
\begin{equation*}A_{1,\sigma,\mu}\in {\bfHom}(\Automd(\bsX_S,\sigma), \Automd(\bsX_{\theta(S)},\theta(\sigma))\end{equation*}
et 
\begin{equation*}A_{2,\sigma,\mu}\in {\bfHom}(\Automd(\bsX_S,\sigma), \Automd(\bsX_S,\sigma))\end{equation*}
v\'erifiant
\begin{equation*}\int_{\bsX_P}K_1(x,y)E^P(y,\Psi,\mu)\dd y = E^{\theta(P)}(x,A_{1,\sigma,\mu}\Psi,\theta(\mu))\end{equation*}
et
\begin{equation*}\int_{\bsX_P}K_2(x,y)E^P(y,\Psi,\mu)\dd y = E^P(x,A_{2,\sigma,\mu}\Psi,\mu)\ptf\end{equation*}
Posons
\begin{equation*}B_{\sigma,\mu}= A_{1,\sigma,\mu}A_{2,\sigma,\mu}^*\in {\bfHom}
(\Automd(\bsX_S,\sigma), \Automd(\bsX_{\theta(S)},\theta(\sigma))\end{equation*}
et
\begin{equation*}H_{\sigma}(x,y;\mu)=\sum_{\Psi\in\Base_S(\sigma)}
E^{\theta(P)}(x,B_{\sigma,\mu}\Psi,\theta(\mu))\overline{E^P(y,\Psi,\mu)}\ptf\end{equation*}
\begin{proposition} \label{decspec}
Le noyau $H(x,y)$ admet la d\'ecomposition spectrale suivante:
\begin{equation*}H(x,y)= \sum_{M \in \ESL/\bfW}\frGM \sum_{\bsigma\in\bsPi_\disc(M)} \cMsig
\int_{\bsmu_M}H_{\sigma}(x,y;\mu)\dd \mu\ptf\leqno{(1)}\end{equation*}
De plus, la somme sur $\Psi$ dans l'expression $H_\sigma(x,y;\mu)$ 
est finie, et en posant
\begin{equation*}h(x,y)=\sum_{M \in \ESL/\bfW}\frGM \sum_{\bsigma\in\bsPi_\disc(M)} \cMsig
\int_{\bsmu_M}\vert H_{\sigma}(x,y;\mu)\vert \dd \mu\vg\end{equation*}
on a la majoration (in\'egalit\'e de Schwartz)
\begin{equation*}\vert H(x,y)\vert\leq h(x,y)\leq K_1K_1^*(x,x)^{1/2}K_2K_2^*(y,y)^{1/2}\ptf\leqno{(2)}\end{equation*}
\end{proposition}
\begin{proof}Comme dans la preuve de \cite[7.2.2]{LW}, 
cela r\'esulte des g\'en\'eralit\'es sur la d\'ecomposition spectrale des noyaux produits 
 \cite[7.1.1~(1)]{LW} et de la forme explicite de la d\'ecomposition spectrale automorphe 
 (th\'eor\`eme \ref{th\'eor\`eme principal de [LW]}). 
\end{proof}


Pour $\delta\in\tG(F)$, 
posons $\theta = \mathrm{Int}_\delta$ et $Q=\theta(P)$. On consid\`ere l'op\'erateur 
\begin{equation*}\Rho(\delta,f,\omega): L^2(\bsX_P)\rightarrow L^2(\bsX_Q)\end{equation*}
d\'efini par
\begin{equation*}\Rho(\delta,f,\omega)\phi(x)=\int_{\tG(\adef)}f(y)(\omega\phi)(\delta^{-1}xy )\dd y\ptf\end{equation*}
Il est donn\'e par le noyau int\'egral
\begin{equation*}K_{Q,\delta}(x,y) =\int_{U_Q(F)\bsl U_Q(\adef)}\omega(x)
\sum_{\eta\in Q(F)} f(x^{-1}u^{-1}\eta^{-1}\delta y)\dd u\ptf\end{equation*}

Soit $S\in\ESP^{\nP}_\st$, et soit $\sigma$ une repr\'esentation automorphe de $M_S(\adef)$. 
Pour $\mu\in\ag_{\nP,\CM}^*$ 
et $f\in C^\infty_\mathrm{c}(\tG(\adef))$, on a d\'efini en \ref{donn\'ees discr\`etes} un op\'erateur
\begin{equation*}\Rho_{S,\sigma,\mu}(\delta, f,\omega): \Automd(\bsX_S,\sigma)\rightarrow
 \Automd(\bsX_{\theta(S)},\theta(\omega\otimes\sigma))\ptf\end{equation*}
Pour $\Psi\in \Automd(\bsX_S,\sigma)$ et $x\in\bsX_{Q}$, on a
\begin{equation*}\Rho(\delta,f,\omega)E^{\nP}(x,\Psi,\mu) = E^Q(x,\Rho_{S,\sigma,\mu}(\delta,f,\omega)\Psi,\theta(\mu))\vg\end{equation*}
d'o
\begin{equation*}\int_{\bsX_{\nP}}K_{Q,\delta}(x,y)E^{\nP}(y,\Psi,\mu)\dd y= 
E^Q(x,\Rho_{S,\sigma,\mu}(\delta,f,\omega)\Psi,\theta(\mu))\ptf\end{equation*}
Rappelons que l'on a fix\'e une base orthonormale $\Base_S(\sigma)$ de l'espace pr\'e-hilbertien 
$\Automd(\bsX_S,\sigma)$. Pour $\mu\in\bsmu_S$, 
on pose
\begin{equation*}K_{Q,\nP\mskip -2mu ,\sigma}(x,y;\mu)=\sum_{\Psi\in\Base_S(\sigma)}
E^Q(x,\Rho_{S,\sigma,\mu}(\delta,f,\omega)\Psi,\theta(\mu))\overline{E^{\nP}(y,\Psi,\mu)}\ptf\end{equation*}
On a les variantes de la proposition \cite[7.3.1]{LW} et de son corollaire 
 \cite[7.3.2]{LW}: 

\begin{proposition}\label{KQd}
La fonction $f\in C^\infty_\mathrm{c}(\tG(\adef))$ \'etant fix\'ee, alors
\begin{enumerate}[(i)]
\item Le noyau $K_{Q,\delta}(x,y)$ admet la d\'ecomposition spectrale suivante:
\begin{equation*}K_{Q,\delta}(x,y) =\mskip -5mu \sum_{M \in \ESL^{\nP}\mskip -2mu /\bfW^{\nP}} \frPpM
\sum_{\sigma\in\bs\Pi_\disc(M)} \cMsig
\int_{\bsmu_M}K_{Q,\nP,\sigma}(x,y;\mu)\dd \mu\ptf\end{equation*}
\item
La restriction \`a $\Siegel\times G(\adef)$ de la fonction
\begin{equation*}(x,y)\mapsto\mskip -15mu\sum_{M \in \ESL^{\nP}\mskip -2mu /\bfW^{\nP}} \frPpM
\sum_{\sigma\in\bs\Pi_\disc(M)} \mskip -2mu \cMsig
\int_{\bsmu_M}
\vert\bs\Lambda^{T,Q}_1 K_{Q,\nP,\sigma}(x,y;\mu)\vert \dd \mu\end{equation*}
est born\'ee et \`a support compact en $x$ et \`a croissance lente en $y$.
\end{enumerate}
\end{proposition}
\begin{proof}
Le point (i) est une cons\'equence de \ref{decspec} et \ref{factorisation du noyau}: 
on choisit un sous-groupe ouvert compact $\bsK'$ de $G(\adef)$ tel que $e_{\bsK'}*f*e_{\bsK'}=f$ 
et $\omega\vert_{\bsK'}=1$; pour $S\in\smash{\bESP}_\st^{\nP}$, $\sigma\in\Pi_\disc(M_S)$ 
et $\mu\in\bsmu_{S}$, on consid\`ere les op\'erateurs
\begin{equation*}A_{1,\sigma,\mu}=\Rho_{S,\sigma,\mu}(\delta,f,\omega)\quad\hbox{et}
\quad A_{2,\sigma,\mu}=\Rho_{S,\sigma,\mu}(e_{\bsK'})\end{equation*}
puis on pose $B_{\sigma,\mu}=A_{1,\sigma,\mu} A^*_{2,\sigma,\mu}$. 
On en d\'eduit que le noyau tronqu\'e $\bs\Lambda^{T,Q}_1K_{Q,\delta}(x,y)$ est \'egal \`a
\begin{equation*}\sum_{M \in \ESL^{\nP}/\bfW}\frac{1}{w^{\nP}(M)}
\sum_{\sigma\in\Pi_\disc(M)} \cMsig\int_{\bsmu_M}
\bs\Lambda^{T,Q}_1K_{Q,\nP,\sigma}(x,y;\mu)\dd \mu\ptf\end{equation*}
On observe que, gr‰ce \`a la factorisation \ref{factorisation du noyau}, on a
\begin{equation*}\bs\Lambda^{T,Q}_1K_{Q,\nP,\sigma}(x,y;\mu) =
\int_{\bsX_G}\bs\Lambda^{T,Q}_1K_{Q,\nP,\sigma}(x,z;\mu)K_{\nP,\nP,\sigma}^*(e_{\bsK'};z,y;\mu)\dd z\end{equation*}
avec
\begin{equation*}K_{\nP,\nP,\sigma}^*(e_{\bsK'};z,y;\mu)=
\sum_{\Psi\in\Base_S(\sigma)}
E^\nP(z,\Psi,\mu)\overline{E^{\nP}(y,\Rho_{S,\sigma,\mu}(e_{\bs{K}'})\Psi,\mu)}\ptf\end{equation*}
On en d\'eduit le point (ii) comme dans la preuve de \cite[7.3.1~(ii)]{LW}, 
gr‰ce \`a l'in\'egalit\'e de Schwarz \ref{decspec}~(2), 
au lemme \ref{prop1nt}~(i) et \`a l'in\'egalit\'e de \ref{noyineg}.
\end{proof}

\begin{corollary}
La restriction \`a $\Siegel\times G(\adef)$ de la fonction
\begin{equation*}(x,y)\mapsto\vert\bs\Lambda_1^{T,Q}K_{Q,\delta}(x,y)\vert\end{equation*}
est born\'ee et \`a support compact en $x$ et \`a croissance lente en $y$. 
\end{corollary}


\part{La formule des traces grossi\`ere}

 \chapter{Formule des traces: \'etat z\'ero}\label{\'etat z\'ero}

 \section{Le cas compact}\label{le cas compact}
Dans cette section nous \'etablissons la formule des traces tordue 
dans le cas o $G_\mathrm{der}$ est anisotrope, c'est-\`a-dire o
\begin{equation*}\ovbsX_{\mskip -2mu \G}=\A_\G(\adef)\G(F)\bsl\G(\adef)\end{equation*}
est compact. On pose
\begin{equation*}\bsY_G\index{YwG@$\bsY_G$}
\bydef\A_\tG(\adef)\G(F)\bsl\G(\adef)\ptf\end{equation*}
Rappelons que l'on a fix\'e un caract\`ere unitaire $\omega$ de $\G(\adef)$ 
trivial sur le groupe $A_\tG(\adef)G(F)$. 
Pour $f\in C_c^\infty(\tG(\adef))$
on considre l'intŽgrale
\begin{equation*}J(f,\omega)\bydef\int_{\bsY_G}K(f,\omega;x,x)\dd x\end{equation*}
avec
\begin{equation*}K(f,\omega;x,y)=\sum_{\delta\in \tG(F)} f(x\mun\delta y)\omega(y)\ptf\end{equation*}
Il est facile de montrer que l'int\'egrale sur $\bsY_G$ est absolument convergente. Indiquons rapidement
comment on en d\'eduit la formule des traces. Pour plus de d\'etails on renvoie
aux chapitres suivants o les r\'esultats de ce paragraphe seront \'etablis dans un cadre plus g\'en\'eral.

On peut d\'evelopper l'int\'egrale suivant les classes de conjugaison. 
On note $\wt{\Gamma}$ un syst\`eme de repr\'esentants des classes de $G(F)$-conjugaison 
dans $\tG(F)$ et $\G^\delta(F)$ le groupe des points $F$-rationnels du centralisateur
$G^\delta$ de $\delta$ dans $G$.
Pour $\delta\in\tG(F)$, on choisit une mesure de Haar sur $G^\delta(\adef)$ et on pose
\begin{equation*}a^G(\delta)= \vol (A_\tG(\adef)G^\delta(F)\bsl G^\delta(\adef))\ptf\end{equation*}
 Si $G^\delta(\adef)\not\subset\ker(\omega)$ 
on pose
\begin{equation*}\ESO_\delta(f,\omega)=0\end{equation*}
et, si $G^\delta(\adef)\subset\ker(\omega)$, on pose
\begin{equation*}\ESO_\delta(f,\omega)= \int_{G^\delta(\adef)\bsl G(\adef)} \omega(g) f(g^{-1}\delta g)\dd \dot{g}\end{equation*}
o $\dd \dot{g}$ est la mesure quotient. 
\begin{proposition}\label{ftcompgeom}
Si $G_\mathrm{der}$ est anisotrope, on a le d\'eveloppement g\'eom\'etrique:
\begin{equation*}J(f,\omega)=\sum_{\delta\in\wt{\Gamma}}a^G(\delta)\ESO_\delta({\tdf},\omega)\ptf\end{equation*}
Seul un nombre fini de $\delta$ (d\'ependant du support de $f$) donne 
une contribution non nulle \`a la somme.
\end{proposition}

Nous allons maintenant consid\'erer le d\'eveloppement spectral.
En g\'en\'eral $J(f,\omega)$ n'est pas une trace car, sauf si $\A_G$ est trivial, l'op\'erateur $\wt\rho(f,\omega)$ 
op\'erant dans $ L^2(\bsX_G)$ n'est pas un op\'erateur \`a trace.

Rappelons qu'on a not\'e $\Xi(G)$\index{XwiG@$\Xi(G)$}
le groupe des caract\`eres unitaires automorphes de $A_G(\adef)$. On note
\begin{equation*}\Xi(G,\tG)\subset \Xi(G)\end{equation*}
le sous-groupe des caract\`eres triviaux sur $A_\tG(\adef)$. 
Les groupes $\Xi(G)$ et $\Xi(G,\tG)$ sont munis 
de mesures de Haar en suivant les conventions de \ref{convention-mesures}: elles donnent le volume $1$ \`a 
$\wh{\ESB}_G$ et $\wh{\ESB}_G^\tG$ respectivement\footnote{Rappelons que $\wh{\ESB}_G^\tG$ est 
le dual de Pontryagin du r\'eseau $\ESB_G^\tG = \ESB_\tG \backslash \ESB_G$ de $\ag_G^\tG$.}. Soit
\begin{equation*}\Xi(G,\theta,\omega)\subset \Xi(G)\end{equation*}
le sous-ensemble form\'e des caract\`eres $\xi$ tels que,
en notant $\omega_{A_G}$ la restriction de $\omega$ \`a $A_G(\adef)$, on ait
\begin{equation*}\xi\circ\theta=\omega_{A_G}\otimes\xi\ptf\end{equation*}
Si $\Xi(G,\theta,\omega)$ est non vide, c'est un espace tordu sous le groupe 
$\Xi(\G)^\theta$ des points fixes sous $\theta$ dans $\Xi(\G)$. 
On observe que $\wh{\ESB}_G^\theta$ est un sous-groupe ouvert de $ \Xi(G)^\theta$.
On munit $\Xi(G)^\theta$ de la mesure de Haar telle que $\mathrm{vol}(\wh{\ESB}_G^\theta)=1$ ce qui fournit 
une mesure $\Xi(G)^\theta$-invariante sur $\Xi(G,\theta,\omega)$.


Consid\'erons un caract\`ere $\xi\in\Xi(\G)$ et posons pour $x$ et $y$ dans $\G(\adef)$:
\begin{equation*}K_\xi(f,\omega;x,y)=
\sum_{\delta\in\A_G(F)\bsl\tG(F)} \int_{\A_\G(\adef)}\overline{\xi(z)} 
f(z\mun x\mun\delta y)\omega(y)\dd z\end{equation*}
soit encore
\begin{equation*}K_\xi(f,\omega;x,y)=\int_{\A_G(F) \bsl\A_\G(\adef)}\overline{\xi(z)}K(f,\omega;zx,y) \dd z\ptf\end{equation*}
Par inversion de Fourier on voit que 
\begin{equation*}K(f,\omega;x,y)=\int_{\xi\in\Xi(\G)}K_\xi(f,\omega;x,y)\dd\xi\vg\end{equation*}
et on observe que
\begin{equation*}K_\xi(f,\omega;zx,zy)=\zeta_\xi(z) K_\xi(f,\omega;x,y)\leqno{(1)}\end{equation*}
o
\begin{equation*}\zeta_\xi=(\xi\circ\theta)\mun\cdot (\omega_{A_G}\otimes \xi)=\omega_{A_G}\otimes\xi^{1-\theta}\end{equation*}
est un \'el\'ement du groupe $\Xi(G,\tG)$.
On observe aussi que, par d\'efinition,
\begin{equation*}\zeta_\xi=1\quad \hbox{\'equivaut \`a}\quad \xi\in\Xi(G,\theta,\omega)\ptf\end{equation*}

Pour $\xi\in \Xi(G)$, on note $L^2(\bsX_G)_\xi$ l'espace de Hilbert des fonctions
sur $\bsX_G$ qui se transforment suivant $\xi$ sur $A_\G(\adef)$. 
Lorsque $\xi\in\Xi(G,\theta,\omega)$ c'est-\`a-dire si $\zeta_\xi=1$, 
l'op\'erateur $\tbsrho(f,\omega)$ induit un endomorphisme de $L^2(\bsX_G)_\xi$.
D'apr\`es \ref{thmfinitude} c'est un op\'erateur de rang fini. 
On pose
\begin{equation*}J(f,\omega,\xi)\bydef \int_{\ovbsX_{\mskip -2mu \G}}K_\xi(f,\omega;x,x)\dd x\end{equation*}
et on a
\begin{equation*}J(f,\omega,\xi)=\trace\Big( \tbsrho(f,\omega)\vert L^2(\bsX_G)_\xi\Big)\ptf\leqno{(2)}\end{equation*}
On note
\begin{equation*}\Pi_\disc(\tG,\omega)\end{equation*}
l'ensemble des classes d'isomorphisme de repr\'esentations automorphes irr\'eductibles de $G(\adef)$ 
discr\`etes modulo le centre, qui admettent un prolongement \`a 
$(\tG(\adef),\omega)$. Pour $\xi\in \Xi(G,\theta,\omega)$, on note
$\Pi_\disc(\tG,\omega)_\xi$ le sous-ensemble de $\Pi_\disc(\tG,\omega)$
form\'e des repr\'esentations dont le caract\`ere central restreint \`a $A_G(\adef)$ est \'egal \`a $\xi$. Enfin pour 
$\pi \in \Pi_\disc(\Pi,\omega)_\xi$, on note
\begin{equation*}\Automd(\bsX_G,\pi)\end{equation*}
la composante isotypique de $\pi$ dans
\begin{equation*}\Automd(\bsX_G)_{\xi}\subset L^2(\bsX_G)_{\xi}\ptf\end{equation*}

\begin{lemma} Pour $\xi\in \Xi(G,\theta,\omega)$, on a
\begin{equation*}J(f,\omega,\xi)= \sum_{\pi \in \Pi_\disc(\tG,\omega)_{\xi}} 
\trace\Big( \tbsrho(f,\omega)\vert \Automd(\bsX_G,\pi) \Big)\ptf\end{equation*}
\end{lemma} 

\begin{proof}
On observe que les repr\'esentations $\pi$ qui n'admettent pas de prolongement \`a $(\tG(\adef),\omega)$
contribuent par z\'ero \`a la trace de l'op\'erateur $\tbsrho(f,\omega)$. 
\end{proof}

On choisit, pour chaque $\pi\in \Pi_\disc(\tG,\omega)$, un prolongement $\wt{\pi}$ \`a $(\tG(\adef),\omega)$
de $\pi$ (plus correctement, 
d'un repr\'esentant $(\pi,V_\pi)$ de la classe $\pi$) et on note $m(\pi,\wt{\pi})$ 
la multiplicit\'e tordue de $(\pi,\wt{\pi})$ dans $L^2(\bsX_G)_{\xi_\pi}$,
d\'efinie dans \cite[2.4]{LW}, o $\xi_\pi$ est la restriction \`a $A_G(\adef)$ du caract\`ere central de $\pi$. 
Le nombre
\begin{equation*}m(\pi,\wt{\pi})\trace(\wt{\pi}(f,\omega)\vert V_\pi)\end{equation*}
ne d\'epend pas du choix de $\wt{\pi}$. Ceci fournit une nouvelle expression:
\begin{equation*}J(f,\omega,\xi)=\sum_{\pi \in \Pi_\disc(\tG,\omega)_\xi}m(\pi,\wt{\pi})
 \trace(\wt{\pi}(f,\omega)\vert V_\pi)\ptf\end{equation*}

\begin{lemma}\label{chgtvar}
On suppose que l'ensemble $\Xi(G,\theta,\omega)$ est non vide. 
Si $\{\psi\}$ est une famille de fonctions sur 
$\Xi(G,\tG)$ qui tend, au sens des distributions, 
 vers la masse de Dirac \`a l'origine, alors pour tout fonction $\kappa$ lisse sur $\Xi(G)$, on a
\begin{equation*}\lim_{\psi}\int_{{\xi\in\Xi(G)}}\psi({\omega_{A_G}}\otimes \xi^{1-\theta})\kappa(\xi)\dd\xi = 
\int_{ \xi\in{\Xi(G,\theta,\omega)}}\kappa(\xi)\dd\xi\ptf\end{equation*}
\end{lemma}

\begin{proof}
Par hypoth\`ese, il existe $\xi_0\in\Xi(G,\theta,\omega)$.
En \'ecrivant $\xi$ sous la forme 
\hbox{$\xi=\xi_0\xi_1\xi_2$} avec $\xi_2\in\Xi(G)^\theta$
on a \begin{equation*}\omega_{A_G}\otimes \xi^{1-\theta}=\xi_1^{1-\theta}\ptf\end{equation*}
On observe alors qu'en posant $\Xi(G)_1=\Xi(G)/\Xi(G)^\theta$ on a
\begin{equation*}\int_{{\xi\in\Xi(G)}}\mskip -5mu\psi(\omega_{A_G}\otimes \xi^{1-\theta})k(\xi)\dd \xi 
=\int_{\xi_1\in\Xi(G)_1}\mskip -5mu\psi(\xi_1^{1-\theta})
\bigg(\int_{ \xi_2\in\Xi(G)^\theta}k(\xi_0\xi_1\xi_2)\dd \xi_2 \bigg)\dd \xi_1\ptf\end{equation*}
On peut supposer que $\psi$ est \`a support dans le tore compact
\begin{equation*}\wh{\ESB}_G^\tG= (1-\theta)\wh{\ESB}_G \;(\subset \Xi(G,\tG))\ptf\end{equation*}
Puisque les mesures sont compatibles \`a la suite exacte courte
\begin{equation*}0 \rightarrow \wh{\ESB}_G^\theta \rightarrow \wh{\ESB}_G \xrightarrow{{1-\theta}} 
 \wh{\ESB}_G^\tG \rightarrow 0\vg\end{equation*}
le lemme en r\'esulte.
\end{proof}

\begin{proposition} \label{cascomp1}
Si $G_\mathrm{der}$ est anisotrope on a l'identit\'e: 
\begin{equation*}J(f,\omega)=\int_{\xi\in\Xi(G,\theta,\omega)}
\trace\Big( \tbsrho(f,\omega)\vert L^2(\bsX_G)_\xi\Big)\dd\xi\end{equation*}
soit encore
\begin{equation*}J(f,\omega)=
\int_{\xi\in\Xi(G,\theta,\omega)}\sum_{\pi \in \Pi_\disc(\tG,\omega)_{\xi}} 
\trace\Big( \tbsrho(f,\omega)\vert \Automd(\bsX_G,\pi) \Big)
\ptf\end{equation*}
\end{proposition}

\begin{proof}
Par d\'efinition
\begin{equation*}J(f,\omega)=\int_{\bsY_G}\bigg( \int_{\xi\in\Xi(\G)} K_\xi(f,\omega;x,x)\dd\xi \bigg)\dd x\end{equation*}
soit encore
\begin{equation*}J(f,\omega)=\int_{\dot x\in\ovbsX_{\mskip -2mu \G}}
\int_{z\in\A_G(F)\A_\tG(\adef) \bsl\A_\G(\adef)}\bigg(\int_{\xi\in\Xi(\G)} K_\xi(f,\omega;zx,zx)\dd\xi
\bigg)\dd z\dd\dot x \ptf\end{equation*}
Consid\'erons une famille $\{\phi\}$ de fonctions \`a support compact sur le groupe ab\'elien localement compact
\begin{equation*}\A_G(F)\A_\tG(\adef) \bsl\A_\G(\adef)\end{equation*}
et tendant vers la fonction $1$, 
de sorte que la famille $\{\wh{\phi}\}$ de leurs transform\'ees de Fourier tende vers la masse de Dirac sur 
$\Xi(G,\tG)$ \`a l'origine.
 Alors $J(f,\omega)$ est \'egal \`a
\begin{equation*}\lim_{\phi\to 1}\int_{\ovbsX_{\mskip -2mu \G}}
\int_{\A_G(F)\A_\tG(\adef) \bsl\A_\G(\adef)}\bigg(\int_{\xi\in\Xi(\G)} K_\xi(f,\omega;zx,zx)\dd\xi
\bigg) \phi(z^{-1})\dd z\dd\dot x\end{equation*}
soit encore, en utilisant (1), \`a
\begin{equation*}\lim_{\phi\to 1}\int_{\ovbsX_{\mskip -2mu \G}}
\int_{\A_G(F)\A_\tG(\adef) \bsl\A_\G(\adef)}\bigg(\int_{\xi\in\Xi(\G)} \phi(z^{-1})\zeta_\xi(z) K_\xi(f,\omega;x,x)\dd\xi
\bigg)\dd z\dd\dot x \ptf\end{equation*}
Comme $\phi$ est \`a support compact, on peut intervertir les int\'egrations en $z$ et $\xi$, et 
on a
\begin{equation*}J(f,\omega)= \lim_{\phi\to 1}\int_{\ovbsX_{\mskip -2mu \G}}\bigg(
\int_{\xi\in\Xi(\G)} \wh\phi(\zeta_\xi)K_\xi(f,\omega;x,x)\dd\xi\bigg)\dd\dot x\ptf\end{equation*}
 Compte tenu de \ref{chgtvar}, cette limite s'\'ecrit
\begin{equation*}\int_{\ovbsX_{\mskip -2mu \G}}\bigg(\int_{\xi\in\Xi(G,\theta,\omega)}K_{\xi}(f,\omega;x,x)\dd\xi\bigg)\dd\dot x\ptf\end{equation*}
Comme $\ovbsX_{\mskip -2mu  G}$ est compact on peut encore intervertir et on obtient
\begin{equation*}J(f,\omega)=\int_{\xi\in\Xi(G,\theta,\omega)}J(f,\omega,\xi)\dd\xi\ptf\end{equation*}
On conclut en invoquant (2).
\end{proof}

On pose \begin{equation*}\bsmu_\tG \bydef \wh{\ESA}_\tG\end{equation*}
et on note $\bsPi_\disc(\tG,\omega)$
le quotient de $\Pi_\disc(\tG,\omega)$ par la relation d'\'equivalence donn\'ee par 
la torsion par les \'el\'ements de $\bsmu_\tG$. Pour \hbox{$\pi\in \Pi_\disc(\tG,\omega)$,} 
on pose
\begin{equation*}\wh{c}_\tG(\pi)= \frac{\vert \wh{\bsbbc}_\tG \vert}{\vert \mathrm{Stab}_\tG(\pi)\vert} \end{equation*}
o $\mathrm{Stab}_\tG(\pi)\subset \wh{\bsbbc}_\tG$ est le stabilisateur de $\pi$ dans $\bsmu_\tG$. 

\begin{lemma}\label{jacobien}
Le morphisme
\begin{equation*}\wh{\ESB}_G^\theta \rightarrow \wh{\ESB}_\tG = \wh{\ESB_G^\theta} \end{equation*}
induit par $ \xi \mapsto \xi\vert_{\ESB_\tG}$
est surjectif et son noyau est fini de cardinal
\begin{equation*}j(\tG)= \vert \det(1-\theta\vert \ag_G^\tG)\vert\ptf\end{equation*}
\end{lemma}
\begin{proof}
Le groupe $\wh{\ESB}_G^\theta$ est le dual de Pontryagin du groupe $(1-\theta)\ESB_G\backslash \ESB_G$. 
Le morphisme compos\'e
\begin{equation*}\ESB_\tG = \ESB_G^\theta \rightarrow \ESB_G \rightarrow (1-\theta)\ESB_G\backslash \ESB_G\end{equation*}
est injectif et son conoyau est \'egal \`a
\begin{equation*}\big((1-\theta)\ESB_G+ \ESB_\tG\big) \backslash \ESB_G = (1-\theta)\ESB_G \backslash \ESB_G^\tG\ptf\end{equation*}
Or l'indice $[\ESB_G^\tG : (1-\theta)\ESB_G]$ est \'egal \`a $\vert \det (1-\theta\vert \ag_G^\tG)\vert$. 
D'o le lemme par dualit\'e de Pontryagin.
\end{proof}
Avec les conventions de \ref{convention-mesures} pour la normalisation des mesures ($\mathrm{vol}(\bsmu_\tG)=1$), 
la proposition \ref{cascomp1} s'\'ecrit aussi:
%
\begin{proposition}\label{cascomp2}
Si $G_\mathrm{der}$ est anisotrope on a l'identit\'e suivante: 
\begin{equation*}J(f,\omega)= 
j(\tG)^{-1}\sum_{\bs\pi\in \bs\Pi_\disc(\tG,\omega)} \wh{c}_\tG(\pi) \int_{\bsmu_\tG} 
\trace\big( \tbsrho(f,\omega)\vert \Automd(\bsX_G,\pi_\lambda) \big)\dd\lambda\ptf\end{equation*}
o, pour chaque classe $\bs\pi$, on a choisi un repr\'esentant $\pi$ dans $\Pi_\disc(\tG,\omega)$. 
Seul un nombre fini de $\bs\pi$ (d\'ependant de $f$) donne une contribution non triviale \`a la somme.
\end{proposition}
\begin{proof}
On peut \'ecrire
\begin{equation*}\int_{\Xi(G,\theta,\omega)}^\oplus L^2(\bsX_G)_\xi \dd \xi =
\bigoplus_{\xi\in\Xi(G,\theta,\omega)^1}\int_{\wh{\ESB}_G^\theta}^\oplus L^2(\bsX_G)_{\xi\star\mu}\dd\mu\end{equation*}
o $\Xi(G,\theta,\omega)^1 \subset \Xi(G)^1$ est l'ensemble des restrictions \`a $A_G(\adef)^1$ 
des \'el\'ements de $\Xi(G,\theta,\omega)$. On a donc
\begin{equation*}J(f,\omega)=\sum_{\xi\in \Xi(G,\theta,\omega)^1} \int_{\wh{\ESB}_G^\theta}
\trace\big( \tbsrho(f,\omega)\vert L^2(\bsX_G)_{\xi\star \mu} \big)\dd\mu\end{equation*}
et
\begin{equation*}\trace\big( \tbsrho(f,\omega)\vert L^2(\bsX_G)_{\xi\star \mu} \big)= 
\sum_{\pi \in \Pi_\disc(\tG,\omega)_{\xi\star \mu}} 
\trace\Big(\tbsrho(f,\omega)\vert \Automd(\bsX_G,\pi) \Big)\ptf\end{equation*}
En remarquant que, pour tout $\nu\in\wh\bsbbc_\tG$,
\begin{equation*}\pi\in\Pi_\disc(\tG,\omega)_{\xi\star\mu}\qquad\hbox{\'equivaut \`a}\qquad
\pi\star\nu\in\Pi_\disc(\tG,\omega)_{\xi\star\mu}\end{equation*}
 puis en passant aux classes d'\'equivalence modulo torsion par les \'el\'ements de $\bsmu_\tG$, 
 on obtient l'expression de l'\'enonc\'e gr‰ce au lemme \ref{jacobien}. 
\end{proof}
%
\begin{corollary}
Si $G_\mathrm{der}$ est anisotrope, la formule des traces tordue est l'identit\'e suivante:
\begin{equation*}\sum_{\delta\in\wt{\Gamma}}a^G(\delta)\ESO_\delta({\tdf},\omega)
= j(\tG)^{-1}\hspace{-0.4cm}\sum_{\bs\pi\in \bs\Pi_\disc(\tG,\omega)}\hspace{-0.2cm} \wh{c}_\tG(\pi)\int_{\bsmu_\tG} 
\trace\Big( \tbsrho(f,\omega)\vert \Automd(\bsX_G,\pi_\lambda) \Big)\dd\lambda\ptf\end{equation*}
\end{corollary}

Ce sont les identit\'es \ref{ftcompgeom}, \ref{cascomp1} et \ref{cascomp2} 
que nous devons g\'en\'eraliser lorsque $\G_\mathrm{der}$ n'est plus n\'ecessairement anisotrope.
Il conviendra d'int\'egrer sur $\bsY_G$ 
des avatars tronqu\'es du noyau. La premi\`ere \'etape est fournie
par le paragraphe suivant.

 \section{L'identit\'e fondamentale}\label{l'identit\'e fondamentale}
Soit $f\in C^\infty_\mathrm{c}(\tG(\adef))$. Pour $\tP\in\wt\ESP$, on pose
\begin{equation*}K_\tP(x,y)=\int_{U_P(F)\bsl U_P(\adef)}\sum_{\delta\in\tP(F)}\omega(y) {\fun}(x^{-1}\delta u y)\dd u\ptf\end{equation*}
C'est le noyau de la repr\'esentation naturelle de $(\tG(\adef),\omega)$ dans $L^2(\bsX_P)$. Pour 
$Q\in\ESP$ tel que $Q\subset P$, on note $\bs\Lambda^{T,Q}_1K_\tP(x,y)$ 
l'op\'erateur de troncature $\bs\Lambda^{T,Q}$ appliqu\'e 
\`a la fonction $x\mapsto K_\tP(x,y)$ pour $y$ fix\'e. Le lemme \cite[8.2.1]{LW} est vrai ici. 

On pose
\begin{equation*}k^T_\geom(x)\index{kTgeom@$k^T_\geom$}=
\sum_{\tP\in\wt\ESP_\st}(-1)^{a_\tP -a_\tG}\sum_{\xi\in P(F)\bsl G(F)}
k^T_{\tP,\mskip 2mu  \geom}(\xi x)\end{equation*}
avec
\begin{equation*}k^T_{\tP\mathrm{,\mskip 2mu  g\acute{e}om}}(x)=\hat{\tau}_P(\bfH_0(x)-T)K_\tP(x,x)\end{equation*}
et
\begin{equation*}k^T_\spec(x)\index{kTspec@$k^T_\spec$}=
\sum_{\tP\in\wt\ESP_\st}(-1)^{a_\tP -a_\tG}
\sum_{\xi\in P(F)\bsl G(F)}k^T_{\tP\mathrm{,\mskip 2mu  spec}}(\xi x)\end{equation*}
avec
\begin{equation*}k^T_{\tP\mathrm{,\mskip 2mu  spec}}(x)=
\sum_{\substack{Q,R\in\ESP_\st\\ Q\subset P\subset R}}\sum_{\xi\in Q(F)\bsl P(F)}
\wt{\sigma}^R_Q(\bfH_0(\xi x)-T)\bs\Lambda^{T,Q}_1K_\tP(\xi x,\xi x)\ptf\end{equation*}
On a donc
\begin{equation*}k^T_\spec(x)=
\sum_{\tP\in\wt\ESP_\st}(-1)^{a_\tP -a_\tG}\sum_{\substack{Q,R\in\ESP_\st\\ Q\subset P\subset R}}
\sum_{\xi\in Q(F)\bsl G(F)}\wt{\sigma}^R_Q(\bfH_0(\xi x)-T)\bs\Lambda^{T,Q}_1K_\tP(\xi x,\xi x)\ptf\end{equation*}
On a la proposition \cite[8.2.1]{LW}: tous ces termes ne d\'ependent que de la projection de $T$ dans 
\begin{equation*}\ag_0^\tG = \ag_0^G \oplus \ag_G^\tG\end{equation*}
et on a les identit\'es
\begin{equation*}k^T_{\tP\mathrm{,\mskip 2mu  g\acute{e}om}}=
k^T_{\tP\mathrm{,\mskip 2mu  spec}}\qquad\hbox{pour tout}\qquad \tP\in\wt\ESP_\st\ptf\end{equation*}

On en d\'eduit l'identit\'e fondamentale:
\begin{proposition}\label{idfond}
On a l'identit\'e $k^T_\geom= k^T_\spec\ptf $
\end{proposition}
Ce r\'esultat, qui est une cons\'equence imm\'ediate de la combinatoire des c™nes, est le point de d\'epart pour la
formule des traces dans le cas non compact.

Chacune des expressions $ k^T_\geom$ et $k^T_\spec$ poss\`ede un d\'eveloppement: 
la premi\`ere suivant les classes d'\'equivalence de paires primitives et la seconde suivant la d\'ecomposition spectrale.
Pour obtenir la formule des traces on int\`egre sur $\bsY_G$ les fonctions $ k^T_\geom$ et $k^T_\spec$.
On montrera que la convergence des int\'egrales (pour $T$ suffisamment r\'egulier) est compatible aux d\'eveloppements
de chacune de ces expressions. Ainsi, l'\'egalit\'e de
\begin{equation*}\Jres^T_\geom(f,\omega)\index{Jageom@$\Jres^T_\geom$}
\bydef\int_{\bsY_G}k^T_\geom(x)\dd x\qquad
\hbox{et de}\qquad \Jres^T_\spec(f,\omega)\index{Jaspec@$\Jres^T_\spec$}
\bydef\int_{\bsY_G}k^T_\spec(x)\dd x\end{equation*}
fournira la formule des traces, c'est-\`a-dire l'\'egalit\'e du d\'eveloppement g\'eom\'etrique 
et du d\'eveloppement spectral. Pr\'ecisons que l'\'egalit\'e
\begin{equation*}\Jres^T_\geom(f,\omega)= \Jres^T_\spec(f,\omega)\end{equation*}
est une \'egalit\'e de fonctions dans $\mathrm{PolExp}$: 
les int\'egrales convergent et sont \'egales pour $T\in \ag_0$ suffisamment r\'egulier et elles 
d\'efinissent un mme \'el\'ement de $\mathrm{PolExp}$ (d'apr\`es \ref{unicit\'e} ).


 \chapter{D\'eveloppement g\'eom\'etrique}
\label{d\'eveloppement g\'eo}

\section{Convergence: c™t\'e g\'eom\'etrique}\label{CVG}
Rappelons qu'on a introduit en \ref{primitifs} l'ensemble $\OO$ des classes d'\'equivalence 
de paires primitives dans $\tG(F)$ et que pour chaque $\oo\in\OO$
on a d\'efini un ensemble $\ESO_{\oo}$ de classes de conjugaison 
de $\tG(F)$. Compte tenu du Lemme~\ref{OO} on peut d\'ecomposer $k^T_\geom(x)$ en 
\begin{equation*}k^T_\geom(x)=\sum_{\oo\in\OO}k^T_{\oo}(x)\end{equation*}
o $k^T_{\oo}$ ne comporte que 
la contribution des \'el\'ements $\delta\in\ESO_{\oo}$:
\begin{equation*}k^T_{\oo}(x)=\sum_{\tP\in\wt\ESP_\st}(-1)^{a_\tP -a_\tG}
\sum_{\xi\in P(F)\bsl G(F)}k^T_{\tP,\oo}(\xi x)\end{equation*}
avec
\begin{equation*}k^T_{\tP,\oo}(x)=\hat{\tau}_P(\bfH_0(x)-T)K_{\tP,\oo}(x,x)\end{equation*}
o
\begin{equation*}K_{\tP,\oo}(x,x)=\int_{U_P(F)\bsl U_P(\adef)}
\sum_{\delta\in\ESO_{\oo}\cap\tP(F)}\omega(x) {\fun}(x^{-1}\delta u x)\dd u\ptf\end{equation*}
On rappelle que d'apr\`es \ref{OO}~(ii), on a la d\'ecomposition 
\begin{equation*}\ESO_\oo \cap \tP(F)= (\ESO_\oo \cap \tM_P(F)) U_P(F)\vg\end{equation*}
ce qui donne un sens \`a l'expression ci-dessus.

On consid\`ere $Q,\mskip 2mu  R\in\ESP_\st$. Rappelons \cite[2.11.1]{LW} qu'il existe un 
$\tP\in\wt\ESP_\st$ tel que $Q\subset P\subset R$ si et seulement 
on a $Q^+\subset R^-$. On a d\'efini en \ref{la fonction F} un 
ensemble $\Siegel_{P_0}^Q(T_1,T)$ d\'ependant d'un compact $C_Q\subset G(\adef)$, 
et on a not\'e $F_{P_0}^Q(\cdot,T)$ la fonction caract\'eristique 
de l'ensemble \begin{equation*}Q(F)\Siegel_{P_0}^Q(T_1,T)\ptf\end{equation*}
Comme en \cite[3.6.3]{LW}, on suppose que $C_Q$ est assez gros, et que $T$ et $-T_1$ sont assez r\'eguliers. 
 
On pose\footnote{Le lecteur prendra garde que dans \cite{LW} on passe au quotient
 par $\BB_G$, qui dans le cas des corps de nombres est identifi\'e \`a un sous-groupe du centre,
car $f$ a \'et\'e int\'egr\'ee sur le centre.} 
\begin{equation*}\bsY_Q\index{YwQ@$\bsY_Q$}
=\A_\tG(\adef)Q(F)\bsl\G(\adef)\ptf\end{equation*}
Le point clef pour la convergence du c™t\'e g\'eom\'etrique (th\'eor\`eme~\ref{convgeom} ci-dessous) 
est le r\'esultat suivant \cite[9.1.1]{LW}: 

\begin{proposition}\label{ygeom}
Supposons $T$ assez r\'egulier, c'est-\`a-dire $\bsd_0(T)\geq c$ o $c$ est une constante d\'ependant du support de $f$. L'int\'egrale
\begin{equation*}\int_{\bsY_Q} F^Q_{P_0}(x,T)\wt{\sigma}_Q^R(\bfH_0(x)-T)
\Bigg\lvert \sum_{\tP\in\wt\ESP_\st,\tQ^+\subset\tP\subset\tR^-}
(-1)^{a_\tP -a_{\tQ}}K_{\tP,\oo}(x,x)\Bigg\rvert  \dd x\end{equation*}
est convergente.
\end{proposition}

\begin{proof}Notons $\Omega_f$ le support de $f$. C'est un compact de 
$\tG(\adef)$, et pour $x\in G(\adef)$ tel que $K_{\tP,\oo}(x,x)\neq 0$, on a
$x^{-1}\delta ux\in\Omega_{{\fun}}$
pour des \'el\'ements $\delta\in\ESO_\oo\cap\tP(F)$ et $u\in U_P(\adef)$. 
Puisque $\bfH_G(x^{-1}\delta u x)=0$ et $\Omega_{{\fun}}\cap G(\adef)^1$ est compact, on peut appliquer 
 \cite[3.6.7]{LW}: si $T$ est assez r\'egulier, pr\'ecis\'ement si $\bsd_0(T)\geq c$ o $c$ 
 est une constante d\'ependant de $\Omega_{{\fun}}$, les 
$\delta\in\ESO_\oo\cap\tP(F)$ qui donnent une contribution non nulle \`a l'expression 
de l'\'enonc\'e appartiennent \`a $\ESO_{\oo}\cap\tQ^+(F)$. En utilisant la d\'ecomposition (\ref{OO}~(ii))
\begin{equation*}\ESO_\oo \cap \wt{Q}^+(F) = (\ESO_\oo\cap \wt{M}_{Q^+}(F))U_{Q^+}(F)\vg\end{equation*}on peut donc comme dans la preuve de 
 \cite[9.1.1]{LW} remplacer $K_{\tP,\oo}(x,x)$ par une expression $\Phi_{\tP,\oo}(x)$ qui s'\'ecrit 
$\Phi_{\tP,\oo}=\sum_{\eta\in\ESO_\oo\cap\tM_{Q^+}}\Phi_{\tP,\eta,\oo}(x)$ avec
\begin{equation*}\Phi_{\tP,\eta,\oo}(x)=\int_{U_P(F)\bsl U_P(\adef)}\sum_{\nu\in U_{Q^+}(F)} {\fun}(x^{-1}\eta\nu ux)\dd u\ptf\end{equation*}
Posons
\begin{equation*}\Xi^R_Q(x)=\wt{\sigma}_Q^R(\bfH_0(x)-T)\sum_{\eta\in \ESO_\oo\cap\tM_{Q^+}}
\Bigg\lvert \sum_{\tP\in\wt\ESP_\st, Q\subset P\subset R} 
(-1)^{a_\tP - a_\tG}\Phi_{\tP,\eta,\oo}(x)\Bigg\rvert  \ptf\end{equation*}
Il s'agit de montrer que l'int\'egrale
\begin{equation*}\int_{\bsY_Q}F^Q_{P_0}(x,T)~\Xi^R_Q(x)~\dd x\end{equation*}
est convergente. Posons 
\begin{equation*}\bs{Z}_Q= A_\tG(\adef)A_Q(F)\bsl A_Q(\adef)\subset \bs{Y}_Q\ptf\end{equation*}
On 
commence par estimer, pour $v\in U_Q(\adef)$ et $x\in G(\adef)$, l'int\'egrale
\begin{equation*}\Theta^R_Q(v,x)=\int_{\bs{Z}_Q}\Xi_Q^R(vax)\bs{\delta}_Q(a)^{-1}\dd a\vg\end{equation*}
de faon uniforme lorsque $x$ reste dans un compact fix\'e. 
Notons que la somme sur $\eta$ dans l'expression $\Xi_Q^R(vax)$ porte sur un ensemble fini, d\'ependant \`a priori de $x$ et $a$. 
 Comme dans la preuve de \cite[9.1.1]{LW}, on montre que pour $x$ dans un compact fix\'e, 
 l'ensemble des $a\in\bs{Z}_Q$ donnant une contribution non triviale \`a l'expression 
 $\Theta_Q^R(v,x)$ est contenu dans un compact; par cons\'equent la somme sur $\eta$ 
 dans l'expression $\Xi_Q^R(vax)$ porte sur un ensemble fini (ind\'ependant de $a$ 
 et $x$ dans un compact fix\'e). Il reste \`a estimer la somme sur $a$ dans l'expression $\Theta_Q^R(v,x)$.
Notons $\bs{Z}_Q^{R^-}$ l'image de
\begin{equation*}A_Q(\adef)\cap A_{R^-}(\adef)^1= \{a\in A_Q(\adef): \bfH_{R^-}(a)=0\}\end{equation*}
dans $\bs{Z}_Q$. Le morphisme
\begin{equation*}1-\theta_0:\bs{Z} _Q\rightarrow\bs{Z}_0 = \bs{Z}_{P_0}\vgq a\mapsto a\theta_0(a^{-1})\end{equation*}
a pour noyau le sous-groupe $\smash{\wt{\bs{Z}}}_Q^{R^-}$ de $\bs{Z} _{Q^+}^{R^-}$ form\'e 
des \'el\'ements $\theta_0$-invariants. D'apr\`es la preuve \cite[9.1.1]{LW}, il 
suffit de consid\'erer les $a\in\smash{\wt{\bs{Z}}}_Q^{R^-}$. 

Soit $\uu$ l'alg\`ebre de Lie de $U_{Q^+}$. On n'a pas ici d'application exponentielle 
mais on peut fixer un $F$-isomorphisme de variet\'es alg\'ebriques $j:\uu\rightarrow U_{Q^+}$ 
compatible \`a l'action de $A_{Q^+}$, i.e. tel que
\begin{equation*}j\circ \mathrm{Ad}_a = \mathrm{Int}_a\circ j\end{equation*}
pour tout $a\in A_{Q^+}$. 
On obtient $j$ par restriction \`a partir d'un $F$-isomorphisme de vari\'et\'es alg\'ebriques 
$j_0: \mathfrak{u}_0\rightarrow U_0$ compatible 
\`a l'action de $A_0=A_{P_0}$, o $\mathfrak{u}_0$ est l'alg\`ebre de Lie de $U_0=U_{P_0}$. 
Pour toute racine $\alpha$ de $A_0$ dans $U_0$, on pose 
$\alpha= \{\alpha,2\alpha\}$ si $2\alpha$ est une racine et $(\alpha)= \{\alpha\}$ sinon. 
On note $U_{(\alpha)}$ le $F$-sous-groupe de $U_0$ correspondant \`a $(\alpha)$ et $\uu_{(\alpha)}$ son alg\`ebre de Lie. 
Soient $\alpha_1,\ldots ,\alpha_r$ les racines primitives (c'est-\`a-dire non divisibles) de $A_0$ dans $U_0$, 
ordonn\'ees arbitrairement. On a la d\'ecomposition en 
produit direct $U_0 = U_{(\alpha_1)} \cdots U_{(\alpha_r)}$, resp. en somme directe 
$\uu_0=\uu_{(\alpha_1)}\oplus\cdots\oplus\uu_{(\alpha_r)}$, 
et il suffit de prouver que pour $i=1,\dots ,r$, il existe un $F$-isomorphisme de 
vari\'et\'es alg\'ebriques $\xi_i:\uu_{(\alpha_i)}\rightarrow U_{(\alpha_i)}$ 
compatible \`a l'action de $A_0$. Alors pour $X\in\uu_0$, on \'ecrit $X=\sum_{i=1}^r X_i$ 
avec $X_i\in\uu_{(\alpha_i)}$ et on pose $j_0(X)= \xi_1(X_1)\cdots \xi_r(X_r)$. Fixons un indice $i$ 
et prouvons l'existence de $\xi_i$. Supposons tout d'abord $(\alpha_i)=\{\alpha_i\}$. 
D'apr\`es \cite[21.17, 21.20]{B2}, $U_{(\alpha_i)}$ est $F$-isomorphe, en tant que vari\'et\'e alg\'ebrique, 
\`a un espace affine $V_i$, la conjugaison par $a\in A_0$ sur $U_{(\alpha_i)}$ correspondant 
\`a la translation par $\alpha_i(a)$ sur $V_i$. 
Si maintenant $(\alpha_i)=\{\alpha_i,2\alpha_i\}$, 
d'apr\`es \cite[21.19]{B2} et la preuve de \cite[21.20]{B2}, il existe un $F$-isomorphisme de vari\'et\'es alg\'ebriques
\begin{equation*}(U_{(\alpha_i)}/U_{(2\alpha_i)})\times U_{(2\alpha_i)}\rightarrow U_{(\alpha_i)}\end{equation*}
compatible \`a l'action de $A_0$ et l'argument pr\'ec\'edent s'applique \`a chacun des deux groupes 
unipotents $U_{(\alpha_i)}/U_{(2\alpha_i)}$ et $U_{(2\alpha_i)}$. Cela prouve l'existence de $\xi_i$ en g\'en\'eral, 
et donc celle de $j_0$. La restriction de $j_0$ \`a $\uu$ donne le 
$F$-isomorphisme $j: \uu\rightarrow U_{Q^+}$ cherch\'e; il est compatible \`a l'action de $A_0$. 
Observons que $j$ induit une bijection $\uu(\adef)\rightarrow U_{Q^+}(\adef)$ 
qui se restreint en une bijection $\uu(F)\rightarrow U_{Q^+}(F)$. 

Soit $\uu^*$ le dual de $\uu$. Fixons un caract\`ere non trivial $\psi$ de $F\backslash \adef$, 
et notons $\uu^\vee$ l'orthogonal de $\uu(F)$ 
dans $\uu^*(\adef)$ pour ce caract\`ere. Pour $\Lambda\in\uu^*(\adef)$ et $u\in U_P(\adef)$, posons
\begin{equation*}g(x,\Lambda,\delta,u)=\int_{\uu(\adef)}\psi(\langle \Lambda,X\rangle) {\fun}(x^{-1}\delta j(X)ux)\dd X\ptf\end{equation*}
Comme dans la preuve de \cite[9.1.1]{LW}, la formule de Poisson permet d'\'ecrire
\begin{equation*}\Xi_Q^R(x)=\wt{\sigma}_Q^R(\bfH_0(x)-T)\sum_{\eta}
\bigg\lvert  \sum_{\Lambda\in\uu^\vee(Q,R)}g(x,\Lambda,\eta,1) \bigg\rvert \end{equation*}
o $\uu^\vee(Q,R)$ est un sous-ensemble de $\uu^\vee$ d\'efini en \textit{loc.~cit}. 
Observons qu'ici $g$ est \`a support compact 
en $\Lambda$ comme transform\'ee de Fourier d'une fonction lisse \`a support compact. 
La suite de la d\'emonstration est identique \`a celle de 
\textit{loc.~cit}.: via l'\'etude de l'action coadjointe de $\smash{\wt{\bs{Z}}}_Q^{R^-}$ sur $\uu^\vee$, 
on obtient que la somme d\'efinissant $\Theta_Q^R(v,x)$ est absolument convergente, 
uniform\'ement lorsque $x$ reste dans un compact.
On conclut comme \`a la fin de la preuve de \textit{loc.~cit}.
\end{proof}

Rappelons que si le compact $C_Q$ est assez gros, et si $T$ et $-T_1$ sont assez r\'eguliers, 
on a la partition \cite[3.6.3]{LW}
\begin{equation*}\sum_{Q\in\ESP_\st,Q\subset P}
\sum_{\xi\in Q(F)\bsl P(F)}F_{P_0}^Q(\xi x,T)\tau_Q^P(\bfH_0(\xi x)-T)=1\ptf\end{equation*}
On en d\'eduit \cite[9.1.2]{LW}:

\begin{theorem}\label{convgeom}
Si $T$ est assez r\'egulier, pr\'ecis\'ement si $\bsd_0(T)\geq c$ o $c$ est une constante 
ne d\'ependant que du support de $f$, l'expression
\begin{equation*}\sum_{\oo\in\OO}\int_{\bsY_G}\vert k_\oo^T(x)\vert \dd x\end{equation*}
est convergente. De plus, seul un ensemble fini de classes $\oo$ (d\'ependant du support de $f$) 
donne une contribution non triviale \`a la somme.
\end{theorem}
L'int\'egrale \'etant absolument convergente, il est loisible de poser
\begin{equation*}\Jres^T_{\oo}\bydef\int_{\bsY_G}k_\oo^T(x)\dd x\qquad\hbox{et}\qquad
\Jres^T_\geom\bydef\int_{\bsY_G}k^T_\geom(x)\dd x\ptf\end{equation*}
On obtient alors le d\'eveloppement g\'eom\'etrique de la formule des traces:
\begin{corollary}\label{devgross}
On a
\begin{equation*}\Jres^T_\geom=\sum_{\oo\in\OO}\Jres^T_{\oo}\ptf\end{equation*}
\end{corollary}

Il ne s'agit ici que de la forme dite {\og grossi\`ere\fg} du d\'eveloppement g\'eom\'etrique.
Nous allons donner une forme plus explicite pour certains termes. 
La th\'eorie des int\'egrales orbitales pour les corps de fonctions est encore \`a \'ecrire. 
Elle sera bien sžr n\'ecessaire pour un d\'eveloppement g\'eom\'etrique {\og fin\fg} au sens de Langlands. 
Il est toutefois possible de traiter les termes primitifs (cf. \ref{partprim}) et les termes
quasi semi-simples comme pour les corps de nombres (cf. \ref{termes qss}). 
Pour les autres termes, on tombe sur des difficult\'es que nous n'essaierons pas de r\'esoudre ici
(cf. \ref{descente}). 
Il est raisonnable d'esp\'erer que pour $p\gg 1$ ces difficult\'es disparaissent (cf. \ref{p grand}).

 \section{Contribution des classes primitives}\label{partprim}
Notons $\OO_\prim\subset\OO$ l'ensemble des classes de $G(F)$-conju\-gaison 
d'\'el\'ements primitifs dans $\tG(F)$. Pour $\oo\in\OO_\prim$, l'expression 
\begin{equation*}k_\oo^T(x)=\sum_{\delta\in\ESO_\oo}\omega(x){\fun}(x^{-1}\delta x)\end{equation*}
ne d\'epend pas de $T$. On la note aussi $k_\oo(x)$. Avec les notations du
paragraphe \ref{l'identit\'e fondamentale}, on a donc
\begin{equation*}k_\prim(f,\omega;x) =\sum_{\oo\in\OO_\prim}k_\oo(x)\end{equation*}
et l'int\'egrale
\begin{equation*}J_\prim^\tG(f,\omega)=\int_{\bsY_G}k_\prim(f,\omega;x)\dd x\leqno{(1)}\end{equation*}
est absolument convergente. Elle d\'efinit une distribution sur $\tG(\adef)$, donn\'ee par
\begin{equation*}J_\prim^\tG(f,\omega)=\sum_{\delta\in\wt{\Gamma}_\prim}\int_{A_\tG(\adef)G^\delta(F)
\bsl G(\adef)}\omega(g){f}(g^{-1}\delta g)\dd g\leqno{(2)}\end{equation*}
o $\wt{\Gamma}_\prim$ est un syst\`eme de repr\'esentants des classes de $G(F)$-conjugaison 
dans $\tG(F)_\prim$. Ici $\G^\delta(F)$ est le groupe des points $F$-rationnels 
du centralisateur\footnote{Vu comme $F$-sch\'ema en groupes, $G^\delta$ n'est \`a priori ni lisse ni connexe.} 
$G^\delta$ de $\delta$ dans $G$, et $\mathrm{d}g$ est le quotient de la mesure de 
Haar sur $A_\tG(\adef)\bsl G(\adef)$ par la mesure de comptage sur $A_\tG(F)\bsl G^\delta(F)$. 
Seul un nombre fini de $\delta$ (d\'ependant du support de $f$) donne une contribution non triviale \`a la somme.

\begin{corollary}\label{CV des IO}
Pour $\delta\in\tG(F)_\prim$, l'int\'egrale orbitale
\begin{equation*}\int_{A_\tG(\adef)G^\delta(F)\bsl G(\adef)}\omega(g) {f}(g^{-1}\delta g)\dd g\end{equation*}
est absolument convergente.
\end{corollary}

\begin{corollary}
Pour $\delta\in \tG(F)_\prim$, le groupe (localement compact) $G^\delta(\adef)$ est unimodulaire et le quotient
\begin{equation*}A_\tG(\adef)G^\delta(F)\bsl G^\delta (\adef)\end{equation*}
est de volume fini.
\end{corollary}
\begin{proof} Consid\'erons le cas $\omega=1$ et $f$ positive. L'int\'egrale orbitale
\begin{equation*}\int_{A_\tG(\adef)G^\delta(F)\bsl G(\adef)}{f}(x^{-1}\delta x)\dd x\end{equation*}
\'etant convergente il en r\'esulte que pour toute fonction lisse $\varphi$ et \`a support compact sur 
\begin{equation*}Y=G^\delta(\adef)\bsl G(\adef)\end{equation*}l'int\'egrale
\begin{equation*}\int_{A_\tG(\adef)G^\delta(F)\bsl G(\adef)}\varphi(x){f}(x^{-1}\delta x)\dd x\end{equation*}
est convergente et d\'efinit une fonctionnelle $G(\adef)$-invariante \`a droite
sur l'espace vectoriel engendr\'e par les fonctions $\psi$ sur $Y$ de la forme $\psi(\dot x)= \varphi(\dot x){f}(x^{-1}\delta x)$.
Mais, en variant $f$ et $\varphi$, on obtient ainsi toutes les fonctions lisses et \`a support compact sur $Y$.
Il existe donc une mesure $G(\adef)$-invariante \`a droite sur $Y$ ce qui implique que 
le groupe $G^\delta(\adef)$ est unimodulaire, puisque $G(\adef)$ l'est. La convergence
de l'int\'egrale orbitale implique que le volume de $A_\tG(\adef)G^\delta(F)\bsl G^\delta(\adef)$ est fini.
\end{proof}

Pour $\delta\in\tG(\adef)_\prim$, on choisit une mesure de Haar sur $G^\delta(\adef)$ et on pose
\begin{equation*}a^G(\delta)= \vol (A_\tG(\adef)G^\delta(F)\bsl G^\delta(\adef))\end{equation*}
o le volume est calcul\'e en prenant la mesure quotient de la mesure de Haar sur $A_\tG(\adef)\bsl G^\delta(\adef)$ par la 
mesure de comptage sur $A_\tG(F)\bsl G^\delta(F)$. Si $G^\delta(\adef)\not\subset\ker(\omega)$, 
on pose
\begin{equation*}\ESO_\delta(f,\omega)=0\end{equation*}
et si $G^\delta(\adef)\subset\ker(\omega)$, on pose
\begin{equation*}\ESO_\delta(f,\omega)= \int_{G^\delta(\adef)\bsl G(\adef)} \omega(g) f(g^{-1}\delta g)\dd \dot{g}\end{equation*}
o $\dd \dot{g}$ est la mesure quotient de la mesure de Haar sur $G(\adef)$ par la mesure de Haar sur $G^\delta(\adef)$. 

On fixe un syst\`eme de repr\'esentants $\wt{\Gamma}\subset\tG(F)$ des classes de 
$G(F)$-conjugaison dans $\tG(F)$, et on note $\wt{\Gamma}_\prim\subset\wt{\Gamma}$ 
le sous-ensemble form\'e des \'el\'ements primitifs.

\begin{proposition}\label{orb}
L'int\'egrale (1) est absolument convergente et d\'efinit une distribution invariante sur $\tG(\adef)$. On a
\begin{equation*}J^\tG_\prim(f,\omega)=\sum_{\delta\in\wt{\Gamma}_\prim}a^G(\delta)\ESO_\delta({\tdf},\omega),\leqno{(3)}\end{equation*}
o la somme porte sur un ensemble fini (d\'ependant du support de $f$).
\end{proposition}

\begin{proof}Un calcul \'el\'ementaire fournit l'\'egalit\'e (3). La finitude r\'esulte du lemme \ref{finitude}. \end{proof}


 \section{Sur la descente centrale}\label{descente} 
Dans \cite[9.2]{LW}, en vue de l'utilisation de la descente centrale de Harish Chandra qui est une technique
essentielle dans les travaux d'Arthur sur le d\'eveloppement g\'eom\'etrique fin,
l'expression $k_\oo^T(x)$ est remplac\'ee par 
une expression $j_\oo^T(x)$ de mme int\'egrale sur 
$\bsY_G$\footnote{Rappelons qu'ici $\bsY_G$ joue le r™le de l'espace 
${\mathbf X}_G= \mathfrak{A}_GG(F)\backslash G(\adef)$ de \cite{LW}. 
Observons aussi que la d\'efinition de l'expression $j_{\wt{P}\mskip -2mu ,\oo}^T(x)$ donn\'ee dans 
\cite[9.2]{LW} est incorrecte et doit tre remplacŽe par celle donn\'ee par \Err(xi) dans l'Annexe \ref{Erratum}.}. 
En caract\'eristique positive le lemme \cite[9.2.1]{LW}, qui permet de faire ce remplacement, n'est
plus vrai mme dans le cas non tordu.
En effet consid\'erons une paire primitive $(M,\delta)$ dans $G$ et $P= MU$. 
Notons $U^{\delta}$ le centralisateur de $\delta$ dans $U$
\footnote{Notons $(1-\delta)U$ l'image du 
$F$-endomorphisme $u\mapsto u\mskip 2mu .\mskip 2mu  \mathrm{Int}_{\delta}(u)^{-1}$. 
Ce morphisme se factorise en un $F$-morphisme bijectif de vari\'et\'es alg\'ebriques 
$U^{\delta}\bsl U\rightarrow (1-\delta)U$ qui n'est en g\'en\'eral pas un isomorphisme. 
Au $F$-sch\'ema en groupes (affine) $U^{\delta}$ correspond 
un sous-groupe alg\'ebrique {$F$-\textit ferm\'e} (au sens de Borel \cite{B2}) de $G$, 
not\'e de la mme mani\`ere. Le $F$-morphisme en question est un isomorphisme si et seulement s'il est s\'eparable, 
auquel cas le $F$-sch\'ema en groupes $U^{\delta}$ est g\'eom\'etriquement r\'eduit (donc lisse) et correspond \`a 
un groupe alg\'ebrique d\'efini sur $F$.}. 
On consid\`ere le $F$-morphisme $\pi_{\delta}$ de vari\'et\'es alg\'ebriques:
\begin{equation*}\pi_{\delta}: U_P\times U_P^{\delta}\rightarrow U_P\qquad\hbox{d\'efini par}\qquad 
(u,v)\mapsto u\mun\mskip 2mu v\mskip 2mu \mathrm{Int}_{\delta}(u)\ptf\end{equation*}
 En g\'en\'eral l'inclusion $\pi_\delta(U_P\times U^{\delta})\subset U_P$ est stricte
et donc le lemme \cite[9.2.2]{LW} est en d\'efaut, comme le montre l'exemple ci-dessous.
\pni\indent
Supposons $F$ de caract\'eristique $p>0$.
Soit $\gamma$ un \'el\'ement de $\mathrm{GL}_p(F)$ qui engendre une extension radicielle non triviale 
$E=F[\gamma]$ de $F$. Cette extension est de degr\'e $p$ et $\gamma$ est primitif dans 
$\mathrm{GL}_p(F)$. Plongeons $M= \mathrm{GL}_p\times \mathrm{GL}_p$ diagonalement dans $\mathrm{GL}_{2p}$ et notons 
$\delta$ l'\'el\'ement $(\gamma,\gamma)$ de $M(F)$. Alors $(M,\delta)$ est une paire primitive dans $G=\mathrm{GL}_{2p}$
et si $P$ le sous-groupe parabolique standard de $G$ de composante de Levi $M$, 
on a $U^\delta(F) \simeq E$. On identifie $U(F)$ \`a $M_p(F)$ et $U^\delta(F)$ \`a $E\subset M_p(F)$. 
On voit que dans ce cas l'application $\pi_\delta$ est donn\'ee par
\begin{equation*}(x,y) \mapsto n(x)+y\qquad\hbox{o}\qquad
n(x)\bydef (\mathrm{Ad}(\gamma)-1)x \ptf\end{equation*}
On peut choisir $\gamma$ tel que $\gamma^p$ soit scalaire et donc $(\mathrm{Ad}(\gamma)-1)^p=0$. 
Il en r\'esulte que $\pi_\delta$ ne peut pas tre surjective.
Par exemple si $p=2$, on a $\gamma n(x)\gamma\mun=n(x)$ et donc $n(x) \in E$ pour tout $x\in M_2(F)$ 
ce qui implique $n(x) + y\in E$ pour tout couple $(x,y)\in M_2(F)\times E$. 

Cet exemple montre que pour les paires primitives 
$(\tM,\delta)$ dans $\tG$ avec $\tM\neq \tG$ et $\delta$ ins\'eparable, 
la descente centrale ne fonctionne plus sans modification.
C'est l'une des principales difficult\'es \`a r\'esoudre du c™t\'e g\'eom\'etrique.

 \section{Contribution des classes quasi semi-simples}\label{termes qss}

Un \'el\'ement $\delta$ de $\tG$ est dit \textit{quasi semi-simple} si l'automorphisme $\tau=\mathrm{Int}_\delta$ de $G$ 
est quasi semi-simple, c'est-\`a-dire s'il stabilise une paire de Borel $(B,T)$ de $G$. Pour l'\'etude des automorphismes quasi semi-simples 
sur un corps quelconque, on renvoie \`a \cite[ch.~2 et 3]{Le}. 
Un automorphisme $\tau$ de $G$ est quasi semi-simple si et seulement l'automorphisme 
$\tau_\mathrm{der}$ de $G_\mathrm{der}$ est quasi semi-simple. 
La composante neutre $G_\tau = (G^\tau)^0$ du centralisateur d'un 
automorphisme quasi semi-simple $\tau$ de $G$ est un groupe alg\'ebrique lin\'eaire r\'eductif (connexe), 
qui est d\'efini sur $F$ si $\tau$ l'est \cite[4.6.3]{Le}. On prendra garde \`a ce que
si $F$ est de caract\'eristique $p>0$ un automorphisme non trivial de $G$ peut tre quasi semi-simple et 
unipotent; toutefois, un tel automorphisme est forc\'ement quasi-central\footnote{En caract\'eristique $p>0$, 
un automorphisme $\tau$ de $G$ est dit \textit{unipotent} s'il existe un entier 
$k\geq 1$ tel que $\tau^{p^k}= \mathrm{Id}$. Par exemple pour $p=2$, l'automorphisme 
$\tau:t\mapsto t^{-1}$ du tore $\GM_\mathrm{m}$ est quasi semi-simple et unipotent. 
De plus le morphisme $1-\tau: t\mapsto t^{2}$ de $\GM_\mathrm{m}$ n'est pas s\'eparable.
Un automorphisme quasi semi-simple $\tau$ de $G$ est dit \textit{quasi-central} 
si $\dim(G_{\tau'})\leq \dim(G_\tau)$ pour tout automorphisme quasi semi-simple $\tau'$ de $G$ 
de la forme $\tau'=\mathrm{Int}_g\circ \tau$ avec $g\in G$.}.

Pour $\delta\in \tG$ et $\tau= \mathrm{Int}_\delta$, notons $(1- \tau)G$ l'image 
du morphisme de $G$ dans $G$:
\begin{equation*}1-\tau: g\mapsto g\tau(g)^{-1}\ptf\end{equation*}
On dit que $\delta$ est \textit{s\'eparable} si le morphisme $1-\tau$ est s\'eparable, c'est-\`a-dire s'il induit 
un isomorphisme de vari\'et\'es alg\'ebriques 
\begin{equation*}G^\delta \backslash G \rightarrow (1-\tau)G\ptf\end{equation*}
 Si $\delta\in \tG(F)$ est s\'eparable, 
le $F$-sch\'ema en groupes $G^\delta$ est lisse et correspond \`a un sous-groupe alg\'ebrique ferm\'e de $G$ 
d\'efini sur $F$.

Soit $\delta \in \tG(F)$ un \'el\'ement quasi semi-simple. En g\'en\'eral $\delta$ n'est pas s\'eparable, mais 
on sait (d'apr\`es \cite[4.6.3]{Le}) que la composante neutre $G_\delta=(G^\delta)^0$ de son 
centralisateur est un groupe alg\'ebrique lin\'eaire (connexe) d\'efini sur $F$. 
On peut donc comme sur un corps de nombres d\'efinir son \textit{centralisateur stable} $I_\delta$ 
(cf. \cite[2.6]{LW}\footnote{Le centre {\og sch\'ematique \fg} $Z_G$ n'est en g\'en\'eral pas r\'eduit. 
On consid\`ere le centre {\og r\'eduit \fg} $\ZZ_G$ de $G$, c'est-\`a-dire le centre au sens de Borel \cite{B2}. 
C'est un groupe alg\'ebrique diagonalisable, \`a priori seulement $F$-ferm\'e, mais qui est en fait d\'efini sur $F$: 
d'apr\`es \cite[18.2]{B2} il existe un tore maximal $T$ de $G$ d\'efini sur $F$; un tel tore se d\'eploie sur 
une extension alg\'ebrique s\'eparable $E$ de $F$, par suite le sous-groupe ferm\'e $\ZZ_G\subset T$ 
est d\'efini sur $E$, donc sur $F$ puisqu'il est $F$-ferm\'e. 
Le centre {\og r\'eduit \fg} $\ZZ_\tG = \ZZ_G^\theta$ de $\tG$ est lui aussi un groupe 
alg\'ebrique diagonalisable d\'efini sur $F$. On prend pour $I_\delta$ le sous-groupe alg\'ebrique ferm\'e de $G$ 
engendr\'e par $G_\delta$ et $\ZZ_\tG$.}):
\begin{equation*}G_\delta \bydef G^{\delta,0} \subset I_\delta \subset G^\delta\ptf\end{equation*}
Consid\'erons 
le tore d\'eploy\'e maximal $A_\delta$ dans le centre de $I_\delta$, ou ce qui revient au mme de $G_\delta$, 
et notons $M_\delta$ le centralisateur de $A_\delta$ dans $G$. C'est un facteur de Levi de $G$, et $\delta$ 
est elliptique (mais pas n\'ecessairement r\'egulier) dans
$\tM_\delta = \delta M_\delta$ 
en d'autres termes\footnote{Observons qu'un un \'el\'ement quasi semi-simple
$\delta$ est primitif si et seulement s'il est elliptique r\'egulier c'est-\`a-dire que $G_\delta$ est un tore et 
le sous-tore d\'eploy\'e maximal de $G_\delta$ est \'egal \`a $A_\tG$.}:
\begin{equation*}A_\delta=A_{\tM_\delta}\ptf\end{equation*}
Notons $\wt\ESP_\st(\delta)$ le sous-ensemble de $\wt\ESP_\st$ form\'e des 
$\tP$ tel que $\tM_P$ contienne un conjugu\'e de $\tM_\delta$ dans $G(F)$.
Soit $\oo=[\tM,\delta]$ une paire primitive dans $\tG$ et soit
 $\cc$ la classe de $G(F)$-conjugaison de $\delta$ dans $\tG(F)$.
La contribution de $\cc$ \`a $k_\oo^T(x)$ est donn\'ee par
\begin{equation*}k^T_{\cc}(x)=\sum_{\tP\in\wt\ESP_\st(\delta)}(-1)^{a_\tP -a_\tG}
\sum_{\xi\in P(F)\bsl G(F)}k^T_{\tP,\cc}(\xi x)\end{equation*}
avec
\begin{equation*}k^T_{\tP,\cc}(x)=\hat{\tau}_P(\bfH_0(x)-T)K_{\tP,\cc}(x,x)\end{equation*}
o
\begin{equation*}K_{\tP,\cc}(x,x)=\int_{U_P(\adef)}
\sum_{\delta\in\cc\cap\tM_P(F)}\omega(x) {\fun}(x^{-1}\delta u x)\dd u\ptf\end{equation*}
Quitte \`a remplacer $\delta$ par un conjugu\'e dans $G(F)$, 
on peut supposer que $\tM_\delta$ est un facteur de Levi standard. 
Pour $\tP\in \wt\ESP_\st$, on d\'efinit l'ensemble \begin{equation*}\bfW(\ag_{\tM_\delta},\tP)\end{equation*}
comme en \cite[9.3]{LW} et on pose
\begin{equation*}j_{\tP,\cc}(x)= \iota(\delta)^{-1} \sum_{s\in \bfW(\ag_{\tM_\delta},\tP)} 
\sum_{\eta \in I_{ s(\delta)}(F)\backslash P(F)}\omega(x)f(x^{-1}\eta^{-1} s(\delta) \eta x)\end{equation*}
o
\begin{equation*}s(\delta) = w_s \delta w_s^{-1} \quad \hbox{et} 
\quad \iota(\delta)= \vert I_\delta(F) \backslash G^\delta(F)\vert \ptf\end{equation*}
Enfin on pose
\begin{equation*}j^T_\cc(x)= \sum_{\tP\in \wt\ESP_\st} (-1)^{a_\tP - a_\tG}
\sum_{\xi \in P(F)\backslash G(F)} \wh{\tau}_P(\bfH_0(\xi x)-T) j_{\tP,\cc}(\xi x)\ptf\end{equation*}
Observons que la somme sur $\tP$ porte en fait sur le sous-ensemble $\wt\ESP_\st(\delta)$. 

\begin{lemma}\label{termes qssc} On a l'\'egalit\'e des int\'egrales
\begin{equation*}\int_{\bs{Y}_G} k^T_\cc(x)\dd x = \int_{\bs{Y}_G} j^T_\cc(x)\dd x\ptf \end{equation*}
\end{lemma}
%
\begin{proof}
Pour $\tP\in \wt\ESP_\st(\delta)$ tel que l'orbite $\cc$ rencontre $\tP(F)$, pour 
$\tM_1\subset \tM_P$ un conjugu\'e de $\tM_\delta$, et pour $s\in \bfW(\ag_{\tM_\delta},\ag_{\tM_1})$, 
d'apr\`es \cite[3.7.6]{Le} le morphime
\begin{equation*}U_P\rightarrow U_P\vgq u \mapsto u^{-1}\mathrm{Int}_{ s(\delta)}(u)\end{equation*}
est s\'eparable. Puisque le centralisateur $U_P^{s(\delta)}= U_P \cap G^{s(\delta)}$ est trivial, 
ce morphisme est un isomorphisme qui induit une application bijective 
$U_P(\adef) \rightarrow U_P(\adef)$. On en d\'eduit l'\'egalit\'e
\begin{equation*}K_{\tP,\cc}(x,x)= \int_{U_P(F)\backslash U_P(\adef)} j_{\tP,\cc}(ux)\dd u\end{equation*}
et le lemme en r\'esulte.
\end{proof}

Posons
\begin{equation*}\ESA_{I_\delta}\bydef \bfH_{\tM_\delta}(I_\delta(\adef)) \subset \ESA_{\tM_\delta}\quad 
\hbox{et} \quad \ESC_{I_\delta}^\tG \bydef \ESB_\tG \backslash \ESA_{I_\delta} \subset \ESC_{\tM_\delta}^\tG\ptf\end{equation*}
Si le caract\`ere $\omega$ de $G(\adef)$ est trivial sur 
$I_\delta(\adef)^1= \ker (I_\delta(\adef) \rightarrow \ESA_{I_\delta})$
il d\'efinit, par restriction \`a $I_\delta(\adef)$, un caract\`ere de $\ESC_{I_\delta}^\tG$ de la forme 
$H \mapsto e^{\langle \mu_\delta,H \rangle}$ pour un \'el\'ement
\begin{equation*}\mu_\delta \in \ker (\bsmu_{I_\delta}\rightarrow \bsmu_\tG)\subset 
\bsmu_{I_\delta}\bydef \wh{\ESA}_{I_\delta}\ptf\end{equation*}
On a introduit en \ref{l'\'el\'ement T_0} la famille orthogonale 
$\mathfrak{Y}(x,T)$ et on pose
\begin{equation*}\bs{v}_{\wt{M}_\delta}^T(\omega,x)= \omega(x) 
\sum_{H \in \ES{C}_{I_\delta}^\tG}\Gamma_{\wt{M}_\delta}^\tG(H,\YY(x,T)) e^{\langle \mu_\delta,H\rangle}\ptf\end{equation*}

\begin{lemma}
La fonction $T\mapsto \bs{v}_{\wt{M}_\delta}^T(\omega,x)$ est dans $\mathrm{PolExp}$.
\end{lemma}

\begin{proof}
On a la suite exacte courte
\begin{equation*}0 \rightarrow \ESB_{\tM_\delta}^\tG 
\rightarrow \ESC_{I_\delta}^\tG \rightarrow \bsbbc_{I_\delta}\bydef \ESB_{\tM_\delta}\backslash \ESA_{I_\delta} 
\rightarrow 0\ptf\end{equation*}
On note $\ESB_{\wt{M}_\delta}^\tG(Z)\subset \ESC_{I_\delta}^\tG$ 
la fibre au-dessus de $Z\in \bsbbc_{I_\delta}$ et on pose
\begin{equation*}\eta_{\wt{M}_\delta, F}^{\tG,\YY(x,T)}(Z;\mu_\delta) = \sum_{H \in \ES{B}_{\wt{M}_\delta}^\tG(Z)} 
\Gamma_{\wt{M}_\delta}^\tG(H,\YY(x,T)) e^{\langle \bsmu_\delta,H\rangle}\ptf\end{equation*}
On a donc
\begin{equation*}\bs{v}_{\wt{M}_\delta}^T(\omega,x)= \omega(x) 
\sum_{Z\in \bsbbc_{I_\delta}}\eta_{\wt{M}_\delta, F}^{\tG,\YY(x,T)}(Z;\mu_\delta)\ptf\end{equation*}
L'assertion r\'esulte alors de \ref{etapol}.
\end{proof}
\pni
Observons que pour $M\in \ESL$, $Q\in \ESF(M)$ et $m\in M(\adef)$ on a
\begin{equation*}Y_{mx,T,Q} = Y_{x,T,Q} - \bfH_Q(m)\end{equation*}
et donc
\begin{equation*}\Gamma_{\tM_\delta}^\tG (H, \YY(mx,T))= \Gamma_{\tM_\delta}^\tG (H + \bfH_{M_\delta}(m), \YY(x,T))\end{equation*}
pour $m\in M_\delta(\adef)$. 
On en d\'eduit que la fonction 
$x\mapsto \bs{v}_{\wt{M}_\delta}^T(\omega,x)$ est invariante par translation \`a gauche par 
$h\in I_\delta(\adef)$.

\begin{proposition}\label{IOqss}
Si $I_\delta(\adef)^1 \subset \ker(\omega)$ on a l'identit\'e \begin{equation*}\int_{\bs{Y}_G} j^T_\cc(x)\dd x=
 \iota(\delta)\mun{\mathrm{vol}\bigg(I_\delta(F)\backslash I_\delta(\adef)^1\bigg)}
\int_{I_\delta(\adef)\backslash G(\adef)} \bs{v}_{\tM_\delta}^T(\omega, x) f(x^{-1} \delta x) \dd \dot{x}\ptf\end{equation*}
L'int\'egrale sur $\bs{Y}_G$ est nulle sinon.
\end{proposition}

\begin{proof} Posons
\begin{equation*}e_{\tM_\delta}(x,T)=
\sum_{s\in \bfW(\ag_{\tM_\delta})} \sum_{\tQ \in \ESF_s(\tM_\delta)}
(-1)^{a_{\tQ}-a_\tG}\; \wh{\tau}_{\tQ} (s^{-1}(\bfH_0(w_s x)-T))\ptf\end{equation*}
Comme dans la preuve de \cite[9.3.1]{LW} on a
\begin{equation*}\int_{\bs{Y}_G}j^T_\cc(x)\dd x= \iota(\delta)^{-1}
\int_{A_\tG(\adef) I_\delta(F)\backslash G(\adef)}\omega(x)e_{\tM_\delta}(x,T)f(x^{-1}\delta x) \dd x\ptf\end{equation*}
Pour que l'int\'egrale sur $\bs{Y}_G$ soit non nulle, il faut que $I_\delta(\adef)^1$
soit inclus dans $\ker(\omega)$. Si tel est le cas, on observe que pour $h\in M_\delta(\adef)$ on a
\begin{equation*}e_{\tM_\delta}(hx,T)=\Gamma_{\tM_\delta}^\tG({\mathbf H}_{M_\delta}(h),\YY(x,T))\end{equation*}
et donc que
\begin{equation*}\int_{A_\tG (\adef) I_\delta(F) \backslash I_\delta(\adef)} \omega(hx) e_{\tM_\delta}(hx,T)
 \dd h= \mathrm{vol}(I_\delta(F)\backslash I_\delta(\adef)^1)\mskip 2mu  \bs{v}_{\tM_\delta}^T(\omega,x)\ptf\end{equation*}
\end{proof}

 \section{Sur le d\'eveloppement g\'eom\'etrique fin}\label{p grand} 
Le d\'eveloppement g\'eom\'etrique fin consiste en l'expression des termes du d\'eveloppement g\'eom\'etrique \ref{devgross}
au moyen d'int\'egrales orbitales pond\'er\'ees. Les propositions \ref{orb} et \ref{IOqss} fournissent une telle expression
pour les termes primitifs ou quasi semi-simples. 

Les autres termes font intervenir des contributions unipotentes et,
comme on a vu en \ref{descente}, la descente centrale ne peut plus tre utilis\'ee en g\'en\'eral sans modification. 
On ne peut donc pas esp\'erer pouvoir reprendre sans efforts 
les travaux d'Arthur, \`a moins d'imposer \`a $p$ d'tre {\og suffisament grand \fg} par rapport au rang de $G$
de sorte que\footnote{Les hypoth\`eses sont probablement redondantes: 
il s'agit d'une liste de de propri\'et\'es toujours vraies 
pour un corps de nombres mais, en g\'en\'eral, fausses pour un corps de fonctions.}: 
\par --  pour toute paire primitive $(\tM,\delta)$, l'\'el\'ement $\delta$ est quasi semi-simple;
\par --   pour tout $\delta\in \tG(F)$, l'automorphisme 
$\mathrm{Int}_\delta$ de $G$ est s\'eparable;
\par --   pour tout $\delta\in \tG(F)$ on a une d\'ecomposition de Jordan 
\begin{equation*}\delta= \delta_\mathrm{s}\delta_\mathrm{u}=\delta_\mathrm{u}\delta_\mathrm{s}\end{equation*}en partie 
quasi semi-simple $\delta_s$ et partie unipotente $\delta_\mathrm{u}$ d\'efinie sur $F$;
\par --   pour tout $\delta\in \tG(F)$ quasi semi-simple et tout ensemble fini $S$ 
de places de $F$, il n'y a qu'un nombre fini de classes de $G_\delta(F_S)$-conjugaison unipotentes.

Observons que si, comme le fait Arthur, on se limite au cas o 
un (et donc tout) $\delta\in\tG(F)$ 
induit un automorphisme ext\'erieur d'ordre fini de $G$, on peut alors demander que
le centre sch\'ematique $Z_G$ soit r\'eduit et que $\tG$ induise un automorphisme de $Z_G$ d'ordre fini 
premier \`a $p$.

Une hypoth\`ese plus forte que les pr\'ec\'edentes, mais aussi plus facile \`a v\'erifier, 
est la suivante. 
On consid\`ere $(\mathrm{GL}_n,\mathrm{GL}_n)$ comme un espace tordu c'est-\`a-dire que $\mathrm{GL}_n$ agit sur lui mme
par conjugaison. On demande qu'il existe un entier $n<p$ et un $F$-morphisme d'espaces tordus alg\'ebriques
\begin{equation*}\iota : (\tG,G) \rightarrow (\mathrm{GL}_n,\mathrm{GL}_n)\end{equation*}
d'image ferm\'ee et qui soit un isomorphisme sur son image.
Sous ces hypoth\`eses, il doit tre possible de reprendre sans grands changements les travaux d'Arthur 
sur le d\'eveloppement g\'eom\'etrique fin: 
on commence par traiter les contributions unipotentes, c'est-\`a-dire les paires primitives $(\tM,\delta)$ 
avec $\delta$ quasi semi-simple et unipotent; puis on traite le cas g\'en\'eral par descente centrale. 
Toutefois, cela reste \`a faire.


 \chapter{Premi\`ere forme du d\'eveloppement spectral}
\label{d\'eveloppement spectral}

 \section{Convergence: c™t\'e spectral}\label{CVspec}
On commence par r\'ecrire l'expression pour
$k^T_\spec(x)$ d\'efinie en \ref{l'identit\'e fondamentale}. 
Pour $Q,\mskip 2mu R\in\ESP_\st$, on pose
\begin{equation*}k^T_\spec(Q,R,x)
=\wt{\sigma}_Q^R(\bfH_0(x)-T)
\sum_{\substack{\tP\in\wt\ESP_\st\\\tQ^+\subset\tP\subset\tR^-}}
(-1)^{a_\tP -a_\tG}\bs\Lambda^{T,Q}_1K_\tP(x,x)\ptf\end{equation*}
On a donc
\begin{equation*}k^T_\spec(x)=\sum_{\substack{Q,R\in\ESP_\st\\ Q\subset R}}
\sum_{\xi\in Q(F)\bsl G(F)}k^T_\spec(Q,R,\xi x)\ptf\end{equation*}
On pose
\begin{equation*}\wt{\epsilon}(Q,R)\index{epsilontildeqr@$\wt\epsilon(Q,R)$}
=\left\{\begin{array}{ll}(-1)^{a_{\tR} -a_\tG} &\hbox{si $Q^+\subset R^-$}\\0 &\hbox{sinon}
\end{array}\right.\end{equation*}
et on note $\tG(Q,R)$ l'ensemble des $\delta\in\tG(F)$ tels que $\delta\in\tP(F)$ 
pour un seul $\tP\in\wt\ESP_\st$ tel que 
$\tQ^+\subset\tP\subset\tR^-$ (autrement dit l'ensemble des $\delta\in\tR^-(F)$ 
tels que $\delta\notin\tP(F)$ pour tout $\tP\in\wt\ESP_\st$ tel que 
$\tQ^+\subset\tP\subsetneq\tR^-$). On pose
\begin{equation*}K_{Q,R}(x,y)=\int_{U_Q(F)\bsl U_Q(\adef)}
\sum_{\delta\in\tG(Q,R)}\omega(y) {\fun}(x^{-1}u_Q^{-1}\delta y)\dd u_Q\ptf\end{equation*}
D'apr\`es \cite[10.1.1]{LW}, on a
\begin{equation*}k^T_\spec(x)=\sum_{\substack{Q,R\in\ESP_\st\\ Q\subset R}}\wt{\epsilon}(Q,R)
\sum_{\xi\in Q(F)\bsl G(F)}\wt{\sigma}_Q^R(\bfH_0(\xi x)-T)
\bs\Lambda^{T,Q}_1K_{Q,R}(x,\xi x)\ptf \leqno{(1)}\end{equation*}
Pour $\delta\in\tG(F)$, on a d\'efini $K_{Q,\delta}(x,y)$ en \ref{KQd}, et on pose
\begin{equation*}Q_\delta= Q\cap \mathrm{Int}_\delta^{-1}(Q)\in\ESP_\st^Q\ptf\end{equation*}
On a donc
\begin{equation*}K_{Q,R}(x,x)=\sum_{\delta\in\wt\bfW(Q,R)}\sum_{\xi\in Q_\delta(F)\bsl Q(F)}K_{Q,\delta}(x,\xi x)\leqno{(2)}\end{equation*}
o $\wt\bfW(Q,R)$ est un ensemble de repr\'esentants de $\tG(Q,R)$ modulo $Q(F)$ 
\`a droite et \`a gauche, i.e. des doubles classes 
$Q(F)\bsl\tG(Q,R)/Q(F)$. 
Rappelons que l'on a pos\'e
\begin{equation*}\bsY_{Q_\delta}=\A_\tG(\adef)Q_\delta(F)\bsl G(\adef)\ptf\end{equation*}
Rappelons aussi que l'on a fix\'e en \ref{Siegel} un sous-ensemble fini $\EE_Q$ de $M_0(\adef)$ 
tel que $M_Q(\adef)= \BB_Q \EE_Q M_Q(\adef)^1$ o $\BB_Q\subset A_Q(\adef)$ 
est l'image d'une section du morphisme $A_Q(\adef) \rightarrow \ESB_Q$. On pose
\begin{equation*}M_Q(\adef)^*\bydef \EE_QM_Q(\adef)^1= M_Q(\adef)^1 \EE_Q\ptf\end{equation*}
On a donc $M_Q(\adef)=\BB_Q M_Q(\adef)^*$. 
On suppose de plus, ce qui est loisible, que $\BB_Q$ est 
de la forme
\begin{equation*}\BB_Q= \BB_\tG\BB_Q^\tG\end{equation*}
o $\BB_Q^\tG$ est l'image d'une section du morphisme compos\'e
\begin{equation*}A_Q(\adef) \rightarrow A_\tG(\adef)\backslash A_Q(\adef)\rightarrow\ESB_Q^\tG=\ESB_\tG \backslash \ESB_Q\end{equation*}
et $\BB_\tG$ est l'image d'une section du morphisme $A_\tG(\adef) \rightarrow \ESB_\tG$.
Les lemmes \cite[10.1.2 \`a 10.1.4]{LW} sont vrais ici, et on a la variante de \cite[10.1.5]{LW}:

\begin{lemma}\label{variante de [LW,1.10.5]}
Soient $\delta\in Q(F)\bsl\tG(Q,R)$, $u\in U_Q(\adef)$, $a\in\BB _Q^\tG$, 
$k_1,\mskip 2mu  k_2\in\bsK$ et $m_1,\mskip 2mu m_2\in M_Q(\adef)^*$. Supposons que, pour un $\xi\in Q(F)$, on ait
\begin{equation*}K_{Q,\delta}(am_1k_1,\xi u a m_2 k_2)\neq 0\quad\hbox{et}\quad\wt{\sigma}_Q^R(\bfH_0(a))=1\ptf\end{equation*}
Alors on a
\begin{equation*}\lVert \smash{\bfH_0(a)}\rVert \leq c(1+\lVert \bfH_0(m_2)\rVert )\end{equation*}
pour une constante $c>0$ ne d\'ependant que du support de $f$. 
\end{lemma}

\begin{proof}
On reprend, en la modifiant, celle de \cite[10.1.5]{LW}. On commence par modifier $\delta$ et $\xi$ 
comme au d\'ebut de la preuve de \textit{loc.~cit}: on suppose que 
$\delta= w_{s_0}\delta_0$ o $w_ {s_0}\in M_{R^-}(F)$ repr\'esente un \'el\'ement $s_0$ du groupe de Weyl 
$\bfW^{M_{R^-}}$ de $M_{R^-}$ tel que $s_0^{-1}\alpha >0$ pour toute racine 
$\alpha\in\Delta_0^Q$, et $\xi\in U_0(F)$. Pour $i=1,\mskip 2mu 2$, on \'ecrit $m_i= m'_ix_i$ avec 
$m'_i\in M_Q(\adef)^1$ et $x_i\in\EE_Q$. Rappelons que
\begin{equation*}K_{Q,\delta}(x,y)=\int_{U_Q(\adef)}\omega(x)\sum_{\mu\in M_Q(F)}{\fun}(x^{-1}u_Q^{-1}\mu\delta y)\dd u_Q\ptf\end{equation*}
On a suppos\'e
\begin{equation*}K_{Q,\delta}(am_1k_1,\xi u a m_2 k_2)\neq 0\end{equation*}
ce qui n'est possible que s'il existe un $u_Q\in U_Q(\adef)$ et un $\mu\in M_Q(F)$ tels que
\begin{equation*}k_1^{-1}x_1^{-1}m'^{-1}_1 a^{-1}u_Q^{-1}\mu\delta\xi uam'_2x_2k_2\end{equation*}
appartient au support de $f$. On en d\'eduit qu'il existe un compact $\Omega$ de $G(\adef)$, ne d\'ependant que du support de $f$, tel que
\begin{equation*}m'^{-1}_1 a^{-1}u_Q^{-1}\mu\delta\xi uam'_2\in\Omega\ptf\end{equation*}
On d\'ecompose $H=\bfH_0(a)$ suivant la d\'ecomposition
\begin{equation*}\ag_Q^\tG = \ag_Q^{R^-}\oplus \mathfrak{b}_{R^-}^G \oplus \ag_{\wt{R}^-}^\tG \oplus \ag_G^\tG\ptf\end{equation*}
o $\mathfrak{b}_{R^-}^G$ est l'orthogonal de $\ag_{\wt{R}}^\tG$ dans $\ag_{R^-}^G$.
On rappelle que $\theta-1: \ag_G^\tG \rightarrow \ag_G^\tG$ est un automorphisme. Comme dans la preuve de \textit{loc.~cit}, 
il suffit de consid\'erer les $a\in \BB_Q$ tels que $H\in \ag_Q^{R^-}$ et la suite de la d\'emonstration est identique.
\end{proof}

\begin{proposition}\label{[LW,10.1.6]}
Supposons $T$ assez r\'egulier, c'est-\`a-dire $\bsd_0(T)\geq c$ o $c$ 
est une constante d\'ependant du support de $f$. Alors pour tous $Q,\mskip 2mu R\in\ESP_\st$ 
tels que $Q\subset R$ et tout $\delta\in\tG(Q,R)$, l'int\'egrale
\begin{equation*}\int_{\bsY_{Q_\delta}}\wt{\sigma}_Q^R(\bfH_0(x)-T)\lvert \bs\Lambda^{T,Q}_1 K_{Q,\delta}(x,x)\rvert  \dd x\end{equation*}
est convergente.
\end{proposition}

\begin{proof} C'est l'analogue de \cite[10.1.6]{LW}.
Rappelons que \begin{equation*}G(\adef)= U_Q(\adef)M_Q(\adef)\bsK\ptf\end{equation*}Puisque $U_Q(F)\bsl U_Q(\adef)$ est compact, 
il existe un compact $\Omega\subset U_Q(\adef)$ tel que $U_Q(\adef)=U_Q(F)\Omega$. 
On a donc
\begin{equation*}G({\adef}) = Q(F) \Omega \Siegel^{M_Q} \bs{K}\end{equation*}
o
\begin{equation*}\Siegel^{M_Q}= \BB_Q\Siegel^{M_Q,*}= (\BB_\tG\BB_Q^\tG)(\EE_Q\Siegel^{M_Q,1})\end{equation*}
est un domaine de Siegel pour le quotient $M_Q(F)\backslash M_Q(\adef)$. On pose $\Siegel_Q^*= \Siegel^{M_Q,*}$. 
{Alors $ \Omega\mskip 2mu \BB_Q^\tG\mskip 2mu \Siegel_Q^*\bs{K}$ est un domaine de Siegel pour le quotient 
$\BB_\tG Q(F)\backslash G(\adef)$.} 
On est donc ramen\'e \`a estimer, pour $\delta\in\tG(Q,R)$, l'expression
\begin{equation*}\sum_{a\in \BB_Q^\tG} \int_{\Omega \times\Siegel_Q^*\times\bsK}
\bs{\delta}_Q(am)\mun\wt{\sigma}_Q^R(\bfH_0(am)-T)\Xi_{Q,\delta}(uamk)\dd u \dd m \dd k
\leqno{(3)}\end{equation*}
avec
\begin{equation*}\Xi_{Q,\delta}(x)=
\sum_{\xi\in Q_\delta(F)\bsl Q(F)}
\big\lvert\bs\Lambda_1^{T,Q}K_{Q,\delta}(\xi x,\xi x)\big\rvert\ptf\end{equation*}
D'apr\`es \cite[10.1.2]{LW}, on a
\begin{equation*}\Xi_{Q,\delta}(uamk)=
\sum_{\xi\in Q_\delta(F)\bsl Q(F)}\big\lvert\bs\Lambda_1^{T,Q}K_{Q,\delta}(amk,\xi u a m k)\big\rvert\ptf\end{equation*}
On d\'eduit (d'apr\`es la d\'efinition de $K_{Q,\delta}$) que l'expression (3) est \'egale \`a
\begin{align*}
\lefteqn{
\sum_{a\in \BB_Q^\tG} \int_{\Omega \times\Siegel_Q^*\times\bsK} 
\bs{\delta}_Q(am)\mun\wt{\sigma}_Q^R(\bfH_0(am)-T) }\\
&&\quad\quad\times
\sum_{\xi\in Q_\delta(F)\bsl Q(F)}
\big\lvert\bs\Lambda_1^{T,Q}K_{Q,\delta}(mk,\xi a^{1-\delta} (a^{-1}ua) m k)\big\rvert\dd u \dd m \dd k.
\end{align*}
avec $a^{1-\delta}= a\mathrm{Int}_\delta^{-1}(a^{-1})$. Notons que $\wt{\sigma}_Q^R(\bfH_0(am)-T)$ 
ne d\'epend que de la projection $\bfH_0(a)+\bfH_Q(m)+ T_Q$ de $\bfH_0(am) -T$ dans $\ag_Q$. 
Puisque $\bfH_Q(M_Q(\adef)^1)=1$ et $\EE_Q$ est fini, $\bfH_Q(m)$ ne prend qu'un nombre fini de valeurs. 
On en d\'eduit qu'il existe une constante $c_1>0$ (ne d\'ependant que de $\EE_Q$) telle que si $\bsd_0(T)\geq c_1$, 
alors la condition 
$\wt{\sigma}_Q^R(\bfH_0(am)-T)=1$ pour un $m\in\Siegel_Q^*$ entra"ne que 
$\wt{\sigma}_Q^R(\bfH_0(a))=1$. 
D'apr\`es le lemme \ref{prop1nt}~(i), l'op\'erateur de troncature fournit un noyau 
\begin{equation*}(m_1,m_2)\mapsto\bs\Lambda^{T,Q}_1K_{Q,\delta}(m_1k,\xi a^{1-\delta} (a^{-1}ua) m_2k)\leqno{(4)}\end{equation*}
sur $M_Q(\adef)\times M_Q(\adef)$ dont la restriction \`a $\Siegel_Q^*\times\Siegel_Q^*$ 
est lisse et \`a support compact, 
donc born\'ee. Choisissons un sous-groupe ouvert distingu\'e $\bsK'$ de $\bsK$ tel que la fonction 
$f\in C^\infty_\mathrm{c}(\tG(\adef))$ d\'efinissant $K_{Q,\delta}$ soit $\bsK'$-bi-invariante. 
Notons $\bsK'_Q$ le groupe $\bsK'\cap M_Q(\adef)$. Pour tous $a$, $u$, $k$ et $\xi$, la fonction
\begin{equation*}(m_1,m_2)\mapsto K_{Q,\delta}(m_1k,\xi a^{1-\delta} (a^{-1}ua) m_2k)\end{equation*}
sur $M_Q(F)\bsl M_Q(\adef)\times M_Q({F})\bsl M_Q(\adef)$ est 
$({\bs K}'_Q\times\bsK'_Q)$-invariante \`a droite. D'apr\`es 
\ref{prop1nt}~(ii), il existe un compact $\Omega_2$ de $\Siegel_Q^*\times\Siegel_Q^*$ tel que pour 
tous $a$, $u$, $k$ et $\xi$, le support de la restriction \`a $\Siegel_Q^*\times\Siegel_Q^*$ 
du noyau tronqu\'e (4) soit contenu dans $\Omega_2$. 
Par restriction \`a la diagonale, on obtient une fonction en $m=m_1=m_2$ born\'ee sur 
$\Siegel_Q^*$, et \`a support dans un compact $C$ de $\Siegel_Q^*$ 
ind\'ependant de $a$, $u$, $k$ et $\xi$. D'apr\`es le lemme \ref{variante de [LW,1.10.5]} (en supposant $\bsd_0(T)\geq c_1$), 
si $\Xi_{Q,\delta} (umak)\neq 0$, alors 
$\lVert \smash{\bfH_0(a)^\tG}\rVert \leq c_2(1+\lVert \bfH_0(m)\rVert )$ pour une constante $c_2>0$. 
Par cons\'equent la somme sur $a\in\BB_Q^\tG$ dans (3) 
est finie. D'apr\`es \cite[10.1.4]{LW}, la somme sur $\xi$ dans
\begin{equation*}\sum_{\xi\in Q_\delta(F)\bsl Q(F)}\big\vert\bs\Lambda_1^{T,Q}K_{Q,\delta}(mk,\xi a^{1-\delta} (a^{-1}ua) m k)\big\vert\end{equation*}
porte sur un ensemble fini, que l'on peut choisir ind\'ependant de $a$, $u$, $m$ et $k$ (puisque $x=mk$ et 
$y= a^{1-\delta} (a^{-1}ua) m k$ varient dans des compacts, cf. la preuve de \textit{loc.~cit}.). 
Cela ach\`eve la d\'emonstration. \end{proof}

On en d\'eduit l'analogue de \cite[10.1.7]{LW}:

\begin{corollary}\label{convspec}
Si $\bsd_0(T)>c$ o $c$ est une constante d\'ependant du support de $f$, l'int\'egrale
\begin{equation*}\int_{\bsX_G}\vert k^T_\spec(x)\vert \dd x \end{equation*}
est convergente.
\end{corollary}

Ce corollaire est aussi impliqu\'e par l'identit\'e fondamentale \ref{idfond} 
et le th\'eor\`eme \ref{convgeom}. 

 \section{Annulations suppl\'ementaires} 
 
 Nous allons maintenant donner une expression un peu diffŽrente pour
\begin{equation*}\Jres^T_\spec
=\int_{\bsY_G}k_\spec^T(x)\dd x\ptf\end{equation*}
De fait, comme dans \cite[10.2.3]{LW}, on a des annulations suppl\'ementaires
qui sont une premi\`ere \'etape essentielle pour le d\'eveloppement spectral fin:

\begin{proposition}\label{cot\'espec}
Si $T$ est assez r\'egulier (comme dans le lemme \cite[10.2.1]{LW}), on a
\begin{equation*}\Jres^T_\spec=\sum_{\substack{Q,R\in\ESP_\st\\ Q\subset R}}\wt{\eta}(Q,R)
\int_{\bsY_{Q_{\delta_0}}}\wt{\sigma}_Q^R(\bfH_0(x)-T)
\bs\Lambda^{T,Q}_1K_{Q,\delta_0}(x,x)\dd x\leqno{(1)}\end{equation*}
avec
\begin{equation*}\wt{\eta}(Q,R)\index{etatildeqr@$\wt{\eta}(Q,R)$}
=\left\{\begin{array}{ll}
\wt{\epsilon}(Q,R) &\hbox{si $Q^+ = R^-$}\\0 &\hbox{sinon}
\end{array}\right. \end{equation*}
\end{proposition}

\begin{proof}
Soient $Q,\mskip 2mu R\in\ESP_\st$ tels qu'il existe un $\tP\in\tP_\st$ avec $Q\subset P\subset R$. 
On suppose que les \'el\'ements 
de $\wt\bfW(Q,R)$ sont de la forme $\delta = w_s$ o $w_s\in\tM_{R^-}(F)$ est un repr\'esentant de 
$s= s_0\rtimes\theta_0$ avec $s_0\in \bfW^{M_{R^-}}$ de longueur minimale 
dans sa double classe $\bfW^{M_Q}\bsl \bfW^{M_{R^-}}/\bfW^{M_Q}$. 
On a donc $s\alpha >0$ et $s^{-1}\alpha >0$ pour toute racine $\alpha\in\Delta_0^Q$, et 
$M_s = Q_\delta\cap M_Q$ est un sous-groupe parabolique standard de $M_Q$ (cf. \cite[10.2]{LW}). 
On note $S$ l'\'el\'ement de $\ESP_\st$ tel que $S\cap M_Q = M_S$, et on pose 
$U_S^Q = U_S\cap M_Q$. Le lemme \cite[10.2.1]{LW} et la proposition \cite[10.2.2]{LW} sont vrais ici. 
Cela implique que si $T$ est assez r\'egulier (comme dans \cite[10.2.1]{LW}), alors pour $Q,\mskip 2mu R\in\ESP_\st$ 
tels que $Q\subset R$, seul l'\'el\'ement $\delta\in\wt\bfW(Q,R)$ appartenant \`a la double classe $Q(F)\delta_0 Q(F)$ 
donne une contribution non triviale \`a l'int\'egrale de $k^T_\spec$ exprim\'e au moyen 
des \'equations \ref{CVspec}.(1) et \ref{CVspec}.(2). 
\end{proof}

Dans le cas non tordu la formule est beaucoup plus simple, puisque la condition 
$1\in \bfW(Q,R)$ implique $Q=R$, et que $\sigma^Q_Q=0$ sauf si $Q=G$ \cite[2.11.4]{LW}. 
On a donc, dans le cas non tordu, et pour $T$ assez rŽgulier
\begin{equation*}\Jres^T_\spec=
\int_{\overline{\bsX}_{\mskip -2mu G}}\bs\Lambda^{T,G}_1K_G(x,x)\dd x\ptf\end{equation*}


 \chapter{Formule des traces: propri\'et\'es formelles}
\label{propri\'et\'es formelles}


 \section{Le polyn™me asymptotique}
\label{le polyn™me asymptotique}

Rappelons que l'on a la d\'ecomposition
\begin{equation*}k^T_\geom(x)=\sum_{\oo\in\OO}k^T_\oo(x)\end{equation*}
et l'identit\'e fondamentale \ref{idfond}
\begin{equation*}k^T_\geom(x)= k^T_\spec(x)\ptf\end{equation*}
Pour $\bullet =\mathrm{spec}$, $\mathrm{g\acute{e}om}$ ou $\oo$, on \'ecrira parfois
\begin{equation*}k^{\tG,T}_\bullet(f,\omega; x)\quad \hbox{en place de}\quad k^T_\bullet(x)\end{equation*}
s'il est n\'ecessaire de pr\'eciser les donn\'ees.
On a vu en \ref{convgeom} et \ref{convspec} que l'int\'egrale
\begin{equation*}\int_{\bsY_G}k^T_\bullet(x)\dd x\end{equation*}
est absolument convergente. En particulier on a la d\'ecomposition
\begin{equation*}\int_{\bsY_G}k^T_\bullet(x)\dd x= 
\sum_{Z\in \bsbbc_\tG} \int_{\bsY_G(Z)}k^T_\bullet(x)\dd x\end{equation*}
o $\bs{Y}_G(Z)$ est l'image dans $\bs{Y}_G$ de 
l'ensemble $\{g\in G(\adef)\mskip 2mu \vert \mskip 2mu \bfH_\tG(g)=Z'\}$ pour un rel\`evement (quelconque) $Z'$ de $Z$ dans $\ESA_\tG$. 
La fonction $f\in C^\infty_\mathrm{c}(\tG(\adef))$ \'etant fix\'ee on consid\`ere, pour chaque $\tQ\in\wt\ESP_\st$, 
la fonction $f_{\tQ}\in C^\infty_\mathrm{c}(\tM_Q(\adef))$ d\'efinie par
\begin{equation*}f_{\tQ}(m)=\int_{U_Q(\adef)\times\bsK}f(k^{-1}m uk)\dd u\dd k\ptf\end{equation*}
On a la suite exacte courte
\begin{equation*}0 \rightarrow \ESB_\tQ^\tG \rightarrow \ESC_\tQ^\tG \rightarrow \bsbbc_\tQ \rightarrow 0\end{equation*}
et pour $Z\in \bsbbc_\tQ$ et $T,\mskip 2mu X \in \ag_{0}$, on a pos\'e (cf. \ref{etapol}) 
\begin{equation*}\eta_{\tQ,F}^{\tG,T}(Z;X)=\sum_{H\in\ESB_\tQ^\tG(Z)}\Gamma_\tQ^\tG(H-X,T)\end{equation*}
o $\ESB_\tQ^\tG(Z)$
 est la fibre au-dessus de $Z\in\bsbbc_\tQ$. On a la variante de \cite[11.1.1]{LW}:

\begin{theorem}\label{propform}
Pour $\bullet =\mathrm{spec}$, $\mathrm {g\acute{e}om}$ ou $\oo$, Il existe une fonction
\begin{equation*}T\mapsto \Jres^T_\bullet= \Jres^{\tG,T}_\bullet(f,\omega)\end{equation*}
dans $\mathrm{PolExp}$ telle que si $\bsd_0(T)\geq c(f)$ pour une constante $c(f)$,
ne d\'ependant que du support de $f$, on ait
\begin{equation*}\Jres^T_\bullet=\int_{\bsY_G}k^T_\bullet(x)\dd x\ptf\end{equation*}
\end{theorem}

\begin{proof}On reprend celle de \cite[11.1.1]{LW}. 
D'apr\`es \cite[8.2.1]{LW}, pour $\wt{P}\in \wt{\ESP}_\mathrm{st}$ on a
\begin{equation*}k^T_{\tP\mathrm{,\mskip 2mu  spec}}(x) 
=\hat{\tau}_\tP(\bfH_0(x)-T)K_\tP(x,x)=k^T_{\tP\mathrm{,\mskip 2mu  g\acute{e}om}}(x)\ptf
\end{equation*}
On a donc
\begin{equation*}k^T_\bullet(x)=\sum_{\tP\in\wt\ESP_\st}(-1)^{a_\tP -a_\tG}
\sum_{\xi\in P(F)\bsl G(F)}\hat{\tau}_\tP(\bfH_{0}(\xi x)-T)K_{\tP,\bullet}(\xi x ,\xi x)\end{equation*}
o \begin{equation*}K_{\tP,\mathrm{spec}}= K_\tP=K_{\tP,\geom}\end{equation*}
est introduit en \ref{l'identit\'e fondamentale} et $K_{\tP,\oo}$ a \'et\'e d\'efini dans \ref{CVG}. 
Comme dans la d\'emonstration de \cite[11.1.1]{LW}, pour $T$ et $X\in\ag_{0,\QM}$ assez r\'eguliers, on obtient 
\begin{equation*}\int_{\bsY_G}k^{T+X}_\bullet (x)\dd x=\sum_{\tQ\in\tP_\st}
\int_{\bsY_Q}\Gamma_{\tQ}^\tG(\bfH_0(x)-X,T)k^{X}_{\tQ,\bullet}(x)\dd x\end{equation*}
avec
\begin{equation*}k^{X}_{\tQ,\bullet}(x)=\sum_{\{\tP\in\wt\ESP_\st\mid \tP\subset\tQ\}}
\sum_{\xi\in P(F)\bsl Q(F)}(-1)^{a_\tP -a_{\tQ}}
\hat{\tau}_\tP^{\tQ}(\bfH_0(\xi x)-X)K_{\tP,\bullet}(\xi x,\xi x)\ptf\end{equation*}
Fixons un $\tQ\in\wt\ESP_\st$. Puisque $\mathrm{vol}(A_\tG(F)\backslash A_\tG(\adef)^1)=1$, 
on peut remplacer l'int\'egrale sur 
$\bs{Y}_Q$ par une int\'egrale sur
\begin{equation*}\bs{Y}'_Q = \BB_\tG Q(F)\backslash G(\adef)\end{equation*}
o $\BB_\tG$ est l'image d'une section du morphisme $A_\tG(\adef)\rightarrow \ESB_\tG$. 
Notons $\BB_\tQ^\tG$ l'image d'une section du morphisme compos\'e
\begin{equation*}A_\tQ(\adef) \rightarrow A_\tG(\adef)\backslash A_\tQ(\adef) \rightarrow 
\ESB_\tQ^\tG =\ESB_\tG\backslash \ESB_\tQ\end{equation*}
et posons
\begin{equation*}\BB_\tQ = \BB_\tG \BB_\tQ^\tG\vgq \bs{Y}''_Q = \BB_\tQ Q(F)\backslash G(\adef)\ptf\end{equation*}
Pour $Z\in \bsbbc_\tQ$, notons $\bs{Y}''_Q(Z)$ l'image de 
l'ensemble $\{g\in G(\adef)\mskip 2mu \vert \mskip 2mu  \bfH_\tQ(g)=Z'\}$ dans $\bs{Y}''_Q$, o $Z'$ 
est un rel\`evement de $Z$ dans $\ESA_\tQ$. On obtient que
\begin{equation*}\int_{\bsY_Q}\Gamma_{\tQ}^\tG(\bfH_0(x)-X,T)k^{X}_{\tQ,\bullet}(x)\dd x= 
\sum_{Z\in \bsbbc_\tQ} \eta_{\tQ,F}^{\tG,T}(Z;X)\int_{\bs{Y}''_Q(Z)}k^{X}_{\tQ,\bullet}(x)\dd x\end{equation*}
avec
\begin{equation*}\int_{\bs{Y}''_Q(Z)}k^{X}_{\tQ,\bullet}(x)\dd x= 
\int_{[\BB_\tQ M_Q(F)\backslash M_Q(\adef)](Z)} k^{\tM_Q,X}_\bullet(f_\tQ,\omega;m)\dd m\ptf\end{equation*}
Pour $\bullet = \oo$, le terme $k^{\tM_Q,X}_\bullet(f_Q,\omega;m)$ 
est d\'efini en remplaant dans la d\'efinition de 
$k^{\tG,X}_\oo$ l'ensemble $G(F)$-invariant $\ESO_\oo$ par 
l'ensemble $M_Q(F)$-invariant $\ESO_\oo\cap \wt{M}_Q(F)$. 
Ce dernier correspond \`a une union finie (\'eventuellement vide) de classes de 
paires primitives dans $\wt{M}_Q$. La finitude r\'esulte de \ref{finiprim}.
Comme plus haut, on peut remplacer l'int\'egrale sur $[\BB_\tQ M_Q(F)\backslash M_Q(\adef)](Z)$ par 
une int\'egrale sur $\bs{Y}_{M_Q}(Z)$. En posant
\begin{equation*}\Jres^{\tG,X}_\bullet(f,\omega)=\int_{\bsY_G}k^{X}_\bullet(f,\omega;x)\dd x\end{equation*}
on a donc
\begin{equation*}\Jres^{\tG,T+X}_\bullet(f,\omega)=\sum_{\tQ\in\wt\ESP_\st}\sum_{Z\in \bsbbc_\tQ}
\eta_{\tQ,F}^{\tG,T}(Z;X)\;\Jres^{\tM_Q,X}_\bullet(Z;f_{\tQ},\omega)\end{equation*}
avec
\begin{equation*}\Jres^{\tM_Q,X}_\bullet(Z;f_{\tQ},\omega)=
\int_{\bsY_{M_Q}(Z)}k^{\tM_Q,X}_\bullet(f_{\tQ},\omega;m)\dd m\ptf\end{equation*}
D'o le r\'esultat puisque d'apr\`es \ref{etapol}
les fonctions $T\mapsto\eta_{\tQ,F}^{\tG,T}(Z;X)$ appartiennent \`a PolExp.
\end{proof}

 \section{Action de la conjugaison}\label{action de la conjugaison}
Pour $y\in G(\adef)$, on note $f^y$ la fonction $f\circ\mathrm{Int}_y$. Soient $y\in G(\adef)$ 
et $T\in\ag_{0,\QM}$. 
Pour $\tQ\in\wt\ESP_\st$ et $Z\in \bsbbc_\tQ$, consid\'erons la fonction 
dans $C^\infty_\mathrm{c}(\tM_Q(\adef))$ d\'efinie par
\begin{equation*}m\mapsto f_{\tQ,y}^{T}(Z;m)=\int_{U_Q(\adef)\times\bsK}
f(k^{-1}m u k)\eta_{\tQ,F}^{\tG,-\bfH_0(ky)}(Z;T)\dd u\dd k\end{equation*}
avec 
\begin{equation*}\eta_{\tQ,F}^{\tG, X}(Z;T)=\sum_{H\in \ESB_\tQ^\tG(Z)}\Gamma_{\tQ}(H-T,X)\ptf\end{equation*}

\begin{proposition}
Soient $y\in G(\adef)$ et $T\in\ag_{0,\QM}$. 
Pour $\bullet =\mathrm{spec}$, $\mathrm{g\acute{e}om}$ ou $\oo$ et pour $T$ assez r\'egulier, on a 
\begin{equation*}\Jres^{\tG,T}_\bullet(f^y,\omega)=\sum_{\tQ\in\wt\ESP_\st}\sum_{Z\in \bsbbc_\tQ}
\Jres^{\tM_Q,T}_\bullet(f_{\tQ,y}^{T}(Z),\omega)\ptf\end{equation*}
\end{proposition}

\begin{proof}
On reprend les notations de la preuve de \ref{propform}. Commenons 
par remplacer $f$ par $f^y$ et $x$ par $xy$ dans l'expression pour $k^T_\bullet(x)$. On obtient
\begin{equation*}k^T_\bullet(f^y\mskip -2mu ,\omega;xy)=\sum_{\tP\in\wt\ESP_\st}(-1)^{a_\tP -a_\tG}
\sum_{\xi\in P(F)\bsl G(F)}\hat{\tau}_\tP(\bfH_0(\xi xy)-T)K_{\tP,\bullet}(\xi x,\xi x)\end{equation*}
o $K_{\tP,\bullet}(x',x')= K_{\tP,\bullet}(f,\omega;x',x')$. Si $x'= u_0 m_0 k$ 
est une d\'ecomposition d'Iwasawa de $x'\in G(\adef)$, 
avec $u_0\in U_0(\adef)$, $m_0\in 
M_0(\adef)$ et $k=k_{x'}\in\bsK$, alors on a
\begin{equation*}\bfH_0(x' y)= \bfH_0(m_0)+\bfH_0(ky)= \bfH_0(x')+\bfH_0(ky). \end{equation*}
D'apr\`es \cite[2.9.4.(2)]{LW}, on en d\'eduit que
\begin{equation*}k^T_\bullet(f^y\mskip -2mu ,\omega;xy)=\sum_{\tQ\in\wt\ESP_\st}\sum_{\eta\in Q(F)\bsl G(F)}
\Gamma_{\tQ}(\bfH_0(\eta x) -T, -\bfH_0(k_{\eta x} y))
k_{\tQ,\bullet}^T(\eta x)\end{equation*}
o $k_{\tQ,\bullet}^T(x')= k_{\tQ,\bullet}^T(f,\omega;x')$,
puis, gr‰ce au changement de variable $x\mapsto xy$, que pour $T\in\ag_{0,\QM}$ assez r\'egulier, on a
\begin{equation*}\int_{\bsY_G}k^T_\bullet(f^y\mskip -2mu ,\omega,x)=\sum_{\tQ\in\wt\ESP_\st}
\int_{\bsY_Q}\Gamma_{\tQ}(\bfH_0( x) -T, -\bfH_0(k_x y))
k_{\tQ,\bullet}^T(x)\dd x\ptf\leqno{(1)}\end{equation*}
On obtient ensuite (comme dans la preuve de \ref{propform}) 
que, toujours pour $T\in\ag_{0,\QM}$ assez r\'egulier, le terme \`a gauche de l'\'egalit\'e dans (1) vaut
\begin{equation*}\sum_{\tQ\in \wt\ESP_\st} 
\sum_{Z\in \bsbbc_\tQ} \int_{\bs{Y}''_Q(Z)}\eta_{\tQ,F}^{\tG,-\bfH_0(k_xy)}(Z;T)k^{T}_{\tQ,\bullet}(x)\dd x\ptf\leqno{(2)}\end{equation*}
Compte-tenu de la d\'efinition 
de $f_{\tQ,y}$ et de la d\'ecomposition d'Iwasawa pour $G(\adef)$, on obtient que l'int\'egrale sur 
$\bsY''_Q(Z)$ dans (2) est \'egale \`a
\begin{equation*}\int_{\bsY_{M_Q}(Z)}k^{\tM_Q,T}_\bullet(f_{\tQ,y}^{T}(Z),\omega;m)\dd m\ptf\end{equation*}
D'o la proposition.
\end{proof}

 \section{La formule des traces: premi\`ere forme}\label{la fdt grossi\`ere}
D'apr\`es \ref{propform},
pour tout r\'eseau $\ESR$ de $\ag_{0,\QM}$, 
la restriction \`a $\ESR$ de la fonction $T\mapsto \Jres_\bullet^{\tG,T}(f,\omega)$ est de la forme
\begin{equation*}\Jres^{\tG,T}_{\ESR,\bullet} (f,\omega)=\sum_{{\nu}\in \wh\ESR} \mskip 2mu \mskip 2mu 
p_{\ESR,{\nu}}(\bullet,f,\omega;T)\mskip 2mu e^{\langle T,{\nu}\rangle}\end{equation*}
o les $p_{\ESR,{\nu}}(\bullet,f,\omega;T)$ sont des polyn™mes en $T$. 
D'apr\`es \ref{passlim} et \ref{etapol} les polyn™mes $p_{\ESR_k,0}(\bullet,f,\omega;T)$ ont une limite
lorsque $k\to\infty$ et on pose:
\begin{equation*}\bsJ^{\tG}_\bullet(f,\omega)\index{JgrastG@$\bsJ^{\tG}_\bullet$}
\bydef\lim_{k\rightarrow +\infty} p_{\ESR_k,0}(\bullet,f,\omega;T_0)\ptf\end{equation*}

\begin{theorem}\label{grossb} On a l'identit\'e:
\begin{equation*}\sum_{\oo\in\OO}\bsJ^{\tG}_{\oo}(f,\omega)=
\bsJ^{\tG}_{\mathrm{spec}}(f,\omega)\ptf\end{equation*}
Pour $M_0$ et $\bsK$ fix\'es, elle est ind\'ependante du choix de $P_0$.
\end{theorem}

\begin{proof}Par int\'egration de l'identit\'e fondamentale \ref{idfond},
ce qui a un sens compte tenu de \ref{convgeom}\ et \ref{cot\'espec}, on obtient l'identit\'e
\begin{equation*}\sum_{\oo\in\OO}\Jres^{\tG,T}_{\ESR,\oo}(f,\omega)=
\Jres^{\tG,T}_{\ESR\mathrm{,~ spec}}(f,\omega)\ptf\end{equation*}
L'ind\'ependance du choix de $P_0$ lorsque l'on prend $T=T_0$
se prouve comme dans \cite[11.3.1]{LW}.
Le th\'eor\`eme en r\'esulte par passage \`a la limite.
\end{proof}

C'est l'analogue du th\'eor\`eme \cite[11.3.2]{LW}. 
Le reste de l'article est consacr\'e au calcul de la limite $\bsJ^{\tG}_{\mathrm{spec}}(f,\omega)$.
L'\'etude des limites $\bsJ^{\tG}_{\oo}(f,\omega)$ des termes g\'eom\'etriques fera l'objet d'un article ult\'erieur.


\part{Forme explicite des termes spectraux}

 \chapter{Estim\'ees uniformes des d\'eveloppements spectraux}
\label{estim\'ees uniformes}


 \section{La formule de d\'epart}
\label{d\'epart}
Soit $Q\in\ESP_\st$. On pose
\begin{equation*}Q'=\theta_0^{-1}(Q),\quad Q_0 = Q\cap Q'\ptf\end{equation*}
Rappelons que l'on a pos\'e \begin{equation*}\bsY_\Qo=\A_\tG(\adef)Q_0(F)\bsl G(\adef)\ptf\end{equation*}
Pour $S\in\ESP^{Q'}_\st$, on note $n^{Q'}(S)$ le nombre de chambres dans $\ag_{S}^{Q'}$. 
Pour $S\in\ESP^{Q'}_\st$, $\sigma\in\Pi_\disc(M_S)$ et 
$\mu\in\ag_{Q',\CM}^*$, on a d\'efini en \ref{donn\'ees discr\`etes} un op\'erateur
\begin{equation*}\tRho_{S,\sigma,\mu}(f,\omega)=\Rho_{S,\sigma,\mu}(\delta_0,f,\omega): 
\Automd(\bsX_S,\sigma)\rightarrow \Automd(\bsX_{\theta_0(S)},\theta_0(\sigma)),\end{equation*}
et on a fix\'e une base orthonormale $\Base_S(\sigma)$ de l'espace pr\'e-hilbertien
$\Automd(\bsX_S,\sigma)$. Pour $\Psi\in\Base_S(\sigma)$, posons
\begin{equation*}\Jres^T_{Q,Q';\Psi}(x,y) =\int_{\bsmu_S}\bs\Lambda^{T,Q}E^Q(x,\tRho_{S,\sigma,\mu}
(f,\omega)\Psi,\theta_0\mu)\overline{E^{Q'}(y,\Psi,\mu)} \dd \mu\end{equation*}et 
\begin{equation*}\Jres^T_{Q,Q';\Psi}(y)= \Jres^T_{Q,Q';\Psi}(y,y)\ptf\end{equation*}
La convergence de l'int\'egrale est claire puisque $\bsmu_S$ est compact. Observons que 
par d\'efinition on a 
\begin{equation*}\sum_{\Psi\in\Base_S(\sigma)} \Jres^T_{Q,Q';\Psi}(x,y)=\int_{\bsmu_S}\bs\Lambda^{T,Q}_1
K_{Q,Q',\sigma}(x,y;\mu)\dd\mu\ptf \end{equation*}
On a la variante suivante de \cite[12.1.1]{LW}:
\begin{proposition} \label{cot\'especb}
Pour $T\in\ag_0$ 
tel que $\bsd_0(T)\geq c(f)$, on a
\begin{align*}
\lefteqn{\Jres^{\tG,T}=\sum_{\substack{Q,R\in\ESP_\st\\ 
Q\subset R}}\wt{\eta}(Q,R)\int_{\bsY_\Qo}\wt{\sigma}_Q^R(\bfH_Q(y)-T)
}\\&&\hspace{1cm}
\times\bigg(\sum_{S\in\ESP_\st^{Q'}}\frac{1}{n^{Q'}(S)}
\sum_{\bsigma\in\bsPi_\disc(M_S)}\cMSsig 
\sum_{\Psi\in\Base_S(\sigma)}\Jres^T_{Q,Q';\Psi}(y)\bigg)\dd y\ptf
\end{align*}
\end{proposition}
\begin{proof}
Elle est identique \`a celle de \textit{loc.~cit}., compte-tenu de la formule \ref{cot\'espec}
et de l'expression pour le noyau $K_{Q,\delta_0,\chi}$ donn\'ee par la proposition \ref{KQd}.
\end{proof}

D'apr\`es \ref{KQd}, l'ensemble des $(\sigma,\Psi)$ tels que 
$\tRho_{S,\sigma,\mu}(f,\omega)\Psi\neq 0$ est fini. Donc, seul un nombre fini de termes non nuls
apparaissent dans l'expression de $\Jres^{\tG,T}$. L'expression 
$\Jres^{\tG,T}$ est une combinaison lin\'eaire finie d'int\'egrales it\'er\'ees (en $y$ et en $\mu$)
\begin{equation*}\int_{\bsY_\Qo}\wt{\sigma}_Q^R(\bfH_Q(y)-T)\Jres^T_{Q,Q';\Psi}(y)\dd y\leqno{(1)}\end{equation*}
 
Ë priori la proposition 
n'affirme que la convergence des int\'egrales dans l'ordre indiqu\'e, 
et pas la convergence absolue de l'int\'egrale multiple.

Pour $\Psi\in \Automd(\bsX_S,\sigma)$ et $\Phi\in \Automd(\bsX_{\theta_0(S)},\theta_0(\sigma))$, 
on consid\'erera aussi les expressions suivantes: 
\begin{equation*}\Jres^T_{Q,Q';\Phi,\Psi}(x,y)= 
\int_{\bsmu_S}\bs\Lambda^{T,Q}E^Q(x,\Phi,\theta_0\mu)\overline{E^{Q'}(y,\Psi,\mu}) \dd \mu\end{equation*}
et \begin{equation*}\Jres^T_{Q,Q';\Phi,\Psi}(y)=\Jres^T_{Q,Q';\Phi,\Psi}(y,y)\ptf\end{equation*}
Ë nouveau la convergence de l'int\'egrale est claire puisque $\bsmu_S$ est compact.

\begin{remark}\label{somfi}
On observe que pour  $\Psi\in\Automd(\bsX_S,\sigma)$ on peut \'ecrire
\begin{equation*}\tRho_{S,\sigma,\mu}(f,\omega)\Psi=
 \sum_{\Phi\in \Automd(\bsX_{\theta_0(S)},\theta_0(\sigma))} \vtt_{\Phi,\Psi}(\mu) \Phi\end{equation*}
o la somme porte sur un ensemble fini et  
o les $\vtt_{\Phi,\Psi}$ sont des fonctions lisses sur le groupe compact $\bsmu_S$.
On voit donc que l'expression $\Jres_{Q,Q';\Psi}^T(y)$ introduite plus haut 
est une combinaison lin\'eaire finie d'expressions du type 
\begin{equation*}\Jres_{Q,Q';\Phi,\Psi;\vartheta}^T(y)\bydef \int_{\bsmu_S}
\bs\Lambda^{T,Q}E^Q(y,\Phi,\theta_0\mu)\overline{E^{Q'}(y,\Psi,\mu}) \vtt(\mu)\dd \mu\end{equation*}
o $\Phi\in \Automd(\bsX_{\theta_0(S)},\theta_0(\sigma))$ 
ne d\'epend pas de $\mu$ et o $\vtt$ est une fonction lisse sur $\bsmu_S$. 
Pour l'\'etude de la convergence des int\'egrales it\'er\'ees (1), 
il suffira de consid\'erer celle des \begin{equation*}\int_{\bsY_\Qo}\wt{\sigma}_Q^R(\bfH_Q(y)-T)\Jres_{Q,Q';\Phi,\Psi;\vtt}^T(y)\dd y\ptf\leqno{(2)}\end{equation*}
Pour prouver la convergence de (2), on peut remplacer la fonction $\vtt$ par le sup de sa valeur absolue et, 
\`a un scalaire pr\`es, on peut supposer que 
ce sup vaut $1$. On est donc ramen\'es \`a prouver la convergence des int\'egrales it\'er\'ees 
\begin{equation*}\int_{\bsY_\Qo}\wt{\sigma}_Q^R(\bfH_Q(y)-T)\Jres_{Q,Q';\Phi,\Psi}^T(y)\dd y\ptf\leqno{(3)}\end{equation*}
En revanche pour effectuer le calcul exact de (1), il faut effectuer le calcul exact de (2) pour n'importe quel $\vtt$. 
\end{remark}

 \section{Estimations}\label{estimations}
Pour $H\in\ESA_\Qo$, soit $M_\Qo(\adef;H)$ l'ensemble des $m\in M_\Qo(\adef)$ 
tels que \begin{equation*}\bfH_\Qo(m)=H\ptf\end{equation*}
On note $\bsY_\Qo(H)$ l'image de ${U_\Qo(\adef)\times }M_\Qo(\adef;H) \times\bsK$ dans $\bsY_\Qo$. 
On pose
\begin{equation*}\ESC_\Qo^\tG\bydef \ESB_\tG\backslash \ESA_\Qo \ptf\end{equation*}
Observons que $\bfH_\Qo$ envoie $\bs{Z}_\Qo= A_\tG({\adef})A_\Qo(F) \backslash A_\Qo({\adef})$
 sur un sous-groupe d'indice fini de $\ESC_\Qo^\tG$. 
L'int\'egrale it\'er\'ee \ref{d\'epart}~(1) est \'egale \`a
\begin{equation*}\sum_{H\in\ESC_\Qo^\tG}\wt{\sigma}_Q^R(H_Q-T) 
\int_{\bsY_\Qo(H)}\Jres^T_{Q,Q';\Psi}(y)\dd y\ptf\end{equation*}

Les estimations \cite[12.2.1]{LW} et \cite[12.2.3]{LW} sont valables ici, mutatis mutandis. On rappelle 
que $\Siegel^*$ est un domaine de Siegel pour le quotient $\BB_GG(F)\backslash G(\adef)$ 
o $\BB_G$ est l'image d'une section du morphisme $A_G(\adef)\rightarrow \ESB_G$. 
%
\begin{lemma}\label{estim1}
Soit $h\in C^\infty_\mathrm{c}(G(\adef))$. Il existe $c>0$ tel que pour tout $x,\mskip 2mu y\in\Siegel^*$, on ait
\begin{equation*}\sum_{\delta\in\BB_GG(F)}\vert h(x^{-1}\delta y)\vert < c\mskip 2mu \delta_{P_0}(x)^{1/2}\delta_{P_0}(y)^{1/2}\ptf\end{equation*}
\end{lemma}

\begin{proof}
Elle est identique \`a celle de \cite[12.2.1]{LW}. Pour passer du groupe $U_R$ 
\`a son alg\`ebre de Lie $\mathfrak{u}_R$, on utilise comme dans la preuve de 
la proposition de \ref{CVG} l'existence d'un $F$-isomorphisme de vari\'et\'es alg\'ebriques 
$\mathfrak{u}_R\rightarrow U_R$ compatible \`a l'action de $A_R$.
\end{proof}

On fixe $S\in\ESP_\st^{Q'}$, une repr\'esentation automorphe $\sigma\in\Pi_\disc(M_S)$, et des vecteurs 
$\Psi\in \Automd(\bsX_S,\sigma)$ et $\Phi\in \Automd(\bsX_{\theta_0(S)},\theta_0(\sigma))$. On pose
\begin{equation*}L= M_Q\vgq L'=M_{Q'}\vgq L_0 = M_\Qo\ptf\end{equation*}
On fixe un domaine de Siegel 
$\bs{\mathfrak{S}}^{L,*}$ pour le quotient $\BB_QL(F)\backslash L({\adef})$. On suppose 
que $\BB_Q\subset A_Q(\adef)$ 
est de la forme
\begin{equation*}\BB_Q = \BB_\G \BB_Q^G\end{equation*}
o $\BB_Q^G\subset A_Q(\adef)$ est l'image d'une section du morphisme compos\'e
\begin{equation*}A_Q(\adef) \rightarrow A_G(\adef)\backslash A_Q(\adef) \rightarrow \ESB_Q^G\ptf\end{equation*}
On fixe aussi un compact $\Omega_Q \subset U_Q({\adef})$ tel 
que $U_Q({\adef})= U_Q(F)\Omega$ et l'on pose
\begin{equation*}\Siegel_Q^G = \Omega_Q \BB_Q^G \Siegel^{L,*}\bs{K}\ptf\end{equation*}
C'est un domaine de Siegel pour le quotient $\BB_G Q(F) \backslash G(\adef)$. 
On fixe de la mme mani\`ere un domaine de Siegel $\Siegel_{Q'}^G = \Omega_{Q'} \BB_{Q'}^G \Siegel^{L'\mskip -2mu ,*}\bs{K}$ 
pour $\BB_G Q'(F)\backslash G({\adef})$.

\begin{proposition}\label{[LW,12.2.3]}
Il existe $c>0$ tel que pour tout $(x,y)\in\Siegel_Q^G\times\Siegel_{Q'}^G$ , on ait
\begin{equation*}\int_{\bsmu_S}\vert E^Q(x,\Phi,\theta_0\mu)\overline{E^{Q'}(y,\Psi,\mu})\vert \dd \mu 
< c\mskip 2mu \delta_{P_0}(x)^{1/2}\delta_{P_0}(y)^{1/2}\ptf\end{equation*}
\end{proposition}
 
\begin{proof}
Comme dans \cite[12.2.3]{LW} on se ram\`ene \`a prouver 
qu'il existe $c>0$ tel que pour tout $y\in\Siegel^*$, on ait
\begin{equation*}\int_{\bsmu_S}\vert E^G(y,\Psi,\mu)\vert^2 \dd \mu < c\mskip 2mu \delta_{P_0}(y)\ptf\end{equation*}
On choisit un sous-groupe ouvert compact $\bsK'$ de $G(\adef)$ tel que la fonction $\Psi $ 
soit invariante \`a droite par $\bsK'$, et on consid\`ere le noyau
\begin{equation*}K_G(e_{\bsK'}; y,y)=\sum_{\delta\in{\BB_G}\G(F)}e_{\bsK'}(y^{-1}\delta y)\ptf\end{equation*}
Son expression spectrale est une somme de termes tous positifs ou nuls et l'un d'eux est l'int\'egrale ci-dessus. 
On conclut gr‰ce au lemme \ref{estim1}. 
\end{proof}

\begin{corollary}\label{12.2.3}
Pour tout $(x,y)\in G(\adef)\times G(\adef)$, l'int\'egrale d\'efinissant
\begin{equation*}{\Jres^T_{Q,Q';\Phi,\Psi}}(x,y)\end{equation*}
est absolument convergente. De plus, il existe $c,\mskip 2mu D >0$ et un sous-ensemble compact 
$C_Q$ de $\Siegel^{L,*}$ tels que {pour tout $x \in\BB_Q^G\Siegel^{L,*}\bsK$ et tout $y\in G(\adef)$}, 
en \'ecrivant $x= a s k $ avec $a\in\BB_Q^G$, $s\in\Siegel^{L,*}$ et $k\in\bsK$, on ait 
\begin{equation*}\vert \Jres^T_{Q,Q';\Phi,\Psi}(x,y)\vert\leq c\mskip 2mu \vert a\vert^D\vert y\vert^D\end{equation*}
si $s\in C_Q$ et $\Jres^T_{Q,Q';\Phi,\Psi}(x,y)=0$ sinon.
\end{corollary}

\begin{proof}
Soit $\bsK'$ un sous-groupe ouvert compact distingu\'e de $\bsK$ tel que la fonction $\phi$ soit invariante \`a droite par $\bsK'$. 
D'apr\`es la proposition \ref{supportcompact}, il existe un sous-ensemble compact $C_Q$ de $\Siegel^{L,*}$ 
tel que {pour tout $(a,k)\in\BB_Q^G\times\bsK$ et tout $\mu\in \bsmu_S$}, le support de la fonction sur $\Siegel^{L,*}$ 
\begin{equation*}h\mapsto\bs\Lambda^{T,Q}E^Q(ask,\Phi,\theta_0(\mu))\end{equation*}
soit contenu dans $C_Q$. D'autre part pour $y\in G(\adef)$, il existe un {$g\in\BB_G Q'(F)$} tel que 
$gy\in\Siegel_{Q'}^G$. On proc\`ede comme dans la preuve de \cite[12.2.4]{LW}, en remarquant que 
puisque la fonction $\bs{\delta}_{P_0}$ est \`a croissance lente, 
$\bs{\delta}_{P_0}(x)^{1/2}\bs{\delta}_{P_0}(gy)^{1/2}$ est essentiellement major\'e par $\vert a\vert^D\vert y\vert^D$.
\end{proof}

Ë priori nous ne pouvons rien dire ici sur le centre, c'est pourquoi nous nous sommes limit\'es aux 
\begin{equation*}x\in\BB_Q^G \Siegel^{L,*}\bs{K}\ptf\end{equation*}Pour $y=x$, on en d\'eduira en \ref{convergence d'une int\'egrale it\'er\'ee} des 
estimations pour $x\in\BB_Q^\tG \Siegel^{L,*}\bs{K}$. 


 \section{Convergence d'une int\'egrale it\'er\'ee}
\label{convergence d'une int\'egrale it\'er\'ee}

On suppose ici que $T$ est un \'el\'ement r\'egulier de $\ag_0^G$ 
tel que\footnote{Cette condition peut sembler inad\'equate, d\`es lors qu'on aura \textit{in fine} 
\`a \'evaluer un \'el\'ement de $\mathrm{PolExp}$ en $T_0\in \ESA_0^G$ qui n'est \`a priori pas $\theta_0$-invariant. 
Mais comme l'\'el\'ement de PolExp \`a \'evaluer ne d\'epend 
que de l'image de $T\in \ag_{0,\QM}^G$ dans le sous-espace $a_{\tM_0,\QM}^\tG= (\ag_{0,\QM}^G)^{\theta_0}$ 
de $\ag_{0,\QM}^G$ form\'e des \'el\'ements $\theta_0$-invariants, 
cette condition n'est pas vraiment gnante.} $\theta_0(T)=T$. On suppose de plus que 
pour tout $\alpha\in\Delta_0$, on a
\begin{equation*}0<\alpha(T)\leq \kappa\mskip 2mu \bsd_0(T)\end{equation*}
pour une constante $\kappa>0$ assez grande. 
Dans un tel c™ne, les fonctions $\bsd_0(T)$, $\lVert  T \rVert $ et $\alpha(T)$  
sont \'equivalentes, pour tout $\alpha \in \Delta_0$.

Fixons des vecteurs $\Psi$ et $\Phi$ comme ci-dessus. Pour all\'eger l'\'ecriture, posons 
\begin{equation*}\Jres(x,y)= \Jres_{Q,Q';\Phi,\Psi}(x,y) \com{et} \Jres(y)=\Jres(y,y)\ptf\end{equation*}
On veut prouver la convergence de l'int\'egrale 
\begin{equation*}\int_{\bsY_\Qo}\wt{\sigma}_Q^R(\bfH_Q(y)-T)\lvert  \Jres(y)\rvert \dd y\vg\leqno{(1)}\end{equation*}
ou, ce qui est \'equivalent, de l'expression
\begin{equation*}\sum_{H\in\ESC_\Qo^\tG}\wt{\sigma}_Q^R(H_Q-T)\int_{\bsY_\Qo(H)}\lvert  \Jres(y))\rvert \dd y\ptf \leqno{(2)}\end{equation*}
L'int\'egration sur $\bsY_\Qo(H)$ se d\'ecompose en une int\'egration sur le produit
\begin{equation*}U_\Qo(F)\bsl U_\Qo(\adef)\times\bsX_{L_0}^\tG(H)\times\bsK\end{equation*}
o $\bs{X}_{L_0}^\tG(H)$ est l'image de $L_0(\adef;H')$ dans $A_\tG(\adef)L_0(F)\backslash L_0(\adef)$ 
pour un rel\`evement $H'\in \ESA_\Qo$ de $H\in \ESC_\Qo^\tG$.
Par ailleurs, la fonction que l'on int\`egre est invariante \`a gauche par le groupe $U_Q(\adef)\cap U_{Q'}(\adef)$. 
Posons $U_\Qo^L = U_Q\cap L$ et $U_\Qo^{L'} = U_\Qo\cap L'$. L'application naturelle
\begin{equation*}U_\Qo\rightarrow U_\Qo^L\times U_\Qo^{L'}\end{equation*}
induit un isomorphime
\begin{equation*}U_\Qo(F)(U_Q(\adef)\cap U_{Q'}(\adef)\bsl U_\Qo(\adef)\rightarrow
U_\Qo^L(F)\bsl U_\Qo^L(\adef)\times U_\Qo^{L'}(F)\bsl U_\Qo^{L'}(\adef)\ptf\end{equation*}
On peut donc remplacer dans (2) la variable $y$ par $uu'xk$ avec 
\begin{equation*}u\in U_\Qo^L(F)\bsl U_\Qo^L(\adef)\vgq
u'\in U_\Qo^{L'}(F)\bsl U_\Qo^{L'}(\adef)\vgq x\in\bsX_{L_0}^\tG(H)\quad \hbox{et}\quad k\in \bsK\ptf\end{equation*}
La mesure $\dd y$ se transforme alors en 
\begin{equation*}\delta_\Qo(x)^{-1}\dd u\dd u' \dd x \dd k\end{equation*}
et on a $\bfH_\Qo(x)= \bfH_\Qo(y)=H$. On obtient:
\begin{equation*}\Jres(uu'xk) = \Jres(uxk,u' xk)
\leqno{(3)}\end{equation*}
On a la variante suivante du lemme \cite[12.3.1]{LW}:

\begin{lemma}\label{[LW,12.3.1]}
Il existe un sous-ensemble fini $\omega\subset {\ag_Q^\tG}$ ind\'ependant de $T$ tel que si
$\Jres(uu'xk)$ est non nul, alors $q_Q(H)=((1-\theta_0)H)_Q^{\tG}$ appartient \`a $\omega$.
\end{lemma}

\begin{proof}Via le choix d'une section de la surjection $\ESA_M\to\ESA_{Q'}$ on dispose 
d'un isomorphisme $\bsmu_M=\bsmu_{Q'}\times \bsmu_S^{Q'}$ et
l'int\'egration sur $\bsmu_S$ se d\'ecompose en une int\'egrale double:
\begin{align*}
\lefteqn{
\Jres(uu'xk)=\int_{\bsmu_S^{Q'}}\int_{\bsmu_{Q'}}
\bs\Lambda^{T,Q}E^Q(uxk,\Phi,\theta_0 (\mu_{Q'}+\mu^{Q'}))
}\\&&\hspace {5,4cm}\times
\quad\overline{E^{Q'}(u'xk,\Psi,\mu_{Q'} +\mu^{Q'})}\dd \mu_{Q'}\dd \mu^{Q'}.
\end{align*}
Soit $\BB_\Qo\subset A_\Qo(\adef)$ un sous-groupe de la forme
\begin{equation*}\BB_\Qo= \BB_\Qo^Q \BB_Q\;(= \BB_\Qo^Q \BB_Q^G \BB_G)\end{equation*}
o $\BB_\Qo^Q\subset A_\Qo(\adef)$ est l'image d'une section du morphisme compos\'e
\begin{equation*}A_\Qo(\adef) \rightarrow A_Q(\adef)\bsl A_\Qo(\adef)\rightarrow \ESB_\Qo^Q\ptf\end{equation*}
Notons $\Siegel^{L_0,*}=\EE_{L_0}\Siegel^{L_0,1}$ un domaine de Siegel pour le quotient 
$\BB_\Qo L_0(F)\backslash L_0(\adef)$ (cf. \ref{Siegel}). 
Fixons de la mme mani\`ere un sous-groupe $\BB'_\Qo= 
\BB_\Qo^{Q'}\BB_{Q' }\subset A_\Qo(\adef)$. 
On ne peut pas en g\'en\'eral s'arranger pour que $\BB_\Qo= 
\BB'_\Qo$, mais puisque $\BB_\Qo \cap \BB'_\Qo$ est d'indice fini 
dans $\BB_\Qo$, quitte \`a grossir $\EE_{L_0}$, on peut toujours 
supposer que 
\begin{equation*}L_0(\adef)= (\BB_\Qo \cap \BB'_\Qo)L_0(F)\Siegel^{L_0,*}\ptf\end{equation*}
On suppose aussi que $\BB_G$ est de la forme
$\BB_G = \BB_G^\tG\BB_\tG$
o $\BB_G^\tG\subset A_G(\adef)$ est l'image d'une section du morphisme compos\'e
\begin{equation*}A_G(A) \rightarrow A_\tG(\adef)\backslash A_G(\adef) \rightarrow \ESB_G^\tG = \ESB_\tG \backslash \ESB_G\end{equation*}
et l'on pose
\begin{equation*}\BB_\Qo^\tG= \BB_\Qo^\tG \BB_Q^\tG= 
\BB_\Qo^Q\BB_Q^G\BB_G^\tG\quad \hbox{et}\quad 
\BB'^\tG_\Qo= \BB_\Qo^{Q'}\BB_{Q'}^\tG=\BB_\Qo^{Q'} \BB_{Q'}^G\BB_G^\tG\ptf\end{equation*}
Choisissons un rel\`evement de $x\in \bsX_{L_0}^\tG(H)$ dans $L_0(\adef;H')$ et \'ecrivons
$x= zas$ avec $z\in \BB_\tG L_0(F)$, $a\in (\BB_\Qo^\tG \cap \BB'^\tG_\Qo)$ 
et $s\in \Siegel^{L_0,*}$. On d\'ecompose $a$ sous la forme
\begin{equation*}a=a_Qa^Q = a_{Q'}a^{Q'}\end{equation*}
avec $a_Q\in \BB_Q^\tG$, $a^Q\in \BB_\Qo^Q$, 
$a_{Q'}\in \BB_{Q'}^\tG$ et $a^{Q'}\in \BB_\Qo^{Q'}$. 
 Posons $H_0=\bfH_\Qo(a)$. Comme dans la d\'emonstration de \cite[12.3.1]{LW}, on obtient que
\begin{align*}
\lefteqn{\Jres(uu'xk)=\bs{\delta}_Q^{1/2}(a_Q)\bs{\delta}_{Q'}^{1/2}(a_{Q'})\int_{\bsmu_{Q'}}
e^{\langle \theta_0(\mu_{Q'}),(H_0)_Q \rangle -\langle\mu_{Q'}, (H_0)_{Q'} \rangle}
\dd \mu_{Q'}}\\&&\hspace{2cm}
\times\int_{\bsmu_S^{Q'}}\bs\Lambda^{T,Q}E^Q(ua^Q s k,\Phi,\theta_0 (\mu^{Q'}))
\overline{E^{Q'}(u'a^{Q'}s k,\Psi,\mu^{Q'})}\dd \mu^{Q'}.
\end{align*}
Or
\begin{equation*}\int_{\bsmu_{Q'}}e^{\langle \theta_0(\mu_{Q'}),(H_0)_Q \rangle -\langle\mu_{Q'}, (H_0)_{Q'} \rangle}
\dd \mu_{Q'}=\int_{\bsmu_{Q'}}e^{\langle\mu_{Q'},\theta_0^{-1}((H_0)_Q)-(H_0)_{Q'} \rangle}\dd \mu_{Q'}\end{equation*}
et cette int\'egrale n'est non nulle que si $\theta_0^{-1}((H_0)_Q)-(H_0)_{Q'}=0$. 
Pour que $\Jres(uu'xk)$ soit non nul, il faut donc que 
\begin{equation*}(\bfH_Q(x) -\theta_0 (\bfH_{Q'}(x)))^\tG = (H- \theta_0(H))_Q^\tG
=(\bfH_Q(s) -\theta_0 (\bfH_{Q'}(s)))^\tG\ptf\end{equation*}
Maintenant, on observe que l'ensemble
\begin{equation*}\omega =\{(\bfH_Q(s) -\theta_0 (\bfH_{Q'}(s)))^\tG\mskip 2mu \vert \mskip 2mu  s\in \Siegel^{L_0,*} \}\subset \ag_Q^\tG\end{equation*}
est fini et le lemme en r\'esulte.
\end{proof}

\begin{proposition}\label{[LW,12.3.2]}
{Pour $T\in (\ag_0^G)^{\theta_0}$,} l'expression
\begin{equation*}\sum_{H\in\ESC_\Qo^\tG}\wt{\sigma}_Q^R(H_Q-T)\int_{\bsY_\Qo(H)}
\lvert  \Jres^T_{Q,Q';\Phi,\Psi}(y)\rvert \dd y\end{equation*}
est convergente et la somme sur $H$ est finie.
\end{proposition}

\begin{proof}
Il s'agit de prouver que pour $H\in\ESC_\Qo^\tG$, la fonction
\begin{equation*}(u,u'\mskip -2mu ,x,k)\mapsto\wt{\sigma}_Q^R(H-T)\delta_\Qo(x)^{-1}\Jres(uu'xk)\end{equation*}
est absolument int\'egrable sur
\begin{equation*}U_\Qo^L(F)\bsl U_\Qo^L(\adef)\times U_\Qo^{L'}(F)\bsl U_\Qo^{L'}(\adef)
\times\bsX_{L_0}^\tG(H)\times\bsK \end{equation*}
et qu'elle est nulle sauf pour un nombre fini de $H$. 
On proc\`ede comme dans la preuve de \cite[12.3.2]{LW}. On commence par d\'ecouper le domaine de sommation en $H$ 
gr‰ce \`a la partition de \cite[1.7.5]{LW} appliqu\'ee au couple $(P,R)= (Q_0,Q)$: 
on peut fixer un $P'\in\ESP_\st$ tel que $Q_0\subset P'\subset Q$ et imposer que
\begin{equation*}\phi_\Qo^{P'}(H-T)\tau_{P'}^Q(H-T)=1\ptf\end{equation*}
On s'int\'eresse donc aux $H\in\ag_\Qo$ tels que
\begin{equation*}\wt{\sigma}_Q^R(H-T)\phi_\Qo^{P'}(H-T)\tau_{P'}^Q(H-T)=1\vgq q_Q(H)\in\omega\ptf \leqno{(4)}\end{equation*}
D'apr\`es \cite[2.13.3]{LW} (l'exposant $G$ est ici remplac\'e par un exposant $\tG$, voir \ref{les fonctions sigma}), 
pour $H$ v\'erifiant (4), on a
\begin{equation*}\lVert  H^{\tG} -T_\Qo \rVert \ll 1 +\lVert  (H -T)_{P'}^Q\rVert \leqno{(5)}\end{equation*}
d'o
\begin{equation*}\lVert  H^{\tG}\rVert \ll 1+\lVert H_{P'}^Q\rVert \ptf\leqno{(6)}\end{equation*}
Les constantes implicites dans (5) et (6) d\'ependent de $T$.
Au lieu d'int\'egrer sur $x\in\bsX_{ L_0}^\tG(H)$, on peut {choisir un rel\`evement $H'\in \ESA_\Qo$ de 
$H\in \ESC_\Qo^\tG$ et int\'egrer sur 
$x\in ((\BB_\Qo\cap \BB'_\Qo)\Siegel^{L_0,*})\cap L_0(\adef;H')$. 
De mme, on peut faire varier $(u,u')$ dans un compact $\Omega_\Qo^L\times\Omega_\Qo^{L'}$ de 
$U_\Qo^L(\adef)\times U_\Qo^{L'}(\adef)$. 
On \'ecrit\footnote{Notons que notre $s$ joue le r™le du $x$ de la d\'emonstration de \cite[12.3.2]{LW}.} 
$x = z a s$ avec 
\begin{equation*}z\in\BB_\tG L_0(F)\vgq a\in (\BB_\Qo^\tG \cap \BB'^\tG_\Qo)\quad \hbox{et} \quad s\in \Siegel^{L_0,*}\vg\end{equation*}
et on d\'ecompose $a$ en $a=a_Qa^Q=a_{Q'}a^{Q'}$ comme dans la preuve 
du lemme \ref{[LW,12.3.1]}. Pour $u\in\Omega_\Qo^{L}$, $u'\in\Omega_\Qo^{L'}$ et 
$k\in\bsK$, en supposant que $a_Q^{-1}\theta_0(a_{Q'})$ appartient \`a $\A_\tG(\adef)A_Q(\adef)^1$ -- sinon $\Jres(uu'xk)=0$ 
(cf. la preuve de \textit{loc.~cit}.) --, on a
\begin{equation*}\Jres(uu'xk) =\bs{\delta}_Q(a_Q)\Jres(ua^Qsk,u'a^{Q'}sk)\ptf\end{equation*}
Quitte \`a grossir $\EE_Q$, on peut supposer que $L_0(\adef)^*= \EE_\Qo L_0(\adef)^1$ est contenu 
dans $L_Q(\adef)^*= \EE_Q L_Q(\adef)^1$. 
Alors pour chaque $u\in\Omega_\Qo^L$, on peut choisir un $\gamma\in L(F)$ tel que
\begin{equation*}y_1= \gamma u a^Q s \in \Siegel ^{L,*}=\EE_Q \Siegel^{L,1}\ptf\end{equation*}
D'apr\`es \ref{12.2.3}, il existe $c,\mskip 2mu D >0$ et un compact $C_L$ dans 
$\Siegel^{L,*}$ tel que pour tout $k\in\bsK$, tout $u\in\Omega_\Qo^L$ et tout $u'\in\Omega_\Qo^{L'}$, 
on ait $\Jres(y_1k,u'a^{Q'}sk)=0$ si $y_1\notin C_L$ et
\begin{equation*}\vert \Jres(y_1k,u'a^{Q'}sk)\vert \leq c\mskip 2mu \vert u'a^{Q'}s\vert^D\quad\hbox{sinon}\ptf\end{equation*}
Comme dans la preuve de \cite[12.3.2]{LW}, on obtient
\begin{equation*}\lVert  H^Q_{P'}\rVert  +\lVert  \bfH_0(s)\rVert \ll 1 +\lVert  \bfH_0(y_1)\rVert \ptf \leqno{(7)}\end{equation*}
Pour $y_1\in C_L$, d'apr\`es (6) et (7), $\lVert  H^{\tG}\rVert $ et $\lVert \bfH_0(s)\rVert $ sont born\'es. On en d\'eduit que 
$H$ varie dans un sous-ensemble fini de $\ESC_\Qo^\tG$.
D'autre part $s$ appartient \`a $\Siegel^{L_0,*}$ et $\lVert \bfH_0(s)\rVert $ est born\'e, 
par cons\'equent $\vert s\vert$ est born\'e.} On obtient que
\begin{equation*}\bs{\delta}_\Qo(x)^{-1}\vert \Jres(u'uxk)\vert\ll 1\ptf\end{equation*}
D'o la proposition, puisque l'ensemble
\begin{equation*}\Omega_\Qo^L\times\Omega_\Qo^{L'}\times ((\BB_\Qo\cap\BB'_\Qo)\Siegel^{L_0,*}\cap L_0(\adef;H))\times \bsK\end{equation*}
est de volume fini.
\end{proof}


 \section{Transformation de l'op\'erateur de troncature}
\label{transformation de l'op\'erateur}

Pour calculer l'expression
\begin{equation*}\sum_{H\in\ESC_\Qo^\tG}\wt{\sigma}_Q^R(H_Q-T) 
\int_{\bsY_\Qo(H)}\Jres^T_{Q,Q';\Phi,\Psi;\vtt}(y)\dd y\leqno{(1)}\end{equation*}
o $\vtt$ est une fonction lisse sur $\bsmu_S$, on d\'ecompose l'int\'egrale sur $\bsY_\Qo(H)$ 
comme en \ref{convergence d'une int\'egrale it\'er\'ee}. On peut permuter l'int\'egrale sur le groupe compact 
\begin{equation*}U_{Q_0^L}(F)\bsl U_{Q_0^L}(\adef)\times U_{Q_0^{L'}}(F)\bsl U_{Q_0^{L'}}(\adef)\end{equation*}
avec celle sur le groupe (lui aussi compact) $\bsmu_S$. Pour $(x,k)\in\bsX_{L_0}^\tG(H)\times\bsK$, 
la compos\'ee de ces deux int\'egrales est \'egale \`a
\begin{equation*}\int_{\bsmu_S}(\bs\Lambda^{T,Q}E^Q)_\Qo(xk,\Phi,\theta_0(\mu))
\overline{E^{Q'}_\Qo(xk,\Psi,\mu)}\vtt(\mu)\dd \mu\leqno{(2)}\end{equation*}
o l'indice $Q_0$ signifie que l'on prend le terme constant le long de $\Qo$.
D'apr\`es \cite[4.1.1]{LW}, cette expression (2) n'est non nulle que si 
\begin{equation*}\phi_\Qo^Q(H-T)= 1\ptf\end{equation*}
Cela entra"ne que dans le d\'ecoupage suivant les $P'\in\ESP_\st$ tel que $Q_0\subset P'\subset Q$ 
dans la preuve de la proposition \ref{[LW,12.3.2]}, seul 
le domaine correspondant \`a $P'=Q$ donne une contribution non nulle. 
D'apr\`es \ref{convergence d'une int\'egrale it\'er\'ee}~(5), il existe $c>0$ 
telle que pour les $H$ v\'erifiant
\begin{equation*}\wt{\sigma}_Q^R(H-T)\phi_\Qo^Q(H-T)=1\vgq q_Q(H)\in\omega \end{equation*}
on ait
\begin{equation*}\lVert  H^{\tG} -T_\Qo\rVert \leq c\ptf \leqno{(3)}\end{equation*}
Le point est qu'ici la constante $c$ est ind\'ependante de $T$ (d'apr\`es \cite[2.13.3]{LW} 
et le lemme \ref{[LW,12.3.1]}). {En particulier, puisque $T_\Qo\in \ag_\Qo^G$, on a
\begin{equation*}H^{\tG} -T_\Qo= H_G^\tG+(H^G -T_\Qo)\end{equation*}
et $\lVert  H^\tG_G\rVert $ est born\'e par une constante ind\'ependante de $T$.}
Fixons un r\'eel $\eta$ tel que $0<\eta <1$. Si $T$ est assez r\'egulier, la condition (3) entra"ne
\begin{equation*}\lVert  H^Q -T_\Qo^Q\rVert \leq\lVert \eta T\rVert  \ptf \leqno{(4)}\end{equation*}
Pour $T\in\ag_0$, on note\footnote{Observons que la fonction $\Kappa{T}$ 
utilis\'ee ici n'est pas tout-\`a-fait la mme que celle de \cite{LW} puisque nous 
ne passons pas au quotient par le centre.} 
$\Kappa{T}$ la fonction caract\'eristique du sous-ensemble des $X\in\ag_0$ tels que $\lVert  X\rVert \leq\lVert  T\rVert $.
 D'apr\`es \cite[4.2.2]{LW} il existe $c'>0$ tel que si
 \begin{equation*}\bsd_0(T)\geq c'(c+1)\vg\end{equation*}
 l'expression (2) multipli\'ee par 
$\wt{\sigma}_Q^R(H-T)$ vaut\footnote{Il faut supprimer le signe $(-1)^{a_Q-a_G}$ 
dans la formule du lemme \cite[4.2.2]{LW} page 86. Voir \Err(iii) de l'Annexe \ref{Erratum}.}
\begin{align*}
\lefteqn{
\Kappa{\eta T}(H^Q- T_\Qo^Q)\wt{\sigma}_Q^R(H-T)\phi_\Qo^Q(H-T)}\\
&&\hspace{2.5cm}
\times\int_{\bsmu_S}\bs\Lambda^{T\Lbra H^Q\Rbra,Q_0}E^Q_\Qo(xk,\Phi,\theta_0(\mu))
\overline{E^{Q'}_\Qo(xk,\Psi,\mu)}\vtt(\mu)\dd \mu
\end{align*}
si $\lVert  H^{\tG} -T_\Qo\rVert \leq c$ et elle est nulle sinon. Ce dernier point 
r\'esulte de l'analogue de la preuve de \ref{convergence d'une int\'egrale it\'er\'ee}~(5) 
pour l'expression ci-dessus. Notons que pour $H$ v\'erifiant $\phi_\Qo^G(H-T)=1$, 
l'\'el\'ement \begin{equation*}T\Lbra H^Q\Rbra = T\Lbra H^Q\Rbra^{Q_0}\in\ag_{P_0}^{Q_0}\end{equation*}
est {\og plus r\'egulier\fg} que $T^{Q_0}$: en effet, d'apr\`es \cite[4.2.1]{LW}, on a
\begin{equation*}\bsd_{P_0\cap L_0}^{L_0}(T\Lbra H^Q\Rbra)\geq\bsd_{P_0\cap L_0}^{L_0}(T^{Q_0})\geq\bsd_0(T)\ptf\end{equation*}
On a donc prouv\'e la

\begin{proposition}\label{[LW,12.4.1]}
Il existe $c,\mskip 2mu c'>0$ tel que pour tout {$T\in (\ag_0^G)^{\theta_0}$} v\'erifiant $\bsd_0(T)\geq c'(c+1)$, 
l'expression (1) soit \'egale \`a
\begin{align*}
\lefteqn{\sum_{H\in\ESC_\Qo^\tG}\Kappa{\eta T}(H^Q- T_\Qo^Q)\wt{\sigma}_Q^R(H-T)
\phi_\Qo^Q(H-T)e^{-\langle 2\rho_\Qo,H\rangle} }\\
&&\hspace{0,2cm}\times\int_{\bsX_{L_0}^\tG(H)\times\bsK}
\int_{\bsmu_S}\bs\Lambda^{T\Lbra H^Q\Rbra,Q_0}E^Q_\Qo(xk,\Phi,\theta_0(\mu))
\overline{E^{Q'}_\Qo(xk,\Psi,\mu)}\vtt(\mu)\dd \mu
\dd x\dd k\ptf
\end{align*}
La somme sur $\ESC_\Qo^\tG$ ne fait intervenir que des $H$ tels que $\lVert  H^{\tG} -T_\Qo\rVert \leq c$.
\end{proposition}

On a aussi l'analogue de \cite[12.5.1]{LW}:
\begin{proposition}\label{[LW,12.5.1]}
Pour $H\in\ESC_\Qo^\tG$, on consid\`ere l'expression
\begin{equation*}\int_{\bsX_{L_0}^\tG(H)\times\bsK}\int_{\bsmu_S} 
\lvert \bs\Lambda^{T\Lbra H^Q\Rbra,Q_0}E^Q_\Qo(xk,\Phi,\theta_0(\mu))
\overline{E^{Q'}_\Qo(xk,\Psi,\mu)}\vtt(\mu)\rvert  \dd \mu
\dd x\dd k\ptf \leqno{(5)}\end{equation*}
On suppose que $T$ est dans le c™ne introduit en \ref{convergence d'une int\'egrale it\'er\'ee}
et que $\bsd_0(T)$ est assez grand. On a les assertions suivantes:
\begin{enumerate}[(i)]
\item Si $\phi_\Qo^Q(H-T)=1$, l'expression (5) est convergente.
\item Il existe $\eta_0$ avec $0<\eta_0 <1$ tel que si $0<\eta<\eta_0$, alors il existe $c>0$ 
tel que pour tout {$T\in (\ag_0^G)^{\theta_0}$ et tout $H\in \ESC_\Qo^\tG$} v\'erifiant
\begin{equation*}\wt{\sigma}_Q^R(H-T)\phi_\Qo^Q(H-T)\Kappa{\eta T}(H^Q- T_\Qo^Q)=1\end{equation*}
l'expression (5) soit major\'ee par\enspace
$c\mskip 2mu  e^{\langle 2\rho_\Qo,H\rangle}\bsd_0(T)^{\dim(\ag_0^{Q_0})}$.
\end{enumerate}
\end{proposition}

\begin{proof}On reprend celle de \textit{loc.~cit}. On veut majorer l'int\'egrale int\'erieure dans (5). 
Pour cela on peut supposer $\vtt \equiv 1$ (cf. la remarque \ref{somfi}).
Comme dans la preuve de \cite[12.2.3]{LW}, on se 
ram\`ene gr‰ce \`a l'in\'egalit\'e de Schwartz \`a majorer deux types d'int\'egrales:
\begin{equation*}\int_{\bsmu_S} 
\bs\Lambda^{T\Lbra H^Q\Rbra,Q_0}E^Q_\Qo(xk,\Phi,\theta_0(\mu))\overline{\bs\Lambda^{T\Lbra H^Q\Rbra,Q_0}
E^Q_\Qo(xk,\Phi,\theta_0(\mu))}\dd \mu\leqno{(6)}\end{equation*}
et
\begin{equation*}\int_{\bsmu_S} 
E^{Q'}_\Qo(xk,\Psi,\mu)\overline{E^{Q'}_\Qo(xk,\Psi,\mu)} \dd \mu\ptf \leqno{(7)}\end{equation*}
Commenons par majorer l'int\'egrale (7). 
Rappelons que l'on a fix\'e en \ref{estimations} un domaine de Siegel 
$\Siegel_{Q'}^G=\Omega_{Q'}\BB _{Q'}^G\Siegel^{L'\mskip -2mu ,*}\bsK$ 
pour le quotient $\BB_G Q'(F)\backslash G(\adef)$. 
Soit $h^{Q'}$ la fonction sur $Q'(F)\backslash G(\adef)/\bsK$ d\'efinie par
\begin{equation*}h^{Q'}(y)=\sum_{\delta\in Q'(F)}\bs{\delta}_{P_0}(\delta y)^{1/2} {\mathbf 1}_{\Siegel_{Q'}^G}(\delta y)\ptf\end{equation*}
On a
\begin{equation*}h^{Q'}(y)\ll\bs{\delta}_{P_0}(y)^{1/2}\ll h^{Q'}(y)\quad \hbox{pour} \quad y\in\Siegel_{Q'}^G\ptf\end{equation*}
D'apr\`es la proposition \ref{[LW,12.2.3]}, pour $b\in \BB_G$ et 
$y,\mskip 2mu y'\in Q'(F)\Siegel_{Q'}^G$, l'int\'egrale
\begin{equation*}\int_{\bsmu_S} E^{Q'}(by,\Psi,\mu)\overline{E^{Q'}(by',\Psi,\mu)} \dd \mu=
\int_{\bsmu_S} E^{Q'}(y,\Psi,\mu)\overline{E^{Q'}(y',\Psi,\mu)} \dd \mu\leqno{(8)}\end{equation*}
est essentiellement major\'ee par $h^{Q'}(y)h^{Q'}(y')$. Ensuite on prend le terme constant 
en chacune des variables $y,\mskip 2mu y'$. Cette op\'eration, 
qui consiste \`a int\'egrer sur un compact, commute \`a l'int\'egrale sur $\bsmu_S$. 
Puis on prend $y=y'= a^G s k$ avec $k\in \bs{K}$, $s\in \Siegel^{L_0,*}$,
$x= as \in L_0(\adef;H')$ pour un rel\`evement $H'\in \ESA_\Qo$ de $H\in \ESC_\Qo^\tG$ et
$a=a_G a^G \in \BB_G \BB_\Qo^G$ 
(on peut mme supposer $a^G\in \BB_\Qo^G \cap \BB'^G_\Qo$ 
comme dans la d\'emonstration de \ref{[LW,12.3.2]}).
L'int\'egrale (7) est donc essentiellement major\'ee par
\begin{equation*}h^{Q'}_\Qo(y)^2=h^{Q'}_\Qo(a^G s )^2\ptf\end{equation*}
Comme la fonction $h^{Q'}$ est \`a croissance lente, son terme constant $h^{Q'}_\Qo$ l'est aussi. 
L'int\'egrale (7) est donc essentiellement major\'ee par 
$\vert a^Gs\vert^D$ pour $D>0$ assez grand.
L'int\'egrale (6) se d\'eduit elle aussi de (8) en prenant les termes constants 
le long de $\Qo$ puis en appliquant 
l'op\'erateur $\bs\Lambda^{T\Lbra H^Q\Rbra,Q_0}$, et enfin en posant $y=y'= a^Gsk$. 
Quand on prend les termes constants, on obtient 
une expression essentiellement major\'ee par \begin{equation*}h_\Qo^Q(y)h_\Qo^Q(y')\ptf\end{equation*}
Rappelons que l'hypoth\`ese \hbox{$\phi_\Qo^Q(H-T)=1$}
assure que l'\'el\'ement $T\Lbra H^Q\Rbra\in\ag_{P_0}^{Q_0}$ est r\'egulier. 
D'apr\`es la proposition \ref{supportcompact}, il existe un sous-ensemble compact 
$\Omega$ de $\Siegel^{L_0,*}$ tel que si $\bs\Lambda^{T\Lbra H^Q\Rbra,Q_0}h_\Qo^Q(a^Gsk)\neq 0$, 
alors $s\in\Omega$. Comme 
\begin{equation*}\bfH_\Qo(x)^G = \bfH_\Qo(a^G)+\bfH_\Qo(s)^G= H^G\end{equation*}
et que $H$ est fix\'e, $a^G$ reste dans un ensemble fini de $\BB_\Qo^G$. D'o le point (i).
Prouvons (ii). On suppose que l'hypoth\`ese
\begin{equation*}\wt{\sigma}_Q^R(H-T)\phi_\Qo^Q(H-T)\Kappa{\eta T}(H^Q-T_\Qo^Q)=1\leqno{(9)}\end{equation*}
est v\'erifi\'ee\footnote{Si $\eta$ est assez petit, l'hypoth\`ese (9) implique $\tau_\Qo^P(H)=1$
o $\tP$ est l'unique \'el\'ement de $\wt\ESP_\st$ v\'erifiant la double inclusion $Q'\subset P \subset R$ (cf. \cite[p.~171]{LW}). 
En particulier cela entra"ne $\tau_\Qo^{Q'}(H)=1$ et $ \tau_\Qo^Q(H)=1$.}. 
On peut prendre $x$ dans un domaine de Siegel $\Siegel^{L_0}=\BB_{L_0}\Siegel^{L_0,*}$. 
On \'ecrit $x=as$ avec $a \in \BB_\Qo$ et $s\in \Siegel^{L_0,*}$. Si $\eta$ est assez petit, l'hypoth\`ese (9) 
implique comme dans la preuve de \cite[12.5.1]{LW} (majoration (4) page 170) 
qu'il existe une constante $D>0$ telle que
\begin{equation*}h_\Qo^Q(a^Gs)\ll \bs{\delta}_\Qo(a)^{\frac{1}{2}} \vert s \vert^D\ptf \end{equation*}
Mais cette majoration est inutile ici, on peut directement passer \`a la page 172. Notons $\bfC$ l'op\'erateur qui multiplie 
une fonction sur $\bs{Y}_\Qo$ par la fonction
\begin{equation*}x \mapsto F_{P_0}^{Q_0}(x,T\Lbra H^Q\Rbra)\end{equation*}
et d\'ecomposons l'op\'erateur $\bs{\Lambda}=\bs\Lambda^{T\Lbra H^Q\Rbra,Q_0}$ en
\begin{equation*}(\bs\Lambda - \bfC)+\bfC\ptf\end{equation*}
Rappelons que $T\Lbra H^Q\Rbra$ est {\og plus r\'egulier\fg} que $T$. D'apr\`es \ref{propK}, si $T$ est assez 
r\'egulier, on peut remplacer l'op\'erateur $\bs\Lambda^{T\Lbra H^Q\Rbra,Q_0}$ par $\bfC$ dans l'int\'egrale int\'erieure 
de l'expression (5). Il nous faut donc majorer l'int\'egrale
\begin{equation*}I_{\bfC}(xk)=\int_{\bsmu_S} 
F_{P_0}^{Q_0}(ask,T\Lbra H^Q\Rbra) E^Q_\Qo(xk,\Phi,\theta_0(\mu))\overline{E^Q_\Qo(xk,\Phi,\theta_0(\mu))}\dd \mu\end{equation*}
sous les hypoth\`eses (9) et
\begin{equation*}F_{P_0}^{Q_0}(xk,T\Lbra H^Q\Rbra)= F_{P_0}^{Q_0}(s,T\Lbra H^Q\Rbra)=1\ptf\leqno{(10)}\end{equation*}
Cela conduit \`a majorer (7) et l'analogue de (7) o $Q$ remplace $Q'$. Sous (9) et (10) on obtient comme dans 
la preuve de \cite[12.5.1]{LW} que $h^{Q'}(a^Gs)$ est essentiellement major\'e par $\bs{\delta}_{P_0}(x)^{1/2}$. 
Donc (7) est essentiellement major\'e par $\bs{\delta}_{P_0}(x)$. Il en est de mme de l'analogue de (7) relatif \`a $Q$, et 
donc aussi de $I_{\bfC}(xk)$. On obtient que l'expression (5) est essentiellement major\'ee 
par
\begin{equation*}\int_{\bsX_{L_0}(H)} F_{P_0}^{Q_0}(x,T\Lbra H^Q\Rbra)\bs{\delta}_{P_0}(x)\dd x\end{equation*}
ou ce qui revient au mme par
\begin{equation*}e^{\langle 2 \rho_\Qo,H \rangle}\int_{\Siegel^{L_0,*}}F_{P_0}^{Q_0}(s,T\Lbra H^Q\Rbra)\bs{\delta}_{P_0}(s)\dd s\ptf\end{equation*}
Rappelons que
\begin{equation*}\Siegel^{L_0,*}=\EE_{L_0} \Siegel^{L_0,1}=\EE_{L_0}\Omega_{L_0}\BB_0^{L_0}(t) \FF_{L_0}\bsK_{L_0}\end{equation*}
 avec $\BB_0^{L_0}(t)= \BB_0(t)\cap L_0(\adef)^1$. On d\'ecompose $s$ en $s= vbk$ avec 
$v\in \EE_{L_0}\Omega_{L_0}$, $b\in \BB_0^{L_0}(t)$ et $k\in \FF_{L_0}\bsK_{L_0}$. 
La d\'ecomposition des mesures 
introduit un facteur $\bs{\delta}_{P_0\cap L_0}^{L_0}(b)^{-1} = \bs{\delta}_{P_0}(b)^{-1}$, et l'int\'egrale sur $\Siegel^{L_0,*}$ 
est essentiellement major\'ee par le cardinal de l'ensemble
\begin{equation*}\{b\in \BB_0^{L_0}(t): F_{P_0}^{Q_0}(b,T\Lbra H^Q\Rbra)=1\}\ptf\end{equation*}
En \'ecrivant $Y= \bfH_0(b)\in \ag_0^{L_0}$, on conclut comme \`a la fin de la d\'emonstration de \cite[12.5.1]{LW}. 
\end{proof}


 \section{Retour \`a la formule de d\'epart}

On suppose d\'esormais 
que $0< \eta < \eta_0$ o $\eta_0$ v\'erifie les conditions de la proposition \ref{[LW,12.5.1]}. 

L'expression pour $\Jres^{G,T}$ 
de la proposition de \ref{cot\'especb} est une combinaison lin\'eaire finie d'int\'egrales it\'er\'ees 
\ref{d\'epart}~(1) (ou, ce qui revient au mme, d'expressions 
 \ref{d\'epart}~(2)) dont la convergence est assur\'ee par 
 la proposition \ref{[LW,12.3.2]}. 

\begin{proposition}\label{defat} 
Il existe une constante absolue $c'>0$ et une constante $c(f)>0$
 telles que, si $\bsd_0(T)\ge c'(c(f)+1)$, on a
\begin{align*}\lefteqn{\hspace{-1cm}
\Jres^{\tG,T}=\sum_{\substack{Q,R\in\ESP_\st\\ Q\subset R}}\wt{\eta}(Q,R)
\sum_{S\in\ESP_\st^{Q'}}\frac{1}{n^{Q'}(S)}\sum_{\bsigma\in\bsPi_\disc(M_S)}\cMSsig}\\
&&\hspace{3,2cm}
\times\sum_{\Psi\in\Base_S(\sigma)}\sum_{H\in\ESC_\Qo^\tG}
A^T(H;\Psi)
\end{align*}
avec
\begin{align*}
\lefteqn{A^T(H;\Psi)=\Kappa{\eta T}(H^Q- T_\Qo^Q)\wt{\sigma}_Q^R(H-T)\phi_\Qo^Q(H-T)}\\
&&\hspace{1cm}\times e^{-\langle 2\rho_\Qo,H\rangle}\int_{\bsX_{L_0}^\tG(H)\times\bsK}
\int_{\bsmu_S}\bs\Lambda^{T\Lbra H^Q\Rbra,Q_0}E^Q_\Qo(xk,\tRho_{S,\sigma,\mu}(f,\omega)\Psi,\theta_0(\mu))\\
&& \hspace{7.6cm}\times\overline{E^{Q'}_\Qo(xk,\Psi,\mu)}\dd \mu
\dd x\dd k\ptf
\end{align*}
\end{proposition}
\begin{proof} Ceci r\'esulte de la proposition \ref{[LW,12.4.1]} 
et de la remarque \ref{somfi}.\end{proof}

\begin{definition}\label{defatb}
Pour $\Psi\in \Automd(\bsX_S,\sigma)$ et $\Phi\in 
\Automd(\bsX_{\theta_0(S)},\theta_0(\sigma))$, et pour $\vtt$ une fonction lisse sur $\bsmu_S$, on pose 
\begin{align*}
\lefteqn{A^T(H;\Phi,\Psi;\vtt)=\Kappa{\eta T}(H^Q- T_\Qo^Q)\wt{\sigma}_Q^R(H-T)\phi_\Qo^Q(H-T)}\\
&&\hspace{1cm}\times e^{-\langle 2\rho_\Qo,H\rangle}\int_{\bsX_{L_0}^\tG(H)\times\bsK}
\int_{\bsmu_S}\bs\Lambda^{T\Lbra H^Q\Rbra,Q_0}E^Q_\Qo(xk,\Phi,\theta_0(\mu))\\
&& \hspace{7cm}\times\overline{E^{Q'}_\Qo(xk,\Psi,\mu)}\vtt(\mu)\dd \mu
\dd x\dd k\ptf
\end{align*}
On \'ecrira simplement $A^T(H)$ pour $A^T(H;\Phi,\Psi;\vtt)$ lorsque les fonctions $\Phi$, $\Psi$ et $\varphi$ sont fix\'ees, 
et on pose
\begin{equation*}A^T=\sum_{H\in\ESC_\Qo^\tG}A^T(H)\ptf\end{equation*}
\end{definition}


 \chapter{Simplification du produit scalaire}
\label{simplification} 


 \section{Une majoration uniforme}

Pour $P\in \ESP$, 
choisissons une section du morphisme $A_P(\adef) \rightarrow \ESB_P$ et notons $\BB_P$ son image. 
Cela permet de relever $\Xi(P)^1$ dans $\Xi(P)$, d'o une identification
\begin{equation*}\Xi(P) = \Xi(P)^1\times \wh\ESB_P \ptf\end{equation*}
Le groupe des caract\`eres 
automorphes, mais non n\'ecessairement unitaires, de $A_P(\adef)$ s'identifie \`a
 $\Xi(P)\times \ag_P^*$.
Un tel caract\`ere $\xi$ peut donc s'\'ecrire 
\begin{equation*}\xi = \xi_\mathrm{u} \vert \xi \vert = (\zeta \star \mu) \star \nu = \zeta \star (\mu+\nu)\end{equation*}
avec $\zeta = \xi\vert_{A_P(\adef)^1}$,
$\mu\in\wh\ESB_P$ et $\nu\in \ag_P^*$. On le notera $\xi=(\zeta,\mu,\nu)$.

Soit $\phi$ une forme automorphe sur $\bs{X}_G$ (discr\`ete ou non). Pour $P\in \ESP_\st$, 
on note $\phi_{P\mathrm{,\mskip 2mu  cusp}}$ le terme constant cuspidal de 
$\phi$ le long de $P$, d\'efini en \cite[I.3.4, I.3.5]{MW1}. C'est une forme automorphe cuspidale sur $\bsX_P$, 
qui s'\'ecrit sous la forme
\begin{equation*}\phi_{P\mathrm{,\mskip 2mu  cusp}}= \sum_{(q,\xi)}q(\bfH_P(x)) \phi_{q,\xi}(x)\leqno{(1)}
\end{equation*}
o $(q,\xi)$ parcourt un sous-ensemble fini de
$\CM[\ag_P] \times \Xi(P)\times \ag_P^* $
et $\phi_{q,\xi}$ est une forme automorphe cuspidale sur $\bsX_P$ se transformant suivant $\xi$. En 
\'ecrivant $\xi=(\zeta,\mu,\nu)$ comme ci-dessus, on voit que
la fonction
\begin{equation*}x \mapsto e^{- \langle \mu+\nu, \bfH_P(x) \rangle} \phi_{q,\xi}(x)\end{equation*}
appartient \`a $\Autom_\disc(\bsX_P)_\zeta$. Notons $\Autom_\cusp(\bsX_P)_{\Xi(P)^1}$ 
le sous-espace de $\Autom_\cusp(\bsX_P)$ engendr\'e par les espaces 
$\Autom_\cusp(\bsX_P)_\xi$ pour $\xi\in \Xi(P)^1$ identifi\'e au sous-groupe de $\Xi(P)$
des caract\`eres triviaux sur $\BB_P$. Il se d\'ecompose en 
\begin{equation*}\Autom_\cusp(\bsX_P)_{\Xi(P)^1}= \bigoplus_{\zeta \in \Xi(P)^1} 
\Autom_\cusp(\bsX_P)_\zeta\ptf\end{equation*}
Pour tout $P\in \ESP_\st$, fixons un sous-ensemble compact $\Gamma_P\subset \ag_{P,\CM}^*/\ESA_P^\vee$, 
deux entiers naturels $\bs{n}_P$ et $d_P$, et un sous-espace de dimension finie $V_P$ de 
$\Autom_\cusp(\bsX_P)_{\Xi(P)^1} $. On note
\begin{equation*}A((V_P,d_P,\Gamma_P,\bs{n}_P)_{P\in \ESP_\st})\end{equation*}
l'ensemble des formes automorphes $\phi$ sur $\bsX_G$ telles que pour tout $P\in \ESP_\st$, 
le terme constant cuspidal $\phi_{P\mathrm{,\mskip 2mu  cusp}}$ puisse s'\'ecrire
\begin{equation*}\phi_{P\mathrm{,\mskip 2mu  cusp}}(x)= \sum_{i=1}^{n_P}e^{\langle \lambda_{P,i}+\rho_P,\bfH_P(x)\rangle} 
\sum_{j=1}^{n_{P,i}}q_{P,i,j}(\bfH_P(x))\phi_{P,i,j}(x)\leqno{(2)}\end{equation*}
o les $n_{P,j}$ sont des entiers positifs ou nuls quelconques, 
$n_P \leq \bs{n}_P$, $\lambda_{P,i}\in \Gamma_P$, $q_{P,i,j}\in \CM[\ag_P]$ avec 
$\deg(q_{P,i,j})\leq d_P$, et $\phi_{P,i,j}\in V_P$. Notons 
que cet ensemble n'est pas un espace vectoriel (\`a cause de la condition $n_P \leq \bs{n}_P$). 
Dans l'expression (2), on peut supposer que les $\lambda_{P,i}$ sont 
deux-\`a-deux distincts. On d\'efinit comme en \cite[13.1]{LW} une norme
\begin{equation*}\lVert  \phi \rVert _\cusp = \sum_{P\in \ESP_\st}\lVert  \phi_{P\mathrm{,\mskip 2mu  cusp}} \rVert _\cusp\ptf\end{equation*}
D'apr\`es \cite[13.1.1]{LW}, on a le
\begin{lemma}
Pour tout $\lambda \in \ag_0^*$, il existe une constante $c>0$ telle que pour tout 
$\phi \in A((V_P,d_P,\Gamma_P,\bs{n}_P)_{P\in \ESP_\st})$ et tout $x\in \Siegel= \BB_G \Siegel^*$, on ait la majoration
\begin{equation*}\vert \phi(x) \vert \leq c \lVert  \phi \rVert _\cusp\sum_{P\in \ESP_\st} \sum_{i=1}^{n_P} 
e^{\langle \lambda^P + \Re(\lambda_{P,i}) +\rho_P , \bfH_0(x)\rangle} (1+ \bfH_P(x))^{d_P}\end{equation*}
o $\lambda^P$ est la projection de $\lambda$ sur $\ag_0^{P,*}$
et les $\lambda_{P,i}$ sont ceux de l'\'egalit\'e (2).
\end{lemma}


 \section{Majoration des termes constants}
\label{majoration des termes constants}

On fixe deux sous-groupe paraboliques standards $Q,\mskip 2mu R\in\ESP_\st$ tels que $Q\subset R$ et $\wt{\eta}(Q,R)=1$. On pose 
$Q'=\theta_0^{-1}(Q)$ et $Q_0= Q \cap Q'$. On fixe aussi 
$S\in\ESP_\st^{Q'}$ et $\sigma\in\Pi_\disc(M_S)$. La repr\'esentation $\sigma$ 
intervient dans le spectre discret de $M_S(F)\backslash M_S(\adef)^1$. Consid\'erons:
\begin{itemize}
\item un sous-groupe parabolique standard $S_\cusp$ tel que $S_\cusp\subset S$;
\item une repr\'esentation automorphe cuspidale $\sigma_\cusp$ de $M_{S_\cusp}(\adef)$ qui est une sous-repr\'esentation 
irr\'eductible de $L^2(\BB_{S_\cusp}\backslash \bsX_{M_{S_\cusp}})$ -- c'est-\`a-dire que 
$\sigma_\cusp$ se r\'ealise 
dans $\Autom_\cusp(\bsX_{{M_{S_\cusp}}})_\zeta$ pour un caract\`ere $\zeta \in \Xi(S_\cusp)^1$;
\item un op\'erateur diff\'erentiel $D$ \`a coefficients polynomiaux sur $\ag_{S_\cusp,\CM}^{S,*}$;
\item un point $\nu_0 \in \ag_{S_\cusp}^{S,*}$.
\end{itemize}
{Rappelons que $\Base_{{S_\cusp \cap M_S}}(\sigma_\cusp)$ est une base orthonormale de l'espace vectoriel 
prŽ-hilbertien $\Automd(\bsX_{S_\cusp \cap M_S},\sigma_\cusp)$.}
Pour $\Phi_\cusp\in \Base_{{S_\cusp \cap M_S}}(\sigma_\cusp)$ et $\nu\in \ag_{S_\cusp,\CM}^{S,*}$, 
formons la s\'erie d'Eisenstein
\begin{equation*}E^{M_S}(y,\Phi_\cusp,\nu)= \sum_{\gamma\in (S_\cusp 
\cap M_S)(F)\backslash M_S(F)}\Phi_\cusp(\gamma y,\nu)\vgq y\in M_S(\adef)\ptf\end{equation*}
On applique l'op\'erateur $D$ sous l'hypoth\`ese que la fonction $\nu \mapsto DE^{M_S}(y,\Phi_\cusp,\nu)$ 
est holomorphe en $\nu=\nu_0$ et on note
\begin{equation*}D_{\nu=\nu_0}E^{M_S}(y,\Phi_\cusp,\nu)\end{equation*}
sa valeur en $\nu=\nu_0$.

Comme dans [LW] on voit qu'en
choisissant convenablement la base $\Base_{{S}}(\sigma)$ de $\Automd(\bsX_{S},\sigma)$, 
on peut supposer que pour tout \'el\'ement $\Psi\in \Base_{{S}}(\sigma)$, il existe des donn\'ees $S_\cusp$, 
$\sigma_\cusp$, $D$, $\nu_0$ et 
$\Psi_\cusp\in \Base_{S_\cusp}(\sigma_\cusp)$ telles que
\begin{equation*}E^{Q'}(y,\Psi,\mu)= D_{\nu=\nu_0}E^{Q'}(y,\Psi_\cusp,\nu+\mu)\end{equation*}
pour tout $\mu \in \ag_{S,\CM}^{*}/\ESA_S^\vee$. Prendre un terme constant et 
prendre un r\'esidu sont deux op\'erations qui commutent. Gr‰ce \`a \cite[5.2.2.(4)]{LW}, on obtient
\begin{equation*}E^{Q'}_\Qo(y,\Psi,\mu)= D_{\nu=\nu_0} \bigg(\sum_{s\in \bfW^{Q'}(\ag_{S_\cusp}, Q_0)}
\hspace{-0.5cm}E^{Q_0}(y,\bfM(s,\nu+\mu)\Psi_\cusp,s(\nu+\mu)) \bigg)\ptf\leqno{(1)}\end{equation*}
Le terme constant $E^{Q'}_{S'_\cusp}(y,\Psi,\mu)$ 
de la forme automorphe $E^{Q'}_\Qo(y,\Psi,\mu)$ relatif \`a un 
sous-groupe parabolique $S'_\cusp\in \ESP_\st^{Q_0}$ associ\'e \`a $S_\cusp$ dans $Q'$ est \'egal \`a:
\begin{equation*}D_{\nu=\nu_0} \bigg(\sum_{s\in \bfW^{Q'}(\ag_{S_\cusp}, Q_0)} 
\sum_{s'\in \bfW^{Q_0}(\ag_{s(S_\cusp)},\ag_{S'_\cusp})}
\hspace{-0.5cm}
\bfM(s's,\nu+\mu)\Psi_\cusp(y,s's(\mu+\nu)) \bigg)\ptf\end{equation*}
Les \textit{exposants cuspidaux} de $E^{Q'}_{S'_\cusp}(y,\Psi,\mu)$ sont les
\begin{equation*}s's(\nu_0+\mu)\in \ag_{S'_\cusp,\CM}^*/\ESA_{S'_\cusp}^\vee\ptf\end{equation*}
Pour $w\in \bfW^{Q'}(\ag_{S_\cusp},\ag_{S'_\cusp})$
notons $Q'_w$ le plus petit sous-groupe parabolique standard de $Q'$ tel que
$\ag_{Q'_w} \subset w(\ag_S)$.
D'apr\`es \cite[13.2~(5)]{LW} et \cite[V.3.16, VI.1.6~(c)]{MW1}, on sait que
pour $\mu \in \bsmu_S $ les parties r\'eelles des exposants cuspidaux de 
$E^{Q'}_{S'_\cusp}(y,\Psi,\mu)$ sont de la forme $w\nu_0$ pour des 
$w\in \bfW^{Q'}(\ag_{S_\cusp},\ag_{S'_\cusp})$ tels que
\begin{equation*}\wh{\tau}_{S'_\cusp}^{\mskip 2mu Q'_w}(-w\nu_0)=1\ptf\leqno{(2)}\end{equation*}
Ainsi dans l'expression $E^{Q'}_{S'_\cusp}(y,\Psi,\mu)$ pour $\mu\in \bsmu_S$, 
les termes index\'es par les couples $(s,s')$ tels que l'\'el\'ement $w=ss'$ ne v\'erifie pas $(2)$ sont nuls. 
On d\'ecompose $E^{Q'}_\Qo(y,\Psi,\mu)$ en
\begin{equation*}E^{Q'}_\Qo(y,\Psi,\mu) = E^{Q'}_{Q_0\mathrm{,\mskip 2mu  unit}}(y,\Psi,\mu)+E^{Q'}_{Q_0,+}(y,\Psi,\mu)\ptf\end{equation*}
ou le terme $ E^{Q'}_{Q_0\mathrm{,\mskip 2mu  unit}}$ est la sous-somme de $(1)$ index\'ee par les $s$ tels que 
$s(\ag_0^S)\subset \ag_0^{Q_0}$ et le terme $E^{Q'}_{Q_0,+}$ est la sous-somme restante. 
On obtient comme en \cite[13.2~(7)]{LW} que pour $\mu\in \bsmu_S$, on a
\begin{equation*}E^{Q'}_{Q_0\mathrm{,\mskip 2mu  unit}}(y,\Psi,\mu)=
\sum_{s\in \bfW^{Q'}(\ag_S,Q_0)} E^{Q_0}(y,\bfM(s,\mu)\Psi,\mu)\ptf\end{equation*}
Rappelons que $\bfW^{Q'}(\ag_S, Q_0)$ est l'ensemble des restrictions \`a $\ag_S$ des 
$s\in \bfW^{Q'}$ tels que $s(\ag_S)\supset\ag_\Qo$ et 
$s$ est de longueur minimale dans sa classe $\bfW^{Q_0}s$ (ce qui signifie que $s(S)\cap L_0$ est standard 
dans $L_0 =M_\Qo$). 
D'apr\`es \cite[13.2~(6)]{LW}, la fonction $ E^{Q'}_{Q_0\mathrm{,\mskip 2mu  unit}}(y,\Psi,\mu)$ est lisse pour $\mu\in \bsmu_S$, 
il en est donc de mme pour la fonction \begin{equation*}
E^{Q'}_{Q_0,+}(y,\Psi,\mu) = E^{Q'}_\Qo(y,\Psi,\mu) - E^{Q'}_{Q_0\mathrm{,\mskip 2mu  unit}}(y,\Psi,\mu)\ptf\end{equation*}

\begin{proposition}\label{unita}
Soient $Z\in \ESA_G$ et $T_1\in {\ag_{0}^{Q_0}}$.
\begin{enumerate}[(i)]
\item Il existe un entier $N\in \NM$ et un r\'eel $c>0$ tels que
\begin{equation*}\vert E^{Q'}_\Qo(xk,\Psi,\mu) \vert \leq c\mskip 2mu \bs{\delta}_{P_0}(a)^{\frac{1}{2}}(1+ \lVert  H \rVert )^N \vert s \vert^N\end{equation*}
pour tout $H\in \ESA_\Qo^G(Z)$ tel que $\tau_\Qo^{Q'}(H)=1$, 
tout $(a,s)\in \BB_{L_0}\times \Siegel^{L_0,*}$ 
tel que $x=as \in L_0(\adef;H)$, tout $k\in \bs{K}$ et tout $\mu\in \bsmu_S$.
\item Il existe un entier $N\in \NM$ et un r\'eel $c>0$ tels que
\begin{equation*}\vert E^{Q'}_\Qo(xk,\Psi,\mu)\vert \leq c\mskip 2mu \bs{\delta}_{P_0}(x)^{\frac{1}{2}}(1+ \lVert  X_\Qo \rVert )^N(1+ \lVert  \bfH_0(s)\rVert )^N\end{equation*}
pour tout $X\in \ESA_{P_0}^G(Z)$ tel que $\tau_{P_0}^{Q'}(X+ T_1)=1$, 
tout $(a,s)\in \BB_{L_0}\times \Siegel^{L_0,*}$ 
tel que $\bfH_0(x)=X$ avec $x=as$, tout $k\in \bs{K}$ et tout $\mu\in \bsmu_S$.
\item Il existe un entier $N\in \NM$ et un r\'eel $c>0$ tels que
\begin{equation*}\vert E^{Q'}_{Q_0\mathrm{,\mskip 2mu  unit}}(xk,\Psi,\mu) \vert \leq c\mskip 2mu \bs{\delta}_{P_0}(x)^{\frac{1}{2}}(1+ \lVert  H \rVert )^N(1+ \lVert  \bfH_0(s)\rVert )^N\end{equation*}
pour tout $H\in \ESA_\Qo^G(Z)$, tout $(a,s)\in \BB_{L_0}\times \Siegel^{L_0,*}$ tel $x=as\in L_0(\adef;H)$, 
tout $k\in \bsK$ et tout $\mu\in \bsmu_S$.
\item Il existe un r\'eel $R>0$, un entier $N>0$ et un r\'eel $c>0$ tels que
\begin{align*}
\lefteqn{\vert E^{Q'}_{Q_0,+}(xk,\Psi,\mu) \vert }\\
&& \leq c\mskip 2mu \bs{\delta}_{P_0}(x)^{\frac{1}{2}}(1+ \lVert  X_\Qo \rVert )^N(1+ \lVert  \bfH_0(s)\rVert )^N
\sup_{\alpha \in \Delta_0^{Q'} \smallsetminus \Delta_0^{Q_0}}\hspace{-0.3cm} 
e^{-R \langle \alpha , \bfH_0(x) \rangle}
\end{align*}
pour tout $X\in \ESA_{P_0}^G(Z)$ tel que $\tau_{P_0}^{Q'}(X+ T_1)=1$, 
tout $(a,s)\in \BB_{L_0}\times \Siegel^{L_0,*}$ 
tel que $\bfH_0(x)=X$ avec $x=as$, tout $k\in \bs{K}$ et tout $\mu\in \bsmu_S$.
\end{enumerate}
\end{proposition}
\begin{proof}
On suit pas \`a pas celle de la proposition \cite[13.2.1]{LW}.
\end{proof}

 \section{Simplication du terme constant}\label{simplication du terme constant}

On a introduit en \ref{defatb} des expressions $A^T(H)$ et $A^T$. On note
\begin{equation*}A^T_{\mathrm{unit}}=\sum_{H\in\ESC_\Qo^\tG}A_{\mathrm{unit}}^T(H)\end{equation*}
les expressions obtenues en rempla\c{c}ant les fonctions $E^{Q'}_\Qo$ par 
$E^{Q'}_{Q_0\mathrm{,\mskip 2mu  unit}}$ et $E^Q_\Qo$ par $E^Q_{Q_0\mathrm{,\mskip 2mu  unit}}$ dans la d\'efinition de $A^T(H)$. 
Alors \cite[13.3.1]{LW} est vrai ici:

\begin{proposition}\label{unitb}
L'int\'egrale d\'efinissant $A_\mathrm{unit}^T(H)$ et la somme d\'efinissant $A_\mathrm{unit}^T$ sont absolument convergentes, 
et pour tout r\'eel $r$, il existe $c>0$ tel que 
\begin{equation*}\vert A^T - A^T_\mathrm{unit}\vert \leq c\mskip 2mu \bsd_0(T)^{-r}\ptf\end{equation*}
\end{proposition}

\begin{proof}
Elle est identique \`a celle de \textit{loc.~cit}, \`a la simplification suivante pr\`es: la d\'ecomposition de 
l'op\'erateur $\bs{\Lambda}= \bs\Lambda^{T\Lbra H^Q\Rbra,Q_0}$ en $(\bs{\Lambda} - \bfC)+\bfC$ conduit \`a 
la d\'ecomposition des expressions 
$A^T(H)$ et $A^T_\mathrm{unit}(H)$ en
\begin{equation*}A^T(H)= A^T_{\bs{\Lambda}-\bfC}(H)+A^T_{\bfC}(H) \quad \hbox{et}\quad 
A^T_\mathrm{unit}(H)= A^T_{\bs{\Lambda}-\bfC\mathrm{,\mskip 2mu  unit}}(H)+A^T_{\bfC\mathrm{,\mskip 2mu  unit}}(H)\ptf\end{equation*}
Comme dans la preuve de la proposition \ref{[LW,12.5.1]}, si $T$ est assez r\'egulier, on a
\begin{equation*}A^T(H)= A^T_{\bfC}(H)\quad \hbox{et}\quad A^T_\mathrm{unit}(H)= A^T_{\bfC\mathrm{,\mskip 2mu  unit}}(H)\ptf\end{equation*}
Seules les expressions $A^T_{\bfC} (H)$ et $A^T_{\bfC\mathrm{,\mskip 2mu  unit}}(H)$ sont \`a comparer. 
Les assertions sont alors cons\'equence de \ref{unita}.
\end{proof}


 \section{Simplification du produit scalaire}
\label{simplification du produit scalaire}
On a d\'efini en \ref{defbrT} un \'el\'ement $T\Lbra H^Q\Rbra$ dans $\ag_{P_0}^{Q_0}$. 
Pour $S\in\ESP_\st^{Q'}$ et 
$H\in\ESA_\Qo$, consid\'erons l'op\'erateur (introduit en \ref{GMspec} mais avec ici $Q_0$ en place de $G$
et $T\Lbra H^Q\Rbra$ au lieu de $T$)
\begin{equation*}\bsO_{S\vert \theta_0(S)}^{T,Q_0}(H;\lambda ,\mu)
=\mskip -6mu\sum_{S'\in\ESP_\st^{Q_0}}
\sum_{\substack{s\in \bfW^Q(\ag_{\theta_0(S)},\ag_{S'})\\ t\in \bfW^{Q'}(\ag_S,\ag_{S'})}}\mskip -17mu
\varepsilon_{S'}^{\mskip 2mu Q_0, T\Lbra H^Q\Rbra_{S'}}(H; s\lambda - t\mu)
\bfM(t,\mu)^{-1}\bfM(s,\lambda)\ptf\end{equation*}
On a fix\'e des fonctions 
\begin{equation*}\Psi\in \Automd(\bsX_S,\sigma)\quad\hbox{et}\quad 
\Phi\in \Automd(\bsX_{\theta_0(S)},\theta_0(\omega\otimes \sigma))\vg\end{equation*}
et une fonction lisse $\vtt$ sur $\bsmu_S$.
Rappelons que l'on a introduit en \ref{Dnu} un op\'erateur de d\'ecalage $\bfD_\nu$.
Pour $\mu,\mskip 2mu \nu \in \bsmu_S$ et $\lambda\in\bsmu_{\theta_0(S)}$, on pose
\begin{equation*}\bso^{\mskip 2mu T,Q_0}(H;\lambda,\mu;\nu)\bydef
\big\langle\bfD_\nu\bsO_{S,\theta_0(S)}^{\mskip 2mu T,Q_0}(H;\lambda,\mu)\Phi,\Psi\big\rangle_S\end{equation*}
 c'est-\`a-dire
\begin{align*}
\lefteqn{\hspace{0.5cm}
\bso^{\mskip 2mu T,Q_0}(H;\lambda,\mu;\nu)=\sum_{S'\in\ESP_\st^{Q_0}}
\sum_{\substack{s\in \bfW^Q(\ag_{\theta_0(S)},\ag_{S'})\\ t\in \bfW^{Q'}(\ag_S,\ag_{S'})}}
\varepsilon_{S'}^{\mskip 2mu Q_0, T\Lbra H^Q\Rbra_{S'}}(H; s\lambda - t\mu)}
\\
&&\hspace{6.5cm}\times\big\langle\bfD_{\nu} \bfM(t,\mu)\mun\bfM(s,\lambda)\Phi, \Psi \big\rangle_{S'}\ptf
\end{align*}
Avec les notations de la proposition \ref{d\'ecalage}, pour tout 
\hbox{$\tu\in \bfW^\tG(\ag_S,\ag_S)$}, on a 
\begin{equation*}\bs{\ESE}(\theta_0(\omega_{A_S}\xi),\xi)=
\{\nu \in \bsmu_S\mskip 2mu \mskip 2mu \mid \mskip 2mu \mskip 2mu \tu(\omega_{A_S} \xi)\star\nu\vert_{\ESB_M} = \xi\}\ptf\end{equation*}
Comme seule la restriction de $\nu$ \`a $\ESB_M$ intervient, et s'il est non vide,
cet ensemble est un espace homog\`ene sous $\wh\bsbbc_M$.

\begin{definition}\label{exisigma}
Lorsque $\xi=\xi_\sigma$ on pose 
\begin{equation*}\bsESEsigma\bydef 
\{\nu\in \bsmu_S\mskip 2mu \mid \mskip 2mu \xi_{\tu(\omega\otimes \sigma)}\star \nu\vert_{\ESB_M}= \xi_\sigma\}
=\bs{\ESE}(\theta_0(\omega_{A_S}\xi_\sigma),\xi_\sigma)\ptf\end{equation*}
\end{definition}

\begin{lemma}
Pour que l'expression $\bso^{\mskip 2mu T,Q_0}(H;\lambda,\mu;\nu)$ soit non nulle, il est n\'ecessaire que 
$\nu$ appartienne \`a $ \bsESEsigma$.
\end{lemma}

\begin{proof}On a
\begin{equation*}\bfD_{\nu} \bfM(t,\mu)\mun\bfM(s,\lambda)\Phi\in\Automd(\bsX_{S},\tu(\omega\otimes \sigma)\star \nu)
\quad\hbox{et}\quad\Psi\in\Automd(\bsX_S,\sigma)\ptf\end{equation*}
 Pour que l'expression $\bso^{\mskip 2mu T,Q_0}(H;\lambda,\mu;\nu)$ soit non nulle, 
il est n\'ecessaire que l'on ait \begin{equation*}\xi_{\tu(\omega\otimes \sigma)}\star \nu\vert_{\ESB_M}=\xi_{\sigma}\ptf\end{equation*}
\end{proof}

\begin{definition}\label{redecal}
On pose
\begin{equation*}\brabso^{\mskip 2mu T,Q_0}(H;\lambda,\mu)=
 \vert \wh\bsbbc_S \vert^{-1}\sum_{\nu \in \bsESEsigma}
 \bso^{\mskip 2mu T,Q_0}(H;\lambda,\mu+\nu;\nu) \end{equation*}
avec $\brabso^{\mskip 2mu T,Q_0}(H;\lambda,\mu)=0$ si $\bsESEsigma =\emptyset$.
\end{definition}
Avec les notations de la proposition \ref{d\'ecalage} on a
\begin{equation*}\brabso^{\mskip 2mu T,Q_0}(H;\lambda,\mu)=
\langle [\bsO]_{S\vert \theta_0(S)}^{T,Q_0}(H, \xi,\xi';\lambda ,\mu) \Phi, \Psi \rangle_S\end{equation*}
mais avec $Q_0$ en place de $G$ et $T\Lbra H^Q\Rbra$ au lieu de $T$.
D'apr\`es \ref{SETR} cette expression est holomorphe en $\lambda$ et $\mu$. 
Pour $\lambda=\theta_0(\mu)$, on \'ecrit
\begin{equation*}\brabso^{\mskip 2mu T,Q_0}(H;\mu)=\brabso^{\mskip 2mu T,Q_0}(H;\theta_0(\mu),\mu)\ptf\end{equation*}
Observons que $\brabso^{\mskip 2mu T,Q_0}(H;\mu)$ ne d\'epend que de l'image de $H$ dans $\ESC_\Qo^\tG= 
\ESB_\tG\backslash \ESA_\Qo$.
On pose 
\begin{equation*}A^T_{\mathrm{pure}}(H)=
\Kappa{\eta T}(H^Q-T_\Qo^Q)\wt{\sigma}_Q^R(H-T)\phi_\Qo^Q(H-T)\int_{\bsmu_S}
\brabso^{\mskip 2mu T,Q_0}(H;\mu) \vtt(\mu)\dd \mu\end{equation*}
et
\begin{equation*}A^T_{\mathrm{pure}}=\sum_{H\in\ESC_\Qo^\tG} A^T_{\mathrm{pure}}(H)\ptf\end{equation*}

\begin{proposition}\label{purea} 
La s\'erie d\'efinissant $A^T_{\mathrm{pure}}$ est convergente et
pour tout r\'eel $r$, on a une majoration
\begin{equation*}\vert A^T_\mathrm{unit} - A^T_\mathrm{pure}\vert \ll e^{-r\bsd_0(T)}\ptf\end{equation*}
\end{proposition}
\begin{proof} 
La preuve suit pas \`a pas les arguments de \cite[13.4.1]{LW}.
Tout d'abord on utilise \cite[2.13.1]{LW} pour prouver la convergence de la s\'erie d\'efinissant $A^T_\mathrm{pure}$.
Puis, gr‰ce au calcul approch\'e du produit scalaire des s\'eries d'Eisenstein tronqu\'ees donn\'e par \ref{PSET}~(ii) 
et compte-tenu de \ref{d\'ecalage} pour le d\'ecalage en $\nu$, 
on montre qu'il existe un r\'eel $c>0$ pour lequel on a la majoration souhait\'ee.
\end{proof}

\begin{corollary}\label{pureb} 
Pour tout r\'eel $r$, on a une majoration
\begin{equation*}\vert A^T- A^T_\mathrm{pure}\vert \ll \bsd_0(T)^{-r}\ptf\end{equation*}
\end{corollary}
\begin{proof}
On invoque de plus \ref{unitb}.
\end{proof}


 \section{D\'ecomposition plus fine}
\label{d\'ecomposition de A_pure}

On va d\'ecomposer la somme sur 
$H\in\ESC_\Qo^\tG$ dans $A^T_\mathrm{pure}$ en une somme sur $\ESC_Q^\tG$ pr\'ec\'ed\'ee d'une somme sur
$\ESA_\Qo^Q$ gr‰ce \`a la suite exacte courte
\begin{equation*}0\to\ESA_\Qo^Q\to\ESC_\Qo^\tG\to\ESC_Q^\tG\to0\ptf\end{equation*}
Consid\'erons $H \in\ESC_\Qo^\tG$, $Z\in \ESC_Q^\tG$ et $Y\in\ag_\Qo^Q$
tels que
\begin{equation*}Z=H_Q\quad \hbox{et}\quad Y= T_\Qo^Q-H^Q\quad
\hbox{et donc}\quad H=Z+T_\Qo^Q-Y\ptf\end{equation*}
Puisque $\Kappa{\eta T}(-Y)=\Kappa{\eta T}(Y)$, on a
\begin{equation*}\Kappa{\eta T}(H^Q-T_\Qo^Q)\wt{\sigma}_Q^R(H-T)\phi_\Qo^Q(H-T)=
\Kappa{\eta T}(Y)\wt{\sigma}_Q^R(Z-T_Q)\phi_\Qo^Q(-Y)\end{equation*}
et
\begin{equation*}T\Lbra H^Q\Rbra=T\Lbra T_\Qo^Q -Y\Rbra \ptf\end{equation*}
Lorsque $\phi_\Qo^Q(-Y)=1$ on a $Y=X_\Qo$ o $X$ est de la forme
\begin{equation*}X=\sum_{\alpha\in\Delta_0^Q\smallsetminus\Delta^{Q_0}_0}x_\alpha\check{\alpha}\com{avec}
x_\alpha\geq 0 \com{pour} \alpha\in \Delta_0^Q\smallsetminus\Delta_0^{Q_0}\ptf\end{equation*}
En d'autres termes, $X$ appartient au c™ne ferm\'e $\mathcal{C}(Q,Q_0)$ de $\ag^Q_0$ 
engendr\'e par les \'el\'ements $\check{\alpha}$ 
pour $\alpha\in\Delta_0^Q\smallsetminus\Delta_0^{Q_0}$. 
D'apr\`es \cite[4.2.1]{LW}, on a
\begin{equation*}T\Lbra H^Q\Rbra=T\Lbra T_\Qo^Q-Y\Rbra = 
T^{Q_0} -\mskip -5mu\sum_{\alpha\in\Delta_0^Q\smallsetminus\Delta_0^{Q_0}} 
x_\alpha\check{\alpha}^{Q_0}=({T-X})^{Q_0}\ptf\end{equation*}
Donc
\begin{equation*}H=\ZTX\quad\hbox{o on a pos\'e}\quad H_Z^U\bydef Z+U_\Qo^Q\ptf\end{equation*}
L'application qui \`a $Y$ associe $X\in\mathcal{C}(Q,Q_0)$ est injective et on note 
\begin{equation*}\mathcal{C}_F(Q,Q_0; T)\subset\mathcal{C}(Q,Q_0)\end{equation*}son image. 
On a ainsi transform\'e la somme sur ${H\in\ESC_\Qo^\tG}$ en une somme sur
\begin{equation*}(Z,X)\in\ESC_Q^\tG\times\mathcal{C}_F(Q,Q_0; T)\end{equation*}
la fonction
\begin{equation*}\Kappa{\eta T}(H^Q-T_\Qo^Q)\wt{\sigma}_Q^R(H-T)\phi_\Qo^Q(H-T)
\quad\hbox{devenant}\quad
\Kappa{\eta T}(X_\Qo)\wt{\sigma}_Q^R(Z-T)\ptf\end{equation*}
Pour $Z\in\ESA_Q$ et $X\in \mathcal{C}_F(Q,Q_0;T)$ on pose 
\begin{align*}
\lefteqn{
\bso^{\mskip 2mu T,Q_0}(Z,X;\lambda,\mu;\nu)\bydef 
\sum_{S'\in\ESP_\st^{Q_0}}
\sum_{s\in \bfW^Q(\ag_{\theta_0(S)},\ag_{S'})}
\sum_{ t\in \bfW^{Q'}(\ag_S,\ag_{S'})}
}\\
&&\hspace{1.5cm}\times\hspace{0,4cm}
\varepsilon_{S'}^{\mskip 2mu Q_0, ({T-X})_{S'}}(\ZTX ; s\lambda - t\mu)
\big\langle\bfM(s,\lambda)\Phi, \bfM(t,\mu)\bfD_{-\nu}\Psi \big\rangle_{S'}\ptf
\end{align*}
On a donc
\begin{equation*}\bso^{T\Lbra H^Q\Rbra,Q_0}(H_Z^{T-X};\lambda ,\mu;\nu)=\bso^{T,Q_0}(Z,X;\lambda,\mu;\nu)\ptf\end{equation*}
En rempla\c{c}ant la variable $t$ 
par $t't$ avec $t\in \bfW^{Q'}(\ag_S,Q_0)$ et $t'\in \bfW^{Q_0}(t(\ag_S),\ag_{S'})$, et 
la variable $s$
par $t'^{-1}s\in \bfW^Q(\theta_0(\ag_S),t(\ag_S))$ cette expression peut s'\'ecrire:
\begin{align*}
\lefteqn{
\sum_{t\in \bfW^{Q'}(\ag_S,Q_0)}
\sum_{s\in \bfW^Q(\theta_0(\ag_S),t(\ag_S))}\sum_{S'\in\ESP_\st^{Q_0}}
\sum_{t'\in \bfW^{Q_0}(t(\ag_S),\ag_{S'})}}
\\
&\hspace{1,2cm}\times\hspace{0,4cm}
\varepsilon_{S'}^{\mskip 2mu Q_0, ({T-X})_{S'}}(\ZTX; t'(s\lambda - t\mu))
\big\langle \bfM(t's,\lambda)\Phi,\bfM(t't,\mu)\bfD_{-\nu}\Psi\big\rangle_{S'}.
\end{align*}
Le sous-groupe parabolique $t(S)$ n'est en g\'en\'eral pas standard mais il existe un unique sous-groupe parabolique standard 
$\ttS\subset Q_0$ tel que $M_{t(S)}= M_\ttS$. On pose $\ttM= M_\ttS$. 
Pour $S'\in\ESP^{Q_0}_\st$ et $t'\in \bfW^{Q_0}(t(\ag_S),a_{S'})$, 
le sous-groupe parabolique $t'^{-1}(S')$ appartient \`a l'ensemble 
$\ESP^{Q_0}(\ttM)$ des 
$S''\in\ESP^{Q_0}$ tels que 
$M_{S''}= \ttM$. On peut remplacer ci-dessus $S'$ par $S''=t'^{-1}(S')$. 
Alors la double somme en $S'$ et $t'$ se transforme en une somme sur 
$S''\in \ESP^{Q_0}(\ttM)$. Pour
\begin{equation*}H' = t'^{-1}(H)\end{equation*}on a
\begin{equation*}\varepsilon_{S'}^{\mskip 2mu Q_0, ({T-X})_{S'}}(H; t'(s\lambda - t\mu))=
\varepsilon_{S''}^{\mskip 2mu Q_0, [{T-X}]_{S''}}(H'; s\lambda - t\mu)\end{equation*}
et
\begin{equation*}\big\langle\bfM(t's,\lambda)\Phi,\bfM(t't,\mu)\bfD_{-\nu}\Psi\big\rangle_{S'}\end{equation*}
\'egale
\begin{equation*}\big\langle \bfM(t',s\lambda)\bfM(s,\lambda)\Phi,\bfM(t',t\mu)\bfM(t,\mu)\bfD_{-\nu}\Psi\big\rangle_{S''}\ptf\end{equation*}
Notons que \begin{equation*}[{T-X}]_{S''}= t'^{-1}(({T-X})_{S'})\qquad\hbox{et donc}\qquad
[{T-X}]_{S''}^{Q_0}= t'^{-1}(({T-X})^{Q_0}_{S'})\end{equation*}
D'apr\`es \cite[4.3.5]{LW}, on a
\begin{equation*}\bfM(t',t\mu)^{-1}\bfM(t',s\lambda)= 
e^{\langle s\lambda - t\mu,\Y_{S''}\rangle}\bfM_{S'\vert S''}(t\mu)^{-1}\bfM_{S'\vert S''}(s\lambda)\end{equation*}
avec $\Y_{S''}= (T_0 - t'^{-1} (T_0))_{S''}$. En d\'efinitive, on obtient
\begin{equation*}\bso^{\mskip 2mu T,Q_0}(Z,X;\lambda,\mu;\nu)=
\sum_{t\in \bfW^{Q'}(\ag_S,Q_0)}\sum_{s\in \bfW^Q(\theta_0(\ag_S),t(\ag_S))}
\bso^{\mskip 2mu T,Q_0}_{ s,t}(Z,X;\lambda,\mu;\nu)\end{equation*}
avec
\begin{align*}
\lefteqn{\bso^{\mskip 2mu T,Q_0}_{ s,t}(Z,X;\lambda,\mu;\nu)=}\\
&&\sum_{S''\in\ESP^{Q_0}(\ttM)}
\varepsilon_{S''}^{\mskip 2mu Q_0, [{T-X}]_{S''}}(\ZTX ; s\lambda - t\mu)
e^{\langle s\lambda - t\mu, \Y_{S''}\rangle}\\
&&\hspace{2cm}\times\big\langle\bfM_{S''\vert \ttS}(s\lambda)
\bfM(s,\lambda)\Phi, \bfM_{S''\vert \ttS}(t\mu)\bfM(t,\mu)\bfD_{-\nu}\Psi\big\rangle_{S''}.
\end{align*}
On doit int\'egrer en $\mu$ la fonction
\begin{equation*}\bso^{\mskip 2mu T,Q_0}(Z,X;\mu;\nu)\bydef \bso^{\mskip 2mu T,Q_0}(Z,X;\theta_0(\mu),\mu+\nu;\nu)\end{equation*}
puis sommer en $Z$ et $X$. Chaque expression $\bso^{\mskip 2mu T,Q_0}_{ s,t}(Z,X;\lambda,\mu;\nu)$ 
est encore une fonction lisse de $\lambda$ et $\mu$, et l'on pose 
\begin{equation*}\bso^{\mskip 2mu T,Q_0}_{ s,t}(Z,X;\mu;\nu)\bydef \bso^{\mskip 2mu T,Q_0}_{ s,t}(Z,X;\theta_0(\mu),\mu+\nu;\nu)\ptf\end{equation*}
L'expression $\bso^{\mskip 2mu T,Q_0}_{ s,t}(Z,X;\mu;\nu)$ ne d\'epend que de l'image de $Z$ 
dans $\ESC_Q^\tG = \ESB_\tG \backslash \ESA_Q$. 

Ces manipulations permettent d'\'ecrire, au moins formellement, $A^T_\mathrm{pure}$
comme une somme index\'ee par des \'el\'ements $s$ et $t$ dans des ensembles de Weyl:
\begin{equation*}A^T_\mathrm{pure}=\mskip -5mu\sum_{t\in \bfW^{Q'}(a_S, Q_0)}
\sum_{s\in \bfW^Q(\theta_0(\ag_S), t(\ag_S))}\mskip -5mu A^T_{s,t}
\qquad\hbox{o}\qquad
A^T_{s,t}= \vert \wh\bsbbc_S \vert^{-1}\mskip -2mu \mskip -2mu \sum_{\nu \in \bsESEsigma}A^T_{s,t,\nu}\end{equation*}
avec 
\begin{equation*}A^T_{s,t,\nu} =\sum_{Z\in\ESC_Q^\tG}\wt{\sigma}_Q^R(Z{-T})\bigg(
\sum_{X\in\mathcal{C}_F(Q,Q_0; T)}\Kappa{\eta T}(X_\Qo)\int_{\bsmu_S}
\bso^{\mskip 2mu T,Q_0}_{ s, t} (Z, X;\mu;\nu)\vtt(\mu)\dd \mu\bigg)\mskip -2mu .\end{equation*}
On peut montrer, en reprenant des arguments de \cite[13.4.1]{LW},
d\'ej\`a utilis\'es pour la preuve de \ref{purea}, que l'expression converge (dans l'ordre indiqu\'e).
Une autre preuve de la convergence de la s\'erie en $Z$ r\'esultera de \ref{major}~(A).

Nous aurons besoin d'une variante 
de l'expression $\bso^{\mskip 2mu T,Q_0}_{ s, t} (Z, X;\lambda,\mu;\nu)$
o
la sommation porte sur $\ESP^Q(\ttM)$, sans variable $X$ et o le sous-groupe parabolique $Q_0$ 
est remplac\'e par $Q$. On pose pour $Z\in\ESA_Q$:
\begin{align*}
\lefteqn{\bso_{s,t}^{\mskip 2mu T,Q}(Z;\lambda,\mu;\nu) =\sum_{S''\in\ESP^Q(\ttM)}
\varepsilon_{S''}^{\mskip 2mu Q,\brT{S''}}(Z; s\lambda - t\mu)e^{\langle s\lambda - t\mu,\Y_{S''}\rangle} }
\\ &&\hspace{2.5cm}\times\big\langle 
\bfM_{S''\vert \ttS}(s\lambda)\bfM(s,\lambda)\Phi,
\bfM_{S''\vert \ttS}(t\mu)\bfM(t,\mu)\bfD_{-\nu}\Psi\big\rangle_{S''}\ptf
\end{align*}
C'est une fonction lisse de $\lambda$ et $\mu$. On pose \begin{equation*}
\bso_{s,t}^{\mskip 2mu T,Q}(Z;\mu;\nu) =\bso_{s,t}^{\mskip 2mu T,Q}(Z;\theta_0(\mu),\mu+\nu;\nu)\ptf
\end{equation*}et
\begin{equation*}\brabso_{s,t}^{\mskip 2mu T,Q}(Z;\mu)= \vert \wh\bsbbc_S \vert^{-1}\sum_{\nu \in \bsESEsigma}
\bso_{s,t}^{\mskip 2mu T,Q}(Z;\mu;\nu)\ptf\end{equation*}
Les expressions $\bso_{s,t}^{\mskip 2mu T,Q}(Z;\mu;\nu) $ ne d\'ependent que de l'image de $Z$ dans 
$\ESC_Q^\tG = \ESB_\tG \backslash \ESA_Q$. 

\begin{proposition}\label{crux}
On pose:
\begin{equation*}\bsA^T_{s,t}=\sum_{Z\in\ESC_Q^\tG}\wt{\sigma}_Q^R(Z-T)\int_{\bsmu_S}
\brabso^{\mskip 2mu T,Q}_{ s, t} (Z;\mu)\vtt(\mu)\dd \mu\ptf\end{equation*}
\begin{enumerate}[(i)]
\item L'expression $\bsA^T_{s,t}$ est convergente 
dans l'ordre indiqu\'e.
\item Pour tout r\'eel $r$, on a une majoration:
$\vert A^T_{s,t} -\bsA^T_{s,t}\vert \ll \bsd_0(T)^{-r}$.
\end{enumerate}
\end{proposition}
Cette proposition est l'analogue de \cite[13.5.1]{LW}, l'un des r\'esultats 
les plus fins du livre. Sa d\'emonstration occupera les deux sections suivantes.


 \section{Premi\`ere \'etape}
\label{preuve de la proposition}

Consid\'erons l'application
\begin{equation*}s\theta_0 - t: \bsmu_S \rightarrow \bsmu_\ttS\ptf \end{equation*}
On note $\bs\kopa_S$ son noyau et $\bs\eta_\ttS$ son image, et l'on pose
\begin{equation*}\bs\eta_S = \bs\kopa_S\backslash \bsmu_S\vgq \bs\kopa_\ttS= \bs\eta_\ttS\backslash \bsmu_\ttS\ptf\end{equation*}
L'application ci-dessus se restreint en un isomorphisme
$\iota:\bs\eta_S \rightarrow \bs\eta_\ttS$.
La suite exacte courte de groupes ab\'eliens compacts
\begin{equation*}0 \rightarrow \bs{\eta}_\ttS \rightarrow \bsmu_\ttS \rightarrow \bs{\kopa}_\ttS \rightarrow 0\end{equation*}
donne par dualit\'e de Pontryagin une suite exacte courte de $\ZM$-modules libres de type fini
\begin{equation*}0 \rightarrow \wh{\bs{\kopa}}_\ttS \rightarrow \ESA_\ttS \rightarrow \wh{\bs{\eta}}_\ttS \rightarrow 0\ptf\end{equation*}
En relevant dans $\ESA_\ttS$ une $\ZM$-base de $\wh{\bs{\eta}}_\ttS$, on d\'efinit un morphisme section 
du morphisme 
$\ESA_\ttS \rightarrow \wh{\bs{\eta}}_\ttS$ ce qui fournit un isomorphisme entre $\ESA_\ttS$ 
et le produit $\wh{\bs{\kopa}}_\ttS\times \wh{\bs{\eta}}_\ttS$. 
Dualement cela permet d'identifier $\bsmu_\ttS$ au produit 
$\bs{\kopa}_\ttS\times \bs{\eta}_\ttS$ et donc d'\'ecrire $\Lambda \in \bsmu_\ttS$ sous la forme
\begin{equation*}
\Lambda= \Lambda_{\bs{\kopa}}+\Lambda_{\bs{\eta}}\in \bs{\kopa}_{_tS}\times 
{\bs{\eta}}_\ttS
\end{equation*}
via cette identification (non canonique), et on identifie de mme $\bsmu_S$ 
au produit $\bs{\kopa}_S\times \bs{\eta}_S $.
On d\'efinit un \'el\'ement de $\bsmu_S$ en posant, pour $(\kopa,\Lambda)\in\bs\kopa_S \times \bsmu_\ttS$,
\begin{equation*}\mu(\kopa,\Lambda)= \kopa+\iota^{-1}(\Lambda_{\bs{\eta}})\ptf\end{equation*}
L'application 
\begin{equation*}\bs\kopa_S \times \bsmu_\ttS \rightarrow \bsmu_S \times \bs\kopa_\ttS\vg \quad (\kopa,\Lambda)
 \mapsto (\mu(\kopa,\Lambda),\Lambda_{\bs\kopa})\end{equation*}
est bijective et on a la relation
\begin{equation*}\theta_0 \mu(\kopa,\Lambda)= s^{-1}(t\mu(\kopa,\Lambda)+\Lambda_{\bs{\eta}})\ptf\leqno{(1)}\end{equation*}
Posons
\begin{equation*}\lambda(\kopa,\Lambda)= s^{-1}(t\mu(\kopa,\Lambda)+ \Lambda)\ptf\end{equation*}
Fixons $\nu\in \bsmu_S$. Rappelons que $\ttM= M_\ttS$ et $L=M_Q$. 
Pour $\kopa \in \bs{\kopa}_S$, $\Lambda \in \bsmu_\ttS$ et $S''\in \ESP^Q(\ttM)$, posons 
\begin{align*}
\lefteqn{
\bsc(\kopa; \Lambda,S'';\nu)= \vtt(\mu(\chi,\Lambda))\big\langle \bfM_{S''\vert\ttS}(s\lambda(\kopa,\Lambda))
\bfM(s,\lambda(\kopa,\Lambda))\Phi, }\\
&&\hspace{4cm} 
\bfM_{S''\vert\ttS}(t\mu(\kopa,\Lambda)+\nu) \bfM(t,\mu(\kopa,\Lambda)+\nu)
\bfD_{-\nu}\Psi \big\rangle_{S''}\ptf
\end{align*}
Les expressions $\bsc(\kopa; \Lambda,S'';\nu)$, consid\'er\'ees comme des fonctions de $\Lambda$ 
d\'ependant des param\`etres $\kopa$ et $\nu$, 
d\'efinissent une $(Q,\ttM)$-famille p\'eriodique $\bsc(\kopa;\nu)$. On d\'efinit aussi une $(Q,\ttM)$-famille p\'eriodique 
$\bsd(\kopa;\nu)= \bsc(\mathfrak{Y},\kopa;\nu)$ par
\begin{equation*}\bsd(\kopa ;\Lambda, S'';\nu)= e^{\langle \Lambda,Y_{S''}\rangle} \bsc(\kopa;\Lambda,S'';\nu)\ptf\end{equation*}
En se limitant aux $S''\in \ESP^{Q_0}(\ttM)$, on obtient des $(Q_0,\ttM)$-familles p\'eriodiques. Pour 
$Z\in \ESA_Q$, resp. $H\in \ESA_\Qo$, et $X'\in \ag_{0,\QM}$, 
on leur associe les fonctions
\begin{equation*}\bsd_{\ttM,F}^{\mskip 2mu Q,X'}(Z,\kopa; \Lambda; \nu)= 
\sum_{S''\in \ESP^{Q}(\ttM)}\varepsilon_{S''}^{Q,[X']_{S''}}(Z;\Lambda)\bsd(\kopa;\Lambda,S'';\nu)\end{equation*}
et 
\begin{equation*}\bsd_{\ttM,F}^{\mskip 2mu Q_0,X'}(H,\kopa; \Lambda; \nu)= \sum_{S''\in \ESP^{Q_0}(\ttM)}
\varepsilon_{S''}^{Q_0,[X']_{S''}}(H;\Lambda)\bsd(\kopa;\Lambda,S'';\nu)\ptf\end{equation*}
Ces fonctions sont lisses en $\kopa$ et $\Lambda$.

\begin{lemma}\label{fourbso}
Soient $Z\in \ESC_Q^\tG$, $\kopa \in \bs{\kopa}_S$, $\Lambda \in \bs{\eta}_\ttS$ et
$X\in \mathcal{C}_F^+(Q,Q_0;T)$. On a les \'egalit\'es suivantes:
\begin{enumerate}[(i)]
\item $\bso^{\mskip 2mu T,Q_0}_{ s,t}(Z,X;\mu(\kopa,\Lambda);\nu)\vtt(\mu(\chi,\Lambda))=
\bsd_{\ttM,F}^{\mskip 2mu Q_0,{T-X}}(\ZTX ,\kopa; \Lambda; \nu)$
\item 
$\bso^{\mskip 2mu T,Q}_{ s,t}(Z;\mu(\kopa,\Lambda);\nu)\vtt(\mu(\chi,\Lambda))=
\bsd_{\ttM,F}^{\mskip 2mu Q,T}(Z,\kopa; \Lambda; \nu)\ptf$
\end{enumerate}
\end{lemma}

\begin{proof} Rappelons que, par d\'efinition, 
\begin{align*}
\lefteqn{\bso^{\mskip 2mu T,Q_0}_{ s,t}(Z,X;\lambda,\mu;\nu)
=\sum_{S''\in\ESP^{Q_0}(\ttM)}
\varepsilon_{S''}^{\mskip 2mu Q_0, [{T-X}]_{S''}}(\ZTX ; s\lambda - t\mu)
e^{\langle\Y_{S''},s\lambda - t\mu\rangle}
}\\&&\hspace{3,5cm}\times\big\langle\bfM_{S''\vert \ttS}(s\lambda)
\bfM(s,\lambda)\Phi, \bfM_{S''\vert \ttS}(t\mu)\bfM(t,\mu)\bfD_{-\nu}\Psi\big\rangle_{S''}\ptf
\end{align*}
Pour $Z\in \ESA_\Qo$ et $\Lambda \in \bsmu_S$ en position g\'en\'erale, on a
\begin{equation*}\bsd_{\ttM,F}^{\mskip 2mu Q,{T-X}}(\ZTX ,\kopa; \Lambda; \nu)
= \bso^{\mskip 2mu T,Q_0}_{ s,t}(Z,X;\lambda(\kopa,\Lambda),\mu(\kopa,\Lambda)+\nu;\nu)\ptf\end{equation*}
Mais, d'apr\`es la relation (1) on a
\begin{equation*}\lambda(\kopa,\Lambda)=\theta_0\mu(\kopa,\Lambda)+s^{-1}(\Lambda_{\bs{\kopa}})\ptf\end{equation*}
On obtient (i) pour $\Lambda_{\bs{\kopa}}=0$. 
La preuve de (ii) est similaire.
\end{proof}

On munit $\bs{\kopa}_S$ et $\bs{\eta}_\ttS$ des mesures de Haar telles que
$\mathrm{vol}(\bs{\kopa}_S)=1= \mathrm{vol}(\bs{\eta}_\ttS)$.
En posant, comme ci-dessus, $\ZTX=Z+({T-X})_\Qo^Q$,
l'expression $A_{s,t,\nu}^T$ se r\'ecrit 
\begin{align*}
\lefteqn{
A^T_{s,t,\nu} =\sum_{Z\in\ESC_Q^\tG}\wt{\sigma}_Q^R(Z{-T})
\sum_{X\in\mathcal{C}_F(Q,Q_0; T)}\Kappa{\eta T}(X_\Qo)}\\
&& \hspace{2.5cm} \times \int_{\bs{\kopa}_S} \bigg(
\int_{\bs{\eta}_\ttS}\bsd_{\ttM,F}^{\mskip 2mu Q_0,{T-X}}(\ZTX ,\kopa; \Lambda; \nu) \dd\Lambda
\bigg) \dd \kopa \ptf
\end{align*}
Pour $Z\in \ESA_Q$, $S''\in \ESP^{Q}(\ttS)$, $V\in\ESA_\ttS$ et $X'\in \ag_{0,\QM}$, on pose
\begin{equation*}\wh{\bsd}(\kopa; V,S'';\nu)= \int_{\bsmu_\ttS}\bsd(\kopa; \Lambda,S''; \nu) 
e^{- \langle \Lambda,V\rangle} \dd \Lambda\end{equation*}
et
\begin{equation*}\wh{\bsd}_{\ttM,F}^{\;Q,X'}(Z,\kopa; V; \nu)= 
\int_{\bsmu_\ttS}\bsd_{\ttM,F}^{\mskip 2mu Q,X'}(Z,\kopa; \Lambda; \nu) 
e^{-\langle \Lambda,V\rangle} \dd \Lambda\ptf\end{equation*}
Pour $H\in \ESA_\Qo$, on d\'efinit de mani\`ere analogue
$ \wh{\bsd}_{\ttM,F}^{\;Q_0,X'}(H,\kopa; V; \nu)$.
Ces fonctions sont \`a d\'ecroissance rapide en $V$. Notons 
\begin{equation*}\ESD_\ttS\bydef \bs{\eta}_\ttS^\vee \subset \ESA_\ttS\end{equation*}
 l'annulateur de $\bs{\eta}_\ttS\;(\subset \bsmu_\ttS)$ dans $\ESA_\ttS$. 
\begin{lemma}\label{[LW,13.6.2]}On a 
\begin{equation*}A^T_{s,t,\nu} =\mskip -5mu\sum_{Z\in\ESC_Q^\tG} \wt{\sigma}_Q^R(Z{-T})\mskip -15mu
\sum_{X\in\mathcal{C}_F(Q,Q_0; T)}\mskip -20mu \Kappa{\eta T}(X_\Qo)
 \int_{\bs{\kopa}_S}\mskip -1mu\sum_{V\in \ESD_\ttS} \wh{\bsd}_{\ttM,F}^{\;Q_0,{T-X}}(\ZTX ,\kopa; V; \nu)\dd\kopa \ptf\end{equation*}
\end{lemma}
\begin{proof}Il suffit d'observer que par inversion de Fourier on a
\begin{equation*}\int_{\bs{\eta}_\ttS}\bsd_{\ttM,F}^{\mskip 2mu Q_0,{T-X}}(\ZTX ,\kopa; \Lambda; \nu) \dd\Lambda
=\sum_{V\in \ESD_\ttS} \wh{\bsd}_{\ttM,F}^{\;Q_0,{T-X}}(\ZTX,\kopa; V; \nu)\ptf\end{equation*}
\end{proof}
Il r\'esultera du lemme \ref{major} (qui est l'analogue de \cite[13.6.3]{LW})
que cette expression est absolument convergente.

\begin{lemma}\label{major}
Fixons un r\'eel $\rho>0$, et consid\'erons les cinq expressions:
\begin{equation*}\sum_{Z\in\ESC_Q^\tG}
\wt{\sigma}_Q^R(Z{-T})\hspace{-5pt}
\sum_{X\in\mathcal{C}_F(Q,Q_0; T)}\int_{\bs{\kopa}_S}
\sum_{V\in \ESD_\ttS} \big\vert \wh{\bsd}_{\ttM,F}^{\;Q_0,{T-X}}(\ZTX,\kopa; V; \nu) \big\vert \dd\kopa ;\leqno{(A)} \end{equation*}
\begin{equation*}\sum_{Z\in\ESC_Q^\tG}\mskip -2mu \mskip -2mu \wt{\sigma}_Q^R(Z{-T})\hspace{-17pt}
\sum_{X\in\mathcal{C}_F(Q,Q_0; T)}\hspace{-17pt}(1-\Kappa{\eta T}(X_\Qo))\int_{\bs{\kopa}_S}
\sum_{V\in \ESD_\ttS} \mskip -5mu\big\vert \wh{\bsd}_{\ttM,F}^{\;Q_0,{T-X}}(\ZTX,\kopa; V; \nu) \big\vert \dd\kopa ;\leqno{(B)} \end{equation*}
\begin{equation*}\sum_{Z\in\ESC_Q^\tG}\mskip -2mu \mskip -2mu \wt{\sigma}_Q^R(Z{-T})\hspace{-15pt}
\sum_{X\in\mathcal{C}_F(Q,Q_0; T)}\hspace{-2pt}\int_{\bs{\kopa}_S}
\sum_{V\in \ESD_\ttS}\hspace{-3pt}
(1- \Kappa{\rho T}(V)) 
\big\vert \wh{\bsd}_{\ttM,F}^{\;Q_0,{T-X}}(\ZTX,\kopa; V; \nu) \big\vert \dd\kopa ;\leqno{(C)}\end{equation*}
\begin{equation*}\sum_{Z\in\ESC_Q^\tG}\wt{\sigma}_Q^R(Z{-T})\int_{\bs{\kopa}_S}
\sum_{V\in \ESD_\ttS} \big\vert \wh{\bsd}_{\ttM,F}^{\;Q,T}(Z,\kopa; V; \nu) \big\vert \dd\kopa ;\leqno{(D)} \end{equation*}
\begin{equation*}\sum_{Z\in\ESC_Q^\tG}\wt{\sigma}_Q^R(Z{-T})\int_{\bs{\kopa}_S}
\sum_{V\in \ESD_\ttS} (1- \Kappa{\rho T}(V)) \big\vert \wh{\bsd}_{\ttM,F}^{\;Q,T}(Z,\kopa; V; \nu) \big\vert 
\dd\kopa \ptf\leqno{(E)} \end{equation*}
Alors on a:
\begin{enumerate}[(i)]
\item Les cinq expressions sont convergentes.
\item Pour tout r\'eel $r$, l'expression (B) est essentiellement major\'ee par $\bsd_0(T)^{-r}$.
\item Il existe une constante absolue $\rho_0>0$ telle que si $\rho>\rho_0$, alors pour tout 
r\'eel $r$, les expressions (C) et (E) sont essentiellement major\'ees 
par $\bsd_0(T)^{-r}$.
\end{enumerate}
\end{lemma}

Admettons provisoirement ce lemme prouv\'e au paragraphe suivant. 
D'apr\`es \ref{invfour}, pour chaque $\kopa\in \bs{\kopa}_S$ (le param\`etre $\nu$ \'etant fix\'e),
il existe une fonction \`a d\'ecroissance rapide 
\begin{equation*}\varphi=\varphi(\kopa;\nu): \UU \mapsto \varphi(\UU) =\varphi(\kopa;\UU;\nu)\end{equation*}
sur $\ESH_{Q,\ttM}$ telle que $\bsc(\kopa;\nu)= \bsc_{\varphi}$. 
Rappelons que $\bsc(\chi;\nu)$ est la $(Q,{_tM})$-famille p\'eriodique d\'efinie par
\begin{align*}
\lefteqn{
\bsc(\kopa; \Lambda,S'';\nu)= \vtt(\chi,\Lambda)) \langle \bfM_{S''\vert\ttS}(s\lambda(\kopa,\Lambda))
\bfM(s,\lambda(\kopa,\Lambda))\Phi, }\\
&&\hspace{3cm}
\bfM_{S''\vert\ttS}(t\mu(\kopa,\Lambda)+\nu) \bfM(t,\mu(\kopa,\Lambda)+\nu)
\bfD_{-\nu}\Psi \rangle_{S''}
\end{align*}
et que l'on a pos\'e
\begin{equation*}\bsd(\chi;\nu)= \bsc(\YY,\chi;\nu)\ptf\end{equation*}
Pour $H\in \ESA_\Qo$ et $X'\in \ag_{0,\QM}$, on a donc
\begin{equation*}\bsd_{\ttM,F}^{\mskip 2mu Q_0,X'}(H,\kopa; \Lambda; \nu)=
\sum_{\UU\in\ESH_{Q_0,\ttM}}\varphi(\UU-\YY)\gamma_{_tM,F}^{Q_0,X'}(H,\UU;\Lambda)\end{equation*}
avec
\begin{equation*}\gamma_{\ttM,F}^{Q_0,X'}(H,\UU;\Lambda)= \sum_{H'\in \ESA_\ttM^{Q_0}(H+U_\Qo)}
\Gamma_\ttM^{Q_0}(H',\UU(X'))e^{\langle\Lambda,H'\rangle}\ptf\end{equation*}
Pour $Z\in \ESA_Q$ et $X\in \ESC_F^+(Q,Q_0;T)$, on obtient
\begin{equation*}\bsd_{\ttM,F}^{\mskip 2mu Q_0,{T-X}}(\ZTX ,\kopa; \Lambda; \nu)=
\sum_{\UU\in\ESH_{Q_0,\ttM}}\varphi(\UU-\YY)\gamma_{\ttM,F}^{Q_0,{T-X}}(\ZTX ,\UU;\Lambda)\ptf\end{equation*}
On a aussi
\begin{equation*}\bsd_{\ttM,F}^{\mskip 2mu Q,T}(Z,\kopa; \Lambda; \nu)=
\sum_{\UU\in\ESH_{Q,\ttM}}\varphi(\UU-\YY)\gamma_{\ttM,F}^{Q,T}(Z,\UU;\Lambda)\ptf\end{equation*}
On introduit comme ci-dessus des transform\'ees de Fourier inverses
\begin{equation*}V \mapsto \wh{\gamma}_{\ttM,F}^{\mskip 2mu Q_0,X'}(H,\UU;V)\quad \hbox{et} \quad 
V \mapsto \wh{\gamma}_{\ttM,F}^{\mskip 2mu Q,X'}(Z,\UU;V)\end{equation*}
le param\`etre $V$ variant dans $\ESA_\ttM$. Par inversion de Fourier, on a
\begin{equation*}\wh{\gamma}^{\mskip 2mu Q,X'}_{\ttM,F}(Z,\UU;V)= \left\{
\begin{array}{ll} \Gamma_\ttM^{Q}(V,\UU(X')) & \hbox{si $ Z+U_Q=V_Q$}\\
0 & \hbox{sinon}
\end{array}\right.\ptf\end{equation*}
On en d\'eduit que
\begin{equation*}\wh{\bsd}_{\ttM,F}^{\;Q_0,{T-X}}(\ZTX,\kopa; V; \nu)=
\sum_{\substack{\UU\in\ESH_{Q_0,\ttM}\\ H_Z^{{T-X}} + U_\Qo = V_\Qo}}\varphi(\kopa;\UU-\YY;\nu)
\Gamma_\ttM^{Q_0}(V,\UU({T-X}))\end{equation*}
et
\begin{equation*}\wh{\bsd}_{\ttM,F}^{\;Q,T}(Z,\kopa; V; \nu)=
\sum_{\substack{\UU\in\ESH_{Q,\ttM}\\ Z + U_Q =
V_Q}}\varphi(\kopa;\UU-\YY;\nu)\Gamma_\ttM^Q(V,\UU(T))\ptf\end{equation*}
Fixons un r\'eel $\rho>\rho_0$ comme dans le point (iii) et posons
\begin{equation*}E^T_1= \sum_{Z\in\ESC_Q^\tG}\wt{\sigma}_Q^R(Z{-T})\int_{\bs{\kopa}_S}
\sum_{V\in \ESD_\ttS} \Kappa{\rho T}(V) \hspace{-10pt}
\sum_{X\in \mathcal{C}_F(Q,Q_0;T)}\hspace{-10pt}
 \wh{\bsd}_{\ttM,F}^{\;Q_0,{T-X}}(\ZTX,\kopa; V; \nu) \dd\kopa\ptf\end{equation*}
D'apr\`es \ref{[LW,13.6.2]} et les assertions du lemme \ref{major} concernant les expressions (A), (B) et (C), 
l'expression $E^T_1$ est absolument convergente, et pour tout r\'eel $r$, on a une majoration
\begin{equation*}\big\lvert A_{s,t,\nu}^T - E^T_1\big\rvert \ll \bsd_0(T)^{-r}\ptf\end{equation*}
Notons $\ESR_\ttS^+$ l'ensemble des racines de $A_\ttM$ qui sont positives pour le sous-groupe parabolique standard $\ttS$. 
Pour tout $S''\in \ESP^{Q_0}(\ttM)$, notons $a(S'')$ le nombre d'\'el\'ement de 
$(-\Delta_{S''})\cap \ESR_\ttS^+$ -- ou encore de $(-\Delta_{S''}^{Q_0})\cap \ESR_\ttS^+$ -- et
\begin{equation*}\mathcal{C}^{Q_0}(S'')\subset \ag_\ttM^{Q_0}\end{equation*}
le c™ne form\'e des
\begin{equation*}\bigg(\sum_{\alpha \in \Delta_{S''}^{Q_0} \cap \ESR_\ttS^+} x_\alpha \check{\alpha}\bigg)+
\bigg( \sum_{\alpha \in (-\Delta_{S''}^{Q_0}) \cap \ESR_\ttS^+}y_\alpha \check{\alpha}\bigg)\end{equation*}
pour des $x_\alpha \geq 0$ et des $y_\alpha >0$. Pour $Y\in \ag_\ttM^{Q_0}+\ESA_\Qo$, on pose
\begin{equation*}\mathcal{C}^{Q_0}_F(Y;S'') = \bigg(Y +\mathcal{C}^{Q_0}(S'')\bigg)\cap 
\ESA_\ttM\subset \ESA_\ttM^{Q_0}(Y_\Qo)\ptf\end{equation*}
Notons que pour $H\in \ESA_\ttM$, on a
\begin{equation*}\mathcal{C}^{Q_0}_F(H+Y;S'')= H+\mathcal{C}^{Q_0}_F(Y;S'')\ptf\end{equation*}
En remplaant les exposants $Q_0$ par $Q$, on d\'efinit de la mme mani\`ere
\begin{equation*}\mathcal{C}^Q(S'') \subset \ag_\ttM^Q\quad \hbox{et} \quad 
\mathcal{C}^Q_F(Y;S'')\subset \ESA_\ttM\ptf\end{equation*}

\begin{lemma}\label{dtmfou}
Pour $Z\in \ESA_Q$, $\kopa \in \bs{\kopa}_{S}$, $V\in \ESA_\ttM$ et $X\in \ESC_F(Q,Q_0;T)$, on a
\begin{equation*}\wh{\bsd}_{\ttM,F}^{\;Q_0,{T-X}}(\ZTX,\kopa; V; \nu)= \sum_{S''\in \ESP^{Q_0}(\ttM)}(-1)^{a(S'')}
\hspace{-0.8cm}
\sum_{V_1\in \mathcal{C}^{Q_0}_F(- H^{{T-X}}_{Z,S''};S'')}
\hspace{-0.8cm}\wh{\bsd}(\kopa; V+V_1,S'';\nu)\end{equation*}
avec
\begin{equation*}H^{{T-X}}_{Z,S''}\bydef Z + [{T-X}]_{S''}^Q\ptf\end{equation*}
\end{lemma}

\begin{proof}
On rappelle que $\ZTX=Z +({T-X})_\Qo^Q \in \ESA_\Qo$. On a
\begin{equation*}\wh{\bsd}_{\ttM,F}^{\;Q_0,{T-X}}(\ZTX,\kopa; V; \nu)=
\sum_{\UU\in\ESH_{Q_0,\ttM}}\varphi(\kopa;\UU-\YY;\nu)
\wh{\gamma}_{\ttM,F}^{\mskip 2mu Q_0,{T-X}}(\ZTX,\UU;V)\end{equation*}
avec
\begin{equation*}\wh{\gamma}_{\ttM,F}^{\mskip 2mu Q_0,{T-X}}(\ZTX,\UU;V)= \left\{
\begin{array}{ll}\Gamma_\ttM^{Q_0}(V,\UU({T-X}))
 & \hbox{si $ \ZTX+U_\Qo=V_\Qo $}\\
0 & \hbox{sinon}
\end{array}\right.\ptf\end{equation*}
Le lemme \ref{gammaphi} nous dit que
\begin{equation*}\Gamma_\ttM^{Q_0}(V,\UU({T-X}))= 
\hspace{-0.5cm}\sum_{S''\in \ESP^{Q_0}(\ttM)} (-1)^{a(S'')} {\mathbf 1}_{\mathcal{C}^{Q_0}(S'')} 
\big(( [{T-X}]_{S''}+ U_{S''} -V)^{Q_0} \big)\end{equation*}
o $ {\mathbf 1}_{\mathcal{C}^{Q_0}(S'')}$ est la fonction caract\'eristique du c™ne $\mathcal{C}^{Q_0}(S'')$. On obtient
\begin{align*}
\lefteqn{\wh{\gamma}_{\ttM,F}^{\mskip 2mu Q_0,{T-X}}(\ZTX,\UU;V)= \sum_{S''\in \ESP^{Q_0}(\ttM)} (-1)^{a(S'')}}\\
&& \times\left\{
\begin{array}{ll}{\mathbf 1}_{\mathcal{C}^{Q_0}(S'')}\big(([{T-X}]_{S''}+U_{S''})^{Q_0}- V\big)
 & \hbox{si $ \ZTX +U_\Qo=V_\Qo$}\\ 0 & \hbox{sinon}
\end{array}\right.
\end{align*}
soit encore
\begin{equation*}\wh{\gamma}_{\ttM,F}^{\mskip 2mu Q_0,{T-X}}(\ZTX,\UU;V)
= \sum_{S''\in \ESP^{Q_0}(\ttM)} (-1)^{a(S'')}{\mathbf 1}_{\mathcal{C}^{Q_0}(S'')}(Y)\end{equation*}
avec
\begin{equation*}Y = Z +({T-X})_\Qo^Q+[{T-X}]_{S''}^{Q_0}+U_{S''}- V = H^{{T-X}}_{Z,S''}+U_{S''}- V\ptf\end{equation*}
La condition $Y\in \mathcal{C}^{Q_0}(S'')$ \'equivaut \`a 
\begin{equation*}U_{S''}\in V + \mathcal{C}^{Q_0}_F(-H^{{T-X}}_{Z,S''};S'')\end{equation*}
et implique que $U_\Qo = V_\Qo - H$. 
D'autre part on a, par d\'efinition,
\begin{equation*}\bsd(\kopa;\Lambda,S'';\nu)= 
\sum_{\UU\in \ESH_{Q_0,\ttM}} \varphi(\kopa;\UU-\YY;\nu)e^{\langle \Lambda, U_{S''} \rangle}\end{equation*}
et donc, par inversion de Fourier,
\begin{equation*}\wh{\bsd}(\kopa;V,S'';\nu) = \sum_{\substack{\UU \in \ESH_{Q_0,\ttM}\\ U_{S''}= V}}
\varphi(\kopa; \UU-\YY; \nu)\ptf\end{equation*}
Par cons\'equent, on a 
\begin{equation*}\sum_{\UU \in \ESH_{Q_0,\ttM}}\mskip -5mu\mskip -5mu \varphi(\kopa; \UU-\YY ; \nu) 
{\mathbf 1}_{V + \mathcal{C}^{Q_0}_F(-H^{{T-X}}_{Z,S''};S'')}(U_{S''})
=\hspace{-0.32cm}\sum_{V_1\in \mathcal{C}^{Q_0}_F(-H^{{T-X}}_{Z,S''};S'')}\hspace{-0.8cm}
\wh{\bsd}(\kopa; V + V_1; \nu)\end{equation*}
ce qui prouve le lemme.
\end{proof}

D'apr\`es \ref{dtmfou}, la somme sur $X$ dans l'expression $E_1^T$ devient
\begin{equation*}\sum_{S''\in \ESP^{Q_0}(\ttM)}(-1)^{a(S'')}\hspace{-0.3cm} \sum_{X\in \mathcal{C}_F(Q_0,Q;T)}
\bigg(\sum_{V_1\in \mathcal{C}_F^{Q_0}(-H^{{T-X}}_{Z,S''};S'')}\hspace{-1cm}
\wh{\bsd}(\kopa; V+ V_1,S'';\nu)\bigg)\mskip -2mu .\leqno{(2)}\end{equation*}
L'expression (2) est bien absolument convergente. 

\begin{lemma}\label{[LW,13.7.1]bis}
Pour $Z\in \ESA_Q$, $\kopa \in \bs{\kopa}_{S}$ et $V\in \ESD_\ttS$ 
\begin{equation*}\wh{\bsd}_{{_tM},F}^{\;Q,T}(Z,\kopa; V; \nu)
\sum_{S''\in \ESP^{Q}(\ttM)}(-1)^{a(S'')}
\hspace{-0.5cm}\sum_{V_2\in \mathcal{C}^{Q}_F(- \ZTSs;S'')}
\hspace{-0.5cm}\wh{\bsd}(\kopa; V + V_2,S'';\nu)\ptf\end{equation*}
avec $\ZTSs\bydef Z + \brT{S''}^Q$.
\end{lemma}

\begin{proof}
Elle est identique \`a celle du lemme \ref{dtmfou}.
\end{proof}

Pour $S''\in \ESP^{Q_0}(\ttM)$, 
il r\'esulte des d\'efinitions que l'application
\begin{equation*}\mathcal{C}(Q,Q_0)\times \mathcal{C}^{Q_0}(S'')
 \rightarrow \ag_\ttM^Q\qquad\hbox{d\'efinie par }\qquad
 (X,V_1) \mapsto [X]_{S''}^Q+V_1\end{equation*}
est injective et a pour image le c™ne $\mathcal{C}^Q(S'')$. Pour $X\in \mathcal{C}_F(Q,Q_0;T)$, 
tout \'el\'ement $V_1\in \mathcal{C}_F^{Q_0}(-H^{{T-X}}_{Z,S''};S'')$ s'\'ecrit
\begin{equation*}V_1 = - H^{{T-X}}_{Z,S''} +V_1^*= -\ZTSs+V_2^*\end{equation*}
avec $V_1^*\in \mathcal{C}^{Q_0}(S'')$ et $V_2^*= [X]_{S''}^Q+V_1^*\in \mathcal{C}^Q(S'')$. 
Par d\'efinition $V_1$ appartient \`a $\mathcal{C}_F^Q(-\ZTSs;S'')$. R\'eciproquement, tout 
\'el\'ement $V_2\in \mathcal{C}_F^{Q}(-\ZTSs;S'')$ s'\'ecrit
\begin{equation*}V_2 = -\ZTSs+[X]_{S''}^{Q}+V_1^* = -H^{{T-X}}_{Z,S''}+V_1^*\end{equation*}
avec $X\in \mathcal{C}(Q,Q_0)$ et $V_1^*\in \mathcal{C}^{Q_0}(S'')$. Donc $V_2$ appartient \`a 
$ \mathcal{C}_F^{Q_0}(-H^{{T-X}}_{Z,S''};S'')$, et comme
\begin{equation*}(V_2)_\Qo= -Z -({T-X})_\Qo^Q = -H^{{T-X}}_Z \vg\end{equation*}
par d\'efinition $X$ appartient \`a $\mathcal{C}_F(Q,Q_0;T)$. L'expression (2) se r\'ecrit donc
\begin{equation*}\sum_{S''\in \ESP^{Q_0}(\ttM)}(-1)^{a(S'')} 
\sum_{V_2\in \mathcal{C}_F^Q(-\ZTSs;S'')}\hspace{-1cm}
\wh{\bsd}(\kopa; V+ V_2,S'';\nu)\end{equation*}
soit encore, d'apr\`es \ref{[LW,13.7.1]bis},
\begin{equation*}\wh{\bsd}_{{_tM},F}^{\;Q,T}(Z,\kopa; V; \nu) - 
\sum_{S''\in \ESP^Q(\ttM) \smallsetminus \ESP^{Q_0}(\ttM)}\hspace{-0.3cm}(-1)^{a(S'')} 
\hspace{-0.5cm}
\sum_{V_2\in \mathcal{C}_F^Q(-\ZTSs;S'')}\hspace{-0.5cm}
\wh{\bsd}(\kopa; V+ V_2,S'';\nu)\ptf\end{equation*}
On en d\'eduit l'\'egalit\'e
\begin{equation*}E^T_1 = E^T_2-E^T_3\leqno{(3)}\end{equation*}
o
\begin{equation*}E^T_2= \sum_{Z\in\ESC_Q^\tG}\wt{\sigma}_Q^R(Z{-T})\int_{\bs{\kopa}_S}
\sum_{V\in \ESD_\ttS} \Kappa{\rho T}(V) \wh{\bsd}_{{_tM},F}^{\;Q,T}(Z,\kopa; V; \nu) \dd\kopa\end{equation*}
et
\begin{align*}
\lefteqn{
E^T_3= \sum_{Z\in\ESC_Q^\tG}\wt{\sigma}_Q^R(Z{-T})\int_{\bs{\kopa}_S}
\sum_{V\in \ESD_\ttS} \Kappa{\rho T}(V) }\\
&& \hspace{3cm}\times
\sum_{S''\in \ESP^Q(\ttM) \smallsetminus \ESP^{Q_0}(\ttM)}\hspace{-0.8cm}(-1)^{a(S'')} 
\hspace{-0.5cm}
\sum_{V_2\in \mathcal{C}_F^Q(-\ZTSs;S'')}\hspace{-0.8cm}
\wh{\bsd}(\kopa; V+ V_2,S'';\nu)\ptf
\end{align*}
La d\'ecomposition (3) est justifi\'ee car l'expression $E^T_2$ est absolument convergente 
d'apr\`es \ref{major}(D) et 
donc $E^T_3$ est convergente au moins dans l'ordre indiqu\'e. 
%

\begin{lemma}\label{remajo}
Pour $S''\in \ESP^Q(\ttM)\smallsetminus \ESP^{Q_0}(\ttM)$, posons
\begin{equation*}E^T_{S''}= \sum_{Z\in\ESC_Q^\tG}\wt{\sigma}_Q^R(Z{-T})\int_{\bs{\kopa}_S}
\sum_{V\in \ESD_\ttS} \Kappa{\rho T}(V) 
\hspace{-0.3cm}\sum_{V_2\in \mathcal{C}_F^Q(-\ZTSs;S'')}\hspace{-0.5cm}
\vert\wh{\bsd}(\kopa; V+ V_2,S'';\nu)\vert \dd\kopa\ptf\end{equation*}
Pour tout r\'eel $r$, on a une majoration de la forme $E^T_{S''} \ll \bsd_0(T)^{-r}$.
\end{lemma}

Admettons ce lemme (qui sera lui aussi prouv\'e dans le paragraphe suivant), et posons
\begin{equation*}E^T_4= \sum_{Z\in\ESC_Q^\tG}\wt{\sigma}_Q^R(Z{-T})\int_{\bs{\kopa}_S}
\sum_{V\in \ESD_\ttS} \wh{\bsd}_{{_tM},F}^{\;Q,T}(Z,\kopa; V; \nu) \dd\kopa\ptf \end{equation*}
L'expression $E^T_4$ est absolument convergente et d'apr\`es l'assertion (iii) 
du lemme \ref{major} concernant l'expression (E), pour tout r\'eel $r$, on a
\begin{equation*}\vert E^T_2 - E^T_4\vert \ll \bsd_0(T)^{-r}\ptf\end{equation*}
Par inversion de Fourier de la somme sur $V$ dans $E^T_4$, on a
\begin{equation*}E^T_4= \sum_{Z\in\ESC_Q^\tG}\wt{\sigma}_Q^R(Z{-T})\int_{\bs{\kopa}_S}
\bigg(\int_{\bs{\eta}_\ttS} \bsd_{\ttM,F}^{Q,T}(Z,\kopa; \Lambda; \nu)\dd\Lambda\bigg) \dd\kopa\ptf \end{equation*}
avec (d'apr\`es \ref{fourbso}~(ii))
\begin{equation*}\bsd_{\ttM,F}^{\mskip 2mu Q,T}(Z,\kopa; \Lambda; \nu)=\bso^{\mskip 2mu T,Q}_{ s,t}(Z;\mu(\kopa,\Lambda);\nu)\vtt(\mu(\chi,\Lambda))\ptf\end{equation*}
L'expression $E^T_4$ est convergente dans l'ordre indiqu\'e. 
On peut regrouper les int\'egrales en $\kopa$ et $\Lambda$ gr‰ce au changement de variables 
$(\kopa,\Lambda)\mapsto \mu(\kopa,\Lambda)$. On obtient
\begin{equation*}E^T_4 = \bsA_{s,t,\nu}^T \bydef 
\sum_{Z\in\ESC_Q^\tG}\wt{\sigma}_Q^R(Z-T)\int_{\bsmu_S}\bso^{T,Q}_{ s, t}(Z;\mu;\nu)\vtt(\mu)\dd \mu\ptf\end{equation*}
On a prouv\'e que pour tout r\'eel $r$, on a
\begin{equation*}\vert A_{s,t,\nu}^T - \bsA_{s,t,\nu}^T \vert \ll \bsd_0(T)^{-r}\ptf\end{equation*}
ce qui ach\`eve la preuve de \ref{crux} modulo les majorations \ref{major} et \ref{remajo}
qui seront \'etablies dans la section suivante.

 \section{Fin de la preuve}
\label{preuve-fin} 

Avant d'attaquer la d\'emonstration proprement dite des lemmes \ref{major} et \ref{remajo}, 
on \'etablit une variante du lemme \ref{TMGMF}. Soit $Z\in \ESA_Q$. 
Pour $\Pp\in \EuScript{F}^{Q}(\ttM)$, $U\in \ESA_{\Pp}$ et $X\in \ag_{\Pp}$, consid\'erons l'expression
\begin{equation*}\gamma_{\Pp\mskip -2mu ,F}^{Q,U}(Z;X,\Lambda)\bydef \sum_{H\in \ESA_{\Pp}^{Q}(Z)} \Gamma_{\Pp}^{Q}(H-X,U) 
e^{\langle \Lambda,H \rangle}\ptf\end{equation*}
Puisque la somme sur $H$ est finie, c'est une fonction enti\`ere en $\Lambda$. 
Pour $V\in \ESA_{\Pp}$, sa transform\'ee de Fourier inverse
\begin{equation*}\wh{\gamma}_{\Pp\mskip -2mu ,F}^{\mskip 2mu Q,U}(Z;X,V) = 
\int_{\bsmu_{\Pp} }\gamma_{\Pp\mskip -2mu ,F}^{Q,U}(Z; X,\Lambda)e^{-\langle \Lambda , V \rangle}\dd \Lambda\end{equation*}
est donn\'ee par
\begin{equation*}\wh{\gamma}^{\mskip 2mu Q,U}_{\Pp\mskip -2mu ,F}(Z;X,V)=\left\{\begin{array}{ll}
\Gamma_{\Pp}^{Q}(V-X ,U) & \hbox{si $Z=V_Q$}\\ 0 & \hbox{sinon}
\end{array}\right.
\end{equation*}
Soit $\bs{e}$ 
une $(Q,\ttM)$-famille p\'eriodique donn\'ee par une fonction \`a d\'ecroissance rapide 
$m$ sur $\ESH_{Q,\ttM}$. La fonction
\begin{equation*}\bs{e}_{\Pp\mskip -2mu ,F}^{Q} (Z;X,\Lambda)= \sum_{\UU \in \ESH_{Q,\ttM}} m(\UU) 
\gamma^{Q,U_{\Pp}}_{\Pp\mskip -2mu ,F}(Z+U_Q; X,\Lambda)\end{equation*}
est lisse en $\Lambda$. Pour $V\in \ESA_{\Pp}$, on d\'efinit comme ci-dessus 
les transform\'ees de Fourier inverses 
$\wh{\bs{e}}_{\Pp\mskip -2mu ,F}^{\mskip 2mu Q}(Z;X,V)$ et $\wh{\bs{e}}(V,\Pp)$. 
Ce sont des fonctions \`a d\'ecroissance rapide en 
$V\in \ESA_{\Pp}$.

\begin{lemma}\label{TMGMFvariante}
Pour $V\in \ESA_{\Pp}$, on a
\begin{equation*}\wh{\bs{e}}_{\Pp\mskip -2mu ,F}^{\mskip 2mu Q}(Z;X,V)=
\sum_{U\in \ESA_{\Pp}^Q(V_Q-Z)} \wh{\bs{e}}(U,\Pp) \Gamma_{\Pp}^Q(V-X,U)\ptf\end{equation*}
De plus pour tout r\'eel $r$, il existe une constante $c>0$ telle que
\begin{equation*}
\vert \wh{\bs{e}}_{\Pp\mskip -2mu ,F}^{\mskip 2mu Q}(Z;X,V) \vert \leq c
 \big(1 + \lVert  V -Z-X^Q \rVert  \big)^{-r}\end{equation*}
\end{lemma}
%
\begin{proof}
On a 
\begin{equation*}\wh{\bs{e}}_{\Pp\mskip -2mu ,F}^{\mskip 2mu Q}(Z;X,V)=
\sum_{\substack{\UU \in \ESH_{Q,\ttM}\\Z+ U_Q = V_Q}} m(\UU) \Gamma_{\Pp}^{Q}(V-X ,U_{\Pp})\end{equation*}
avec, pour $U\in \ESA_{\Pp}$,
\begin{equation*}\sum_{\substack{\UU \in \ESH_{Q,\ttM}\\ U_{\Pp}=U}}m(\UU)= 
\int_{\bsmu_{\Pp}} \bs{e}(\Lambda,\Pp) e^{-\langle \Lambda, U \rangle}\dd\Lambda 
=\wh{\bs{e}}(U,\Pp) \ptf\end{equation*}
D'o la premi\`ere assertion du lemme. 
Quant \`a la majoration, pour $V,\mskip 2mu  U \in \ESA_{\Pp}$ tels que $\Gamma_{\Pp}^{Q}(V-X ,U)$, on a
$\lVert  (V-X)^Q \rVert  \ll \lVert  U^Q \rVert $. Si de 
plus $Z+ U_Q = V_Q$, alors puisque $V - Z - X^Q= U_Q + (V-X)^Q$, on a
\begin{equation*}\lVert  V - Z-X^Q \rVert  \ll \lVert  U_Q \rVert  + \lVert  (V-X)^Q \rVert  \ll \lVert  U \rVert  \ptf\end{equation*}
On obtient que pour tout r\'eel $r>1$, l'expression
\begin{equation*}\vert \wh{\bs{e}}_{\Pp\mskip -2mu ,F}^{\mskip 2mu Q}(Z;X,V) \vert 
(1 + \lVert  V - Z-X^Q \rVert )^r \end{equation*}
est essentiellement major\'ee par
\begin{equation*}\sum_{U \in \ESA_{\Pp}^Q(V_Q-Z)} \wh{\bs{e}}(U,\Pp) \vert (1 + \lVert  U \rVert )^r \end{equation*}
Cette somme converge car $\wh{\bs{e}}(U,\Pp)$ est \`a d\'ecroissance rapide en $U$.
\end{proof}

\vskip2mm
\ni \textit{D\'emonstration du lemme \ref{major}.} On reprend en l'adaptant celle de \cite[13.6.3]{LW}. 
Commenons par l'expression 
(D). D'apr\`es \ref{TMGMF}, pour $V\in \ESA_\ttM$, on a
\begin{equation*}\wh{\bsd}_{{_tM},F}^{\;Q,T}(Z,\kopa; V; \nu)= 
\sum_{\substack{\UU \in \ESH_{Q,\ttM}\\ Z + U_Q =V_Q}} 
\varphi(\kopa;\UU-\YY;\nu) \Gamma_\ttM^Q(V,\UU(T))\end{equation*}
avec (d'apr\`es \cite[1.8.6]{LW})
\begin{equation*}\Gamma_\ttM^{Q}( V,\UU(T)) = \sum_{\Pp\in \EuScript{F}^Q(\ttM)} 
\Gamma_\ttM^{\Pp}( V,\TT)\Gamma_{\Pp}^Q(V_{\Pp}-\brT{\Pp},U_{\Pp})\ptf\end{equation*}
On obtient
\begin{equation*}\wh{\bsd}_{{_tM},F}^{\;Q,T}(Z,\kopa; V; \nu)= 
\sum_{\Pp\in \EuScript{F}^Q(\ttM)} \Gamma_\ttM^{\Pp}(V,\TT)
\wh{\bsd}_{\Pp\mskip -2mu ,F}^{\mskip 2mu Q}(Z,\kopa;\brT{\Pp},V_{\Pp};\nu)\end{equation*}
avec, pour $X\in \ag_{\Pp}$ et $V'\in \ESA_{\Pp}$,
\begin{equation*}\wh{\bsd}_{\Pp\mskip -2mu ,F}^{\mskip 2mu Q}(Z,\kopa;X,V';\nu)\bydef 
\sum_{\substack{\UU \in \ESH_{Q,\ttM}\\ U_Q = V'_Q-Z}} \varphi(\kopa;\UU-\YY;\nu)
\Gamma_{\Pp}^Q( V'-X,U_{\Pp})\ptf\end{equation*}
soit encore (d'apr\`es \ref{TMGMFvariante})
\begin{equation*}\wh{\bsd}_{\Pp\mskip -2mu ,F}^{\mskip 2mu Q}(Z,\kopa;X,V';\nu) = 
\sum_{U \in \ESA_{\Pp}^Q(V'_Q-Z)}\wh{\bsd}(\chi; U,\Pp;\nu)\Gamma_{\Pp}^Q(V'-X,U)\ptf\end{equation*}
L'expression (D) est donc major\'ee par
\begin{equation*}\sum_{\Pp\in \EuScript{F}^Q(\ttM)} I^T_{(D)}(\Pp)\end{equation*}
avec\footnote{Rappelons que puisque l'\'el\'ement $T$ est r\'egulier, la famille orthogonale $\TT$ 
est r\'eguli\`ere, et d'apr\`es \cite[1.8.7]{LW} la fonction 
$H \mapsto \Gamma_{_tM}^{P'}(H,\TT)$ est 
la fonction caract\'eristique d'un ensemble qui se projette sur un compact convexe de $\ag_{_tM}^{P'}$. }
\begin{equation*}I^T_{(D)}(\Pp)=\sum_{Z\in\ESC_Q^\tG}\wt{\sigma}_Q^R(Z{-T})\int_{\bs{\kopa}_S}
\sum_{V\in \ESD_\ttS} \Gamma_\ttM^{\Pp}(V^{\Pp}\mskip -2mu ,\TT) 
\vert \wh{\bsd}_{\Pp\mskip -2mu ,F}^{\mskip 2mu Q}(Z,\kopa;\brT{\Pp},V_{\Pp};\nu)\vert \dd\kopa\ptf\end{equation*}
Fixons un $\Pp\in \EuScript{F}^{Q}(\ttM)$. L'\'el\'ement $T$ \'etant fix\'e, d'apr\`es \cite[1.8.5]{LW} 
il existe une constante $c>0$ telle que pour 
tout $V\in \ag_\ttS$ tel que $\Gamma_\ttM^{\Pp}( V^{\Pp}\mskip -2mu ,\TT)\neq 0$, on ait $\lVert  V^{\Pp} \rVert  \leq c$. 
Pour $X\in \ag_{\Pp}$, 
la fonction $\wh{\bsd}_{\Pp\mskip -2mu ,F}^{\mskip 2mu Q}(Z,\kopa;X,V';\nu)$ 
est \`a d\'ecroissance rapide en $V'\in \ESA_{\Pp}$, uniform\'ement en $\chi$,
par cons\'equent l'expression
\begin{equation*}\int_{\bs{\kopa}_S}\sum_{V\in \ESD_\ttS}\Gamma_\ttM^{\Pp}(V^{\Pp}\mskip -2mu ,\TT) 
 \vert \wh{\bsd}_{\Pp\mskip -2mu ,F}^{\mskip 2mu Q}(Z,\kopa;\brT{\Pp},V_{\Pp};\nu)\vert \dd \kopa\end{equation*}
est convergente 
et, en posant
\begin{equation*}\phi(Z,X,V')\bydef \int_{\bs{\kopa}_S}\vert \wh{\bsd}_{\Pp\mskip -2mu ,F}^{\mskip 2mu Q}(Z,\kopa;X,V';\nu)\vert \dd \kopa\vg\end{equation*}
on a
\begin{equation*}I^T_{(D)}(\Pp)=\sum_{Z\in\ESC_Q^\tG}\wt{\sigma}_Q^R(Z{-T})
\sum_{V\in \ESD_\ttS} \Gamma_\ttM^{\Pp}( V^{\Pp}\mskip -2mu ,\TT) \phi (Z,\brT{\Pp},V_{\Pp})\ptf\end{equation*}
Le groupe $\ESD_\ttS$ est par d\'efinition l'annulateur de
\begin{equation*}\bs{\eta}_\ttS= (s\theta_0 - t)\bsmu_S \subset \bsmu_\ttS\end{equation*}
dans $\ESA_\ttS$. On a donc
\begin{equation*}\ESD_\ttS= \ker\Big(\theta_0^{-1}s^{-1}(1-w\theta_0): \ESA_\ttS \rightarrow \ESA_{S}\Big)
\quad \hbox{avec} \quad w = s\theta_0(t)^{-1}\end{equation*}
soit encore
\begin{equation*}\ESD_\ttS= \ker\Big(1-w\theta_0: \ESA_\ttS \rightarrow \ESA_{\ttS})\Big)\ptf\end{equation*}
Posons
\begin{equation*}\mathfrak{d}_\ttS= \ker\Big(1-w\theta_0: \ag_\ttS \rightarrow \ag_{\theta_0(\ttS)})\Big)\quad \hbox{et} \quad 
\mathfrak{d}_{\Pp}= \mathfrak{d}_\ttS \cap \ag_{\Pp}\ptf\end{equation*}
Soit $\mathfrak{e}_{\Pp}$ l'orthogonal de $\mathfrak{d}_{\Pp}$ dans $\ag_{\Pp}$. 
On note $V'\mapsto V'_d$, resp. $V'\mapsto V'_e$, la projection orthogonale de 
$\ag_{\Pp}= \mathfrak{d}_{\Pp}\oplus \mathfrak{e}_{\Pp}$ sur $\mathfrak{d}_{\Pp}$, 
resp. $\mathfrak{e}_{\Pp}$.
Posons
\begin{equation*}\mathfrak{d}_\ttS^{(\Pp)} = \mathfrak{d}_\ttS \cap (\ag_\ttS^{\Pp} \oplus \mathfrak{e}_{\Pp})\ptf\end{equation*}
On a la d\'ecomposition
\begin{equation*}\mathfrak{d}_\ttS = \mathfrak{d}_{\Pp} \oplus \mathfrak{d}_\ttS^{(\Pp)}\leqno{(1)}\end{equation*}
et la projection $\ag_\ttS \rightarrow \ag_\ttS^{\Pp},\mskip 2mu  V \mapsto V^{\Pp}$ est injective sur $\mathfrak{d}_\ttS^{(\Pp)}$. 
Posons
\begin{equation*}\ESD_{\Pp} \bydef \ESD_\ttS \cap \ag_{\Pp}= \ESA_\ttS \cap \mathfrak{d}_{\Pp}\vg\end{equation*}
et notons $\ESD_{\Pp}^\flat$ et $\ESD_\ttS^{(\Pp)}$ les projections orthogonales de $\ESD_\ttS$ 
sur $\mathfrak{d}_{\Pp}$ et $\mathfrak{d}_\ttS^{(\Pp)}$ 
pour la d\'ecomposition (1). 
On a l'inclusion $\ESD_{\Pp} \subset \ESD_{\Pp}^\flat$ (avec \'egalit\'e si $\Pp={_tS}$) et la suite exacte courte
\begin{equation*}0 \rightarrow \ESD_{\Pp}\rightarrow \ESD_\ttS \rightarrow \ESD_\ttS^{(\Pp)} \rightarrow 0\ptf\leqno{(2)}\end{equation*}
On d\'ecompose la somme $\sum_{V\in \ESD_\ttS}$ en une double somme 
\begin{equation*}\sum_{V_1\in \ESD_\ttS^{(\Pp)}}\quad\sum_{V\in \ESD_{\Pp}(V_1)}\end{equation*}o 
$\ESD_{\Pp}(V_1)\subset \ESD_\ttS$ est la fibre au-dessus de $V_1$ pour la suite exacte courte (2). 
L'expression $I^T_{(D)}(\Pp)$ se r\'ecrit
\begin{equation*}I^T_{(D)}(\Pp)=\sum_{Z\in\ESC_Q^\tG}\wt{\sigma}_Q^R(Z{-T}) 
\sum_{V_1 \in \ESD_\ttS^{(\Pp)}} \Gamma_\ttM^{\Pp}(V_1^{\Pp}\mskip -2mu ,\TT)\phi_e(Z,\brT{\Pp},V_1)\end{equation*}
avec
\begin{equation*}\phi_e(Z,X,V_1) \bydef \sum_{V\in \ESD_{\Pp}(V_1)}\phi(Z,X,V_{\Pp}) \ptf\end{equation*}
On a
\begin{equation*}\phi_e(Z,X,V_1) \ll \phi_e^\flat(Z,X) \bydef \sum_{V'\in \ESD_{\Pp}^\flat} \phi(Z,X,V')\ptf\end{equation*}
Observons que 
\begin{equation*}\phi(Z,X,V')=\phi(0,X-Z',V'-Z')\end{equation*}
o $Z'$ est un rel\`evement de $Z$ dans $\ESA_{\Pp}$. On en d\'eduit que les fonctions $\phi_e(Z,X,V_1)$ et 
$\phi_e^\flat(Z,X) $ ne d\'ependent que $Z_e$ et qu'elles sont \`a d\'ecroissance rapide en $Z_e$. 
D'autre part, puisque la projection $V\mapsto V^{\Pp}$ est injective sur 
$\ESD_\ttS^{(\Pp)}$, la somme 
\begin{equation*}\sum_{V_1 \in \ESD_\ttS^{(\Pp)}} \Gamma_\ttM^{\Pp}(V_1^{\Pp}\mskip -2mu ,\TT)\end{equation*}
est finie. D'o la majoration
\begin{equation*}I^T_{(D)}(\Pp)\ll
\sum_{Z\in\ESC_Q^\tG}\wt{\sigma}_Q^R(Z{-T})\phi_e^\flat(Z,\brT{\Pp})\ptf \end{equation*}
D'apr\`es \cite[13.6.(10)]{LW}, on a l'inclusion
\begin{equation*}\mathfrak{d}_\ttS\subset \ker (q_Q)\leqno{(3)}\end{equation*}
o $q_Q: \ag_0 \rightarrow \ag_Q^\tG$ est l'application d\'efinie en \ref{qqq}. 
Rappelons que cette application est l\'eg\`erement 
diff\'erente de celle de \cite[2.13]{LW} (au lieu de projeter sur $\ag_Q^G$, 
on projette ici sur $\ag_Q^\tG$). L'inclusion (3) 
entra"ne l'analogue de la majoration \cite[13.6.(9)]{LW}:
\begin{equation*}\lVert  (Z-T_Q)^\tG \rVert  \ll \lVert  (Z-T_Q)_e \rVert  \quad \hbox{pour tout $Z\in \ag_Q$ 
tel que $\wt{\sigma}_Q^R(Z-T)=1$}\ptf\leqno{(4)}\end{equation*}
On en d\'eduit que $\lVert  Z^\tG \rVert  \ll 1+ \lVert  Z_e \rVert $ pour tout $Z\in \ESA_Q$ 
tel que $\wt{\sigma}_Q^R(Z-T)=1$. Cela entra"ne la convergence de 
$I^T_{(D)}(\Pp)$ et ach\`eve la preuve de la convergence de (D).

\medskip
Consid\'erons maintenant l'expression (E). On voit comme ci-dessus qu'elle est major\'ee par
\begin{equation*}\sum_{\Pp\in \EuScript{F}^Q(\ttM)} I^T_{(E)}(\Pp)\end{equation*}
avec
\begin{equation*}I^T_{(E)}(\Pp)=\sum_{Z\in\ESC_Q^\tG}\wt{\sigma}_Q^R(Z{-T})
\sum_{V\in \ESD_\ttS} \big(1- \Kappa{\rho T}(V)\big) \Gamma_\ttM^{\Pp}(V^{\Pp}\mskip -2mu ,\TT)
\phi(Z,\brT{\Pp},V_{\Pp})\end{equation*}
soit encore
\begin{align*}
\lefteqn{
I^T_{(E)}(\Pp)=
\sum_{Z\in\ESC_Q^\tG}\wt{\sigma}_Q^R(Z{-T})
\sum_{V_1 \in \ESD_\ttS^{(\Pp)}}\Gamma_\ttM^{\Pp}(V_1^{\Pp}\mskip -2mu ,\TT)}\\
&& \hspace{3cm}\times
\sum_{V\in \ESD_{\Pp}(V_1)}\big(1- \Kappa{\rho T}(V)\big)\phi(Z,\brT{\Pp},V_{\Pp})\ptf
\end{align*}
Fixons $\rho'>0$, pour l'instant arbitraire. Pour all\'eger l'\'ecriture, posons
\begin{equation*}Z_T \bydef Z-T_Q\in \ag_Q\ptf\end{equation*}
Observons que \begin{equation*}\wt{\sigma}_Q^R(Z-T)= \wt{\sigma}_Q^R(Z_T) = \wt{\sigma}_Q^R(Z_T^\tG)\ptf\end{equation*}
On majore $I^T_{(E)}(\Pp)$ par
\begin{equation*}I^T_{(E),\geq}(\Pp)+ I^T_{(E),<}(\Pp)\end{equation*}
o $I^T_{(E),\geq}(\Pp)$, resp. $I^T_{(E),<}(\Pp)$, est l'expression obtenue en remplaant la fonction 
$\wt{\sigma}_Q^R(Z-T)$ par $\wt{\sigma}_Q^R(Z_T)(1-\kappa^{\rho'T}(Z_T))$, resp. 
$\wt{\sigma}_Q^R(Z_T)\kappa^{\rho'T}(Z_T)$, dans $I^T_{(E)}(\Pp)$. On commence par 
majorer $I^T_{(E),\geq}(\Pp)$. On peut choisir $\rho''>0$ tel que $(1-\kappa^{\rho' T}(Z_T))=1$ 
(c'est-\`a-dire $ \lVert  Z_T \rVert  > \rho' \lVert  T \rVert $) implique 
$ \lVert  Z_T^\tG \rVert  > \rho'' \lVert  T \rVert  $. Alors on a
\begin{equation*}I^T_{(E),\geq}(\Pp)\ll 
\sum_{Z\in\ESC_Q^\tG}\wt{\sigma}_Q^R(Z_T^\tG)(1-\kappa^{\rho''T}(Z_T^\tG))
\hspace{-0.2cm}\sum_{V_1 \in \ESD_\ttS^{(\Pp)}} \Gamma_\ttM^{\Pp}(V_1^{\Pp}\mskip -2mu ,\TT)
\phi_e(Z,\brT{\Pp},V_1)\ptf\end{equation*}
Pour tout $V\in \ESD_{\Pp}(V_1)$ la projection orthogonale $V_{\Pp\mskip -2mu ,e}$ de 
$V_{\Pp}$ sur $\mathfrak{e}_{\Pp}$ ne d\'epend que de $V_1$, et on la note $V_{1,e}$. 
D'apr\`es le lemme \ref{TMGMFvariante}, pour tout r\'eel $r$ on a une majoration
\begin{equation*}\phi_e(Z,X,V_1) \ll \Big(1 + \lVert  V_{1,e} -Z_{T,e}-X_e \rVert  \Big)^{-r}\leqno{(5)}\end{equation*}
o la constante implicite est absolue, c'est-\`a-dire ne d\'epend d'aucune variable. 
La constante implicite dans la majoration (4) est elle aussi absolue. 
Comme dans la preuve de 
\cite[13.6.3, page~203]{LW}\footnote{Voir toutefois les \textit{errata} \Err(xix) et \Err(xx) de l'Annexe \ref{Erratum}.}
on montre que l'on peut choisir $\rho''$ tel que la condition
\begin{equation*}\wt{\sigma}_Q^R(Z_T^\tG)(1-\kappa^{\rho''T}(Z_T^\tG))\Gamma_{_tM}^{\Pp}(V^{\Pp}\mskip -2mu ,\TT)=1\end{equation*}
entra"ne une majoration 
\begin{equation*}\lVert  Z_T^\tG \rVert  \ll \lVert  V_{1,e} - Z_{T,e}- \brT{\Pp\mskip -2mu ,e} \rVert \ptf\end{equation*}
Pour tout r\'eel $r$ on a donc une majoration
\begin{equation*}I^T_{(E),\geq}(\Pp) \ll
\sum_{Z\in\ESC_Q^\tG}(1-\kappa^{\rho''T}(Z_T^\tG))(1+\lVert  Z_T^\tG \rVert )^{-r}
\hspace{-0.2cm} \sum_{V_1 \in \ESD_\ttS^{(\Pp)}}\Gamma_\ttM^{\Pp}(V_1^{\Pp}\mskip -2mu ,\TT)\ptf\end{equation*}
La somme en $V_1$ est essentiellement major\'ee par $\lVert  T \rVert ^D$ pour un certain entier $D$, et la somme en $Z$ 
est essentiellement major\'ee par $\lVert  T \rVert ^{-r}$. D'o la majoration
\begin{equation*}I^T_{(E),\geq}(\Pp) \ll \bsd_0(T)^{-r}\ptf\end{equation*}
Traitons maintenant $I^T_{(E),>}(\Pp)$. Gr‰ce \`a la suite exacte courte
\begin{equation*}0 \rightarrow \ESD_{_tS}^{\Pp}\bydef \ESD_{_tS} \cap \mathfrak{d}_{_tS}^{(\Pp)}
 \rightarrow \ESD_{_tS} \rightarrow \ESD_{\Pp}^\flat \rightarrow 0 \vg\leqno{(6)}\end{equation*}
on peut d\'ecomposer la somme $\sum_{V\in \ESD_{tS}}$ en une double somme
\begin{equation*}\sum_{V'\in \ESD_{\Pp}^\flat} \sum_{V\in \ESD_{t_S}^{\Pp}(V')} \end{equation*}
o $\ESD_{t_S}^{\Pp}(V')$ est la fibre au-dessus de $V'$ dans $\ESD_{_tS}$ pour la suite exacte courte (6). 
On a donc
\begin{align*}
\lefteqn{
I^T_{(E),<}(\Pp)=
\sum_{Z\in\ESC_Q^\tG}\wt{\sigma}_Q^R(Z{-T})\kappa^{\rho' T}(Z_T)
\sum_{V'\in \ESD_{\Pp}^\flat}\phi(Z,\brT{\Pp},V')}\\
&& \hspace{3cm}\times
\sum_{V\in \ESD_{_tS}^{\Pp}(V')}(1- \Kappa{\rho T}(V))\Gamma_\ttM^{\Pp}(V^{\Pp}\mskip -2mu ,\TT)\ptf
\end{align*}
Comme dans la preuve de \cite[13.6.3]{LW}, il existe une constante $c_1>0$ 
telle que pour tout $V\in \ESD_{_tS}$ tel que $\Gamma_{_tM}^{\Pp}(V^{\Pp}\mskip -2mu ,\TT)=1$, on ait la majoration
$\lVert  V_1 \rVert  \leq c_1 \lVert  T \rVert  $ o $V_1= V-V_d$ est l'image de $V$ dans $\ESD_{_tS}^{(\Pp)}$. 
Si $\rho > c_1$, en ajoutant la condition 
$(1- \Kappa{\rho T}(V))=1$ c'est-\`a-dire $\rho \lVert  T \rVert  < \lVert  V \rVert  $, on obtient 
$ \lVert  V_d \rVert >(\rho -c_1) \lVert  T \rVert  $ c'est-\`a-dire $(1 - \kappa^{(\rho-c_1)T}(V_d))=1$. 
En particulier $V_d\neq 0$ et l'espace $\mathfrak{d}_{\Pp}$ n'est pas nul. 
Il existe $c_2>0$ tel que la condition $\kappa^{\rho 'T}(Z_T)=1$ c'est-\`a-dire $\lVert  Z_T \rVert  \leq \rho' \lVert  T \rVert $ 
entra"ne $ \lVert  Z_{T,d} + \brT{\Pp,d} \rVert  \leq c_2 \lVert  T \rVert $. 
En prenant $\rho > c_1 + c_2$, on obtient que la condition
\begin{equation*}\kappa^{\rho 'T}(Z_T) (1- \Kappa{\rho T}(V))\Gamma_{_tM}^{\Pp}(V^{\Pp}\mskip -2mu ,\TT)=1\end{equation*}
entra"ne l'in\'egalit\'e
\begin{equation*}\lVert  V_d - Z_{T,d} - \brT{\Pp\mskip -2mu ,d} \rVert 
 \geq (\rho - (c_1+c_2)) \lVert  T \rVert  > \Big(1 - \frac{c_2}{\rho-c_1}\Big) \lVert  V_d \rVert  \ptf\end{equation*}
Gr‰ce au lemme \ref{TMGMFvariante}, on en d\'eduit que pour tout r\'eel $r$ 
l'expression $I^T_{(E),<}(\Pp)$ est essentiellement major\'ee par
\begin{equation*}\sum_{Z\in\ESC_Q^\tG}\kappa^{\rho' T} (Z_T^\tG)
\hspace{-0.2cm} \sum_{V'\in \ESD_{_tS}^\flat}(1 + \lVert  V' \rVert )^{-r}(1- \kappa^{(\rho-c_1)T}(V'))
 \sum_{V_1\in \ESD_{_tS}^{(\Pp)}} \Gamma_\ttM^{\Pp}(V_1^{\Pp}\mskip -2mu ,\TT)\ptf\end{equation*}
Les sommes en $Z$ et en $V_1$ sont essentiellement major\'ees par $\lVert  T \rVert ^D$ 
pour un entier $D$ convenable, et pour tout r\'eel $r$ la somme sur $V'$ est essentiellement major\'ee par 
$\lVert  T \rVert ^{-r}$. D'ou une majoration
\begin{equation*}I^T_{(E),<}(\Pp) \ll \bsd_0(T)^{-r}\ptf\end{equation*}
qui, jointe \`a la majoration $I^T_{(E),\geq}(\Pp) \ll \bsd_0(T)^{-r}$, assure la convergence de l'expression (E) et 
l'assertion de (iii) la concernant.

Consid\'erons maintenant l'expression (A). Comme pour (D), on obtient qu'elle est essentiellement major\'ee par
\begin{equation*}\sum_{\Pp\in \EuScript{F}^{Q_0}({_tM})} I^T_{(A)}(\Pp)\end{equation*}
avec
\begin{align*}
\lefteqn{
I^T_{(A)}(\Pp)=
\sum_{Z\in\ESC_Q^\tG}
\wt{\sigma}_Q^R(Z{-T})\hspace{-5pt}
\sum_{X\in\mathcal{C}_F(Q,Q_0; T)}}\\
&& \hspace{3cm}\times
\sum_{V\in \ESD_\ttS} \big\vert \Gamma_{_tM}^{\Pp}(V^{\Pp}\mskip -2mu ,\TT- \XX) 
\big\vert \psi(H_Z^{T-X}, \brTX{\Pp},V_{\Pp}) 
\end{align*}
et
\begin{equation*}\psi(H,Y,V')\bydef \int_{\bs{\kopa}_S}
\big\vert \wh{\bsd}_{\Pp\mskip -2mu ,F}^{\mskip 2mu Q_0}(H,\kopa;Y,V';\nu)\big\vert \dd \kopa\ptf\end{equation*}
Ici $\TT - \XX$ est la famille orthogonale $(\brTX{\Pp})$. 
Elle est rationnelle si $T\in \ag_{0,\QM}$. Fixons un $\Pp\in \EuScript{F}^{Q_0}({_tM})$. 
Rappelons que $\mathfrak{e}_{\Pp}$ est l'orthogonal de 
$\mathfrak{d}_{\Pp}= \mathfrak{d}_{_tS} \cap \ag_{\Pp}$ dans $\ag_{\Pp}$, et qu'on 
a not\'e $V'\mapsto V'_e$ la projection orthogonale de $\ag_{\Pp}= 
\mathfrak{d}_{\Pp}\oplus \mathfrak{e}_{\Pp}$ sur $\mathfrak{e}_{\Pp}$. 
Comme pour (D), l'expression $I^T_{(A)}(\Pp)$ se r\'ecrit
\begin{align*}
\lefteqn{
I^T_{(A)}(\Pp)=\sum_{Z\in\ESC_Q^\tG}\wt{\sigma}_Q^R(Z{-T}) \sum_{X\in\mathcal{C}_F(Q,Q_0; T)}}\\
&& \hspace{3cm}\times
\sum_{V_1 \in \ESD_\ttS^{(\Pp)}}\big\vert \Gamma_\ttM^{\Pp}(V_1^{\Pp}\mskip -2mu ,\TT-\XX)\big\vert
\psi_e(H_Z^{T-X},\brTX{\Pp},V_1)
\end{align*}
avec
\begin{equation*}\psi_e(H,Y,V_1) \bydef \sum_{V\in \ESD_{\Pp}(V_1)}\psi(H,Y,V_{\Pp}) \ptf\end{equation*}
D'apr\`es le lemme \ref{TMGMFvariante}, pour tout r\'eel $r$ on a une majoration
\begin{equation*}\psi_e(H,Y,V_1) \ll \big(1 + \lVert  V_{1,e} - (H + Y^{Q_0})_e\rVert \big)^{-r}\ptf\end{equation*}
Pour $H=H_Z^{T-X}= Z +(T-X)_\Qo^Q$ et $Y = \brTX{\Pp}$, on a
\begin{equation*}H + Y^{Q_0}= Z + \brTX{\Pp}^Q = Z_{T-X} + \brTX{\Pp} \end{equation*}
avec $Z_{T-X}= Z- (T-X)_Q$. Comme $X$ appartient \`a $\mathcal{C}_F(Q,Q_0;T)\subset \ag_0^Q$, on a $Z_{T-X}= Z_T$. 
Notons $\mathfrak{d}_\Qo^\flat \subset \ag_\Qo$ l'image de $\mathfrak{d}_{_tS}$ par 
la projection $V \mapsto V_\Qo$, et soit $\mathfrak{h}$ l'orthogonal de $\mathfrak{d}_\Qo^\flat$ dans $\ag_\Qo$. Puisque
\begin{equation*}\mathfrak{d}_\Qo= \mathfrak{d}_{P'} \cap \ag_\Qo \subset \mathfrak{d}_\Qo^\flat\vg\end{equation*}
on a l'inclusion $\mathfrak{h} \subset \mathfrak{e}_{P'}$. 
Notons $\mathfrak{h}^\perp$ l'orthogonal de $\mathfrak{h}$ dans $\mathfrak{e}_{\Pp}$, et 
$V \mapsto V_h=V_{P'\mskip -2mu ,h}$ la projection orthogonale de
\begin{equation*}\ag_0= \ag_0^{\Pp} \oplus \mathfrak{d}_{\Pp} \oplus \mathfrak{h} \oplus \mathfrak{h}^\perp\end{equation*}
sur $\mathfrak{h}$. Pour $V\in \mathfrak{d}_{_tS}$, on a $V_h=0$. 
D'autre part puisque la projection $V \mapsto V_h$ se factorise 
\`a travers $V \mapsto V_\Qo$, on a $\brTX{\Pp\mskip -2mu ,h}= (T-X)_h = T_h -X_h$. 
Pour tout r\'eel $r$, on obtient une majoration
\begin{equation*}\psi_e(H_Z^{T-X},\brTX{P'},V_1) \ll \big(1 + \lVert  Z_{T,h} + T_h-X_h\rVert \big)^{-r}\ptf\end{equation*}
Or d'apr\`es \cite[13.6.(13), p.~204]{LW}, pour tout $Z\in \ESC_Q^\tG$ tel que 
$\wt{\sigma}_Q^R(Z_T)=1$ et tout $X\in \mathcal{C}(Q,Q_0)$, on a une majoration
\begin{equation*}\lVert  Z_T^\tG \rVert  + \lVert  X \rVert  \ll \lVert  Z_{T,h} + T_h-X_h\rVert \ptf \leqno{(7)}\end{equation*}
Comme $\lVert  Z^\tG \rVert  \ll 1 + \lVert  Z_T^\tG \rVert  $ (la constante implicite d\'ependant de $T$), pour tout r\'eel $r$ 
on obtient une majoration 
\begin{equation*}I^T_{(A)}(\Pp)\ll \sum_{Z\in\ESC_Q^\tG} \sum_{X\in\mathcal{C}_F(Q,Q_0; T)}
\big(1 + \lVert  Z^\tG \rVert  + \lVert  X \rVert \big)^{-r}
\sum_{V_1 \in \ESD_\ttS^{(\Pp)}}\big\vert \Gamma_\ttM^{\Pp}(V_1^{\Pp}\mskip -2mu ,\TT-\XX)\big\vert \ptf\end{equation*}
Puisque l'application $V_1 \mapsto V_1^{P'}$ est injective, la somme en $V_1$ est essentiellement major\'ee par 
\begin{equation*}\lVert  T \rVert ^D + (1 + \lVert  X \rVert  )^D\end{equation*}
pour un entier $D$ convenable. On en d\'eduit que pour pour tout r\'eel $r$, on a une majoration
\begin{equation*}I^T_{(A)}(\Pp)\ll \sum_{Z\in\ESC_Q^\tG} 
\sum_{X\in\mathcal{C}_F(Q,Q_0; T)}\lVert  T \rVert ^D (1+ \lVert  Z^\tG \rVert )^{-r} (1+ \lVert  X \rVert )^{-r}\ptf \leqno{(8)}\end{equation*}
Cela prouve la convergence de l'expression (A).

Quant aux deux expressions restantes ((B) et (C)), leur convergence se d\'eduit 
des raisonnements pr\'ec\'edents comme dans la preuve de 
du lemme \cite[13.6.3]{LW}. Idem pour la majoration du point (ii) de l'\'enonc\'e. 
Cela ach\`eve la preuve du lemme \ref{major}.\hfill \qed

\vskip2mm
\ni \textit{D\'emonstration du lemme \ref{remajo}.} 
Le sous-groupe parabolique \begin{equation*}S''\in \ESP^Q(\ttM)\smallsetminus \ESP^{Q_0}(\ttM)\end{equation*}\'etant fix\'e, 
on consid\`ere la transform\'ee de Fourier inverse
\begin{equation*}V \mapsto \wh{\bsd}(\chi; V,S'' ; \nu)\ptf\end{equation*}
C'est une fonction \`a d\'ecroissance rapide en $V\in \ESA_{_tM}$, uniform\'ement en $\chi$. Par cons\'equent la fonction
\begin{equation*}V \mapsto \xi(V)=\int_{\bs{\kopa}_S}\vert\wh{\bsd}(\chi; V,S'' ; \nu)\vert \dd \kopa\end{equation*}
sur $\ESA_{_tM}$ est encore \`a d\'ecroissance rapide, et on a une majoration
\begin{equation*}E^T_{S''}\ll \sum_{Z\in\ESC_Q^\tG}\wt{\sigma}_Q^R(Z-T)
\sum_{V\in \ESD_\ttS} \Kappa{\rho T}(V) 
\hspace{-0.3cm}\sum_{V_2\in \mathcal{C}_F^Q(-\ZTSs;S'')}\hspace{-0.5cm}\xi(V+V_2)\ptf\leqno{(9)}\end{equation*}
Rappelons que pour $Y\in \ag_\ttM^{Q}+\ESA_Q$, on a pos\'e
\begin{equation*}\mathcal{C}^{Q}_F(Y;S'') = \big(Y +\mathcal{C}^{Q}(S'')\big)\cap 
\ESA_\ttM\subset \ESA_\ttM^{Q}(Y_Q)\ptf\end{equation*}
On note $\mathcal{C}^Q(S'')_\Qo$ et $\mathcal{C}^{Q}_F(Y;S'')_\Qo$ les images (projections orthogonales) 
de $\mathcal{C}^Q(S'')$ et $\mathcal{C}_F^Q(Y;S'')$ dans 
$\ag_\Qo$. Par d\'efinition $\mathcal{C}^Q(S'')_\Qo$ est un sous-ensemble de $\ag_\Qo^Q$, 
$Y_\Qo$ appartient \`a $\ag_\Qo^Q + \ESA_Q$, et on a les inclusions
\begin{equation*}\mathcal{C}^{Q}_F(Y;S'')_\Qo \subset 
\big( Y_\Qo + \mathcal{C}^Q(S'')_\Qo\big)\cap \ESA_\Qo \subset \ESA_\Qo^Q(Y_Q)\ptf\end{equation*}
Pour $X\in \mathcal{C}^{Q}_F(Y;S'')_\Qo$, on note 
$\mathcal{C}^{Q}_F(Y;S'')_X^{Q_0}\subset \mathcal{C}^{Q}_F(Y;S'')$ la fibre au-dessus de $X$. Cette fibre est contenue 
dans $\ESA_{_tM}^{Q_0}(X)$. On peut 
donc d\'ecomposer la somme $\sum_{V_2\in \mathcal{C}_F^Q(-\ZTSs;S'')}$ en une double somme
\begin{equation*}\sum_{X\in \mathcal{C}_F^Q(-\ZTSs;S'')_\Qo}\sum_{V_2 \in \mathcal{C}_F^Q(-\ZTSs;S'')_X^{Q_0}}\end{equation*}
puis majorer brutalement la seconde somme par $\sum_{V_2 \in \ESA_{_tM}^{Q_0}(X)}$. On obtient
\begin{equation*}E^T_{S''}\ll \sum_{Z\in\ESC_Q^\tG}\wt{\sigma}_Q^R(Z-T)
\sum_{V\in \ESD_\ttS} \Kappa{\rho T}(V) 
\hspace{-0.3cm}\sum_{X\in \mathcal{C}_F^Q(-\ZTSs;S'')_\Qo}\hspace{-0.5cm}
\overline{\xi}(V_\Qo+X)\ptf\leqno{(10)}\end{equation*}
avec, pour $\overline{V}\in \ESA_\Qo$,
\begin{equation*}\overline{\xi}(\overline{V})= \sum_{V_2\in \ESA_{_tM}^{Q_0}(\overline{V})} \xi(V)\ptf\end{equation*}
La fonction $\overline{\xi}$ est \`a d\'ecroissance rapide en $\overline{V}\in \ESA_\Qo$. 

Notons $\mathfrak{k}$ le noyau de l'application $q_Q: \ag_0 \rightarrow \ag_Q^\tG$ d\'efinie en \ref{qqq}, et 
$\mathfrak{k}_t$ sa projection sur $\ag_{_tS}$ ou ce qui revient au mme 
(puisque $\ag_0^{_tS} \subset \ag_0^{Q_0} \subset \mathfrak{k}$) 
son intersection avec cet espace. On note $\mathfrak{f}$ l'orthogonal de $\mathfrak{k}_t$ dans $\ag_{_tS}$. Puisque 
$\mathfrak{k}= \mathfrak{k}_t \oplus \ag_0^{_tS}$, c'est aussi l'orthogonal de 
$\mathfrak{k}$ dans $\ag_0$. C'est donc un sous-espace de $\ag_\Qo$. Pour $V\in \ag_{_tS}=
 \mathfrak{k}_t \oplus \mathfrak{f}$, on note $V_f=V_{Q_0,f}$ la projection orthogonale de 
$V$ sur $\mathfrak{f}$. D'apr\`es l'inclusion (3), on a $\mathfrak{d}_{_tS}\subset \mathfrak{k}_t$, 
par cons\'equent $V_f=0$ pour tout $V\in \mathfrak{d}_{_tS}$. 
D'apr\`es (10), pour tout r\'eel $r$, on obtient une majoration
\begin{equation*}E^T_{S''}\ll \sum_{Z\in\ESC_Q^\tG}\wt{\sigma}_Q^R(Z_T)
\sum_{V\in \ESD_\ttS} \Kappa{\rho T}(V) 
\hspace{-0.3cm}\sum_{X\in \mathcal{C}_F^Q(-\ZTSs;S'')_\Qo}\hspace{-0.5cm}
(1 + \lVert  X_f \rVert )^{-r}\ptf\leqno{(11)}\end{equation*}
La somme sur $V$ est essentiellement major\'ee par $ \lVert  T\rVert ^D$ pour un $D$ convenable. L'\'el\'ement
$\ZTSs$ est par d\'efinition \'egal \`a $Z + \brT{S''}^Q= Z_T + \brT{S''}$. 
Tout \'el\'ement $X\in \mathcal{C}_F^Q(-\ZTSs;S'')_\Qo$ s'\'ecrit 
$X= - H_{Z;S''}^T + X'$ avec $X' \in \mathcal{C}^Q(S'')_\Qo$, et l'on a
\begin{equation*}X_f = -Z_{T,f} - T_f + X'_f\ptf\end{equation*}
D'apr\`es \cite[13.7~(4)]{LW}, pour $Z\in \ESA_Q$ tel que $\wt{\sigma}_Q^R(Z_T)=1$ et 
$X'\in \mathcal{C}_F^Q(S'')_\Qo$, on a une majoration absolue
\begin{equation*}\lVert  T \rVert  + \lVert Z_T^\tG \rVert  + \lVert X' \rVert  \ll 1 + \lVert  -Z_{T,f}- T_f + X'_f \rVert \ptf\end{equation*}
On en d\'eduit que pour $Z\in \ESA_Q$ tel que $\wt{\sigma}_Q^R(Z_T)=1$ et 
$X\in \mathcal{C}_F^Q(-\ZTSs;S'')_\Qo$, on a une majoration absolue
\begin{equation*}\lVert  T \rVert  + \lVert Z^\tG \rVert  + \lVert X\rVert  \ll 1 + \lVert  X_f \rVert \ptf\end{equation*}
D'apr\`es (11), pour tout r\'eel $r$, on obtient une majoration
\begin{equation*}E^T_{S''}\ll \lVert  T \rVert  ^{D-r}
\sum_{Z\in\ESC_Q^\tG}(1 + \lVert  Z^\tG\rVert )^{-r}
\hspace{-0.3cm}\sum_{X\in \mathcal{C}_F^Q(-\ZTSs;S'')_\Qo}\hspace{-0.5cm}
(1+ \lVert  X \rVert )^{-r}\ptf\leqno{(12)}\end{equation*}
Ceci est essentiellement major\'e par $\bsd_0(T)^{-r}$, ce qui d\'emontre le lemme. \hfill\qed


 \section{\'Elargissement des sommations}
\label{\'elargissement}

D'apr\`es \cite[13.8.1]{LW}, on a l'inclusion
\begin{equation*}\bfW^{Q'}(\ag_S,Q_0)\subset \bfW^G(\ag_S,Q)\ptf \leqno{(1)}\end{equation*}
On rel‰che les hypoth\`eses sur $Q$ et $R$: on suppose seulement $P_0\subset Q\subset R$ et 
on abandonne l'hypoth\`ese $\wt{\eta}(Q,R)\neq 0$. Pour $t\in W^G(\ag_S,Q)$, on pose
\begin{equation*}\wt{\eta}(Q,R;t)=\sum_\tP(-1)^{a_\tP - a_\tG}\end{equation*}
o la somme porte sur l'ensemble des $\tP\in\wt\ESP_\st$ tels que 
$Q\subset P\subset R$ et $t\in W^P$. Cet ensemble peut tre vide. 
S'il est non vide, alors il existe deux espaces paraboliques 
standards $\tP_1\subset \tP_2$ tel que ce soit l'ensemble des 
$\tP\in \wt\ESP_\st$ v\'erifiant $\tP_1 \subset \tP \subset \tP_2$ 
(on a alors $P_2 = R^-$). On en d\'eduit que $\wt{\eta}(Q,R;t)\neq 0$ 
si et seulement s'il existe un \textit{unique} $\tP\in \wt\ESP_\st$ 
tel que $Q\subset P\subset R$ et $t\in W^P$, auquel cas on a
\begin{equation*}\wt{\eta}(Q,R;t) = (-1)^{a_{\widetilde{P}} - a_\tG}\ptf\end{equation*}
Rappelons que pour $\Pp\in \EuScript{F}^Q(\ttM)$ et $w= s\theta_0(t)^{-1}$ on a pos\'e:
\begin{equation*}\mathfrak{d}_\ttS= \ker\Big(1-w\theta_0: \ag_\ttS \rightarrow \ag_{\theta_0(\ttS)})\Big)\ptf\end{equation*}
Rappelons aussi que pour $V\in \ag_{\Pp}=\mathfrak{d}_{\Pp}\oplus \mathfrak{e}_{\Pp}$, on a not\'e $V_e$ 
la projection orthogonale de $V$ sur $\mathfrak{e}_{\Pp}$. 

\begin{lemma}\label{majorations page 211}
On suppose $\wt{\eta}(Q,R;t)\neq 0$. 
Soient $Z\in \ESA_Q$ et $V\in \ESD_{\ttS}$
tels que \begin{equation*}\wt{\sigma}_Q^R(Z-T)\Gamma_\ttM^{\Pp}(V^{\Pp}\mskip -2mu ,\TT)=1\ptf\end{equation*}
Alors:
\begin{enumerate}[(i)]
\item $\lVert  (Z-T_Q)^\tG \rVert  \ll 1 + \lVert V_{\Pp\mskip -2mu ,e} -(Z-T_Q)_e - \brT{\Pp\mskip -2mu ,e} \rVert $;
\item et, si $t\notin \bfW^{Q'}(\ag_S ,Q_0)$,
\begin{equation*} \lVert  T \rVert  + \lVert  (Z-T_Q)^\tG\rVert  \ll 1 + \lVert V_{\Pp\mskip -2mu ,e} -(Z-T_Q)_e
 - \brT{\Pp\mskip -2mu ,e} \rVert \ptf\end{equation*} 

\end{enumerate}
\end{lemma}

\begin{proof}
Ce sont les analogues des assertions (3)(i) et (3)(ii) en bas de la page~211 de \cite{LW}, 
dont la preuve occupe les pages 212 \`a 215 de \textit{loc.~cit}. 
\end{proof}

\begin{proposition}\label{convprop}
Soient $t\in \bfW^G(\ag_S,Q)$ et $s\in \bfW^Q(\theta_0(\ag_S),t(\ag_S))$. 
On pose
\begin{equation*}\bsA^T_{s,t}=\sum_{Z\in\ESC_Q^\tG}\wt{\sigma}_Q^R(Z-T)\int_{\bsmu_S}
\brabso^{\mskip 2mu T,Q}_{ s, t} (Z;\mu)\vtt(\mu)\dd \mu\ptf\end{equation*}
On suppose $\wt{\eta}(Q,R;t)\neq 0$. 
\begin{enumerate}[(i)]
\item L'expression $\bsA^T_{s,t}$ est convergente dans l'ordre indiqu\'e.
\item Supposons $t\notin \bfW^{Q'}(\ag_S,Q_0)$. Alors pour tout r\'eel $r$, on a une majoration
\begin{equation*}\vert\bsA^T_{s,t}\vert\ll \bsd_0(T)^{-r}\ptf\end{equation*}
\end{enumerate}
\end{proposition}

\begin{proof} 
Puisque $\bs{A}_{s,t}= \vert \wh\bsbbc_S \vert^{-1}\sum_{\nu \in \bsESEsigma}
\bs{A}_{s,t,\nu}$ avec \begin{equation*}\bsA^T_{s,t,\nu}=\sum_{Z\in\ESC_Q^\tG}\wt{\sigma}_Q^R(Z-T)\int_{\bsmu_S}
\bso^{\mskip 2mu T,Q}_{ s, t} (Z;\mu;\nu)\vtt(\mu)\dd \mu\vg\end{equation*}il suffit de prouver les r\'esultats pour $\bs{A}_{s,t,\nu}^{T}$ 
avec $\nu \in \bsESEsigma$ fix\'e. Le lemme \ref{fourbso}~(ii) s'applique ici encore et 
on en d\'eduit l'analogue de \ref{[LW,13.6.2]}:
\begin{equation*}\bsA^T_{s,t,\nu}=\sum_{Z\in\ESC_Q^\tG}\wt{\sigma}_Q^R(Z{-T})
\bigg( \int_{\bs{\kopa}_S}
\sum_{V\in \ESD_\ttS} \wh{\bsd}_{\ttM,F}^{\;Q,T}(Z ,\kopa; V; \nu)\dd \kopa \bigg)\ptf\end{equation*}
L'expression est convergente dans l'ordre indiqu\'e. Il s'agit de prouver 
qu'elle est absolument convergente puis de la majorer lorsque $t\notin \bfW^{Q'}(\ag_S,Q_0)$.
On observe que dans la preuve de 
la convergence de l'expression (D) du lemme \ref{major},
ce n'est qu'\`a partir de la relation (3) que l'hypoth\`ese $t\in \bfW^{Q'}$ est utilis\'ee. 
On a donc ici aussi la majoration
\begin{equation*}\bsA^T_{s,t,\nu} \ll \sum_{\Pp\in \EuScript{F}^Q(\ttS)} I^T_{(D)}(\Pp)\end{equation*}
avec
\begin{equation*}I^T_{(D)}(\Pp)=\sum_{Z\in\ESC_Q^\tG}\wt{\sigma}_Q^R(Z{-T})
\sum_{V_1\in \ESD_\ttS^{(\Pp)}} \Gamma_\ttM^{\Pp}( V_1^{\Pp}\mskip -2mu ,\TT) \phi_e (Z,\brT{\Pp},V_1)\ptf\end{equation*}
D'apr\`es \ref{preuve-fin}.(5) et le lemme \ref{majorations page 211}, pour tout r\'eel $r$, 
en posant $C_r=1$ sans hypoth\`ese sur $t$ et $C_r = \lVert  T \rVert ^{-r}$ sous l'hypoth\`ese de (ii), on a 
une majoration
\begin{equation*}I^T_{(D)}(\Pp) \ll C_r \sum_{Z\in\ESC_Q^\tG} \big(1 + \lVert  (Z-T_Q)^\tG \rVert  \big)^{-r} 
\sum_{V_1\in \ESD_\ttS^{(\Pp)}} \Gamma_\ttM^{\Pp}( V_1^{\Pp}\mskip -2mu ,\TT)\end{equation*}
o la constante implicite est absolue. La somme en $V_1$ est essentiellement major\'ee par $\lVert  T \rVert ^D$ 
pour un certain entier $D$. La somme en $Z$ 
est convergente ce qui d\'emontre le point (i). Sous l'hypoth\`ese de (ii) on obtient 
$I^T_{(D)}(\Pp) \ll \lVert  T \rVert ^{-r}$ pour tout r\'eel $r$, ce qui d\'emontre (ii). 
\end{proof}

On pose
\begin{equation*}\bsA^T=\sum_{t\in \bfW^G(\ag_S,Q)}\wt{\eta}(Q,R;t)
\sum_{s\in \bfW^Q(\theta_0(\ag_S),t(\ag_S))}\bsA^T_{s,t}\ptf\end{equation*}

\begin{corollary}\label{lastcor} 
Pour tout r\'eel $r$, on a les majorations suivantes:
\begin{enumerate}[(i)]
\item Si $\wt{\eta}(Q,R)\neq 0$ alors
$\vert\wt{\eta}(Q,R)A^T -\bsA^T\vert\ll\bsd_0(T)^{-r}$.
\item Si $\wt{\eta}(Q,R)= 0$ alors $
\vert\bsA^T\vert \ll \bsd_0(T)^{-r} $.
\end{enumerate}
\end{corollary}


\begin{definition}\label{pose}
On consid\`ere $Q,\mskip 2mu S\in\ESP_\st$ tels que $S\subset Q' =\theta_0^{-1}(Q)$. Soient 
$t\in \bfW^G(\ag_S,Q)$, $s\in \bfW^Q(\theta_0(\ag_S), t(\ag_S))$, $Z\in \ESA_Q$, 
$\mu\in \bsmu_S$ et $\lambda\in \bsmu_{\theta_0(\ag_S)}$. Pour
$\sigma\in\Pi_\disc(M_S)$ on d\'efinit l'op\'erateur
\begin{align*}
\lefteqn{\bsO^{\mskip 2mu T,Q}_{s,t}(Z; S,\sigma;\lambda,\mu)
=\sum_{S''\in\ESP^Q(\ttM)}
\varepsilon_{S''}^{\mskip 2mu Q,\brT{S''}}(Z; s\lambda - t\mu)e^{\langle s\lambda - t\mu, \Y_{S''}\rangle}}\\
&&\hspace{3,8cm}\times 
\bfM(t,\mu)^{-1}\bfM_{S''\vert \ttS}(t\mu)^{-1}\bfM_{S''\vert \ttS}(s\lambda)
\bfM(s,\lambda)\ptf
\end{align*}
\end{definition}
C'est une fonction lisse de $\lambda$ et $\mu$. 
On a introduit en \ref{exisigma} l'ensemble $\bsESEsigma$ qui, s'il est non vide,
est un espace principal homog\`ene sous $\wh\bsbbc_M$. Pour $\mu\in \bsmu_S$, on pose:
\begin{equation*}\brabsO^{\mskip 2mu T,Q}_{s,t}(Z; S,\sigma;\mu)=
\vert \wh\bsbbc_S \vert^{-1}\sum_{\nu \in \bsESEsigma}\bfD_\nu\mskip 2mu 
\bsO^{\mskip 2mu T,Q}_{s,t}(Z; S,\sigma;\theta_0(\mu),\mu+\nu)\ptf\end{equation*}
La fonction $\mu\mapsto\brabsO^{\mskip 2mu T,Q}_{s,t}$ est lisse.
On rappelle que l'on a d\'efini en \ref{cot\'especb} une expression 
$\Jres^{\tG,T}= \Jres^{\tG,T}(f,\omega)$. Nous allons en introduire une variante. 
Pour all\'eger un peu les notations nous aurons recours au lemme suivant:

\begin{lemma}\label{spur}
Consid\'erons deux espaces pr\'e-hilbertiens $\bs{\mathcal E}\subset\bs{\mathcal F}$
o $\bs{\mathcal E}$ est un facteur direct et un op\'erateur $A: \bs{\mathcal E} \rightarrow \bs{\mathcal F}$
de rang fini.
On suppose que $\bs{\mathcal E}$ est muni d'une base (au sens alg\'ebrique) orthonormale $\Base$. 
L'expression \begin{equation*}\spur(A)=\sum_{\Psi\in\Base}\langle A\Psi,\Psi\rangle_{\bs{\mathcal F}}\end{equation*}
donn\'ee par une s\'erie convergente, est ind\'ependante du choix de la base.
Si de plus $A$ stabilise $\bs{\mathcal E}$, c'est-\`a-dire si $A$ est un endomorphisme de $\bs{\mathcal E}$,
alors
\begin{equation*}\spur(A)=\trace(A)\ptf\end{equation*}
\end{lemma}

\begin{proof} Puis que $\bs{\mathcal E}$ est un facteur direct, 
tout  $\Phi\in {\bs{\mathcal F}}$ peut s'\'ecrire $\Phi=\Phi_1+\Phi_2$
avec $\Phi_1\in \bs{\mathcal E}$ et $\Phi_2$ est orthogonal \`a $\bs{\mathcal{E}}$.
 La s\'erie
\begin{equation*}\sum_{\Psi\in\Base}\langle \Phi,\Psi\rangle_{\bs{\mathcal F}}=\sum_{\Psi\in\Base}\langle \Phi_1,\Psi\rangle_{\bs{\mathcal E}}\end{equation*}
se r\'eduit \`a une somme finie et il en est de mme de la s\'erie 
d\'efinissant $\spur(A)$
puisque $A$ est de rang fini.  L'ind\'ependance du choix de la base se ram\`ene au cas de la dimension finie.
\end{proof}

Nous appliquerons ce lemme au cas o $\bs{\mathcal E}=\Automd(\bsX_S,\sigma)$ et o $\bs{\mathcal F}$ 
est l'espace engendr\'e par $\Automd(\bsX_S,\sigma)$ et
 les $\Automd(\bsX_S,\tu(\sigma\otimes\omega)\star\nu)$
pour $\nu\in\bsESEsigma$. Nous poserons
\begin{equation*}\spurs(A)=\sum_{\Psi\in\Base_S(\sigma)}\langle A\Psi,\Psi\rangle_S\ptf\end{equation*}

\begin{proposition}\label{\'ecriture finale}
On consid\`ere l'expression 
\begin{align*}
\lefteqn{\bJres_\spec^{\tG,T}(f,\omega)=\sum_{\substack{Q,R\in\ESP_\st\\ Q\subset R}}
\sum_{S\in\ESP_\st^{Q'}}\frac{1}{n^{Q'}(S)}
\sum_{\bsigma\in\bsPi_\disc(M_S)}\cMsig
}\\&&\hspace{1cm}\times\sum_{t\in \bfW^G(\ag_S,Q)} 
\sum_{s\in \bfW^Q(\theta_0(\ag_S),t(\ag_S))}\wt{\eta}(Q,R;t) 
\sum_{Z\in\ESC_Q^\tG}\wt{\sigma}_Q^R(Z-T_Q)\\
&&\hspace{4cm}\times\int_{\bsmu_S}
\spurs\mskip -2mu \bigg(\brabsO_{s,t}^{\mskip 2mu T,Q}(Z; S,\sigma;\mu)\tRho_{S,\sigma,\mu}(f,\omega)\bigg)
\dd \mu\ptf
\end{align*}
\begin{enumerate}[(i)]
\item L'expression $\bJres_\spec^{\tG,T}=\bJres_\spec^{\tG,T}(f,\omega)$ est convergente.
\item Pour tout r\'eel $r$, on a une majoration
\begin{equation*}\big\vert \Jres_\spec^{\tG,T} -\bJres_\spec^{\tG,T}\big\vert\ll\bsd_0(T)^{-r}\ptf\end{equation*}
\end{enumerate}
\end{proposition} 

\begin{proof} On observe que,
d'apr\`es \ref{thmfinitude}, l'op\'erateur $\tRho_{S,\sigma,\mu}(f,\omega)$ est de rang fini;
l'assertion (i) 
r\'esulte alors de \ref{convprop} (en utilisant la remarque \ref{somfi}). Compte tenu de l'expression pour $\Jres^{\tG,T}$
donn\'ee en \ref{defat}, on voit que
la majoration (ii) r\'esulte de la conjonction des in\'egalit\'es \ref{pureb}, \ref{crux} et \ref{lastcor}.
\end{proof}

 \chapter{Formules explicites}\label{formules explicites} 

 \section{Combinatoire finale: \'etape 1}\label{combfin}
Soient $S,\mskip 2mu S_0,\mskip 2mu Q\in\ESP_\st$ tels que $S_0=\theta_0(S)\subset Q$. 
On a donc, comme pr\'ec\'edemment, $S\subset\theta_0^{-1}(Q)=Q'$. 
Soit aussi $\tu\in \bfW^\tG(\ag_S,\ag_S)$. D'apr\`es \cite[14.1.1]{LW}, 
$\tu$ s'\'ecrit d'une mani\`ere et d'une seule sous la forme
\begin{equation*}\tu= u\theta_0,\quad u=t^{-1}s\in \bfW^G(\ag_{S_0},\ag_S)\end{equation*}
avec $t\in \bfW^G(\ag_S,Q)$ et $s\in \bfW^Q(\ag_{S_0}, t(\ag_S))$. Soit $S''\in\ESP_\st$ tel que 
$\ag_{S''}= t(\ag_S)$, et soit $S_1=t^{-1}(S'')$. On a donc $S''\subset Q$. 
On consid\`ere des param\`etres $\mu$ et $\nu$ dans $\bsmu_S$, et l'on pose
$\Lambda=\tu\mu -\nu$. Rappelons que l'on a pos\'e en \ref{l'\'el\'ement T_0}
\begin{equation*}\Y_u=T_0-u^{-1}T_0= \bfH_0(w_u^{-1})\ptf\end{equation*}
On introduit une variante tordue:
\begin{equation*}\Y_{\tu}=\theta_0^{-1}\Y_u =\theta_0^{-1}T_0 -\tu^{-1}T_0\end{equation*}
ainsi que le scalaire
\begin{equation*}a_S(\mu,\tu)= 
e^{\langle\mu +\rho_S,\Y_{\tu}\rangle}= e^{\langle \theta_0\mu+\rho_{S_0},Y_u\rangle}\ptf\end{equation*}
On pose enfin
\begin{equation*}M=M_S\vgq Q_1= t^{-1}Q = t^{-1}sQ =\tu Q'\in\ESF(M)\ptf\end{equation*}
Soit $H\in\ESA_{Q_1}$. Pour $S_1\in\ESP^{Q_1}(M)$, 
on a d\'efini en \ref{GMspec} et \ref{SETR}
\begin{equation*}\ESM({\YY};S,\mu;\Lambda, S_1)= e^{\langle\Lambda,\Y_{S_1}\rangle}
\bfM_{S_1\vert S}(\mu)^{-1} \bfM_{S_1\vert S}(\mu+\Lambda)\end{equation*}
et
\begin{equation*}\ESM_{M,F}^{\mskip 2mu Q_1,T}(H,{\YY};S,\mu;\Lambda)=
\sum_{S_1\in\ESP^{Q_1}(M)}\varepsilon_{S_1}^{\mskip 2mu Q_1,\brT{S_1}}(H;\Lambda)\mskip 2mu 
\ESM({\YY};S,\mu;\Lambda, S_1)\ptf\end{equation*}
Pour $\mu\in \bsmu_{M}$ et $\lambda\in \bsmu_{\theta_0(M)}$,
on a d\'efini en 
 \ref{pose} l'op\'erateur
\begin{align*}
\lefteqn{\bsO^{\mskip 2mu T,Q}_{s,t}(tH;S,\sigma;\lambda,\mu)
=\sum_{S''\in\ESP^Q(\ttM)}
\varepsilon_{S''}^{\mskip 2mu Q,\brT{S''}}(tH; s\lambda - t\mu)e^{\langle s\lambda - t\mu,\Y_{S''}\rangle}
}\\&&\hspace{4,3cm}\times \hspace{0,3cm}
\bfM(t,\mu)^{-1}\bfM_{S''\vert \ttS}(t\mu)^{-1}\bfM_{S''\vert \ttS}(s\lambda)\bfM(s,\lambda)\ptf
\end{align*}
%

\begin{proposition}\label{aths}
Soient $\nu \in \bsmu_M$ et $\Lambda = u\lambda -\mu$. Avec les notations de \ref{rhotu} on a
\begin{align*}
\lefteqn{
\bsO^{\mskip 2mu T,Q}_{s,t}(tH; S,\sigma;\lambda,\mu+\nu)\tRho_{S,\sigma,\mu}(f,\omega)}\\
&&=\frac{a_S(\theta_0^{-1}\lambda,\tu)}{ a_S(\mu,\tu)}
\ESM_{M,F}^{\mskip 2mu Q_1,T}(H,{\YY};S,\mu+\nu;\Lambda-\nu)
\bfM_{S\vert\tu S}(\mu+\Lambda)\Rho_{S,\sigma,\mu}(\tu,f,\omega)\ptf
\end{align*}
\end{proposition}

\begin{proof}
Posons $\mu'= \mu+\nu$ et $\Lambda'= u\lambda - \mu'=\Lambda -\nu$. Puisque
\begin{equation*}\Lambda'= t^{-1}(s\lambda - t\mu')\vg\end{equation*}
pour $S''\in\ESP^Q(\ttM)$ et $S_1= t^{-1}(S'')\in\ESP^{Q_1}(M)$, on a
\begin{equation*}\varepsilon_{S''}^{\mskip 2mu Q,\brT{S''}}(tH; s\lambda - t\mu')e^{\langle s\lambda - t\mu'\mskip -2mu ,\Y_{S''}\rangle}= 
\varepsilon_{S_1}^{\mskip 2mu Q_1, t^{-1}\brT{S''}}(H;\Lambda')e^{\langle \Lambda'\mskip -2mu ,t^{-1}\Y_{S''}\rangle}\ptf\end{equation*}
Or on a $t^{-1}\brT{S''}= \brT{S_1}$, et
\begin{equation*}t^{-1}\Y_{S''}=\Y_{S_1} -\Y_{t^{-1}(S)}\end{equation*}
o $\Y_{t^{-1}(S)}$ est la projection de $\Y_t= T_0 - t^{-1}T_0$ sur $\ag_M$. 
Gr‰ce \`a \cite[6.1.1, 14.1.2]{LW}, on obtient, comme dans la preuve de \cite[14.1.3]{LW}, que
\begin{equation*}\bfM(t,\mu')^{-1}\bfM_{S''\vert \ttS}(t\mu')^{-1}\bfM_{S''\vert \ttS}(s\lambda)
\bfM(s,\lambda)\tRho_{S,\sigma,\mu}(f,\omega)\end{equation*}
est \'egal \`a
\begin{equation*}\frac{a_S(\theta_0^{-1}\lambda,\tu)}{ a_S(\mu,\tu)}e^{\langle \Lambda'\mskip -2mu , Y_{t^{-1}(S)}\rangle}
\bfM_{S_1\vert S}(\mu')^{-1}
\bfM_{S_1\vert S}(\mu'+\Lambda') \bfM_{S\vert \tu S}(\mu'+\Lambda')
\Rho_{S,\sigma,\mu}(\tu,f,\omega)\ptf\end{equation*}
D'o le r\'esultat. 
\end{proof}
Pour $\nu \in \bsmu_M$ et $\Lambda=(\tu-1)\mu$, on pose 
\begin{align*}
\lefteqn{
\bsA^{T,Q_1}_M(H;\sigma,\tu,\mu;\nu)}\\
&& \hspace{0,8cm}
=\spurs\bigg(\bfD_\nu\mskip 2mu \ESM_{M,F}^{\mskip 2mu Q_1,T}(H,{\YY};S,\mu+\nu;\Lambda-\nu)
\bfM_{S\vert\tu S}(\mu +\Lambda)\Rho_{S,\sigma,\mu}(\tu,f,\omega) \bigg)
\ptf
\end{align*}
Pour $\nu\in \bsESEsigma$, 
l'espace principal homog\`ene sous $\wh\bsbbc_M$ introduit en \ref{exisigma}, 
on a $\nu\vert_{\ESB_\tG}=0$ et $\Lambda\vert_{\ESB_\tG}=0$. L'expression 
\begin{equation*}\bsA^{T,Q_1}_M(H;\sigma,\tu,\mu;\nu)\end{equation*}
ne d\'epend donc que de l'image de 
$H$ dans $\ESC_{Q_1}=\ESB_\tG\backslash \ESA_{Q_1}$.
 On pose alors
\begin{equation*}[\bsA]^{T,Q_1}_M(H;\sigma,\tu,\mu)=
\vert\wh\bsbbc_M\vert^{-1}\sum_{\nu \in \bsESEsigma}
\bsA^{T,Q_1}_M(H;\sigma,\tu,\mu;\nu)\ptf\end{equation*}
Rappelons que $S$ est l'unique \'el\'ement de $\ESP_\st$ tel que $M_S=M$.
\begin{lemma}\label{indepconj}
On a l'\'egalit\'e
\begin{enumerate}[(i)]
\item 
$[\bsA]^{T,Q_1}_M(H;\sigma,\tu,\mu) =
\spurs\mskip -2mu \bigg(\brabsO^{\mskip 2mu T,Q}_{s,t}(tH; S,\sigma;\mu)\tRho_{S,\sigma,\mu}(f,\omega)\bigg)\ptf$
\item Cette expression 
est invariante si l'on remplace $M$, $Q_1$, $S$, $\tu$ et $H$ par leurs conjugu\'es sous l'action 
d'un \'el\'ement $w\in \bfW^G$, et simultan\'ement $\sigma$ et $\mu$ par $w\sigma= \sigma \circ \mathrm{Int}_{w}^{-1}$ et 
$w\mu = \mu \circ \mathrm{Int}_{w}^{-1}$.
\end{enumerate}
\end{lemma}
\begin{proof}
Pour (i), puisque $\ESM_{M,F}^{\mskip 2mu Q_1,T}(H,{\YY};S,\mu;\Lambda)$ est lisse pour les valeurs 
imaginaires pures de $\mu$ et $\Lambda$, on peut prendre $\lambda=\theta_0(\mu)$ dans la proposition \ref{aths}. 
Pour (ii), rappelons que $w$ d\'efinit un op\'erateur
\begin{equation*}\bs{w}: \Automd(\bsX_S,\sigma)\rightarrow \Autom(\bsX_{ wS},w\sigma).\end{equation*}
Cet op\'erateur est une isom\'etrie et (ii) est une cons\'equence des \'equations fonctionnelles satisfaites 
par les op\'erateurs d'entrelacement. 
\end{proof}

On observe que $\brT{Q_1}= t^{-1}(T_Q)$ puisque $Q_1= t^{-1}Q$. 
Soient aussi $R\in\ESP_\st$ tel que $Q'\subset R$, et $R_1= t^{-1}R$. On pose 
\begin{equation*}\bJres_{ M,Q_1}^{T,R_1}(\sigma,\tu) 
=\sum_{H\in\ESC_{Q_1}^\tG}\widetilde{\sigma}_{Q_1}^{R_1}(H-\brT{Q_1})
\int_{\bsmu_M}[\bsA]^{T,Q_1}_M(H;\sigma,\tu,\mu)\dd \mu\ptf \end{equation*}
On a d\'efini cette expression pour $M=M_S$ avec $S\in\ESP_\st$. 
Plus g\'en\'eralement, elle est bien d\'efinie pour 
tout $M\in\ESL$, tout $S\in\ESP$ tel que $M_S=M$, 
tout $Q_1$ et tout $R_1$ dans $\ESP$ tels que $M\subset Q_1\subset R_1$, 
et tout $\tu\in \bfW^\tG(\ag_M,\ag_M)$. 
On introduit alors l'expression
\begin{equation*}\bJres_M^{\tG,T}(\sigma,f,\omega,\tu)\index{JrestGTM@$\bJres_M^{\tG,T}$}
=\sum_{\substack{Q_1,\mskip 2mu  R_1\in\ESP\\ M\subset Q_1\subset R_1}}
\wt{\eta}(Q_1,R_1;u)\bJres_{ M,Q_1}^{T,R_1}(\sigma,\tu)\end{equation*}
et on a l'analogue de la proposition \cite[14.1.5]{LW}:

\begin{proposition}\label{rejres}
L'expression $\bJres_\spec^{\tG,T}(f,\omega)$ de \ref{\'ecriture finale} se r\'ecrit
\begin{equation*}\bJres_\spec^{\tG,T}(f,\omega)
=\sum_{M\in\ESL^G/\bfW^G}\frac{1}{w^G(M)}
\sum_{\bsigma\in\bsPi_\disc(M)}\cMsig\mskip -5mu
\sum_{\tu\in \bfW^\tG(\ag_M,\ag_M)}\mskip -5mu
\bJres_M^{\tG,T}(\sigma,f,\omega,\tu)\ptf\end{equation*}
\end{proposition}
Maintenant, on d\'efinit une $(G,M)$-famille $\bs{c} = \bs{c}(\sigma,\tu,\mu;\nu)$ par
\begin{equation*}\bsc(\Lambda,S_1)=
\spurs\bigg(\bfD_\nu\mskip 2mu \ESM({\YY};S,\mu+\nu;\Lambda, S_1)
\bfM_{S\vert\tu S}(\mu+\nu+\Lambda) 
\Rho_{S,\sigma,\mu}(\tu,f,\omega)\bigg)\mskip -2mu \ptf\end{equation*}
Elle est p\'eriodique car la famille orthogonale $\YY$ est enti\`ere. Pour 
 $\Lambda = (\tu-1)\mu$, on a donc
\begin{equation*}\bsA^{T,Q_1}_M(H;\sigma,\tu,\mu;\nu)= \bscMF^{Q_1,T}(H;\Lambda-\nu)\ptf\end{equation*}
D'apr\`es \ref{invfour}, il existe une fonction \`a d\'ecroissance rapide
\begin{equation*}\UU\mapsto\varphi(\UU)=\varphi(\sigma,\tu,\mu;\nu;\UU)\end{equation*}
sur $\ESH_M$ telle que $\bsc=\bsc_{\varphi}$.
D'apr\`es la formule d'inversion de Fourier \ref{formule d'inversion pour les (G,M)-familles}, 
on a
\begin{equation*}\bsA^{T,Q_1}_M(H;\sigma,\tu,\mu;\nu)= 
\sum_{\UU\in\ESH_M}\varphi(\UU)\gammaMF^{Q_1,T}(H,\UU;\Lambda-\nu)\end{equation*}
avec 
\begin{equation*}\gammaMF^{Q_1,T}(H,\UU;\Lambda-\nu)=\gammaMF^{Q_1,\UU(T)}(H+U_{Q_1}; \Lambda-\nu)\ptf\end{equation*}

\begin{lemma}\label{supportTF}
Le support de
$\varphi$ est contenu dans le r\'eseau \begin{equation*}\ESH_M^G= \ESH_M \cap \HH_M^G\end{equation*}du sous-espace 
$\HH_M^G$ de $\HH_M$ form\'e des familles orthogonales $\UU=(U_P)$ telles que $U_G=0$. 
\end{lemma}
\begin{proof}
Il suffit de voir que la $(G,M)$-famille p\'eriodique $\bsc$ est invariante par translations 
par les \'el\'ements de $\widehat{\ag}_G$ (en fait de $\bsmu_G$). Par d\'efinition
\begin{equation*}\bsc(\Lambda,S_1)=
\spurs\mskip -2mu \bigg( \bfD_\nu\mskip 2mu \ESM({\YY};S,\mu+\nu;\Lambda, S_1)
\bfM_{S\vert\tu S}(\mu+\nu+\Lambda) 
\Rho_{S,\sigma,\mu}(\tu,f,\omega)\bigg)\end{equation*}
avec
\begin{equation*}\ESM({\YY};S,\mu;\Lambda, S_1)= e^{\langle\Lambda,\Y_{S_1}\rangle}
\bfM_{S_1\vert S}(\mu)^{-1} \bfM_{S_1\vert S}(\mu+\Lambda)\ptf\end{equation*}
D'autre part on a
\begin{equation*}\bsc(\Lambda,S_1)= \int_{\bsmu_G} \bsc(\Lambda+\nu,S_1) \dd \nu= 
\sum_{\UU \in \ESH_M} \bigg( \int_{\bsmu_G} e^{\langle \nu,U_G\rangle} \dd \nu\bigg)
e^{\langle \Lambda,U_{S_1} \rangle} \varphi(\UU)\ptf\end{equation*}
On en d\'eduit que
\begin{equation*}\varphi(\UU)= \bigg( \int_{\bsmu_G} e^{\langle \nu,U_G\rangle}\dd \nu \bigg) \varphi(\UU)\ptf\end{equation*}
Cela prouve le lemme.\end{proof}

Conform\'ement \`a nos conventions on pose 
$\ESC_M^\tG = \ESB_\tG\backslash \ESA_M$.
Soit $\tL\in\ESL^\tG$ le sous-ensemble de Levi minimal contenant $M\tu$. En particulier
\begin{equation*}\tu\in \bfW^\tL(\ag_M,\ag_M)\ptf\end{equation*}

\begin{proposition}\label{LW,14.1.7}
On a l'\'egalit\'e
\begin{align*}
\lefteqn{\bJres_M^{\tG,T}(\sigma,f,\omega,\tu)=\vert\wh\bsbbc_M\vert^{-1}
\sum_{\nu \in \bsESEsigma} \sum_{H\in\ESC_M^\tG}
e^{-\langle \nu, Y_{\tu}+H\rangle} }\\&&\hspace{3cm}
 \times\sum_{\UU\in\ESH_M}\int_{\bsmu_{M}} e^{\langle (\tu -1)\mu,H\rangle} 
\Gamma_\tL^\tG(H,\UU(T)){\varphi}(\sigma,\tu,\mu;\nu;\UU) \dd \mu.
\end{align*}
\end{proposition}

\begin{proof}La preuve ci-apr\`es reprend pour l'essentiel les argument de \cite[14.1.7]{LW}.
On rappelle que par d\'efinition on a
\begin{equation*}\bJres_{ M,Q_1}^{T,R_1}(\sigma,\tu)=\vert\wh\bsbbc_M\vert^{-1}\sum_{\nu \in \bsESEsigma}
e^{-\langle \nu, Y_{\tu}\rangle}\bJres_{ M,Q_1}^{T,R_1}(\sigma,\tu;\nu)\end{equation*}
avec
\begin{equation*}\bJres_{ M,Q_1}^{T,R_1}(\sigma,\tu;\nu)=
\sum_{H_1\in\ESC_{Q_1}^\tG}\widetilde{\sigma}_{Q_1}^{R_1}(H_1-\brT{Q_1})
\int_{\bsmu_M}\bsA^{T,Q_1}_M(H_1;\sigma,\tu,\mu+\nu;\nu)\dd \mu\ptf\end{equation*}
Fixons $\nu \in \bsESEsigma$ et posons $\varphi= \varphi(\sigma,\tu,\mu;\nu)$. 
Pour $H_1\in \ESA_{Q_1}$ et $\UU\in \ESH_M$, on a (par d\'efinition)
\begin{equation*}\gammaMF^{Q_1,T}(H_1,\UU;\Lambda)= \sum_{H\in \ESA_M^{Q_1}(H_1+ U_{Q_1})}
\Gamma_M^{Q_1}(H,\UU(T))e^{\langle\Lambda,H\rangle}\ptf\end{equation*}
Par cons\'equent, en rappelant que $\bsc$ est la $(G,M)$-famille $\bsc_{\varphi}=\bs{c}(\sigma,\tu,\mu;\nu)$, on a
\begin{equation*}\bscMF^{Q_1,T}(H_1;\Lambda)=\sum_{\UU\in\ESH_M}
\sum_{H\in\ESA_M^{Q_1}(H_1 +U_{Q_1})}\varphi(\UU)
\Gamma_M^{Q_1}(H,\UU(T))e^{\langle\Lambda,H\rangle}\ptf\leqno{(1)}\end{equation*}
Pour $H\in \ESA_M^{Q_1}(H_1 +U_{Q_1})$, on a $H_{Q_1}= H_1+U_{Q_1}$, 
et il existe une constante $c>0$ (ind\'ependante de $\UU$) 
telle que si $\Gamma_M^{Q_1}(H,\UU(T))\neq 0$, on ait 
\begin{equation*}\lVert  H^{Q_1}\rVert \leq c\sup_{P\in\ESP^{Q_1}(M)}\lVert  (U_{P}+\brT{P})^{Q_1}\rVert \ptf\end{equation*}
Par cons\'equent la somme 
\begin{equation*}\sum_{H\in\ESA_M^{Q_1}(H_1+U_{Q_1})}\vert\Gamma_M^{Q_1}(H,\UU(T))\vert\end{equation*}
est finie, et puisque $\varphi$ est \`a d\'ecroissance rapide sur $\ESH_M$, on en d\'eduit que l'expression (1) 
est absolument convergente. En prenant $\Lambda = (\tu -1)\mu -\nu$, on obtient que
\begin{align*}
\lefteqn{\int_{\bsmu_S}\bsA^{T,Q_1}_M(H_1;\sigma,\tu,\mu;\nu) \dd \mu=
\sum_{\UU\in\ESH_M^G}
\sum_{H\in\ESA_M^{Q_1}(H_1 +U_{Q_1})}}\\
&&\hspace{3cm}
\times 
\int_{\bsmu_{M}}\varphi(\sigma,\tu,\mu;\nu;\UU)
\Gamma_M^{Q_1}(H,\UU(T))e^{\langle (\tu -1)\mu - \nu,H\rangle}\dd \mu.
\end{align*}
On a donc
\begin{equation*}\bJres_{ M,Q_1}^{T,R_1}(\sigma,\tu;\nu)
=\sum_{H_1\in\ESC_{Q_1}^\tG}\sum_{\UU\in\ESH_M}
\sum_{H\in\ESA_M^{Q_1}(H_1 +U_{Q_1})}\int_{\bsmu_{M}}\bs{G}_{Q_1}^{R_1}(H,\mu,\nu,\UU)\dd \mu\end{equation*}
avec
\begin{equation*}\bs{G}_{Q_1}^{R_1}(H,\mu,\nu ,\UU)
= e^{\langle (\tu-1)\mu -\nu,H\rangle}\widetilde{\sigma}_{Q_1}^{R_1}(H- \UU(T))
\Gamma_M^{Q_1}(H,\UU(T))\varphi(\sigma,\tu,\mu;\nu;\UU)\ptf\end{equation*}
L'expression
\begin{equation*}\sum_{\UU\in\ESH_M}\sum_{H\in\ESA_M^{Q_1}(H_1 +U_{Q_1})}
\int_{\bsmu_{M}}\bs{G}_{Q_1}^{R_1}(H,\mu,\nu,\UU)\dd \mu\leqno{(2)}\end{equation*}
est absolument convergente. Pour chaque $\UU\in \ESH_M$, puisque $(U_S)_{Q_1}= U_{Q_1}$, 
on a $\ESA_M^{Q_1}(H_1+U_{Q_1})= \ESA_M^{Q_1}(H_1)+ U_S$. 
L'expression (2) est donc \'egale \`a
\begin{equation*}\sum_{\UU\in\ESH_M}\sum_{H\in\ESA_M^{Q_1}(H_1)}
\int_{\bsmu_{M}}\bs{G}_{Q_1}^{R_1}(H+U_S,\mu,\nu,\UU)\dd \mu\end{equation*}
et $\bJres_{ M,Q_1}^{T,R_1}(\sigma,\tu;\nu)$, c'est-\`a-dire la somme sur $H_1\in \ESC_{Q_1}^\tG$ 
des expressions (2), se r\'ecrit
\begin{equation*}\sum_{H\in \ESC_M^\tG}
\sum_{\UU\in\ESH_M}\int_{\bsmu_{M}}\bs{G}_{Q_1}^{R_1}(H+U_S,\mu,\nu,\UU)\dd \mu\ptf\leqno{(3)}\end{equation*}
Pour cela, il suffit de d\'emontrer la convergence absolue d'une somme it\'er\'ee de la forme
\begin{equation*}\sum_{H\in \ESC_M^\tG}
\sum_{\UU\in\ESH_M}\widetilde{\sigma}_{Q_1}^{R_1}(H-\brT{Q_1})\Gamma_M^{Q_1}(H+U_S,\UU(T))
\psi ((\tu^*-1)(H+U_S),\UU)\end{equation*}
o $\tu^*$ est le transpos\'e de $\tu$ (identifi\'e \`a $\tu^{-1}$) et
\begin{equation*}\psi(X,\UU)=\int_{\bsmu_{M}}e^{\langle \mu,X\rangle}\varphi(\sigma,\tu,\mu;\nu;\UU)\dd \mu\ptf\end{equation*}
D'apr\`es ce qui pr\'ec\`ede, c'est-\`a-dire la convergence absolue de l'expression (2), pour $H_1\in\ESA_{Q_1}$, l'expression
\begin{equation*}\sum_{H\in\ESA_M^{Q_1}(H_1)}\sum_{\UU\in\ESH_M}\Gamma_M^{Q_1}(H+U_S,\UU(T))
\psi ((\tu^*-1)(H+U_S),\UU)\end{equation*}
est absolument convergente. De plus sa valeur en $H_1$ est donn\'ee par $\xi((\tu^*-1)H_1)$ pour une fonction 
$\xi$ sur $\ESA_{Q_1}$ qui est la transform\'ee anti-Laplace d'une fonction lisse sur 
$\bsmu_{Q_1}=\widehat{\ESA}_{Q_1}$.
Il suffit donc d'\'etablir la convergence absolue de la somme
\begin{equation*}\sum_{H_1\in\ESC_{Q_1}^\tG}\widetilde{\sigma}_{Q_1}^{R_1}(H_1-\brT{Q_1})\xi((\tu^*-1)H_1)\end{equation*}
qui r\'esulte de \cite[2.12.2]{LW} (voir aussi \ref{qqq}). 
On peut donc effectuer le changement de variable $H\mapsto H-U_S$ dans (3). On obtient que
\begin{equation*}\bJres_{ M,Q_1}^{T,R_1}(\sigma,\tu;\nu)= 
\sum_{H\in\ESC_M^\tG}\sum_{\UU\in\ESH_M^G} 
\bigg(\int_{\bsmu_{M}}\bs{G}_{M,Q_1}^{T,R_1}(H,\mu,\nu,\UU)\dd \mu\bigg)\ptf\end{equation*}
On a
\begin{equation*}\bJres_M^{\tG,T}(\sigma,f,\omega,\tu)=\sum_{\substack{Q_1,\mskip 2mu  R_1\in\ESP\\ M\subset Q_1\subset R_1}}
\wt{\eta}(Q_1,R_1;u)\bJres_{ M,Q_1}^{T,R_1}(\sigma,\tu)\end{equation*}
avec
\begin{equation*}\wt{\eta}(Q_1,R_1;u)\bJres_{ M,Q_1}^{T,R_1}(\sigma,\tu)=
\sum_{\substack{\tP\in\wt\ESP,\mskip 2mu \tu\in \bfW^\tP\\ Q_1\subset P\subset R_1}}
(-1)^{a_\tP -a_\tG}\bJres_{ M,Q_1}^{T,R_1}(\sigma,\tu)\ptf\end{equation*}
Donc 
\begin{align*}
\lefteqn{\bJres^{\tG,T}_{ M}(\sigma,f,\omega,\tu)= \vert\wh\bsbbc_M\vert^{-1}\sum_{\nu \in \wh{\bsbbc}_M}
e^{-\langle \nu, Y_{\tu}\rangle} 
\sum_{H\in\ESC_M^\tG}\sum_{\UU\in\ESH_M^G}}\\
&&\times\sum_{\substack{Q_1,\mskip 2mu  R_1\in\ESP\\ M\subset Q_1\subset R_1}}
\sum_{\substack{\tP\in\wt\ESP,\mskip 2mu \tu\in \bfW^\tP\\ 
Q_1\subset P\subset R_1}}(-1)^{a_\tP -a_\tG}
\widetilde{\sigma}_{Q_1}^{R_1}(H -\UU(T)_{Q_1})\Gamma_M^{Q_1}(H,\UU(T))\\
&&\times\int_{\bsmu_{M}}e^{\langle (\tu -1)\mu-\nu,H\rangle}
\varphi(\sigma,\tu,\mu;\nu;\UU)\dd \mu.
\end{align*}
La double somme \begin{equation*}\sum_{\substack{Q_1,\mskip 2mu  R_1\in\ESP\\ M\subset Q_1\subset R_1}}
\sum_{\substack{\tP\in\wt\ESP,\mskip 2mu \tu\in \bfW^\tP\\ Q_1\subset P\subset R_1}}\end{equation*}
s'\'ecrit aussi 
\begin{equation*}\sum_{\substack{\tP\in\wt\ESP,\mskip 2mu \tu\in \bfW^\tP\\ 
M\subset P}}\sum_{\substack{Q_1,\mskip 2mu  R_1\in\ESP\\ M\subset Q_1\subset P\subset R_1}}=
\sum_{\tP\in\ESF(\tL)}\sum_{\substack{Q_1,\mskip 2mu  R_1\in\ESP\\ M\subset Q_1\subset P\subset R_1}}\end{equation*}
et d'apr\`es \cite[2.11.5]{LW}, la double somme
\begin{equation*}\sum_{\tP\in\ESF(\tL)}\sum_{\substack{Q_1,\mskip 2mu  R_1\in\ESP\\ M\subset Q_1\subset P\subset R_1}}
(-1)^{a_\tP -a_\tG}\widetilde{\sigma}_{Q_1}^{R_1}(H -\UU(T)_{Q_1})\Gamma_M^{Q_1}(H,\UU(T))\end{equation*}
est \'egale \`a
\begin{equation*}\sum_{\tP\in\ESF(\tL)}\sum_{\substack{Q_1\in\ESP\\ M\subset Q_1\subset P}}
(-1)^{a_\tP -a_\tG}\hat{\tau}_\tP(H-\UU(T)_{P})
\tau_{Q_1}^P(H-\UU(T)_{Q_1})\Gamma_M^{Q_1}(H,\UU(T))\ptf \leqno{(4)}\end{equation*}
D'apr\`es \cite[1.8.4.(3)]{LW} et \cite[2.9.3]{LW}, cette double somme (4) est \'egale \`a
\begin{equation*}\sum_{\tP\in\ESF(\tL)}
(-1)^{a_\tP -a_\tG}\hat{\tau}_\tP(H-\UU(T)_{P})=\Gamma_\tL^\tG(H,\UU(T))\ptf\end{equation*}
Cela ach\`eve la preuve la proposition. \end{proof}

 \section{Combinatoire finale: \'etape 2}\label{combfin2}

Ce paragraphe n'a pas d'\'equivalent pour les corps de nombres: 
il s'agit, \`a l'aide du lemme d'inversion de Fourier \ref{[W,3.20]}, 
de remplacer la somme sur les $H\in \ESC_M^\tG$ dans la 
proposition \ref{LW,14.1.7} par une somme sur les $\tZ \in \wh{\bsbbc}_G$.

Rappelons que 
$\tL\in\ESL^\tG$ est le sous-espace de Levi minimal (dans $\tG$) contenant $M\tu$. 
L'espace $\ag_\tL$ 
est le sous-espace de $\ag_M$ form\'e des points fixes sous $\tu= u\theta_0$. 
On note $\ag_M^\tL$ l'orthogonal de $\ag_\tL$ dans $\ag_M$:
\begin{equation*}\ag_M =\ag_\tL\oplus \ag_M^\tL\quad\hbox{et l'on a}\quad
\ag_M^\tL= (\tu -1)\ag_M\ptf\end{equation*}
Par d\'efinition, $\ESA_\tL$ est l'image et $\ESA_M^\tL$ le noyau 
de la projection de $\ESA_M$ sur $\ag_\tL$. On a donc une suite exacte
\begin{equation*}0 \rightarrow \ESA_M^\tL \rightarrow \ESA_M \rightarrow \ESA_\tL\rightarrow 0\ptf\end{equation*}
La dualit\'e de Pontryagin fournit alors la suite exacte
\begin{equation*}0 \rightarrow \bsmu_\tL \rightarrow \bsmu_M \rightarrow \bsmu_M^\tL \rightarrow 0\ptf\end{equation*}
Conform\'ement \`a nos conventions, 
les groupes compacts $\bsmu_M$, $\bsmu_\tL$ et $\bsmu_M^\tL$
sont munis de la mesure de Haar qui leur donne le volume $1$.
Consid\'erons le morphisme entre groupes de Lie ab\'eliens compacts connexes
\begin{equation*}\phi: \bsmu_M^\tL\times\bsmu_\tL\to\bsmu_M\quad\hbox{d\'efini par}\quad
\phi(\dmu,\lambda)=(\tu -1)\mu-\lambda\end{equation*}
o $\mu\in\bsmu_M$ est un rel\`evement de $\dmu\in\bsmu_M^\tL$. 

\begin{lemma}\label{ku} Le morphisme  $\phi$ est surjectif et son noyau  
$\mathfrak{K}_{\tu}$ est un groupe fini dont le cardinal est \'egal \`a la valeur absolue du jacobien de $\phi$:
\begin{equation*}\mid \mathrm{Jac}(\phi)\mid =\vert \mathfrak{K}_{\tu}\vert = \vert \det(\tu-1\mskip 2mu \vert\mskip 2mu  \ag_M^\tL)\vert\ptf\end{equation*}
\end{lemma} 

\begin{proof}
Le morphisme $\phi$ induit un isomorphisme pour les alg\`ebres de Lie
car les espaces tangents \`a l'origine, de l'image de $(\tu-1)$ et de $\bsmu_\tL$, 
sont des suppl\'ementaires orthogonaux 
(on identifie $\tu$ et son transpos\'e inverse et on utilise 
que $\tu^{-1}(\tu -1)= -(\tu -1)$). 
\end{proof}

Pour $\nu\in \bsmu_M$, on note
\begin{equation*}\xi_{\tu,\nu}:\bsmu_M \rightarrow \bsmu_M\qquad
\hbox{l'application}\qquad\mu\mapsto (\tu-1)\mu-\nu\ptf\end{equation*}
Consid\'erons le sous-ensemble des $\mu\in\bsmu_M$ dont l'image par 
$\xi_{\tu,\nu}$ appartient \`a $\bsmu_\tL$:
\begin{equation*}\bmMtunu= \{\mu\in\bsmu_M\mskip 2mu \mid \mskip 2mu (\tu-1)\mu-\nu\in\bsmu_\tL\}\ptf\end{equation*}
C'est une union finie de translat\'es de $\bsmu_\tL$. En effet, 
son quotient 
\begin{equation*}\bsmu_{M,\tu}^\tL(\nu)= \bmMtunu/\bsmu_\tL\end{equation*}est un espace principal homog\`ene 
sous le groupe $\bsmu_{M,\tu}^\tL(0)\simeq \mathfrak{K}_\tu$. On a donc
\begin{equation*}\vert\bsmu_{M,\tu}^\tL(\nu)\vert =\vert \mathfrak{K}_{\tu}\vert \ptf\end{equation*}
On munit $\bmMtunu$ de la mesure $\bsmu_\tL$-invariante induite par celle sur $\bsmu_\tL$.

\begin{lemma}\label{[W,3.20]} Pour $X\in \ESA_\tL$, on note $\ESA_M^\tL(X)$ 
l'ensemble des $H\in \ESA_M$ tels que $H_\tL=X$. 
Soient $\nu\in \bsmu_M$ et $\psi$ une fonction lisse sur $\bsmu_M$. On a l'identit\'e
\begin{equation*}\sum_{H\in\ESA_M^\tL(X)}\bigg(\int_{\bsmu_M} 
e^{\langle (\tu -1)\mu - \nu,H\rangle}\psi (\mu)\dd \mu\bigg)=\vert \mathfrak{K}_{\tu}\vert\mun 
\int_{\bmMtunu}e^{\langle \xi_{\tu,\nu}(\mu),X\rangle}\psi(\mu) \dd \mu\ptf\end{equation*}
\end{lemma} 
\begin{proof}
C'est imm\'ediat par inversion de Fourier (cf. \cite[3.20]{W}).
\end{proof}

Rappelons que
\begin{equation*}\bsESEsigma\bydef
\{\nu\in \bsmu_M\mskip 2mu \vert\mskip 2mu  \xi_{\tu(\omega\otimes \sigma)}\star\nu\vert_{\ESB_M}= \xi_\sigma\}\ptf\end{equation*}
En particulier, tout $\nu\in\bsESEsigma$ est trivial sur $\ESB_\tG$, 
et donc, pour tout $\mu \in \bmMtunu$, le 
caract\`ere $\xi_{\tu,\nu}(\mu)$ de $\ESA_\tL$ est lui aussi trivial sur $\ESB_\tG$.
Soit $\tZ\in \ESA_\tG$. Pour $\nu,\mskip 2mu \mu \in \bsmu_M$ et $\Lambda=(\tu-1)\mu$, on pose
\begin{align*}
\lefteqn{
\bsA^{T,\tG}_\tL(\tZ;\sigma,\tu,\mu;\nu)}\\&& \hspace{1cm} =
\spurs\bigg(\bfD_\nu\mskip 2mu \ESM_{\tL,F}^{\tG,T}(\tZ,{\YY};S,\mu+\nu;\Lambda-\nu)
\bfM_{S\vert\tu S}(\mu +\Lambda)\Rho_{S,\sigma,\mu}(\tu,f,\omega)\bigg)
\ptf
\end{align*} 
Pour $\nu\in \bsESEsigma$, puisque $\Lambda-\nu$ est trivial sur $\ESB_\tG$, cette expression ne d\'epend que de l'image de 
$\tZ$ dans $\bsbbc_\tG = \ESB_\tG \backslash \ESA_\tG$.

\begin{proposition}\label{propcombfin}\label{admnu}
On a
\begin{equation*}\bJres^{\tG,T}_{ M}(\sigma,f,\omega,\tu)=
\vert\wh\bsbbc_M\vert^{-1}\vert \mathfrak{K}_{\tu}\vert\mun\sum_{\nu \in \bsESEsigma}
\sum_{\tZ\in \bsbbc_\tG}
\int_{\bmMtunu}\bsA^{T,\tG}_\tL(\tZ;\sigma,\tu,\mu;\nu)\dd \mu\ptf\end{equation*}
\end{proposition}

\begin{proof}
D'apr\`es \ref{LW,14.1.7}:
\begin{align*}
\lefteqn{\bJres^{\tG,T}_{ M}(\sigma,f,\omega,\tu)=
\vert\wh\bsbbc_M\vert^{-1}\sum_{\nu \in \bsESEsigma} \sum_{H\in\ESC_M^\tG}
e^{-\langle \nu, Y_{\tu}+H\rangle}
 }\\&&\hspace{2cm} \times
 \sum_{\UU\in\ESH_M}\int_{\bsmu_{M}} e^{\langle (\tu -1)\mu,H\rangle} 
\Gamma_\tL^\tG(H,\UU(T)){\varphi}(\sigma,\tu,\mu;\nu;\UU) \dd \mu.
\end{align*}
On d\'ecompose la somme sur les ${H\in\ESC_M^\tG}$ en une double somme sur 
${X\in\ESC_\tL^\tG}$ pr\'ec\'ed\'ee de la somme en ${H\in\ESA_M^\tL(X)}$
o
$\ESC_\tL^\tG\bydef \ESB_\tG\backslash \ESA_\tL$
et on obtient
\begin{align*}
\lefteqn{\bJres^{\tG,T}_{ M}(\sigma,f,\omega,\tu)={\vert\wh\bsbbc_M\vert^{-1}}
\sum_{\nu\in \bsESEsigma }\sum_{X\in\ESC_\tL^\tG}
\sum_{\UU\in\ESH_M}}\\&&\hspace{2cm}\times
\sum_{H\in\ESA_M^\tL(X)}\int_{\bsmu_M} 
e^{\langle (\tu -1)\mu-\nu,H\rangle} 
\Gamma_\tL^\tG(H,\UU(T))\varphi(\sigma,\tu,\mu;\nu;\UU) \dd\mu\ptf
\end{align*}
En utilisant \ref{[W,3.20]} on voit que
\begin{align*}
\lefteqn{\bJres^{\tG,T}_{ M}(\sigma,f,\omega,\tu)=
\vert\wh\bsbbc_M\vert^{-1}\vert \mathfrak{K}_{\tu}\vert\mun 
\sum_{\nu\in \bsESEsigma}\sum_{X\in\ESC_\tL^\tG}
}\\ &&\hspace{2,5cm}\times 
\int_{\bmMtunu} \sum_{\UU\in\ESH_M}
e^{\langle \xi_{\tu,\nu}(\mu),X\rangle}\Gamma_\tL^\tG(X,\UU(T))
\varphi(\sigma,\tu,\mu;\nu;\UU) \dd \mu.
\end{align*}
On d\'ecompose la somme sur $X\in \ESC_\tL^\tG$ 
en une double somme sur ${\tZ\in \bsbbc_\tG}$ pr\'ec\'ed\'ee de la somme sur
$X\in \ESA_\tL^\tG(\tZ)$.
Pour $\nu\in \bsESEsigma$, $\tZ\in \bsbbc_\tG$ et $\mu\in \bmMtunu$, on a
\begin{equation*}\sum_{X\in \ESA_\tL^\tG(\tZ)} 
e^{\langle \xi_{\tu,\nu}(\mu),X\rangle}\Gamma_\tL^\tG(X;\UU(T))=
\gamma_{\tL,F}^{\tG,\UU(T)}(\tZ; \xi_{\tu,\nu}(\mu))\ptf\end{equation*}
Rappelons que la somme sur $\ESH_M$ est en fait une somme sur $\ESH_M^G$ (cf. \ref{supportTF}). 
D'apr\`es 
la formule d'inversion de Fourier pour les $(\tG,\tL)$-familles 
(\ref{inversion (cas tordu)}~(1)), 
pour $\mu\in \bmMtunu$, 
on a
\begin{equation*}\sum_{\UU\in \ESH_M^G}
\gamma_{\tL,F}^{\tG,\UU(T)}(\tZ;\xi_{\tu,\nu}(\mu))\varphi(\sigma,\tu,\mu;\nu;\UU) 
= \bsc_{\tL,F}^{\tG,T}(\tZ;\xi_{\tu,\nu}(\mu))\end{equation*}
o $\bsc$ est la $(G,M)$-famille d\'efinie par la fonction \`a d\'ecroissance 
rapide sur $\ESH_M$:
$\UU\mapsto\varphi(\sigma,\tu,\mu ;\nu;\UU)$.
Puisque $\Lambda-\nu=\xi_{\tu,\nu}(\mu)$ 
le terme $\bsc_{\tL,F}^{\tG,T}(\tZ;\xi_{\tu,\nu}(\mu))$ est \'egal \`a l'expression
$\bsA^{T,\tG}_\tL(\tZ;\sigma,\tu,\mu;\nu)$.
\end{proof}

On rappelle que l'on est parti de la formule \ref{cot\'espec}
pour $\Jres^{\tG,T}_\spec (f,\omega)$
 r\'ecrite sous la forme \ref{cot\'especb}. Puis on a introduit une variante $\bJres_\spec^{\tG,T}(f,\omega)$ 
 en \ref{\'ecriture finale}~(i) que l'on a d\'ecompos\'e en une somme 
de termes $\bJres^{\tG,T}_{ M}(\sigma,f,\omega,\tu)$ (cf. \ref{rejres})
pour lesquels on obtient une nouvelle expression en \ref{propcombfin}. 
Ces formules ont \'et\'e obtenues pour $T\in \ag_0$ dans le translat\'e d'un c™ne ouvert non vide. 
Pour $T\in \ag_{0,\QM}$, nous pouvons maintenant les comparer. 
\begin{proposition}\label{polasymp}
On a l'\'egalit\'e de fonctions dans $\mathrm{PolExp}$:
\begin{equation*}\Jres^{\tG,T}_\spec(f,\omega)=\sum_{M\in\ESL^G/\bfW^G}\frac{1}{w^G(M)}
\sum_{\bsigma\in\bsPi_\disc(M)}\cMsig\mskip -5mu
\sum_{\tu\in \bfW^\tG(\ag_M,\ag_M)}\mskip -5mu
\bJres^{\tG,T}_{ M}(\sigma,f,\omega,\tu)\ptf\end{equation*}
\end{proposition}
\begin{proof} Les deux membres de cette \'equation sont, comme fonctions
de $T\in \ag_{0,\QM}$, des \'el\'ements de {\nobreak PolExp}: on invoque \ref{propform} pour $\Jres^{\tG,T}$
et \ref{lissgm} pour les $\bJres^{\tG,T}_{ M}$ 
vu leur d\'efinition en \ref{propcombfin}. Compte tenu de la d\'ecomposition
\ref{rejres}, la majoration \ref{\'ecriture finale}~(ii) montre que les deux membres ne diff\`erent que
par une quantit\'e qui tend vers z\'ero lorsque $T$ tend vers l'infini dans un c™ne ouvert.
Leur \'egalit\'e r\'esulte alors de \ref{unicit\'e}.
\end{proof}
Cette proposition peut tre vue comme
 l'analogue de \cite[14.1.11]{LW} (raffin\'e en \cite[14.2.1]{LW}). Toutefois 
 c'est une \'egalit\'e de fonctions dans PolExp et non une \'egalit\'e de polyn™mes 
 et par ailleurs les sommes sur $\nu$ et $\tZ$ viennent compliquer 
 l'expression \ref{propcombfin}; une nouvelle \'etape est n\'ecessaire ici.

 \section{Le polyn™me limite au point central}\label{le polyn™me limite}
Les preuves dans cette section sont inspir\'ees de celle du lemme de \cite[3.23]{W}.
Pour $S\in \ESP(M)$ et $\mu\in \ag_{M,\CM}^*$, rappelons que $\ESM(S,\mu)$ est la $(G,M)$-famille spectrale 
\`a valeurs op\'erateurs d\'efinie par 
\begin{equation*}\ESM(S,\mu;\Lambda, S_1)= 
\bfM_{S_1\vert S}(\mu)^{-1} \bfM_{S_1\vert S}(\mu+\Lambda)\end{equation*}
et qu'on lui a associ\'e l'op\'erateur 
\begin{equation*}\ESM_\tL^\tG(S,\mu;\Lambda)=
\sum_{\tS_1\in \wt\ESP(\tL)}\epsilon_{\tS_1}^\tG(\Lambda)\mskip 2mu 
\ESM(S,\mu;\Lambda, S_1)\ptf\end{equation*}
C'est une variante de l'op\'erateur $\ESM_{M,F}^{G,T}(Z,\YY;P,\lambda;\Lambda)$  
introduit en \ref{lissgm}. Ici, comme dans le cas des corps de nombres,
on int\`egre les fonctions caract\'eristiques 
de c™nes sur l'espace vectoriel plut™t que de sommer sur le r\'eseau (cf. \ref{epspq}). 
Les fonctions $\epsilon_{\tS_1}^\tG(\Lambda)$ sont d\'efinies via le choix d'une mesure de Haar sur l'espace vectoriel 
$\ag_{\tL}^{\tG}$ mais le produit 
\begin{equation*}\mathrm{vol}(\ESA_{\tL}^{\tG}\backslash \ag_{\tL}^{\tG})^{-1} \ESM_{\tL}^{\tG}(S,\mu;\Lambda)\end{equation*}
est ind\'ependant de ce choix.
Pour all\'eger l'\'ecriture, posons (comme dans \cite[14.1]{LW})
\begin{equation*}\ESM_\tL^\tG(S,\mu)\bydef\ESM_\tL^\tG(S,\mu;0)\ptf\end{equation*}

\begin{definition}\label{defA}
On pose
\begin{equation*}\bsA^{\tG}_\tL(\sigma,\tu,\mu;\nu,\Lambda)=\spurs\bigg(\bfD_\nu\mskip 2mu \ESM_\tL^\tG(S,\mu+\nu;\Lambda)
\bfM_{S\vert\tu S}(\mu+\nu)\Rho_{S,\sigma,\mu}(\tu,f,\omega)\bigg)\ptf\end{equation*}
\end{definition}

Fixons $\nu\in \bsESEsigma$,
$\tZ\in \ESA_\tG$ et $\mu\in \bmMtunu$.
Consid\'erons la fonction (introduite juste avant \ref{admnu})
\begin{equation*}T\mapsto a_{\tZ,\mu}(T)=\bsA^{T,\tG}_\tL(\tZ;\sigma,\tu,\mu;\nu)\end{equation*}
sur $\ag_{0,\QM}$. D'apr\`es l'analogue tordu de \ref{passlim}, 
cette fonction appartient \`a $\mathrm{PolExp}$. 
Pour tout r\'eseau $\ESR$ de $\ag_{0,\QM}$, on peut \'ecrire 
\begin{equation*}a_{\tZ,\mu}(T)=\sum_{\tau}e^{\langle \tau\mskip -2mu ,T\rangle} p_{\ESR,\tau}(\azx T)\com{pour tout} T\in \ESR\end{equation*}
o la somme porte sur un sous-ensemble fini de caract\`eres $\tau\in \wh\ESR$. 
D'apr\`es l'analogue tordu de \ref{passlim},  la limite
\begin{equation*}\bs{\alpha}_{\tZ,\mu}\bydef \lim_{k\rightarrow +\infty}p_{\ESR_k,0} (\azx T_0)\end{equation*}
existe et  est ind\'ependante du r\'eseau $\ESR$. 

\begin{lemma}\label{limite}On a
\begin{equation*}\bs{\alpha}_{\tZ,\mu}=\left\{\begin{array}{ll}\mathrm{vol}
 \Big(\ESA_\tL^\tG\backslash \ag_\tL^\tG\Big)\mun
e^{\langle \xi_{\tu,\nu}(\mu),\tZ\rangle}\bsA^{\tG}_\tL(\sigma,\tu,\mu;\nu,\xi_{\tu,\nu}(\mu))& 
\hbox{si $\xi_{\tu,\nu}(\mu)\in \wh{\bsbbc}_\tG$}\\
\mskip 5mu 0 & \hbox{sinon}
\end{array}\right..\end{equation*}
\end{lemma}
\begin{proof}
La $(G,M)$-famille p\'eriodique $\bs{c}(\sigma,\tu,\mu;\nu)$ introduite en \ref{combfin} est de la forme
$\bs{c}(\sigma,\tu,\mu;\nu) = \bsd(\mathfrak{Y})$
pour une $(G,M)$-famille p\'eriodique $\bsd$ donn\'ee par
\begin{equation*}\bsd(\Lambda,S_1)=
\spurs\Big( \bfD_\nu\mskip 2mu \ESM(S,\mu+\nu;\Lambda,S_1)
\bfM_{S\vert\tu S}(\mu+\nu+\Lambda)\Rho_{S,\sigma,\mu}(\tu,f,\omega)\Big)\end{equation*}
et 
\begin{equation*}\bsd({\YY};\Lambda,S_1)=e^{\langle \Lambda, Y_{S_1}\rangle}\bsd(\Lambda,S_1)\ptf\end{equation*}
On a donc
\begin{equation*}a_{\tZ,\mu}(T)=\bsd_{\tL,F}^{\tG,T}(\mathfrak{Y};\xi_{\tu,\nu}(\mu))\ptf\end{equation*}
Rappelons que $\xi_{\tu,\nu}(\mu)$ est un \'el\'ement de $ \bsmu_\tL$, et plus 
pr\'ecis\'ement un caract\`ere de $\ESC_\tL^\tG =\ESB_\tG \backslash \ESA_\tL$. 
Le dual de Pontryagin $\wh{\ESC}_\tL^\tG$ s'ins\`ere dans la suite exacte courte
\begin{equation*}0 \rightarrow \wh{\bsbbc}_\tG \rightarrow \wh{\ESC}_\tL^\tG 
\rightarrow \bsmu_\tL^\tG \rightarrow 0\ptf\end{equation*}
L'image de $\xi_{\tu,\nu}(\mu)$ dans $\bsmu_\tL^\tG$ est nulle
 si et seulement si $\xi_{\tu,\nu}(\mu) \in \wh{\bsbbc}_\tG$.
On d\'eduit alors de \ref{passlim} que
 \begin{equation*}\bs{\alpha}_{\tZ,\mu}= \left\{\begin{array}{ll} 
\mathrm{vol} \Big(\ESA_\tL^\tG\backslash \ag_\tL^\tG \Big)^{-1}\hspace{-0.2cm} e^{\langle \xi_{\tu,\nu}(\mu),\tZ\rangle}
\bsd_\tL^\tG(\mathfrak{Y}+ \mathfrak{T}_0;\xi_{\tu,\nu}(\mu)) & \hbox{si $\xi_{\tu,\nu}(\mu)\in \wh{\bsbbc}_\tG$}\\
0 & \hbox{sinon}\\
\end{array}
\right.\end{equation*}
o, pour $\Lambda\in \wh{\ag}_\tL$ en dehors des murs,
\begin{equation*}\bsd_\tL^\tG(\mathfrak{Y}+ \mathfrak{T}_0;\Lambda)
= \sum_{\tP\in\ESP(\tL)}
e^{\langle \Lambda,Y_\tP +\brTo\tP \rangle}\epsilon_\tP^\tG(\Lambda) \bsd(\Lambda,\tP)\ptf\end{equation*}
Puisque $ \Y_\tP+ \brTo\tP$ est \'egal \`a l'image $T_{\smash{0,\tL}}$ de $T_0$ dans $\ag_\tL$, on a
\begin{align*}
\lefteqn{%
e^{\langle \Lambda, Y_\tP+\brTo\tP\rangle} \bsd(\Lambda,\tP)=
e^{\langle \Lambda,T_{\smash{0,\tL}}\rangle}
}\\&& 
\hspace{1cm}\times\spurs\Big(\bfD_\nu\mskip 2mu \ESM(S,\mu+\nu; \Lambda,\tP)
\bfM_{S\vert\tu S}(\mu+\nu +\Lambda)\Rho_{S,\sigma,\mu}(\tu,f,\omega)\Big)
\ptf
\end{align*}
Comme $T_{\smash{0,\tL}}$ appartient \`a $\ag_\tL^\tG$, si $\xi_{\tu,\nu}(\mu)\in \wh{\bsbbc}_\tG$ alors  
$e^{\langle \xi_{\tu,\nu}(\mu), T_{\smash{0,\tL}}\rangle}=1$ et, 
compte tenu de  de la d\'efinition \ref{defA},
\begin{equation*}\bs{\alpha}_{\tZ,\mu}=\mathrm{vol}\Big(\ESA_\tL^\tG\backslash \ag_\tL^\tG\Big)\mun
e^{\langle \xi_{\tu,\nu}(\mu),\tZ\rangle}\bsA^{\tG}_\tL(\sigma,\tu,\mu;\nu,\xi_{\tu,\nu}(\mu))
\ptf\end{equation*}  
\end{proof}

\begin{definition}\label{defJ}
Introduisons les notations suivantes: 
\begin{equation*}v_\tL = \mathrm{vol}\Big(\ESB_\tL^\tG\backslash \ag_\tL^\tG \Big)\vgq
V_M(\tu)\mskip 2mu \mskip 2mu  ={v_\tL^{-1}\vert\wh\bsbbc_\tL\vert \vert\wh\bsbbc_M\vert^{-1}}\vert \mathfrak{K}_{\tu}\vert\mun
\end{equation*}
et
\begin{equation*}\bmMtunuo =\{\mu\in\bmMtunu\mskip 2mu \mskip 2mu \mid \mskip 2mu \mskip 2mu \xi_{\tu,\nu}(\mu)=0\mskip 2mu \}\ptf\end{equation*}
L'ensemble $\bmMtunuo$ est stable par translations par $\bsmu_\tL$ et muni de la mesure $\bsmu_\tL$-invariante
induite par celle sur $\bsmu_\tL$. 
On pose
\begin{equation*}\bsJ_{ M}^\tG(\sigma,f,\omega,\tu;\nu)=V_M(\tu)\int_{\bmMtunuo}\bsA^{\tG}_\tL(\sigma,\tu,\mu;\nu)
\dd \mu\ptf\end{equation*}
avec
\begin{equation*}\bsA^{\tG}_\tL(\sigma,\tu,\mu;\nu)=\bsA^{\tG}_\tL(\sigma,\tu,\mu;\nu,0)\end{equation*}
c'est-\`a-dire
\begin{equation*}\bsA^{\tG}_\tL(\sigma,\tu,\mu;\nu)=\spurs\Big(\bfD_\nu\mskip 2mu \ESM_\tL^\tG(S,\mu+\nu)
\bfM_{S\vert\tu S}(\mu+\nu)\Rho_{S,\sigma,\mu}(\tu,f,\omega)\Big)\ptf\end{equation*}
Enfin on pose
\begin{equation*}\bsJ_{ M}^\tG(\sigma,f,\omega,\tu)=\sum_{\nu \in \bsESEsigma}
\bsJ_M^\tG(\sigma,f,\omega,\tu;\nu)\ptf\end{equation*}
\end{definition}

\begin{proposition} \label{expspec}
La fonction
\begin{equation*}T\mapsto h(T)=\bJres^{\tG,T}_{ M}(\sigma,f,\omega,\tu)\end{equation*}
est dans $\mathrm{PolExp}$ et pour tout r\'eseau $\ESR$ de $\ag_{0,\QM}$, on a l'\'egalit\'e
\begin{equation*}\lim_{k\rightarrow +\infty}p_{\ESR_k,0}(h,T_0)=\bsJ^\tG_{M}(\sigma,f,\omega,\tu)\ptf\end{equation*}
\end{proposition}

\begin{proof}
Rappelons que pour $\tZ\in \bsbbc_\tG$, la fonction  $T\mapsto a_{\wt{Z},\mu}(T)$ 
introduite avant le lemme \ref{limite} appartient \`a 
PolExp et que $\bs{\alpha}_{\wt{Z},\mu}$ en est le polyn™me limite en $T=T_0$. 
Comme l'expression $a_{\wt{Z},\mu}(T)$ est lisse en $\mu$ 
lorsque $\mu$ varie dans le compact $\bsmu_\tL$,
on en d\'eduit que la fonction 
\begin{equation*}h_{\tZ,\nu}:T\mapsto \int_{\bmMtunu}a_{\tZ,\mu}(T)\dd \mu\end{equation*}
appartient elle aussi \`a PolExp. 
Le contr™le du terme d'erreur dans le passage \`a la limite (cf. \ref{passlim}) montre qu'il 
est uniforme et on obtient que
\begin{equation*}\lim_{k\rightarrow +\infty}p_{\ESR_k,0}(h_{\tZ,\nu},T_0)=
\int_{\bmMtunu}\lim_{k\rightarrow +\infty}p_{\ESR_k,0} (\azx T_0)\dd\mu=
\int_{\bmMtunu}\bs{\alpha}_{\tZ,\mu}\dd\mu\ptf\end{equation*}
La somme sur les $\tZ\in \bsbbc_\tG$ de l'expression ci-dessus \'etant finie,
on peut permuter la somme et l'int\'egrale.  Mais
la somme sur $\tZ\in \bsbbc_\tG$ des $ e^{\langle \xi_{\tu,\nu}(\mu),\tZ \rangle}$
 vaut $\vert \wh{\bsbbc}_\tG\vert$ si $ \xi_{\tu,\nu}(\mu)= 0$, et $0$ sinon. 
 On a donc montr\'e que 
 \begin{equation*}\sum_{\tZ\in \bsbbc_\tG}\lim_{k\rightarrow +\infty}p_{\ESR_k,0}(h_{\tZ,\nu},T_0)\end{equation*}
 est \'egal \`a
 \begin{equation*}
 \vert \wh\bsbbc_\tG \vert 
\mathrm{vol}\Big(\ESA_\tL^\tG\backslash \ag_\tL^\tG \Big)^{-1}
 \int_{\bmMtunuo}\bsA^{\tG}_\tL(\sigma,\tu,\mu;\nu)
\dd \mu\ptf\end{equation*}
Compte tenu de la d\'efinition \ref{defJ}, on a
\begin{equation*}
\vert \wh\bsbbc_\tG \vert 
\mathrm{vol}\Big(\ESA_\tL^\tG\backslash \ag_\tL^\tG \Big)^{-1}= 
\vert \wh\bsbbc_\tL\vert \mathrm{vol}\Big(\ESB_\tL^\tG\backslash \ag_\tL^\tG \Big)^{-1}
 = \vert \wh\bsbbc_\tL\vert v_\tL^{-1}\ptf\end{equation*}
 D'o le lemme puisque d'apr\`es \ref{propcombfin}, on a 
 \begin{equation*}h(T)= \vert\wh\bsbbc_M\vert^{-1}\vert \mathfrak{K}_{\tu}\vert\mun 
 \sum_{\nu \in \bsESEsigma}\sum_{\wt{Z}\in\bsbbc_\tG} h_{\tZ,\nu}(T)\ptf\end{equation*}
\end{proof}
\section{D\'eveloppement spectral fin}\label{spefin}
Soient $\nu\in \bsmu_M$ et $\mu\in \bmMtunu$ tels que $\xi_{\tu,\nu}(\mu)=0$. 
Puisque $\nu=(\tu-1){\mu}$ on a
\begin{equation*}\bfD_\nu=\bfD_{-{\mu}} \bfD_{\tu{\mu}}\ptf\end{equation*}
D'apr\`es l'\'equation fonctionnelle \ref{eqf}, on a
\begin{equation*}\bfD_{\tu{\mu}}\mskip 2mu \ESM_\tL^\tG(S,\tu{\mu})\bfM_{S\vert\tu S}(\tu{\mu})= \ESM_\tL^\tG(S,0)
\bfM_{S\vert\tu S}(0)\bfD_{\tu{\mu}}\ptf\end{equation*}
On en d\'eduit que 
\begin{equation*}\bfD_\nu\mskip 2mu \ESM_\tL^\tG(S,\mu+\nu)\bfM_{S\vert\tu S}(\mu+\nu) 
\Rho_{S,\sigma,\mu}(\tu,f,\omega)\leqno{(1)}\end{equation*}
est \'egale \`a
\begin{equation*}\bfD_{-{\mu}}\ESM_\tL^\tG(S,0)\bfM_{S\vert\tu S}(0)\bfD_{\tu{\mu}}\mskip 2mu 
\Rho_{S,\sigma,\mu }(\tu,f,\omega)\ptf\end{equation*}
Puisque d'apr\`es \ref{romu} on a
\begin{equation*}\bfD_{\tu{\mu}}\mskip 2mu \Rho_{S,\sigma,{\mu}}(\tu,f,\omega)= 
\Rho_{S,\sigma{\star {\mu}},0}(\tu,f,\omega)\bfD_{\mu} \end{equation*}
l'expression (1) est encore \'egale \`a 
\begin{equation*}\bfD_{\mu}\mun \bsA(\sigma\star\mu,\tu,0)\mskip 2mu \mskip 2mu  \bfD_{\mu}\leqno\end{equation*}
o l'on a pos\'e
\begin{equation*}\bsA(\sigma,\tu,\mu)=\ESM_\tL^\tG(S,\mu)\bfM_{S\vert\tu S}(\mu)
\Rho_{S,\sigma,\mu}(\tu,f,\omega)\ptf\leqno\end{equation*}
En observant que
\begin{equation*}\spurs\mskip -2mu \Big(\bfD_{\mu}\mun \bsA(\sigma\star\mu,\tu,0)\mskip 2mu \mskip 2mu  \bfD_{\mu}\Big)=
\spur_{\sigma\star\mu}( \bsA(\sigma\star\mu,\tu,0))\end{equation*}
l'expression de $\bsJ(\sigma,f,\omega,\tu)$  donn\'ee dans \ref{defJ} peut donc encore s'\'ecrire
\begin{equation*}\bsJ^\tG_{ M}(\sigma,f,\omega,\tu)=V_M(\tu)\mskip 2mu \mskip 2mu  \sum_{\nu \in \bsESEsigma}
\int_{\bmMtunuo}\spur_{\sigma\star\mu}(\bsA(\sigma\star\mu,\tu,0))\dd \mu\ptf\end{equation*}

\begin{lemma}\label{translation}
Pour que l'expression $\spur_{\sigma\star\mu}(\bsA(\sigma\star{\mu},\tu,0))$ soit non nulle, 
il faut que
\begin{equation*}\tu(\omega \otimes (\sigma\star\mu))\simeq \sigma\star\mu\leqno(2)\end{equation*}
c'est-\`a-dire que $\sigma\star\mu$ se prolonge en une repr\'esentation de $(M\tu,\omega)$. 
En ce cas on a
\begin{equation*}\spur_{\sigma\star\mu}(\bsA(\sigma\star{\mu},\tu,0))=\trace(\bsA(\sigma\star{\mu},\tu,0))\ptf\end{equation*}
\end{lemma}

\begin{proof}
L'op\'erateur $\bsA(\sigma\star\mu,\tu,0)$ envoie l'espace $\Automd(\bsX_S,\sigma\star \mu)$ dans l'espace 
$\Automd(\bsX_S,\tu(\omega\otimes(\sigma\star \mu)))$ et pour que ces deux espaces 
soient d'intersection non nulle, il faut qu'ils soient \'egaux c'est-\`a-dire que la condition (2) soit v\'erifi\'ee.
\end{proof}

D'apr\`es \ref{translation}, la contribution de $\bsigma$ (l'orbite de $\sigma$ sous torsion par $\bsmu_M$)
est nulle sauf peut-tre si
\begin{equation*}\bsigma\in\bsPi_\disc(M;\tu,\omega)\subset\bsPi_\disc(M)\end{equation*}
o $\bsPi_\disc(M;\tu,\omega)$ est le sous-ensemble des $\bsigma$ ayant un repr\'esentant $\sigma$ 
v\'erifiant la condition
\begin{equation*}\tu(\omega \otimes \sigma) \simeq \sigma\ptf \leqno{(3)}\end{equation*}
On supposera d\'esormais que $\sigma$ est un tel repr\'esentant\footnote{Observons que 
l'existence d'un tel repr\'esentant exige que le 
caract\`ere $\omega$ soit trivial sur le groupe $A_{M\tu}(\adef)=A_\tL(\adef)$.
En g\'en\'eral il n'est pas possible de 
r\'ealiser la condition (3) pour tous les $\tu$ simultan\'ement.}.
Consid\'erons l'ensemble
\begin{equation*}\fetu\bydef\left\{\lambda\in\bsmu_M\mskip 2mu \mskip 2mu \mid \mskip 2mu \mskip 2mu \tu(\omega \otimes (\sigma\star\lambda))\simeq \sigma\star\lambda\right\}\end{equation*}
qui, compte tenu de (3), est \'egal \`a \begin{equation*}\left\{\lambda\in\bsmu_M\mskip 2mu \mskip 2mu \mid \mskip 2mu \mskip 2mu (\tu-1)\lambda\in\Stab\right\}\ptf\end{equation*}
On observe que $\Stab\subset\bsESEsigma$.
On a donc une injection
\begin{equation*}\fetu\to \bigcup_{\nu\in\bsESEsigma} \bmMtunuo\end{equation*}
On note $\Fetu$ l'ensemble fini quotient de $\fetu$ par l'action de $\bsmu_\tL$. 
En choisissant des repr\'esentants $\lambda$ dans $\fetu\subset\bsmu_M$ pour les \'el\'ements $\dot\lambda\in\Fetu$
 on d\'efinit une injection
\begin{equation*}\Fetu\times\bsmu_\tL\to\bigcup_{\nu\in\bsESEsigma}\bmMtunuo\quad
\hbox{par}\quad (\dot\lambda,\mu)\mapsto \lambda+\mu\ptf\end{equation*}
Dans la d\'efinition \ref{defJ} de $\bsJ(\sigma,f,\omega,\tu)$ 
et compte tenu de \ref{translation}
on peut donc remplacer 
la somme sur les $\nu\in\bsESEsigma$ pr\'ec\'ed\'ee de l'int\'egrale 
sur les ${\mu} \in \bmMtunuo$ 
par la somme sur les ${{\dot\lambda} \in \Fetu}$ pr\'ec\'ed\'ee de l'int\'egrale sur $\bsmu_\tL$:
\begin{equation*}\bsJ^\tG_{ M}(\sigma,f,\omega,\tu)=
V_M(\tu)\sum_{{\dot\lambda} \in \Fetu} \int_{\bsmu_\tL}
\trace(\bsA(\sigma\star \lambda,\tu,\mu))\dd \mu \ptf\leqno(4)\end{equation*}
L'expression  ne d\'epend pas du choix des rel\`evements puisque, 
d'apr\`es (1) 
et \ref{romu}, pour $\mu,\mskip 2mu \mu' \in \bsmu_\tL$ on a
\begin{equation*}\bsA(\sigma\star \mu',\tu,\mu)= \bfD_{\mu'} \bsA(\sigma,\tu, \mu+\mu')\bfD_{\mu'}^{-1}\end{equation*}
et donc
\begin{equation*}\trace(\bsA(\sigma\star \mu',\tu,\mu))= \mathrm{trace}(\bsA(\sigma,\tu, \mu+\mu') )\ptf\end{equation*}
Pour chaque classe $\bsigma$ on a choisi un repr\'esentant $\sigma$ v\'erifiant
la condition (3). 
D'apr\`es (4) et compte tenu de \ref{defA} on a 
\begin{equation*}\bsJ^\tG_{ M}(\sigma,f,\omega,\tu)
=V_M(\tu)\mskip -7mu \sum_{{\dot\lambda} \in \Fetu}
\int_{\bsmu_\tL}\mskip -7mu\trace\Big(\ESM_\tL^\tG(S,\mu)
\bfM_{S\vert\tu S}(\mu)\Rho_{S,\sigma\star \lambda,\mu}(\tu,f,\omega)
\Big)\mskip -4mu \dd \mu\ptf\end{equation*}
Cette expression ne d\'epend pas du choix de $S\in \ES{P}(M)$ ni du choix du repr\'esentant $\sigma$ de $\bsigma$.


On rappelle que $\bsPi_\disc(M;\tu,\omega)$ est l'ensemble des classes modulo torsion 
par les \'el\'ements de $\bsmu_M$ des (classes d'isomorphisme de) repr\'esentations $\sigma$ v\'erifiant la condition (1).
Notons $\bsPi_\disc(M\tu,\omega)$ l'ensemble des classes modulo torsion par les \'el\'ements de $\bsmu_\tL$
 des (classes d'isomorphisme de) repr\'esentations verifiant (1). On pose
\begin{equation*}\bsJ^\tG_{M,\tu}(f,\omega) = \sum_{\bsigma\in\bsPi_\disc(M;\tu,\omega)}\cMsig
\bsJ^\tG_{ M}(\sigma,f,\omega,\tu)\end{equation*}
et 
\begin{equation*}\wh{c}_\tL(\sigma)\bydef \frac{\vert \wh{\bsbbc}_\tL\vert}{ \vert \mathrm{Stab}_\tL(\sigma) \vert}\ptf\end{equation*}

 \begin{proposition}\label{Jtu}
 L'expression $\bsJ^\tG_{M,\tu}(f,\omega)$ se r\'ecrit
\begin{align*}
\lefteqn{\bsJ^\tG_{M,\tu}(f,\omega)=v_\tL^{-1} \sum_{\bsigma\in \bsPi_\disc(M\tu,\omega)}
 \wh{c}_\tL(\sigma) \vert\mskip -2mu \det (\tu -1\vert \ag_M^\tL)\vert^{-1}}\\
&& \hspace{3cm}\times \int_{\bsmu_\tL}\trace\Big(
\ESM_\tL^\tG(S,\mu)
\bfM_{S\vert\tu S}(0)\Rho_{S,\sigma,\mu}(\tu,f,\omega)\Big) \dd \mu\ptf
\end{align*}
 \end{proposition}
 \begin{proof}
 L'application naturelle
\begin{equation*}\bsPi_\disc(M\wt{u},\omega)\rightarrow \bsPi_\disc(M;\tu,\omega)\end{equation*}
est surjective. On peut donc remplacer, dans l'expression $\bsJ^\tG_{M,\tu}(f,\omega)$, la somme sur 
les $\bs{\sigma}\in \bs\Pi(M;\tu,\omega)$ suivie de la somme sur les $\lambda\in \Fetu$ par une simple somme sur les 
$\bs{\sigma}\in \bs\Pi(M\tu,\omega)$ multipli\'ee par le cardinal de l'image 
dans $\bsmu_M^\tL$ du stabilisateur $\mathrm{Stab}_M(\sigma)$. Cette image est le groupe quotient
\begin{equation*}\mathrm{Stab}_M(\sigma)/\mathrm{Stab}_\tL(\sigma) \quad \hbox{o} \quad 
\mathrm{Stab}_\tL(\sigma) \bydef \mathrm{Stab}_M(\sigma)\cap \bsmu_\tL\ptf\end{equation*}
Compte tenu de \ref {defJ} et de \ref{ku} on a 
\begin{equation*}V_M(\tu) \frac{\vert \mathrm{Stab}_M(\sigma)\vert }{\vert \mathrm{Stab}_\tL(\sigma)\vert } = 
v_\tL^{-1} \wh{c}_\tL(\sigma) \wh{c}_M(\sigma)^{-1} \vert\mskip -2mu \det (\tu -1\vert \ag_M^\tL)\vert^{-1}\ptf\end{equation*}
On observe enfin que pour $\mu\in\bsmu_\tL$ on a
\begin{equation*}\bfM_{S\vert\tu S}(\mu)=\bfM_{S\vert\tu S}(0)\ptf\end{equation*}
\end{proof}

Rappelons que $\bfW^\tG(\ag_M,\ag_M)$ est l'ensemble des restrictions \`a $\ag_M$ des \'el\'ements 
$\tu\in \wt\bfW=\bfW^\tG$ tels que $\tu(\ag_M)=\ag_M$. C'est aussi le quotient 
\begin{equation*}\bfW^\tG(M)=N^{\wt\bfW}(M)/\bfW^M\end{equation*}
de l'ensemble $N^{\wt\bfW}(M)$ des $\tu\in \wt\bfW$ tels que $\tu(M)=M$ 
par le groupe de Weyl de $M$: $\bfW^M=N^M(M_0)/M_0$.

\begin{proposition}\label{speclimK}
La valeur en $T_0$ du polyn™me limite introduite en \ref{grossb} v\'erifie l'identit\'e 
\begin{equation*}\bsJ^\tG_{\mathrm{spec}}(f,\omega)=\sum_{M\in\ESL^G/\bfW^G}\frac{1}{w^G(M)}
 \sum_{\tu\in \bfW^\tG(M)} \bsJ^\tG_{M,\tu}(f,\omega)\end{equation*}
\end{proposition}

\begin{proof} Par passage \`a la limite, \ref{polasymp} fournit l'identit\'e
\begin{equation*}\bsJ^\tG_\mathrm{spec}(f,\omega)=\sum_{M\in\ESL^G/\bfW^G}\frac{1}{w^G(M)}
\sum_{\bsigma\in\bsPi_\disc(M)}\cMsig
\sum_{\tu\in \bfW^\tG(M)}\bsJ^\tG_{ M}(\sigma,f,\omega,\tu)\end{equation*}
puis on invoque \ref{Jtu}.
\end{proof}
 
Pour $\tL \in \wt{\ESL}$ et $M\in \ESL^L$,
un \'el\'ement $\tu\in \bfW^\tL(M)$ est dit \textit{r\'egulier} si l'espace de Levi (dans $\tL$) 
qui lui est associ\'e est $\tL$ lui-mme. Ceci \'equivaut \`a demander que \begin{equation*}\det(\tu-1 \vert \ag_M^L)\neq 0\ptf\end{equation*}
On note $\bs{W}^\tL(M)_\mathrm{reg}$ l'ensemble des $\tu\in \bfW^\tL(M)$ qui sont r\'eguliers. 
Observons que pour $\tu \in \bs{W}^\tL(M)_\mathrm{reg}$, $\tL$ est le sous-espace de Levi (dans $\tG$) associ\'e \`a $\tu$ 
c'est-\`a-dire le plus petit sous-espace de Levi contenant $M\tu$; par cons\'equent d'apr\`es \ref{ku} on a aussi 
\begin{equation*}\det(\tu-1\vert \ag_M^\tL)\neq 0 \com{et} \ag_M^\tL = (\tu-1)\ag_M\end{equation*}o $\ag_\tL$ 
est le sous-espace de $\ag_M$ form\'e des points fixes sous $\tu$ et 
$\ag_M^\tL$ est l'orthogonal de $\ag_\tL$ dans $\ag_M$.

L'expression $J^\tG_{M,\tu}(f,\omega)$ a d\'ej\`a \'et\'e d\'efinie pour $M\in \ES{L}$ et $\tu\in \bfW^\tG(M)$. 
Pour $\tL\in \wt{\ESL}$, posons
\begin{equation*}\bsJ^\tG_\tL(f,\omega)= \sum_{M\in \ESL^L} \frac{\vert \bfW^M\vert }{\vert \bfW^L\vert}
\sum_{\tu\in \bfW^\tL(M)_\mathrm{reg}}\bsJ^\tG_{M,\tu}(f,\omega)\ptf\end{equation*}
Notons $\wt\bfW^L= N_L(\tM_0)/M_0$
le quotient du normalisateur de $\tM_0$ dans $L$ par $M_0$. Avec ces notations, on a la
forme finale du d\'eveloppement spectral:

\begin{theorem}
On a
\begin{equation*}\bsJ^\tG_{\mathrm{spec}}(f,\omega)= \sum_{\tL\in \wt{\ESL}}
\frac{\vert \wt\bfW^L \vert}{\vert \wt\bfW^G\vert} \bsJ^\tG_\tL(f,\omega)\ptf\end{equation*}
\end{theorem}

\begin{proof} 
Pour tout $\tu\in\bfW^\tG(M)$ il existe un $\tL$
tel que $\tu$ appartienne \`a $\bs{W}^\tL(M)_\mathrm{reg}$
et $\bsJ^\tG_{M,\tu}(f,\omega)$ est donn\'e par \ref{Jtu}.
On invoque enfin \ref{speclimK}.
\end{proof}

\section{Partie discr\`ete modulo le centre}

Pour $M\in \ESL$, on s'int\'eresse \`a la partie discr\`ete (modulo le centre) c'est-\`a-dire aux contributions $\bsJ^\tG_{M,\tu}(f,\omega)$ 
avec $\tL=\tG$. 

On pose
\begin{equation*}\bsJ^\tG_{M,\disc}(f,\omega)= \sum_{\tu\in \bfW^\tG(M)_\mathrm{reg}}\bsJ^\tG_{M,\tu}(f,\omega)\ptf\end{equation*}
Pour $\tu\in \bfW^\tG(M)_\mathrm{reg}$, on a
\begin{align*}
\lefteqn{\bsJ^\tG_{M,\tu}(f,\omega)= \sum_{\bsigma\in \bsPi_\disc(M\tu,\omega)}
\wh{c}_\tG(\sigma) \vert\mskip -2mu \det (\tu -1\vert \ag_M^\tG)\vert^{-1}
}\\&& \hspace{5cm}\times 
\int_{\bsmu_\tG}\trace\Big(\bfM_{S\vert\tu S}(0)\Rho_{S,\sigma,\mu }(\tu,f,\omega)\Big)\dd \mu\ptf
\end{align*}

On souhaite permuter la somme et l'int\'egrale dans l'expression $\bsJ^\tG_{M,\tu}(f,\omega)$ ci-dessus. 
On commence par regrouper les termes suivant le groupe $\Xi(\tG)$
des caract\`eres unitaires automorphes de $A_{\smash\tG}(\adef)$. On le munit 
d'une mesure de Haar suivant les conventions de \ref{convention-mesures};
elle v\'erifie $\mathrm{vol}(\wh{\ESB}_\tG)=1$. 
Pour $\xi \in \Xi(\tG)$, on note
\begin{equation*}\Pi_\disc(M\tu,\omega)_\xi\end{equation*}
le sous-ensemble de $ \Pi_\disc(M\tu,\omega)$ form\'e de des $\sigma$ tels que 
$\xi_\sigma\vert_{A_\tG(\adef)}= \xi$, et on pose
\begin{equation*}\trace\Big(\bfM_{S\vert\tu S}(0)\Rho_{S,\disc,\xi}(\tu,f,\omega)\Big)
=\mskip -2mu \mskip -2mu \sum_{\sigma \in \Pi_\disc(M\tu,\omega)_\xi} \mskip -5mu\mskip -5mu\trace\Big(
\bfM_{S\vert\tu S}(0)\Rho_{S,\sigma,0}(\tu,f,\omega)\Big)\mskip -2mu \ptf\end{equation*}
Ici encore, l'expression est ind\'ependante du choix de $S\in \ESP(M)$. 
Posons
\begin{equation*}\bsJ^\tG_{M,\tu}(f,\omega,\xi)= \vert \det (\tu -1; \ag_M^{{G}})\vert^{-1} 
 \trace\Big(\bfM_{S\vert\tu S}(0)\Rho_{S,\disc,\xi}(\tu,f,\omega)\Big) \ptf\end{equation*}
Puisque
\begin{equation*}\vert \det (\tu -1; \ag_M^\tG)\vert = j(\tG)\mskip 2mu  \vert \det (\tu -1; \ag_M^{G})\vert \end{equation*}
on obtient le

\begin{lemma}\label{JMdiscu1}
Pour $M\in \ESL$ et $\tu\in\bfW^\tG(M)_\mathrm{reg}$, on a
\begin{equation*}\bsJ^\tG_{M,\tu}(f,\omega)= j(\tG)\mun\mskip 2mu \int_{\Xi(\tG)}\bsJ^\tG_{M,\tu}(f,\omega,\xi)\dd\xi \ptf\end{equation*}
\end{lemma}

Rappelons qu'on a not\'e $\Xi(G,\theta,\omega)$ le sous-ensemble de $\Xi(\G)$ form\'e des caract\`eres $\xi$ tels que,
en notant $\omega_{A_G}$ la restriction de $\omega$ \`a $A_G(\adef)$, on ait
\begin{equation*}\xi\circ\theta=\omega_{A_G}\otimes\xi\ptf\end{equation*}
Si $\Xi(G,\theta,\omega)$ est non vide on le munit, comme 
en \ref{le cas compact}, de la mesure $\Xi(G)^\theta$-invariante d\'eduite de la mesure de Haar 
sur $\Xi(G)^\theta$ telle que $\mathrm{vol}(\wh{\ESB}_G^\theta)=1$. 

\begin{lemma}\label{JMdiscu2}
Pour $M\in \ESL$ et $\tu\in\bfW^\tG(M)_\mathrm{reg}$, on a
\begin{equation*}\bsJ^\tG_{M,\tu}(f,\omega)= \int_{\Xi(G,\theta,\omega)}\bsJ^\tG_{M,\tu}(f,\omega,\xi)\dd \xi \ptf\end{equation*}
\end{lemma}

\begin{proof}
Cela r\'esulte de \ref{JMdiscu1} et \ref{jacobien}.
\end{proof}
On peut reformuler ce qui pr\'ec\`ede en une
\begin{proposition}
La partie {\og discr\`ete modulo le centre \fg} de la formule des traces est donn\'ee par
\begin{equation*}\bsJ^\tG_\disc(f,\omega)= \sum_{M\in \ESL^G/\bfW^G} \frac{1}{w^G(M)}\bsJ^\tG_{M,\disc}(f,\omega)\end{equation*}
avec
\begin{equation*}\bsJ^\tG_{M,\disc}(f,\omega)=\sum_{\tu\in \bfW^\tG(M)_\mathrm{reg}} \int_{\Xi(G,\theta,\omega)}
\bsJ^\tG_{M,\tu}(f,\omega,\xi)\dd \xi\ptf\end{equation*}
Cela s'\'ecrit aussi
\begin{equation*}\bsJ^\tG_\disc(f,\omega)= 
\sum_{M\in \ESL^G} \frac{\vert \bfW^M \vert}{\vert \bfW^G \vert}J^\tG_{M,\disc}(f,\omega)\ptf \end{equation*}
\end{proposition}
 En particulier, pour $M=G$, on a
\begin{equation*}\bsJ_{G,\disc}^\tG(f,\omega)= \int_{\Xi(G,\theta,\omega)}
\trace\Big(\tbsrho(f,\omega)\vert L^2_\disc(\bsX_G)_\xi\Big) \dd \xi\end{equation*}
et on retrouve la formule de la proposition \ref{cascomp1} si $G_\mathrm{der}$ est anisotrope. On observera que 
$\bsJ^\tG_{\disc}(f,\omega)$ n'est autre que $\bsJ^\tG_\tG(f,\omega)$.

\vfill\eject
\appendix\bigskip
\chapter{Erratum pour \cite{LW}}\label{Erratum}

\renewcommand{\adef}{\AM_F}

\newtheorem{Eproposition}[section]{\textsc{Proposition}}
\newtheorem{Elemma}[section]{\textsc{Lemme}}
\newtheorem{Ecorollary}[section]{\textsc{Corollaire}}

Dans cette annexe, les notations sont celles de \cite{LW}. 

\newcor 
Il a \'et\'e observ\'e par P.-H. Chaudouard que la preuve de \cite[1.8.1]{LW} pages~22-23
n'est valable que si les  sous-espaces $\ag_Q^R$ et $\ag_S^R$ sont orthogonaux.
Un argument diff\'erent est donn\'e ci-dessous.
On commence par \'etablir un r\'esultat auxiliaire. 
Soit $V$ un espace vectoriel r\'eel euclidien de dimension finie
On note $(\mskip 2mu ,\mskip 2mu )$ le produit scalaire, $\lVert \mskip 2mu \rVert $ 
la norme et $V^*$ l'espace dual (que l'on peut identifier \`a $V$). 
Soit $\oD$ une base de $V^*$ (que l'on ne suppose pas n\'ecessairement obtuse) 
et soit $\oD^*$ la base de $V^*$ formŽe des $\xi_\alpha\in V^*$ tels que
$(\alpha,\xi_\beta)=\delta_{\alpha,\beta}$ o $\delta_{\alpha,\beta}$ est le symbole de Kronecker.
En d'autres termes $\Delta^*$ est la base duale de $\Delta$ et
  $\alpha\mapsto\xi_\alpha$ est la bijection naturelle entre $\oD$ et 
$\oD^*$.  
\begin{Elemma} \label{ocd}
Soit $X\in V$. On note $C(\oD,X)$ l'ensemble des
$H\in V$ tels que pour tout $\alpha \in \oD$ on ait 
\begin{equation*}\alpha(H)\ge0 \com{et}
\xi_\alpha(H-X) \le0\ptf\end{equation*}
Il existe une constante $c>0$ telle que
\begin{equation*}\mid \mid H\mid \mid \le c\mid \mid X\mid \mid \ptf\end{equation*}
\end{Elemma}

\begin{proof}Pour $H\in C(\oD,X)$ on a
\begin{equation*}( H,H )=\sum_{\alpha\in\oD} \alpha(H) \xi_\alpha(H)\le \sum_{\alpha\in\oD} \alpha(H) \xi_\alpha(X)\end{equation*}
ce qui implique l'existence d'une constante $c>0$ telle que
\begin{equation*}\mid \mid H\mid \mid ^2\le c\mid \mid H\mid \mid \mskip 2mu \mid \mid X\mid \mid 
\comm{et donc}\mid \mid H\mid \mid \le c\mid \mid X\mid \mid \ptf\end{equation*}
\end{proof}

\begin{Ecorollary}\label{cor}
L'ensemble $C(\oD,X)$ est compact et si $X=0$ alors $C(\oD,X)$ est r\'eduit \`a $\{0\}$.
\end{Ecorollary}

On consid\`ere maintenant trois sous-groupes paraboliques $P\subset Q\subset R$ 
dans $G$. Soit $V=\ag_P^R$. On dispose de la base $\oD_P^R$ de $V^*=\ag_P^{R,*}$.
On note encore $\alpha\mapsto\xi_\alpha$ la bijection naturelle entre $\oD_P^R$ et 
sa base duale\footnote{La base  $(\oD_P^R)^*$ formŽe des $\xi_\alpha$
n'est pas en g\'en\'eral identique ˆ la base $\hat\Delta_P^R$ formŽe des $\varpi_\alpha$. 
En effet les $\xi_\alpha$ et $\varpi_\alpha$ 
sont colin\'eaires avec coefficients de proportionnalit\'e positifs
mais non n\'ecessairement \'egaux \`a 1.}
D'apr\`es \cite[1.2.2]{LW}, il existe un sous-groupe parabolique $S$ tel 
que $\oD_P^R$ soit l'union disjointe  
de $\oD_P^Q$ et $\oD_P^S$. Soit $X\in\V$. On note $C(P,Q,R,X)$ l'ensemble des $H\in V$ tels que
\begin{equation*}\alpha(H)>0 \com{pour tout} \alpha\in\oD_P^Q\;,\quad
\alpha(H)\le0 \com{pour tout} \alpha\in\oD_P^S\end{equation*}
et
\begin{equation*}\xi_\alpha(H-X)>0 \com{pour tout} \alpha\in\oD_P^S\;, \quad
\xi_\alpha(H-X)\le0 \com{pour tout} \alpha\in\oD_P^Q\ptf\end{equation*}
Le lemme \cite[1.8.1]{LW} affirme que

\begin{Elemma} \label{cpqrx}
On a les assertions suivantes:
\begin{enumerate}[(i)]
\item L'ensemble $C(P,Q,R,X)$ est relativement compact
et est contenu dans une boule de rayon $c\mid \mid X\mid \mid $ pour une constante $c>0$. 
\item $C(P,Q,R,0)$ est vide sauf si $P=R$.
\item Si $X$ est dans l'adh\'erence la chambre de Weyl positive, 
i.e. si $\alpha(X)\ge0$ pour tout $\alpha\in\oD_P^R$, alors $C(P,Q,R,X)$ est vide si $Q\subsetneq R$.
\end{enumerate}
\end{Elemma}

\begin{proof}  
On observe que l'adh\'erence de $C(P,Q,R,X)$ est un ferm\'e dans
$C(\oD,X)\subset\V=\ag_P^R$ o $\oD$ est la base de $V$ form\'ee des $\alpha\in\oD_P^Q$
et des $-\alpha$ pour $\alpha\in\oD_P^S$.
L'assertion (i) r\'esulte de \ref{ocd}. Lorsque $X=0$ il r\'esulte de \ref{cor} que $C(P,Q,R,0)\subset\{0\}$ et
 si $P\ne R$ les in\'egalit\'es strictes excluent $H=0$. Pour \'etablir (iii) on observe que
\begin{equation*}\alpha(H)= \alpha(X)+
\sum_{\beta\in\oD_P^Q}(\alpha,\beta)\xi_\beta(H-X) 
+\sum_{\beta\in\oD_P^S}(\alpha,\beta)\xi_\beta(H-X)\ptf
\end{equation*}
Consid\'erons $H\in C(P,Q,R,X)$,
$\alpha\in\oD_P^S$ et $X$  dans l'adh\'erence la chambre de Weyl positive.
Comme $\oD_P^R$ est une base obtuse, on a
\begin{equation*}\alpha(H) \le0\;,\quad \alpha(X)\ge0\com{et}
\sum_{\beta\in\oD_P^Q}(\alpha,\beta)\xi_\beta(H-X)\ge0\end{equation*}
ce qui implique 
\begin{equation*}\sum_{\beta\in\oD_P^S}(\alpha,\beta)\langle \xi_\beta,H-X\rangle\le0\ptf\end{equation*}
Pour toute famille de scalaires positifs $\{c_\alpha\mskip 2mu \vert\mskip 2mu  \alpha \in \Delta_P^S\}$ on a donc
\begin{equation*}\sum_{\alpha\in\oD_P^S}(\eta,\alpha) \xi_\alpha(H-X)\le0\com{avec}
\eta=\sum_{\alpha\in\oD_P^S}c_\alpha\alpha\ptf\end{equation*}
Les restrictions $\overline\xi_\alpha$ des $\xi_\alpha$ au sous-espace
$\ag_P^S\subset \ag_P^R$ forment la base duale  de $\oD_P^S$ qui est aig\"ue puisque 
$\oD_P^S$ est obtuse \cite[1.2.5]{LW}. On  choisit
\begin{equation*}\eta=\sum_{\alpha\in\oD_P^S}\overline\xi_\alpha \com{c'est-\`a-dire}
c_\alpha=\sum_{\beta\in\oD_P^S}(\overline\xi_\alpha,\overline\xi_\beta)>0\ptf\end{equation*}
Dans ce cas
 $(\eta,\alpha)=1$ pour tout $\alpha\in\oD_P^S$, ce qui fournit l'in\'egalit\'e
 \begin{equation*}\sum_{\alpha\in\oD_P^S}\xi_\alpha(H-X)\le0\ptf\end{equation*}
Comme $\xi_\alpha(H-X)>0$ pour tout $\alpha\in\oD_P^S$, 
c'est une contradiction si $\oD_P^S$ est non vide.  
\end{proof}

\newcor Dans l'ŽnoncŽ du corollaire \cite[3.6.7]{LW} page 74, lire  $n, n'\in N_0(\adef)$.

\newcor Dans l'\'enonc\'e de \cite[4.2.2]{LW} page 86 appara"t un signe $(-1)^{a_Q-a_G}$ qui est erron\'e:
 il faut supprimer ce signe. L'erreur est engendr\'ee
 par une faute de frappe dans la preuve: dans la formule (1) en haut de la page 87 
 on doit lire $(-1)^{a_R-a_Q}$ au lieu de $(-1)^{a_R-a_G}$.

\newcor 
Il a \'et\'e observ\'e par Delorme que
dans la proposition \cite[4.3.3]{LW} il faut ajouter une condition sur $T$. Non seulement
$T$ doit tre {\og r\'egulier\fg} mais il doit tre {\og loin des murs\fg}.
De fait, il est suppos\'e implicitement que ${\mathbf d}_{P_0}(T)> 0$, mais il faut de plus fixer 
une constante $c$ (positive assez grande) et se limiter \`a des $T$ v\'erifiant \begin{equation*}\lVert  T \rVert  \le c\mskip 2mu  {\mathbf d}_{P_0} (T )\end{equation*}ou, ce qui 
revient au mme, imposer (pour une autre constante)
\begin{equation*}\langle\alpha,T \rangle \le c \langle \beta,T \rangle\end{equation*}
pour tous les $\alpha$ et $\beta \in \Delta$. %
L'argument qui suit corrigeant la preuve de [LW, 4.3.3] est emprunt\'e \`a un message de 
Waldspurger \`a Delorme.

\begin{proof}On reprend la preuve et les notations de la proposition \cite[4.3.2]{LW}. 
On doit majorer \begin{equation*}\lVert x\rVert ^B \mid A_{Q,R}(x)\mid \end{equation*}pour $Q \subsetneq R$. L'expression $A_{Q,R}(x)$ est donn\'ee 
par une somme sur $\xi\in Q(F)\backslash G(F)$, mais pour $x$ fix\'e il y a au plus un $\xi$ modulo $Q(F)$ tel que
\begin{equation*}F^Q_{P_0}(\xi x, T)\sigma_Q^R(\bfH_0(\xi x -T)=1\ptf\end{equation*}
Quitte \`a remplacer $x$ par $\xi x$, on peut supposer
\begin{equation*}F_{P_0} ^Q (x, T ) = 1\vg \leqno{(1)}\end{equation*}
\begin{equation*}\sigma_Q^R(\bfH_0(x) - T) = 1\ptf\leqno{(2)}\end{equation*}
D'apr\`es \cite[3.6.2]{LW}, on peut aussi supposer $x = n_1ak$ avec $n_1 \in N_Q(\adef)$, 
$a \in \AA^G_0$ et $k \in \Omega$ un compact 
(qui ne d\'epend que de du choix du sous-groupe parabolique $Q$). 
Pour les majorations, on peut aussi bien supposer $k = 1$. On impose aussi que $n_1$ est dans un compact, 
ce que l'on peut assurer en le multipliant \`a gauche par un \'el\'ement de $N_Q(F)$. 
D'apr\`es \cite[4.3.1]{LW}, pour tout entier $r \ge 1$, il existe un op\'erateur diff\'erentiel $X_{Q,R}$ de degr\'e $r$, qui ne
d\'epend que de $Q$, $R$, $r$, tel que
\begin{equation*}\mid A_{Q,R}(x)\mid  \le \sup_{n'_1\in N_Q(\adef)} \mid \bs{\rho}(\mathrm{Ad}(a\mun)X_{Q,R})\varphi(n'_1a)\mid \ptf\end{equation*}
L'op\'erateur $X_{Q,R}$ est de la forme $X_{Q,R}= \sum_m Z_m$ avec
\begin{equation*}\mathrm{Ad}(e^{H_1})Z_m = e^{\lambda_m (H_1)} Z_m \quad \hbox{pour} \quad H_1\in \ag_Q^R\end{equation*}
o $\lambda_m$ est une combinaison lin\'eaire \`a coefficients entiers $\geq r$ des racines dans $\Delta_Q^R$.
On en d\'eduit bien une majoration
\begin{equation*}\mid A_{Q,R}(x)\mid \ll \sup_{n'_1,m}e^{-r \sum_{\alpha\in \Delta_Q^R} \langle \alpha,\bfH_0(a)\rangle }
\mid \bs{\rho}(Z_m)\varphi(n'_1a)\mid \end{equation*}
o les $Z_m$ ne d\'ependent que de $Q$, $R$, $r$. 
D'o 
\begin{equation*}\qquad \mid x\mid ^B \mid A_{Q,R}(x)\mid  \ll e^{C(a)}\sup_{n'_1,m}\mid \bs{\rho}(Z_m)\varphi(n'_1a)\mid \mid x\mid ^{-A}\leqno(3)\end{equation*}
\pni
o on a pos\'e
\begin{equation*}C(a)=c(B+A) \lVert  H_0(a) \rVert  -r \sum_{\alpha\in \Delta_Q^R} \langle \alpha,H_0(a)\rangle\vg\end{equation*}
$c$ \'etant une constante telle que $\mid a\mid  \le e^{c\lVert  H_0(a)\rVert  }$. La proposition 4.3.3 r\'esulte de (3)
\`a condition de prouver que 
\begin{equation*}C(a) \le - A\cdot{\mathbf d}_{P_0}(T)\ptf\end{equation*}
Ecrivons
\begin{equation*}H_0(a)=H_Q +H_Q^R +H_R\end{equation*}
avec $H_Q \in \AA_Q$, $H_Q^R \in \AA^R_Q$ et $H_R \in \AA_R$. La condition (1) impose $\lVert  H_Q\rVert  \ll \lVert  T\rVert  $. La
condition (2) impose (d'apr\`es \cite[2.11.6]{LW}) que $\lVert  H_R - T_R\rVert  \ll \lVert  H_Q^R - T_Q^R\rVert  $, 
avec des d\'efinitions \'evidentes de $T_R$ et $T_Q^R$. D'o $\lVert  H_R\rVert  \ll \lVert  H_Q \rVert +\lVert  T \rVert $. Elle impose
aussi que $\langle\alpha,H_Q^R -T_Q \rangle>0$ pour tout $\alpha\in \Delta_Q^R $. On en d\'eduit
\begin{equation*}C(a) \le c_1(A+B)\lVert  H_Q^R \rVert  +c_2(A+B)\lVert  T \rVert  -\frac{r}{2}
\sum_{\alpha\in \Delta_Q^R} \langle \alpha, H_Q+T_Q \rangle\end{equation*}
o $c_1$, $c_2$ sont des constantes positives qui ne d\'ependent de rien. 
En choisissant $r \ge c_3(A+B)$ pour une constante $c_3$ convenable, on a
\begin{equation*}c_1(A+B)\lVert  H_Q \rVert  - \sum_{\alpha\in \Delta_Q^R} <\langle\alpha, H_Q \rangle\le 0\ptf\end{equation*}
Pour $\alpha \in \Delta^R_Q$, soit $\beta \in \Delta$ se projetant sur $\alpha$. On sait que $\alpha$ est la somme de 
$\beta$ et d'une combinaison lin\'eaire des $\gamma \in \Delta_Q$ 
\`a coefficients positifs ou nuls. Donc \begin{equation*}\langle \alpha,T_Q \rangle\ge \langle \beta,T \rangle\ge {\mathbf d}_{P_0}(T)\ptf\end{equation*}
 Alors \begin{equation*}C(a) \le c_2(A+B)\lVert  T\rVert  -\mid \Delta^R_Q\mid \frac{r}{2}{\mathbf d}_{P_0}(T)\end{equation*}
Si on impose, comme dit plus haut, une condition $\lVert T\rVert  \le c\mskip 2mu {\mathbf d}_{P_0}(T)$, alors pour 
$r \ge c4(2A+B)$, on a bien
\begin{equation*}C(a) \le - A\cdot{\mathbf d}_{P_0}(T)\ptf\end{equation*}
\end{proof} 

\newcor La formule pour le produit scalaire $\langle \Phi,\Psi \rangle_P$,
en bas de la page 93, donn\'ee par une int\'egrale sur 
${\mathbf X}_P=\AA_PP(F)N_P(\adef)\backslash G(\adef)$ n'a pas de sens si $P\neq G$. En effet le groupe 
$\AA_PP(F)N_P(\adef)$ n'est pas unimodulaire et il n'y a pas de mesure $\G(\adef)$-invariante \`a droite sur 
${\mathbf X}_{P}$ (cf. \ref{formesautom}). 
Il convient de poser
\begin{equation*}\langle \Phi,\Psi \rangle_P
= \int_{\mathbf K}\int_{{\mathbf X}_M} \Phi(mk) \overline{\Psi(mk)} \dd m \dd k\ptf\end{equation*}
En revanche sur ${\mathbf X}_{P,G}= \AA_G P(F)N_P(\adef)\backslash G(\adef)$ on dispose de la mesure quotient
et on a la formule d'int\'egration
\begin{equation*}\int_{{\mathbf X}_{P,G}}e^{\langle 2\rho_P,\bfH_P(x) \rangle} \phi(x)\dd x 
= \int_{\mathbf K}\int_{{\mathbf X}_{M,G}} \phi(mk) \dd m \dd k\end{equation*}
 qui est utilis\'ee dans la preuve du th\'eor\`eme \cite[5.4.2]{LW}. 

\newcor Page 104.
Dans la d\'efinition de l'op\'erateur $\bs\rho_{P,\sigma,\mu}(\delta,y,\omega)$
il manque le caract\`ere $\omega$ dans la d\'efinition de l'espace d'arriv\'ee.

\newcor Page 106, ligne 16 il faut lire
\begin{equation*}K_\tG(f,\omega;x,y)= \sum_i \int_{{\mathbf X}_G}K_\tG(g_i,\omega;x,z) 
K_G^*(\omega h_i;z,y)\dd z \ptf\end{equation*}

\newcor Page 111.
Dans l'\'enonc\'e de la proposition \cite[7.2.2]{LW} 
il faut prendre pour $K_1$ et $K_2$ des noyaux sur $\bsX_{\theta(P),G}\times\bsX_{P,G}$ et 
sur ${\mathbf X}_{P,G}\times{\mathbf X}_{P,G}$ et remplacer les int\'egrales sur ${\mathbf X}_P$ par des int\'egrales sur ${\mathbf X}_{P,G}$. 

\newcor Page 112. Il faut remplacer $\mathcal{B}^P(\sigma)$ par $\mathcal{B}^S(\sigma)$ dans 
la d\'efinition de $H_\sigma$: ici $M$ est un sous-groupe de Levi de $S$ avec $S\subset P$. 
Cette erreur due \`a un copier-coller intempestif s'est propag\'ee dans toute la suite de \cite{LW}.

\newcor Page 126, ligne 2. Il faut remplacer $\mathfrak{n}(\adef)$ par $\mathfrak{n}_P(\adef)$ o $\mathfrak{n}_P$ est l'alg\`ebre de Lie 
de $N=N_P$ (rappel: on a not\'e $\mathfrak{n}$ l'alg\`ebre de Lie de $N_{Q^+}$).

\newcor Page 128, ligne -15. La d\'efinition de $j_{\tP\mskip -2mu ,\oo}(x)$ n'a pas de sens. 
Il faut la remplacer par 
\begin{equation*}j_{\tP\mskip -2mu ,\oo}(x)= 
\sum_{\delta \in \oo \cap \tM_P(F)} \int_{N(\delta,\adef)}\omega(x)f^1(x^{-1}\delta n x) \dd n\ptf\end{equation*}
De mme, dans l'\'enonc\'e du lemme 9.2.2 page 129 il faut remplacer la premi\`ere expression par
\begin{equation*}\int_{N(F)\backslash N(\adef)} \sum_{\delta \in \oo\cap \tM_P(F)} 
\int_{N(\delta,\adef)} \phi(n^{-1} \delta n_1 n)\dd n_1 \dd n \ptf\end{equation*}

\newcor Page 128, ligne -13. Lire:
\textit{la partie quasi semi-simple $\delta_s$ de $\delta$.}

\newcor Page 129. Dans la preuve du lemme 9.2.2, il convient de dire (mme si c'est implicite) 
que l'application du lemme \cite[9.2.1]{LW} induit une application surjective sur les groupes des points ad\'eliques.

\newcor Page 130, ligne -10. Lire $s(\delta)$ au lieu de $\delta_s$.

\newcor Page 130, ligne -11. Lire $I_{s(\delta)}(F)$ au lieu de $I_{\delta_s}$.

\newcor Page 130, ligne -6. Pour une classe de conjugaison 
quasi semi-simple $\cc$ on doit pr\'eciser que $k_\cc^T$ est 
d\'efinie en remplaant, dans la d\'efinition de $k_{\oo}^T$ \cite[9.2]{LW}, la somme sur $P\supset P_0$
 par la somme sur les $P\supset P_0$ tels que $\tM_P$ 
contienne un conjugu\'e de $\tM_\delta$ et l'ensemble $\oo$ par l'orbite $\cc$.

\newcor Page 150. Dans la formule pour $\hat{\tau}_\tP(H-X)$ 
il faut sommer sur les $\tQ\in\wt\ESP_\st$ tels que $\tP\subset\tQ$, i.e. prendre $\tR=\tG$.

\newcor Page 168, ligne 2. Lire: \textit{et sera \`a l'\oe uvre...}

\newcor Page 182. Dans l'\'enonc\'e du point (ii) du lemme 13.1.1, il faut lire: \textit{Pour tout $\lambda \in \ag_0^*$, etc}.

\newcor Pages 183--186: l'indice $Q_0$ devrait tre en exposant dans les termes \`a droite des \'equations (2) 
page 183, (7) page 185 et (8) page 186 de [LW]. Elles correspondent ici aux \'equations (3) et (5) de la section 13.2.

\newcor Page 194, ligne 11, il faut lire: \textit{holomorphe en $\lambda$ et $\nu$}.

\newcor Page 197, lignes 2 et 6: les $+Y$ devraient tre des $-Y$.

\newcor Page 203. Dans l'in\'egalit\'e ligne~12 et dans celle ligne~-3, il faut remplacer $Y_{P'\mskip -2mu ,e}(T)$ par $-Y_{P'\mskip -2mu ,e}(T)$.

\newcor Page 203, ligne -9. L'in\'egalit\'e est stricte: il faut lire $\lVert  U+V \rVert  > \rho \lVert  T \rVert $.

\newcor Page 205, majoration (1): il faut lire $\vert A^T_{s,t} - E^T_1\vert$.
 
\newcor Page 207: $\wt{Y}_{S'}(X)$ est la projection de $u^{-1}X$ sur $t(\ag_S)^Q$.

\newcor Page 219, ligne -6: l'expression pour $\bfA_M^{Q_1}(T,\sigma,\tu,\mu)$ faisant intervenir une trace
n'a de sens que si l'op\'erateur est un endomorphisme de l'espace ${\mathcal A}({\mathbf X}_S,\sigma)$. 
La formulation correcte est celle du point ($i$) du lemme \cite[14.1.4]{LW}.


\renewcommand{\indexname}{Index des notations}
\printindex

\backmatter


\vfill\eject

\end{document}